\documentclass[letterpaper]{article}
\usepackage{url}
\usepackage{exscale,relsize}
\usepackage[dvips]{graphicx,color}

\usepackage{amsmath}
\usepackage{amssymb}
\usepackage{bm}
\usepackage[mathscr]{eucal}
\usepackage{amsfonts}
\usepackage{latexsym}
\usepackage{amsxtra}
\usepackage{amsbsy}
\usepackage{amsthm}
\usepackage{amscd}
\usepackage{amsopn}
\usepackage{amstext}
\usepackage{upref}
\newcommand{\me}{\mathrm{e}}
\newcommand{\mi}{\mathrm{i}}
\newcommand{\md}{\mathrm{d}}

\renewcommand{\Im}{\mathrm{Im}}
\renewcommand{\Re}{\mathrm{Re}}

\numberwithin{equation}{section}
\allowdisplaybreaks[4]

\let \Im \relax
\let \Re \relax
\DeclareMathOperator{\Im}{Im}
\DeclareMathOperator{\Re}{Re}

\theoremstyle{definition}
\newtheorem{remark}{Remark}[section]
\theoremstyle{plain}
\newtheorem{theorem}{Theorem}[section]
\newtheorem{definition}{Definition}[section]
\newtheorem{corollary}{Corollary}[section]
\newtheorem{conjecture}{Conjecture}[section]

\newtheorem{proposition}{Proposition}[section]
\newtheorem{lemma}{Lemma}[section]
\begin{document}
\fontsize{10pt}{11pt}\selectfont
\fontencoding{T1}\selectfont
\baselineskip=11pt
\frenchspacing
\title{Algebroid Solutions of the Degenerate Third Painlev\'e Equation for Vanishing Formal Monodromy Parameter}
\author{A.~V.~Kitaev\thanks{\texttt{E-mail: kitaev@pdmi.ras.ru}} \, and \,
A.~Vartanian \\
Steklov Mathematical Institute, Fontanka 27, St. Petersburg 191023, Russia}
\date{April 11, 2023}
\maketitle
\enlargethispage*{2.17\baselineskip}
\begin{abstract}
\noindent
Various properties of algebroid solutions of the degenerate third Painlev\'e equation,
\begin{equation*}
u^{\prime \prime}(\tau) \! = \! \frac{(u^{\prime}(\tau))^{2}}{u(\tau)} \! - \! \frac{u^{\prime}(\tau)}{\tau}
\! + \! \frac{1}{\tau} \! \left(-8 \varepsilon (u(\tau))^{2} \! + \! 2ab \right) \! + \! \frac{b^{2}}{u(\tau)},\qquad
\varepsilon=\pm1,\quad\varepsilon b>0,
\end{equation*}
for the monodromy parameter $a=0$ are studied. The paper contains connection results for asymptotics
as $\tau\to+0$ and as $\tau\to+\infty$ for $a\in\mathbb{C}$. Using these results, the simplest algebroid solution with
asymptotics $u(\tau)\to c\tau^{1/3}$ as $\tau\to0$, where $c\in\mathbb{C}\setminus\{0\}$, together with its
associated integral $\smallint_0^\tau {(u(t))^{-1}\,\md t}$, are considered in detail, and their basic asymptotic behaviours
are visualized.
\vspace{0.30cm}


20F55

\vspace{0.11cm}

\textbf{Abbreviated Title.} Algebroid Solutions of the Degenerate Third Painlev\'e Equation

\vspace{0.11cm}


equation
\end{abstract}
\vspace{-0.1cm}
\tableofcontents
\pagebreak
\section{Introduction} \label{sec:introduction}
We consider the degenerate third Painlev\'{e} equation~\cite{Kit87,KitVar2004} in the form
\begin{equation} \label{eq:dp3u}
u^{\prime \prime}(\tau) \! = \! \frac{(u^{\prime}(\tau))^{2}}{u(\tau)} \! - \! \frac{u^{\prime}(\tau)}{\tau}
\! + \! \frac{1}{\tau} \! \left(-8 \varepsilon (u(\tau))^{2} \! + \! 2ab \right) \! + \! \frac{b^{2}}{u(\tau)},
\end{equation}
where the prime denotes differentiation with respect to $\tau$, $\varepsilon=\pm1$, and
$b\in\mathbb{R}\setminus\lbrace 0\rbrace$ and $a\in\mathbb{C}$ are parameters. Equation~\eqref{eq:dp3u} is also
refered as the third Painlev\'e equation of $D_7$ type \cite{OKSO2006}.

Algebroid solutions of equation~\eqref{eq:dp3u} can be viewed as meromorphic solutions of the Painlev\'e-type equations
that are equivalent, in the sense of Ince's classification~\cite{Ince}, to equation~\eqref{eq:dp3u}; therefore, it is natural
to extend some results and ideas developed by one of the authors of this work in \cite{KitSIGMA2019} for the
study of the meromorphic solutions of \eqref{eq:dp3u} to this wider class of solutions. The algebroid solutions represent
an interesting class of solutions from the point of view of their asymptotics, because their large-$\tau$ asymptotic behaviour
can be explicitly expressed in terms of the initial values of the associated meromorphic functions.
Recall that, in a generic situation, such explicit formulae are not obtainable for any Painlev\'e equation.
At the same time, however, the behaviour of the algebroid  solutions at the point at infinity resembles the behaviour of
generic solutions; so, we take  this opportunity to ``visualize the asymptotics'', namely, we consider several examples of
initial values for the simplest algebroid solution and compare the graphs of the numerical solutions with their asymptotics.
This comparison elucidates many interesting features of the numeric-asymptotic correspondence. In view of the present
asymptotic study, we've included updated and reformulated connection results obtained in \cite{KitVar2004, KitVar2010}
for asymptotics of solutions of equation~\eqref{eq:dp3u} for generic $a\in\mathbb{C}$ in Appendices~\ref{app:asympt0} and
\ref{app:infty}. A detailed description of the contents of this paper is given below, after a brief account of the literature.


We now mention some works that are related to the topic of our study.
Gromak \cite{Gromak1979} proved that the general third Painlev\'e equation has algebraic solutions iff it reduces
(with, perhaps, the help of the transformation $u(\tau)\to1/u(\tau)$) to the degenerate case~\eqref{eq:dp3u} with
$\mi a=n\in\mathbb{Z}$: for each $n$, equation~\eqref{eq:dp3u} has exactly three solutions of the form $R(x)$, where $R$
is a rational function and $(2\varepsilon x)^3=b^2\tau$. From the functional point of view, we have one multi-valued function,
and the three solutions are obtained via a cyclic permutation of the sheets of the Riemann surface $(2\varepsilon x)^3=b^2\tau$.
The function $R(x)$ can be constructed via a successive application of the B\"acklund transformations to the three
different solutions $u(\tau)=\tfrac{\varepsilon}{2}\sqrt[3]{b^2\tau}$
of \eqref{eq:dp3u} for the simplest case $a=0$. Recently, Buckingham and Miller \cite{BuckinghamMiller2022} studied
a double-scaling limit of the algebraic solution as $n\to\infty$ and $\tau/n^{3/2}=\mathcal{O}(1)$.

Among other asymptotic results for equation~\eqref{eq:dp3u} that concern its general solutions, we mention the recent paper
by Shimomura~\cite{ShShBoutroux2022} on the elliptic asymptotic representation of the general solution of \eqref{eq:dp3u}
in terms of the Weierstrass $\wp$-function in cheese-like strip domains along generic directions in
$\mathbb{C}\setminus\left(\mathbb{R}\cup\mi\mathbb{R}\right)$. Another interesting paper by Gamayun, Iorgov,
and Lisovyy \cite{GILy2013} gives, in particular, a derivation of the asymptotic expansions via a proper double-scaling limit
{}from the sixth Painlev\'e equation to the degenerate third Painlev\'e equation,
with emphasis placed on the combinatorial properties of the coefficients of the asymptotic expansions.

In the last decade, an ever-increasing number of papers dedicated to the application of the degenerate third Painlev\'e equation
and its generalizations, e.g., the cylindrical reduction of the Toda system, to some models in applied and theoretical physics
and in geometry have appeared; see, for example,
\cite{BLM2020,ContattoDorigoni2015,Contatto2017,DunPla2009,Dun2012,GIL2015I,GIL2015,GIL2020,GIL2023,Hildebrand2022,
Suleimanov2017,TW1997,TW1998}.
The majority of these works refer to, or report, some novel results
in the asymptotic description of some special solutions appearing in particular applications. In this paper, we can not,
nor do we attempt to, give an overview of these works, as such a presentation would lead us too far astray from our goals.
As a matter of fact, it would be of considerable interest to prepare an account of these works in the form of a review
article dedicated to the multifarious manifestations of the degenerate third Painlev\'e equation \eqref{eq:dp3u}.
Hereafter, we discuss only those works that are of primary relevance for our current research.

The main illustrative object of study in this article is the holomorphic at $r=0$ function $H(r)$ which solves the
ordinary differential equation (ODE)
\begin{equation} \label{eq:hazzidakis}
H^{\prime \prime}(r) \! = \! \frac{(H^{\prime}(r))^{2}}{H(r)} \! - \! \frac{H^{\prime}(r)}{r} \! + \!
\frac{1}{r} \! \left((H(r))^{2} \! - \! \frac{1}{H(r)} \right),
\end{equation}
where the prime denotes differentiation with respect to $r$. This function, in the real case
$H(r)\in\mathbb{R}$ for $r\in\mathbb{R}$, was introduced by Bobenko and Eitner in \cite{BobEitLMN2000}
as the function defining the Blaschke metrics of the two-dimensional regular indefinite affine sphere in $\mathbb{R}^3$ with
two affine straight lines. They proved, in particular, that, for this special class of the affine spheres, the function $H(r)$
is a similarity solution of the general Tzitz\'{e}ica equation describing regular indefinite affine spheres in $\mathbb{R}^3$.
The authors of \cite{BobEitLMN2000} formulated a special Goursat boundary-value problem for the Tzitz\'eica
equation: the solution of this problem is a similarity function which solves equation~\eqref{eq:hazzidakis}.
Exploiting the unique solvability of the Goursat problem for
second-order hyperbolic partial differential equations, Bobenko and Eitner proved the existence and the uniqueness of the
smooth at $r=0$ real solution $H(r)$ of equation~\eqref{eq:hazzidakis}. Assuming, then, the existence of the smooth solution
$H(r)$, they deduced {}from equation~\eqref{eq:hazzidakis} that $H^{\prime\prime}(r)$  is also smooth, and, substituting
the expansion
\begin{equation} \label{eq:Hat0short}
H(r)\underset{r\to0}{=}H(0)+H^{\prime}(0)r+\mathcal{O}(r^{2}),
\end{equation}
where $\mathcal{O}(r^{2})$ is a smooth function, into equation~\eqref{eq:hazzidakis}, proved that
\begin{equation} \label{eq:Hprime0-H0}
H^{\prime}(0)= (H(0))^{2}-\frac{1}{H(0)}.
\end{equation}
Bobenko and Eitner also showed that the Painlev\'e property of this equation allows one to make several useful, for
the geometry of the affine sphere, conclusions about the qualitative behaviour of this solution; for example,
if $H(0)>0$, then the solution has neither poles nor zeros on the negative-$r$ semi-axis, and, for $H(0)<0$, the smooth
solution is growing monotonically from some---largest---pole on the negative-$r$ semi-axis to the first zero on the
positive-$r$ semi-axis.

We now commence with the detailed discussion of the contents of this work.

Section~\ref{sec:2} consists of two subsections.
In Subsection~\ref{subsec:existence}, we prove that, for $H(0)\in\mathbb{C}\setminus\{0\}$, there exists a unique solution of
equation~\eqref{eq:hazzidakis} which is holomorphic at $r=0$: we use the straightforward method that is based on the proof of
the convergence of a formal power series. This choice for the method of the proof is adopted because, in the following
Subsection~\ref{subsec:H0coeffs} and in Section~\ref{sec:asymptnumerics}, we use the recurrence relation that is analysed
in Subsection~\ref{subsec:existence}. The main goal that we pursue in Subsection~\ref{subsec:H0coeffs} is to study the
coefficients of the Taylor-series expansion of $H(r)$ introduced in Subsection~\ref{subsec:existence}. Actually, our goal
in this respect is two-fold: (i) to formulate some number-theoretic properties of these coefficients; and (ii) to give an
effective tool for their calculation.
After some preliminary results and experimentation with \textsc{Maple}, we were able to formulate a conjecture regarding
the \emph{content} of the polynomials in $a_0:=-H(0)$ defining the Taylor-series coefficients. A substantial part of
Subsection~\ref{subsec:H0coeffs} is devoted to the technique of generating functions for the calculation of the coefficients.
This technique was suggested in \cite{KitSIGMA2019} for the study of the Taylor coefficients of holomorphic solutions of
equation~\eqref{eq:dp3u} in the case where these coefficients are rational functions of the parameter
$a\in\mathbb{C}\setminus\{0\}$. Equation~\eqref{eq:hazzidakis} is related (see the discussion
below) to equation~\eqref{eq:dp3u}, but with $a=0$, and our coefficients are functions of the parameter $H(0)$. Nevertheless,
we show that the technique of generating functions is also applicable in this situation; moreover, in Appendix~\ref{app:g2},
we present a stratagem that actually helps in the situation where straightforward calculations lead to cumbersome formulae.

Via the change of variables
\begin{equation} \label{eq:hazzidakis-dP3y}
y(t) \! = \! t^{1/3}H(r),\qquad
r=\left(\frac34\right)^2t^{4/3},
\end{equation}
one shows that equation~\eqref{eq:hazzidakis} transforms into the canonical form of the third Painlev\'e equation \cite{Ince}
with the coefficients $(1,0,0,-1)$,
\begin{equation} \label{eq:dP3y}
y^{\prime \prime}(t) \! = \! \frac{(y^{\prime}(t))^{2}}{y(t)} \! - \! \dfrac{y^{\prime}(t)}{t} \! + \!
\frac{(y(t))^{2}}{t} \! - \! \dfrac{1}{y(t)},
\end{equation}
where the prime denotes differentiation with respect to $t$. Equation~~\eqref{eq:dp3u} can also be identified as a special case
of the canonical form of the third Painlev\'e equation with the following set of coefficients, $(-8\varepsilon,2ab,0,b^2)$.
One can, of course, identify equations~\eqref{eq:dp3u} and \eqref{eq:dP3y} by setting $\varepsilon=-1/8$, $b=\pm\mi$, and $a=0$;
however, since we are planning on using the asymptotic results obtained in \cite{KitVar2004,KitVar2010}, where it is assumed
that $\varepsilon=\pm1$ and $\varepsilon b>0$, we identify these equations by choosing, for the coefficients in
equation~\eqref{eq:dp3u}, the values
\begin{equation}\label{eq:a-b-epsilon-conditions}
\varepsilon \! = \! b \! = \! +1
\qquad \text{and}\qquad
a=0,
\end{equation}
and making the following change of variables
\begin{equation}\label{eq:u-y-transformation}
\tau=2^{-3/2} \me^{3\pi\mi/4}t \qquad\text{and}\qquad
u(\tau)=-2^{-3/2}\me^{-3\pi\mi/4}y(t).
\end{equation}

For future reference, we rewrite the expansion~\eqref{eq:Hat0short} in terms of the functions $y(t)$ and $u(\tau)$;
the expansion for $y(t)$ follows immediately from equations~\eqref{eq:Hat0short} and \eqref{eq:hazzidakis-dP3y}:
\begin{equation}\label{eq:y-Hat0}
y(t)\underset{t\to0}{=} t^{1/3}\!\left(H(0)+H^{\prime}(0)\!\left(3/4\right)^2t^{4/3}+\mathcal{O}\big(t^{8/3}\big)\right),
\end{equation}
where the branches of $t^{1/3}$ and $t^{4/3}$ are defined to be positive for $t>0$.
Now, equations~\eqref{eq:u-y-transformation} imply that
\begin{equation}\label{eq:u(tau)asympt0}
u(\tau)\underset{\tau\to+0}{=}\frac{1}{2}\tau^{1/3} \! \left(H(0)-H^{\prime}(0)(3/2)^{2}\tau^{4/3}
+\mathcal{O}(\tau^{8/3}) \right),
\end{equation}
where the branches of $\tau^{1/3}$ and $\tau^{4/3}$ are defined analogously as for the powers of $t$ above, and
the coefficient values~\eqref{eq:a-b-epsilon-conditions} are assumed.

Section~3 begins with the general description of the algebroid solutions of equation~\eqref{eq:dP3y}. This consideration
is  based on the fact that equation~\eqref{eq:dP3y} possesses the Painlev\'e property, and its solutons have only one branching
point at the origin $t=0$; therefore, a solution is algebroid when the exponent defining its behaviour
at the origin ($\tau=t=r=0$) is a rational number; moreover, there are no logarithmic terms in the complete asymptotic
expansion as $t\to0$ of this solution. After that, we define two particular series (infinite sequences) of algebroid solutions
and show how to match the classical definition of the algebroid function (see, for example, \cite{Steinmetz2017}) as the solution
of an algebraic equation with meromorphic coefficients. This approach is based on the study of the expansion of the solution as
$t\to0$. We derive a structure for the coefficients of this expansion as functions of the initial value $a_0$. This analysis is
similar to the corresponding part of Subsection~\ref{subsec:H0coeffs} for the study of the function $H(r)$; however,
the combinatorics of the coefficients proves to be more interesting. We then deviate from the course of study of
Subsection~\ref{subsec:H0coeffs}, and, instead of using generating functions for the coefficients, define, for each
algebroid solution,  with the help of its small-$t$ expansion, a set of meromorphic functions that are holomorphic at the origin.
In terms of these meromorphic functions, we present an explicit construction for the algebraic equations of the algebroid
solutions of the aforementioned series. This construction leads to the study of some interesting functional determinants.
Finally, we derive systems of second-order differential equations which allow one to determine the meromorphic functions
used for the construction of the algebraic equations mentioned above as the unique solutions of these systems that are
holomorphic at the origin. We expect that the methodology expounded in Section~3 will work for the case of generic algebroid
solutions of equation~\eqref{eq:dp3u}; technically, however, the explicit construction may prove to be unwieldy.

In Section~\ref{sec:mondata}, we return to the study of the function $H(r)$. We recall the definition of the monodromy
manifold defined in our work~\cite{KitVar2004} corresponding to the function $u(\tau)$ that solves \eqref{eq:dp3u} for
the coefficient values \eqref{eq:a-b-epsilon-conditions}. This manifold uniquely describes the pair of functions $u(\tau)$
and $\me^{\mi\varphi(\tau)}$, where the function $\varphi(\tau)$ is an indefinite integral, that is,
$\varphi'(\tau)=2a/\tau+b/u(\tau)$.
We see that the function $v(\tau):=\me^{\mi\varphi(\tau)}$ depends on an additional multiplicative constant of
integration compared to the function $u(\tau)$; moreover, as long as the function $v(\tau)$ is known, the function $u(\tau)$
can be readily obtained via differentiation. One can actually parametrize our system of isomonodromy deformations
\cite{KitVar2004} in terms of the function $v(\tau)$ which solves, in its own right, a third-order Painlev\'e equation that
can be derived {}from equation \eqref{eq:dp3u} via the substitution $u(\tau)=-b\left(2a/\tau+\mi v'(\tau)/v(\tau)\right)^{-1}$.
Our goal, however, is to compare the asymptotic results of the papers \cite{Kit87} and  \cite{KitVar2004}. Towards this end,
we eliminate the multiplicative constant from the monodromy data of the monodromy manifold of \cite{KitVar2004} to
arrive at the so-called contracted monodromy manifold, and show that the contracted manifold is equivalent to the
monodromy manifold considered in the paper~\cite{Kit87}.
The main purpose of Section~\ref{sec:mondata} is to explain how one can use the results of \cite{KitVar2004} in order
to calculate the monodromy data corresponding to the solution~\eqref{eq:u(tau)asympt0} associated with $H(r)$ in terms of
the initial value $H(0)$. For the reader's convenience, some basic results from \cite{KitVar2004} that are necessary for
understanding the material of this section are formulated in Appendix~\ref{app:asympt0}. These results concern the asymptotic
behaviour as $\tau\to0$ of the function $u(\tau)$ and of its indefinite integral, $\varphi(\tau)$, related to
equation~\eqref{eq:dp3u} for generic parameter $a\in\mathbb{C}$, with $|\mathrm{Im}\,a|<1$. As discussed above, the function
$\varphi(\tau)$ has a closely-knit relationship to the system of isomonodromy deformations studied in
\cite{KitVar2004,KitVar2010}, so that it is very helpful for the study of asymptotics of some definite integrals related to
$u(\tau)$.
Compared to \cite{KitVar2004}, we have, in Appendix~\ref{app:asympt0}, simplified the notation and some formulae, and
have presented explicit asymptotics for the function $\varphi(\tau)$. It is our expectation that the detailed derivation
presented in Section~\ref{sec:mondata}, in conjunction with the improved presentation for the asymptotic results given in
Appendix~\ref{app:asympt0}, will be of benefit to those readers for whom the derivation of analogous parametrizations for other
types of solutions of \eqref{eq:dp3u} is required.

In Section~\ref{sec:Coxeter}, we give a group-theoretical characterization of the algebroid solutions of \eqref{eq:dp3u}
for $a=0$. The contracted monodromy manifold ``enumerating'' the solutions of \eqref{eq:dp3u} for generic values of
$a$ is a cubic surface in $\mathbb{C}^3$. The projectivization of this surface is a singular cubic surface in $\mathbb{CP}^3$
with a singularity of type $A_3$. With the help of a rational parametrization for this cubic surface, we derive its group of
automorphisms, $G$. A formula for one of the generators of this group is not properly defined, thus we consider
its regularization: this regularization is obtained for the restriction of the group $G$, denoted as $G(s)$,
where $s$ is a Stokes multiplier, acting on a disjoint sum of two conics. These disjoint sums of two conics do not intersect
for different values of $s$, so that continuing the action of $G(s)$ as the identity transformation on the complement
of the contracted monodromy manifold to the conics on which $G(s)$ is acting non-trivially, we can present $G$ as an infinite
direct sum of groups $G(s)$ with $s\in\mathbb{C}\setminus\{-1,3\}$. The group $G(s)$ is isomorphic to a Coxeter group of the
type
$$\bullet\underset{4}{\tfrac{\phantom{xxxx}}{\phantom{xxxx}}}\bullet\underset{m}{\tfrac{\phantom{xxxx}}
{\phantom{xxxx}}}\bullet$$
which has a normal subgroup isomorphic to the dihedral group $Dih_m$, $m\in\mathbb{N}\cup\{\infty\}$.
The solution is algebroid iff the corresponding monodromy data belongs to the conics where $m$ is finite.
This condition is equivalent to the statement that the corresponding Stokes multiplier $s$ is a real algebraic number
that solves one of the polynomial equations $q_m(s)$, $m\in\mathbb{N}$, defined recursively in Section~\ref{sec:Coxeter}
with the help of the Chebyshev polynomials.
The polynomials $q_m(s)$ are known in the mathematical literature as representing the ``trigonometric'' algebraic
numbers; in particular, their Galois group is solvable, so that their roots which coincide with the Stokes multipliers of
the algebroid solutions can be presented in terms of radicals.

Section~\ref{sec:asymptnumerics} is devoted to the visualization of the large-$r$ asymptotics on the negative-$r$ semi-axis.
In this section, we compare the numerical plots of the functions $H(r)$ and
$I(r):=\smallint_r^0\tfrac{1}{\sqrt{-r}H(r)}\,\md r=(\varphi(\tau)-\varphi(0))\left.\right\vert_{\tau^{2/3}=\frac23\sqrt{-r}}$
with the plots of their asymptotics, where the function $\varphi(\tau)$ is addressed above. For the convenience of the reader,
we present a summary of our previous results \cite{KitVar2004,KitVar2010} on the small- and large-$\tau$ asymptotics for
solutions of equation~\eqref{eq:dp3u} for generic $a\in\mathbb{C}$ in Appendices~\ref{app:asympt0} and \ref{app:infty},
respectively; furthermore, in these appendices, the reader will also find several new results for asymptotics of the
function $\varphi(\tau)$. In our previous papers~\cite{KitVar2010,KitSIGMA2019,KitVar2019}, we found and corrected some
mistakes in \cite{KitVar2004}; therefore, we have amalgamated the corrected results in Appendices~\ref{app:asympt0}
and \ref{app:infty}. In addition, in these appendices, we improve the notation and simplify some of the formulae:
all these changes are indicated therein.

By the locution ``visualization of asymptotics'' we mean the visual comparison of the plots of the numerical solutions with
their asymptotics. The primary goal of this comparison is three-fold: (i) for different initial values $H(0)$, to observe
the behaviour of the functions $H(r)$ and $I(r)$ at ``finite'' distances; (ii) to verify the correctness of
the asymptotic formulae; and (iii) to understand at what rate these functions achieve their asymptotic behaviour.
These goals are fundamental {}from the point of view of applications of these functions; however, it is by no means a trivial
matter to put these concepts on rigorous mathematical footing.

Section~\ref{sec:asymptnumerics} consists of six examples that were deliberately chosen in order to exhibit the dependence of
the functions $H(r)$ and $I(r)$ on the initial value $H(0)$ on the negative-$r$ semi-axis, together with the features of their
asymptotic approximations. These examples do not represent the complete list of known solutions: the solutions that are
not mentioned here represent some special classes of solutions (in our case, solutions that depend on one real parameter) with
some specific asymptotic behaviours; for example, solutions which are singular on the negative-$r$ semi-axis, or so-called
truncated solutions~\cite{av3}.

There is an ancient proverb which states: ``One look is worth a thousand words in a book''. In the context of our studies,
it can be rephrased as the locution ``visualization of asymptotics''. A ``present-day look'', however, is
not possible without the help of computer simulations, where the presentation of the results undergo a correction by the
corresponding computer programs; thus, it is important to discuss some features of this correction. These features meddle
with fact that equation~\eqref{eq:hazzidakis} has a singularity at $r=0$, where the initial data is specified.

At first glance, everything appears to be simple: if one observes that the plots of a function and its asymptotics approach
one another on some segment of reasonable length, then, one concludes that the asymptotics is correct..., or, most likely
correct... A subtle point here, of course, is the notion of ``reasonable length'', which is not that apparent. Since we are
dealing with asymptotics, we have to verify them over relatively large distances, because, in certain situations, even though
the numerical solution and its proposed asymptotics are close to one another over short distances, they may diverge over
longer ones. This may occur, for example, if there is a minor mistake in the asymptotic formula: the behaviour of the solution
has not yet stabilized and, at this stage, is partially compensated by a mistake, and some residual discrepancies between the
plots can be explained by the fact that the asymptotics is not supposed to coincide exactly with the solution. In order to
exclude the possibility of a mistake in the asymptotics, we have to increase the length of the interval of comparison. Doing so,
however, may compromise the accuracy of the numerical solution, so that the asymptotics is correct, but the discrepancy
between the two starts to grow. Both of these problems can be rectified, but as a result, one will have to compare the plots
over relatively large intervals. Then, in order to fit into a standard page, the plots are appropriately scaled by a computer
programme. As a result, some of the features of the plots may be lost, as, say, in our case, where in some of the figures
presented in Section~\ref{sec:asymptnumerics} (see, also, Figs.~\ref{fig:H0=-02+i0045+ReHintro} and
\ref{fig:H0=-02+i0045+ImHintro} below) the reader will see sharp peaks and icicles, even though all functions are,
\emph{de facto}, smooth,  and, at their extremal points, the derivatives vanish, although ``a look'' shows that they are
close to infinity. Another problem related to the scaling is the length of the first two peaks/icicles, which are the
narrowest ones. The plot is built on the basis of a finite number of points of which very few land inside the narrow
peaks/icicles. Actually, when the number of plot points are not sufficient, these peaks/icicles resemble a fence constructed
from sticks of random lengths. Looking at such a plot, one can conclude that this situation occurs because the poles near the
negative-$r$ semi-axis are located at random distances, which, however, is not the case!
Of course, the behaviour of solutions at distances relatively close to the origin is not a particularly important problem from
the point of view of asymptotics, because we want to see that everything is correct over distances longer than that of the
location of the first two peaks/icicles; nevertheless, in many cases, the asymptotics resembles the behaviour of the numerical
solutions even over these very small distances, which is related with the problem of how quickly the solutions attain their
asymptotic behaviour; therefore, it is intriguing to observe the discrepancy between the corresponding plots even over such
small distances. The true length of the first peaks/icicles can be determined by considering close-up pictures of these
peaks/icicles,
and in the most complicated cases by making numerical calculations with smaller step sizes. Then, one varies the number of
points that are calculated for the generation of the plots in order to find a better correspondence for the lengths of the
peaks/icicles. For the examples presented in Section~\ref{sec:asymptnumerics}, the discrepancy between the lengths of
the first two peaks/icicles on both the scaled and close-up figures is about $5$ to $6$ percent, or less, whereas  for the
subsequent peaks/icicles this difference is not observable.
We present such close-up pictures in some of the examples so that the reader can compare the real lengths of the peaks/icicles
with the ones on the scaled figures. Such close-up pictures also show the reader the quality of the approximations
of the numerical solutions by their large-$r$ asymptotics even over small distances: on the non-close-up figures, such
approximations look better as a result of scaling.

One may ask: whence a notion of approximating a function at finite or even small distances by its large-$r$ asymptotics
emanate? Well, ``finite'' and ``small'' are not well-defined notions; after all, on the practical level, if one wants to
compare a function to its asymptotics, then one has to start this comparison from some ``finite'' value of the argument!
Recalling the aforementioned proverb, the reader may ask: where does one have to cast ``a look'' to be sure that the
asymptotics correctly approximates the function?
Another point is that all mathematical models of practical value are applicable in some bounded domains
(times, distances, etc.); therefore, it is imperative to know whether or not our asymptotics are actually applicable in,
or far away from, the domains of interest.

The gist of the discussion in the previous paragraph can be visualized with the help of
Figs.~\ref{fig:H0=-02+i0045+ReHintro} and \ref{fig:H0=-02+i0045+ImHintro} below.
\begin{figure}[htpb]
\begin{center}
\includegraphics[height=50mm,width=100mm]{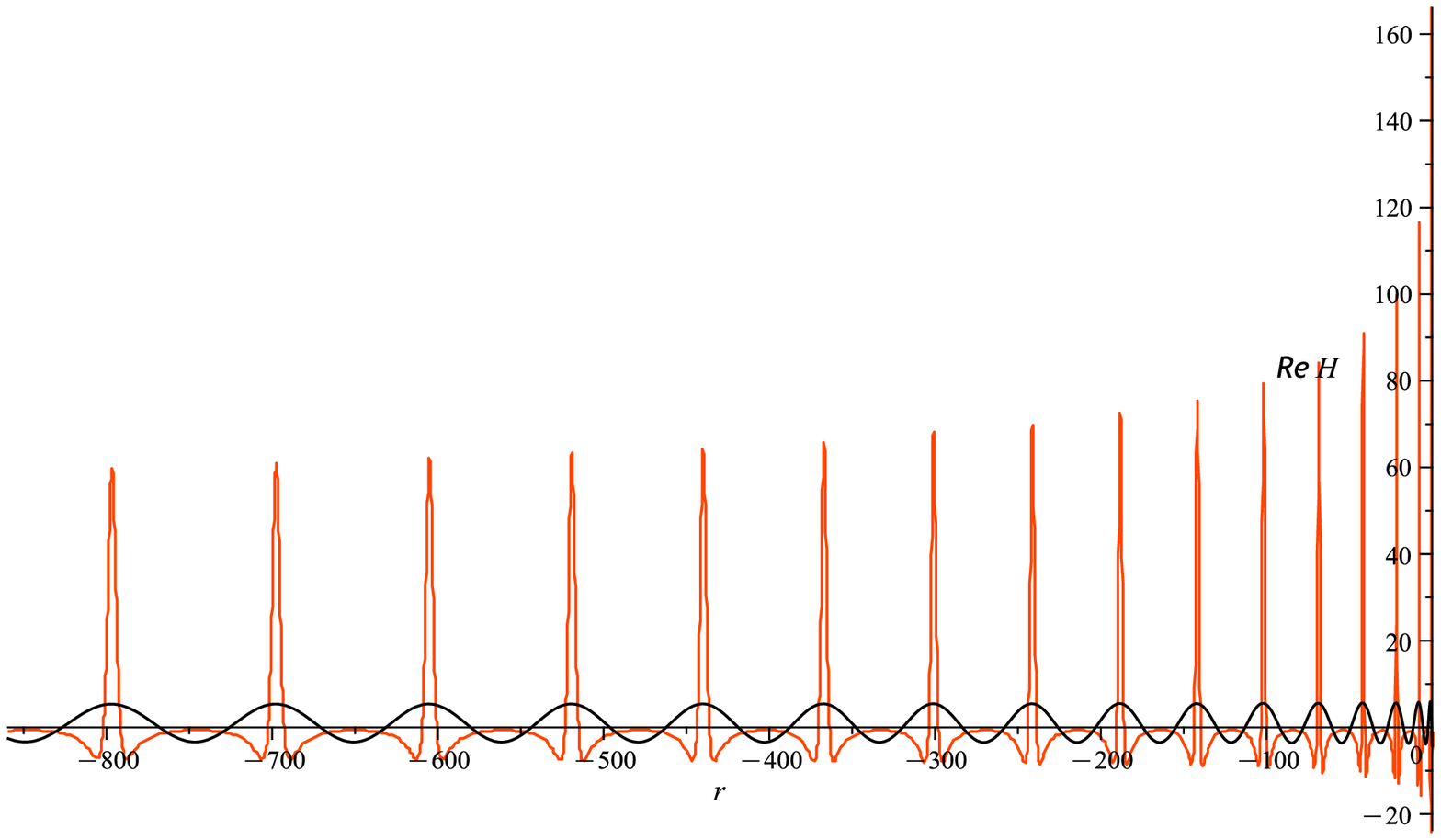}
\caption{The red and black plots are, respectively, the real parts of the numeric and large-$r$ asymptotic
(cf. equation~\eqref{eq:H-asympt-Large-regular}) values of the function $H(r)$ for $r\leqslant-0.1$ corresponding
to the initial value $H(0)=-0.2+\mi0.045$ (see Subsection~\ref{subsec:example6}).}
\label{fig:H0=-02+i0045+ReHintro}
\end{center}
\end{figure}

\begin{figure}[htpb]
\begin{center}
\includegraphics[height=50mm,width=100mm]{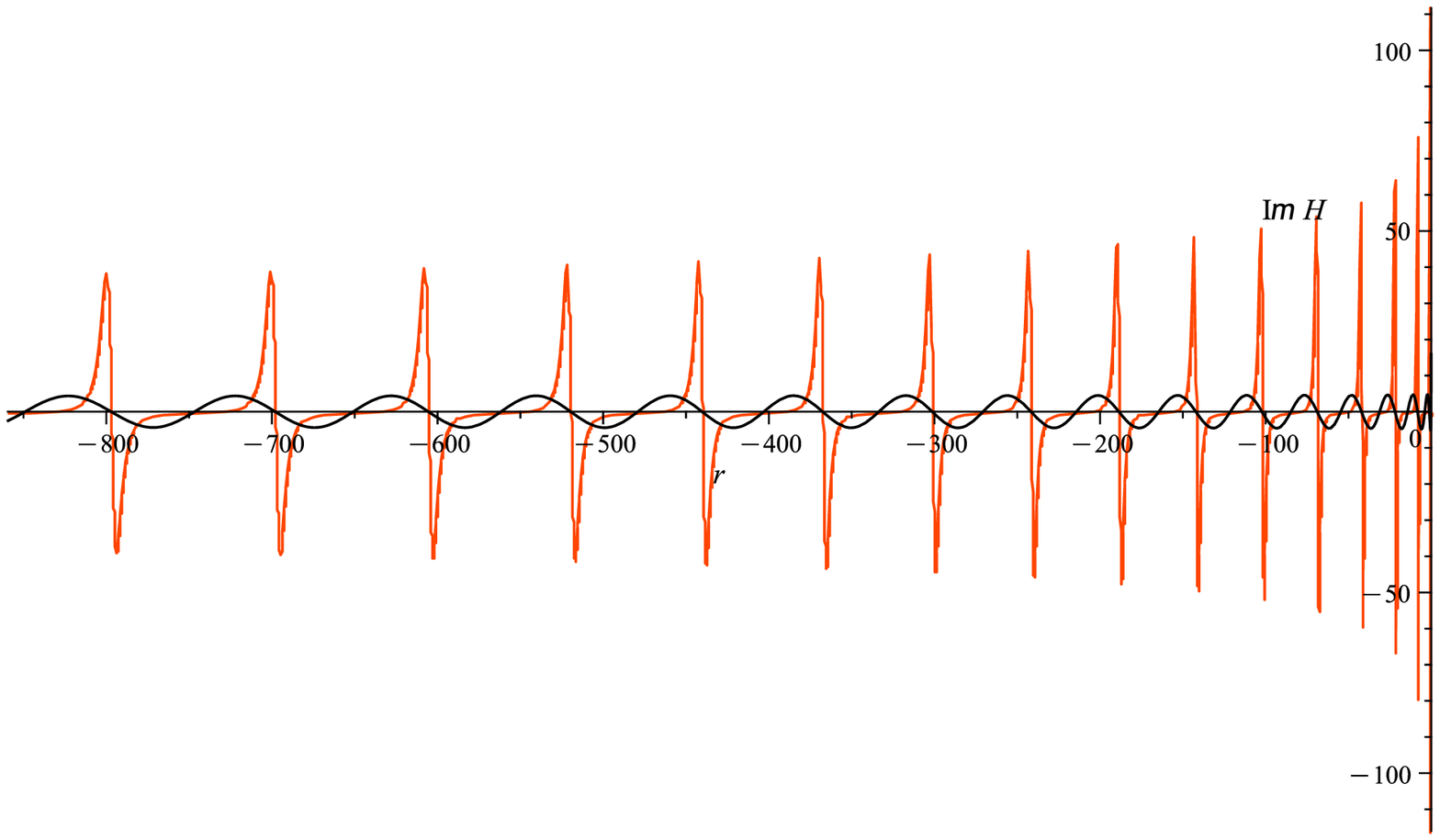}
\caption{The red and black plots are, respectively, the imaginary parts of the numeric and large-$r$ asymptotic
(cf. equation~\eqref{eq:H-asympt-Large-regular}) values of the function $H(r)$ for $r\leqslant-0.1$ corresponding to the initial
value $H(0)=-0.2+\mi0.045$ (see Subsection~\ref{subsec:example6}).}
\label{fig:H0=-02+i0045+ImHintro}
\end{center}
\end{figure}

In Figs.~\ref{fig:H0=-02+i0045+ReHintro} and \ref{fig:H0=-02+i0045+ImHintro}, the red plots are the real and imaginary parts,
respectively, of the function $H(r)$ considered in Example 6 of Section~\ref{sec:asymptnumerics}. The black plots in these
figures are the corresponding leading terms of asymptotics.
``A look'' seems to suggest that something is wrong: (i) perhaps it's the asymptotics; (ii) perhaps it's the distance whence
the asymptotics start to work; or (iii) perhaps it's the absence of the correction terms? Our claim is the following:
(i) the asymptotics are, in fact, correct; (ii) the proper distances over which these facts can be visualized can not be
attained numerically; and (iii) a finite number of correction terms will not help to visualize that the asymptotics
are correct, even though they may be beneficial for improving the correspondence of the plots on larger distances
relative to the origin.
We now justify our claims. How do we know that the asymptotics are correct? We have, in fact, two asymptotic formulae
with overlapping domains of applicability, one of which obtained in \cite{KitVar2004}, the other in \cite{KitVar2010}.
The formula taken from \cite{KitVar2004} is visualized in Figs.~\ref{fig:H0=-02+i0045+ReHintro} and
\ref{fig:H0=-02+i0045+ImHintro}, where as the formula taken from \ref{fig:H0=-02+i0045+ImHintro} is visualized in
Figs.~\ref{fig:H0=-02+i0045+ReH} and \ref{fig:H0=-02+i0045+ImH} of Subsec~\ref{subsec:example6}: this visualization shows
that the latter formulae approximate the numerical solution with a very high degree of accuracy for $r\leqslant-0.1$.
This formula from \cite{KitVar2010} can be further simplified for very large values of $|r|$, so that we can find, so to say,
the ``asymptotics of asymptotics'', and thus obtain  a simplified asymptotic formula that coincides with the one plotted in
Figs.~\ref{fig:H0=-02+i0045+ReHintro} and \ref{fig:H0=-02+i0045+ImHintro}. The simplified asymptotics shows that
$\Re\,H(r)\to1$ as $|r|\to\infty$.
In Subsection~\ref{subsec:example6}, we evaluated the distance over which the plots presented in
Figure~\ref{fig:H0=-02+i0045+ReHintro} become positive: this evaluation shows that the required time and accuracy for the
calculation of the numerical solution goes well beyond the possibilities of modern computers; furthermore, even if we could
execute such a calculation, it wouldn't be possible to visualize it, because we would only be able to see small fragments
of the corresponding plot with no possibility whatsoever of being able to discern the connection between the fragments.
Without the knowledge of the interplay of these asymptotics, one could, in principle, continue the calculations to values
like $r=-10^4,-10^5\ldots$, and observe that the plots are changing very slowly, namely, the maxima of the numerical solution
decrease slightly whilst the minima increase slightly, the distance between maxima grows like $n$, where
$n$ is the number of quasi-periods (a part of the plot between two neighbouring peaks), the scaling eats away at the distances,
but not entirely, and, visually, the general pictures remain very similar to the ones presented. The correction terms may
shift the location of the extrema and render the plot narrower in their neighbourhoods; but, since the plots are changing
very slowly, to achieve such sharp peaks with a finite number ($5$ to $10$, say) of correction terms is simply not
possible.\footnote{\label{foot:correctionsintro} It is over-arching and time consuming to explicitly calculate more then
10 correction terms (see Appendix~\ref{app:infty}). The more correction terms one keeps, the asymptotics provides a better and
better approximation for the functions $H(r)$ and $I(r)$ as $r$ continues to shift  farther and farther away from the origin.}
If we did not have the second asymptotic formula from the paper~\cite{KitVar2010}, then, we would probably illustrate,
with the help of Figs.~\ref{fig:H0=-02+i0045+ReHintro} and \ref{fig:H0=-02+i0045+ImHintro}, that the asymptotics from
\cite{KitVar2004} is not valid for the solution $H(r)$ presented in these figures! The reader will note that there are initial
values of $H(r)$ for which the first asymptotics taken from \cite{KitVar2004}, contrary to the example discussed above, better
approximates the solution for finite values of $r$, and therefore more instrumental for the study of the solutions for such
initial values.

Each of the six examples considered in Section~\ref{sec:asymptnumerics} is supplemented with three types of comments:
(i) the settings used in the corresponding \textsc{Maple} programs that would enable the reader to reproduce our plots, or
to generate plots for other solutions of equation~\eqref{eq:dp3u}; (ii) our understanding of the qualitative behaviour of
the solutions and the corresponding asymptotics; and (iii) some comments of an emotional nature---embedded in the
footnotes---when we encounter unexpected behaviours of the solutions. We now discuss these items in succession.

(i) The construction of the numerical solutions is based on the Cauchy problem with initial data $H(0)$, $H'(0)$, and $I(0)=0$,
where $H'(0)$ given in \eqref{eq:Hprime0-H0}. The problem involved is that equation~\eqref{eq:hazzidakis} has a singularity at
$r=0$. It seems that \textsc{Maple} has an algorithm that allows one to make calculations in this case, but it must understand
that the initial condition \eqref{eq:Hprime0-H0} is satisfied exactly. For some selected initial data (see
Appendix~\ref{app:Kit87}, Figs.~\ref{fig:H0=5over7}--\ref{fig:H0=100}), we were able to use standard \textsc{Maple} programs
for the numerical calculations in order to generate the corresponding plots, but for minor changes of the initial value, $H(0)$,
{}from, say, $5/7$ to $4/7$ or $3/7$, the standard programs did not work. Consequently, we had to consider the Cauchy
problem set at some point $r_1$ close to the origin. Since we are dealing with large-$r$ asymptotics, we have to
construct our solutions over relatively long intervals. For a longer interval, a lower accuracy of the solution is attained
closer to the far-end of the interval. In our calculations, therefore, we have to guarantee, somehow, the accuracy of the
calculations for the numerical solutions sufficient enough for the purposes of comparison with their asymptotics. To achieve
this goal, we used the methodology adopted in \cite{AK2020}: we do the calculations for some rather small, in our opinion,
value $r_1$, say, with good accuracy, plot the graph of the function we are analysing, consider yet another value for $r_1$,
usually 10 times closer to the origin, compare the plots, then increase the accuracy of the calculations two-fold and
compare the plots, and in the event that the plots coincide on some interval, we then increase its length by 1.5 to 2 times
and see whether or not there is a discrepancy. The plots are illustrated with different colours, so that it is easy to observe whether
or not the plots coincide visually.
This algorithmic-in-nature procedure is not as complicated as it may appear at first glance, because, after some experience,
the first approximation is already good enough, and only $2$ to $3$ additional calculations are necessary to confirm that the
numerical solution is accurate enough over the chosen interval. Calculations of the asymptotics for the function $H(r)$ do not
present any problems, since the errors are not accumulating with the distance of the calculation, so that a very high degree of
accuracy is not required here.
At the same time, though, the numerical calculation of asymptotics for the function $I(r)$ is more subtle, and even requires
considerably more execution time than for the calculation of the corresponding numerical solution.

(ii) We also supplement our examples with explanations of the behaviour of the solutions and of some features of their
approximations via the asymptotics. These explanations should be viewed upon as providing preliminary observations that,
hopefully, could be developed to the level of qualitative analysis of equation~\eqref{eq:dp3u}.
As mentioned above, the qualitative analysis of smooth real solutions is quite simple, and, in fact, was done in
\cite{BobEitLMN2000}. The qualitative analysis for complex solutions is more interesting; in this case,
we encounter an ``interaction'' between the stable and unstable attracting curves (the straight lines for the function $H(r)$ and
parabolae for $I(r)$), which complicates considerably the behaviour of the solutions. One may have a question about
the appearance of finite poles and/or zeros destroying the numerical calculations of the solutions on the negative-$r$
semi-axis. For some special classes of solutions, for example, real solutions with $H(0)<0$ (see the discussion above),
such a problem actually exists. For complex
solutions with randomly-chosen initial data $H(0)$, this problem does not appear in practice: even though we have studied
many examples of complex solutions, several of which have been presented in this paper, we have yet to encounter this
problem.

(iii) Comparing the asymptotics with the numerical solutions is, for us, a surprisingly emotional endeavour: recall that the
method we used to derive our asymptotic formulae, namely, The Isomonodromy Deformation Method, does not
``suggest'' any direct involvement of equation \eqref{eq:dp3u} in the asymptotic analysis, while the numerical methods are
based on difference schemes for the approximations of solutions of this equation; therefore, when we see that the numeric and
asymptotic plots are practically merging into one and the same curve, it resembles a manifestation of the integrity of
Mathematics. At the same time, we have an example presented in
Figs.~\ref{fig:H0=-02+i0045+ReHintro} and \ref{fig:H0=-02+i0045+ImHintro} where, as discussed above, the correctness of
the asymptotics can only be justified theoretically. In these figures, the asymptotics and numerical solutions are,
at least, located in the same `domains of the pictures', not far from the
negative-$r$ semi-axis.
A more astonishing situation occurs for the corresponding integral $I(r)$. According to our studies,
$I(r)\genfrac{}{}{0pt}{3}{=}{r \to -\infty}2\sqrt{-r}+\mathcal{O}(1)$. Actually, in Examples $1$ and $2$ of
Section~\ref{sec:asymptnumerics}, we corroborate this asymptotic behaviour; however, in Examples $3$--$6$ of
Section~\ref{sec:asymptnumerics}, for randomly chosen initial values, $H(0)$, the initial---and quite substanial---part
of the plots for $I(r)$ appear, rather unexpectedly, to be located below the negative-$r$ semi-axis!
A considerable increase of the interval of integration, and yet, $I(r)$ continues to follow the wrong tendency!
After a further increase of the integration interval, the function $I(r)$ abruptly changes its behaviour
to the correct---asymptotic---one! As a result, the plot of the numerical solution resembles an ``underground bunker'' with
two staircases leading to the surface, but in opposite directions. One may think that we do not have a formula that approximates
the right staircase; but, it happens that we do, in fact, have one that, possibly with a $2\pi k$ shift, for some integer $k$,
provides us with a reasonable approximation for the right staircase!
Who, or what kind of entity, can reside in such a bunker and manage to spoil the correct behaviour of the
function $I(r)$?\footnote{\label{foot:intro:mole}
So, what to think? Yes, this is reminiscent of the middle of the 18th century, Carlo Goldoni, Truffaldino's home! In the 21st
century, this can be interpreted as the home of someone who works for two intelligence agencies, namely, a ``mole''. We,
however, have a vague idea regarding these types of moles, whilst well acquainted with moles that live in gardens. Further
thinking in this direction (a further investigation of the plots) leads us to the understanding that the function $I(r)$ can
be interpreted as a simplified mathematical model describing the underground movements of the mole (see the detailed discussion
in the final paragraph of Subsection~\ref{subsec:stair-stringer}). After this presentation, the reader may elect to call $I(r)$
\emph{the mole function}.}
For reasons explained in Footnote~\ref{foot:intro:mole}, we coined the name ``mole's dwelling'' for this underground
bunker. On the other hand, we found some initial values for which we were not able to numerically reach
the mole's dwelling: one of these initial values is discussed in Example 6 of Section~\ref{sec:asymptnumerics}.
With the help of the asymptotics, we evaluated the location of the mole's dwelling: this location suggests that it can not be
inhabited by an ordinary mole.\footnote{\label{foot:intro:fallen}
This reminds us of the concept of the ``fallen angel'', well known in the Abrahamic religions.
We see that for different initial values the function $I(r)$ may have different interpretations.}
So, the main intrigue underlying the generation of emotions in the visualization studies is related to the rate at which
the solution attains its large-$r$ asymptotic behaviour. When this attainment is realized for values of $r$ close
to the origin, there is also a question of why it takes place so quickly. When we originally derived the asymptotics for the
function $u(\tau)$, and consequently for $H(r)$, many intermediate expressions were considerably simplified, or neglected,
under the assumption that $\tau$ was very large; for small values of $\tau$, though, such terms are close to those that
eventually form the leading term of asymptotics! In our case, a very helpful circumstance is that we have two asymptotic
formulae for each of the functions $H(r)$ and $I(r)$ which are valid in overlapping domains of the initial
values $H(0)$. We found that at least one of the asymptotics provides a good approximation for the corresponding solutions
beginning from small values of $r$.
In our opinion, the emotional component of these visualization studies indicates that there is a need for the qualitative
analysis of complex solutions of equation~\eqref{eq:dp3u} which would serve as a bridge connecting numerical and indirect
asymptotic methods.

In the main body of the paper we deal with equation~\eqref{eq:dp3u} for $a=0$; thence, we decided to verify some of the results
obtained in the work~\cite{Kit87} in Appendix~\ref{app:Kit87} by taking, as an example, the function $H(r)$. In contrast to the
examples discussed in Section~\ref{sec:asymptnumerics}, we consider, in Appendix~\ref{app:Kit87},  regular \textbf{real}
solutions $H(r)$. At the time when the paper~\cite{Kit87} was
published, there were strict limitations on the pagination count for publications, and all formulae were presented in
handwritten form; this, unfortunately, resulted in a number of misprints. In Appendix~\ref{app:Kit87}, we've corrected all
such misprints that are obvious at first glance (without any additional calculations),
and then show the consistency of the results with those presented in Appendices~\ref{app:asympt0} and \ref{app:infty}.
In Appendix~\ref{app:Kit87}, we comment on the Russian version of the paper~\cite{Kit87}. Two years after the appearance of
\cite{Kit87}, the English translation emerged; when compared with the original version, it contained additional misprints:
these misprints are also addressed in Appendix~\ref{app:Kit87}.

\section{The Solution $H(r)$ Holomorphic at the Origin} \label{sec:2}

This section consists of two subsections. In Subsection~\ref{subsec:existence}, we prove the existence of the solution of
equation~\eqref{eq:hazzidakis}, $H(r)$, holomorphic at $r=0$. In Subsection~\ref{subsec:H0coeffs}, some number-theoretic
properties of the coefficients of the Taylor-series expansion for $H(r)$ are studied.
\subsection{Existence}\label{subsec:existence}
\begin{proposition}\label{prop:FormalExpansion}
For any $a_0\in\mathbb{C}\setminus\{0\}$, there exists a unique formal solution
of equation~\eqref{eq:hazzidakis},
\begin{equation}\label{eq:Hat0-expansion}
H(r)=-a_0+\sum_{k=1}^{\infty}a_kr^k,
\end{equation}
where the coefficients $a_k\in\mathbb{C}$ are independent of $r$.
\end{proposition}
\begin{proof}
For $a_0\in\mathbb{C}\setminus\{0\}$, substituting the expansion~\eqref{eq:Hat0-expansion} into equation~\eqref{eq:hazzidakis}
and equating to zero the coefficients of like powers of $r^k$, $k=-1, 0,\ldots$, one obtains
\begin{equation}\label{eq:a1}
a_1=\frac{a_0^3+1}{a_0},
\end{equation}
and the following recurrence relation for the coefficients $a_k$,
\begin{equation}\label{eq:ak-recurrence}
\begin{split}
(n+1)^2a_0a_{n+1}&=\big((n-1)^2a_1-3a_0^2\big)a_n+3a_0\sum_{i=1}^{n-1}a_ia_{n-i}+\sum_{i=2}^{n-1}(n+1-i)(n+1-2i)a_ia_{n+1-i}\\
&{}-\sum_{\substack{i+j+k=n\\i,j,k\geqslant1}}a_ia_ja_k, \qquad n=1,2,\ldots,
\end{split}
\end{equation}
where the conventions $\sum_{i=2}^0X_i=-X_1$ and  $\sum_{i=1}^0X_i=0$ are used.
\end{proof}
\begin{lemma}\label{lem:a-n-estimate}
For any $a_0\in\mathbb{C}\setminus\{0\}$, there exist $R>0$ and $N>0$ such that for any $n\in\mathbb{N}$,
\begin{equation}\label{ineq:a-n-estimate}
|a_n|<N\frac{R^n}{n^2}.\qquad
\end{equation}
If $\mathcal{D}$ is a compact subset of $\mathbb{C}\setminus\{0\}$, then there exist $N>0$ and $R>0$ such that
estimate~\eqref{ineq:a-n-estimate} is valid for all $a_0\in\mathcal{D}$.
\end{lemma}
\begin{proof}
The proof proceeds via mathematical induction.
The basis of the induction argument consists of the inequalities (cf. \eqref{eqs:a2-a6} below)
\begin{equation}\label{eq:a1a2baseInduction}
|a_1|<NR,
\qquad
|a_2|<NR^2/4.
\end{equation}
Multiplying equation~\eqref{eq:a1} by $3|a_0|/4$, we see that the inequalities~\eqref{eq:a1a2baseInduction}
are satisfied provided
\begin{equation}\label{eq:condition0}
R>\max\left\{3|a_0|, \frac{|a_0|^3+1}{|a_0|\,N}\right\},
\end{equation}
where $N>0$ is, thus far, arbitrary.
Assume now that the inequality~\eqref{ineq:a-n-estimate} holds for all $n=1,2,\ldots, m$, and prove that it is true for $n=m+1$.
Substituting $n=m$ into the recursion relation~\eqref{eq:ak-recurrence} and dividing both sides by
$a_0$, we proceed to successively estimate the four entries on the right-hand side:
\begin{enumerate}
\item[\pmb{(1)}]
\begin{equation}\label{ineq:estimate1}
\begin{gathered}
\left|\big((m-1)^2a_1-3a_0^2\big)a_m/a_0\right|\leqslant
\left((m-1)^2\frac{|a_1|}{|a_0|}+3|a_0|\right)|a_m|\leqslant
\left(\frac{(m-1)^2}{m^2}\frac{NR}{|a_0|}+\frac{3|a_0|}{m^2}\right)NR^m\\
<\left(\frac{N}{|a_0|}+\frac{3|a_0|}{Rm^2}\right)NR^{m+1}\leqslant\left(\frac18+\frac18\right)NR^{m+1}=\frac{NR^{m+1}}{4},
\end{gathered}
\end{equation}
where we took into account that $m\geqslant2$ and imposed the following conditions on $N$ and $R$:
\begin{equation}\label{eq:condition1}
0<N\leqslant\frac{|a_0|}{8},\qquad
R\geqslant6|a_0|.
\end{equation}
\item[\pmb{(2)}]
\begin{equation}\label{ineq:estimate2}
\begin{gathered}
\left|3\sum_{i=1}^{m-1}a_ia_{m-i}\right|\leqslant
3N^2R^m\sum_{i=1}^{m-1}\frac1{i^2(m-i)^2}=
\frac{3N^2R^m}{m^2}\sum_{i=1}^{m-1}\left(\frac1{i}+\frac1{m-i}\right)^2\\
=\frac{3N^2R^m}{m^2}\sum_{i=1}^{m-1}\left(\frac1{i^2}+\frac1{(m-i)^2}+\frac2{m}\left(\frac1{i}+\frac1{m-i}\right)\right)=
\frac{3N^2R^m}{m^2}\sum_{i=1}^{m-1}\left(\frac2{i^2}+\frac4m\cdot\frac1{i}\right)\\
<\frac{3N^2R^m}{m^2}\left(\frac{\pi^2}{3}+\frac4m\left(1+\frac{m-2}{2}\right)\right)=
\frac{N^2R^m(\pi^2+6)}{m^2}\leqslant
\frac{16N^2R^m}{4}\leqslant
\frac{NR^{m+1}}{4},
\end{gathered}
\end{equation}
where, in the last inequality, we assumed that
\begin{equation}\label{eq:condition2}
R\geqslant 16N.
\end{equation}
\item[\pmb{(3)}]
\begin{equation}\label{ineq:estimate3}
\begin{gathered}
\left|\frac1{a_0}\sum_{i=2}^{m-1}(m+1-i)(m+1-2i)a_ia_{m+1-i}\right|\leqslant
\frac{N^2R^{m+1}}{|a_0|}\left(\sum_{i=2}^{m-1}\frac{|m+1-2i|}{(m+1-i)i^2}\right)\\
\leqslant\frac{N^2R^{m+1}}{|a_0|}\left(\sum_{i=2}^{m-1}\frac1{i^2}+\sum_{i=2}^{m-1}\frac1{i(m+1-i)}\right)=
\frac{N^2R^{m+1}}{|a_0|}\left(\sum_{i=2}^{m-1}\frac1{i^2}+\frac2{m+1}\sum_{i=2}^{m-1}\frac1{i}\right)\\
<\frac{N^2R^{m+1}}{|a_0|}\left(\frac{\pi^2}{6}-1+\frac{m-2}{m+1}\right)<
\frac{\pi^2N^2R^{m+1}}{6|a_0|}\leqslant
\frac{NR^{m+1}}{4},
\end{gathered}
\end{equation}
where the following inequality was imposed,
\begin{equation}\label{eq:condition3}
0<N\leqslant\frac{3|a_0|}{20}.
\end{equation}
\item[\pmb{(4)}]
\begin{equation}\label{ineq:estimate4}
\begin{gathered}
\left|\frac1{|a_0|}\sum_{\substack{i+j+k=m\\i,j,k\geqslant1}}a_ia_ja_k\right|=
\left|\frac1{|a_0|}\sum_{i=1}^{m-2}a_i\sum_{j=1}^{m-i-1}a_ja_{m-i-j}\right|\leqslant
\frac{N^3R^m}{|a_0|}\sum_{i=1}^{m-2}\frac1{i^2}\sum_{j=1}^{m-i-1}\frac1{j^2(m-i-j)^2}\\
<\frac{N^3R^m}{|a_0|}\sum_{i=1}^{m-2}\frac1{i^2(m-i)^2}\left(\frac{\pi^2}{3}+2\right)<
\frac{N^3R^m}{|a_0|m^2}\left(\frac{\pi^2}{3}+2\right)^2<\frac{NR^{m+1}}{4},
\end{gathered}
\end{equation}
where the last inequality is predicated on
\begin{equation}\label{eq:condition4}
\frac{N^2}{|a_0|R}\left(\frac{\pi^2}{3}+2\right)^2\leqslant1.
\end{equation}
\end{enumerate}
Now, one verifies that there exist $N>0$ and $R>0$ such that the conditions~\eqref{eq:condition0}--\eqref{eq:condition4}
are valid. Actually, choose any $N$ satisfying the inequality
\begin{equation}\label{ineq:N-final}
0<N<\frac{|a_0|}8,
\end{equation}
and, for that choice of $N$, take
\begin{equation}\label{ineq:R-final}
R=\max\left\{16N,6|a_0|, \frac{|a_0|^3+1}{|a_0|\,N}\right\}.
\end{equation}
In this case, the conditions~\eqref{eq:condition0} and \eqref{eq:condition2}--\eqref{eq:condition4} are satisfied
automatically.

If $a_0\in\mathcal{D}$, then the functions $|a_0|$ and $(|a_0|^3+1)/|a_0|$ have minima and maxima.
In this case, we substitute $\min(|a_0|)$ in lieu of $|a_0|$ in condition \eqref{ineq:N-final}, and the maximum values for
$|a_0|$ and $(|a_0|^3+1)/|a_0|$ in \eqref{ineq:R-final}. Thus, the estimates
\eqref{ineq:estimate1}--\eqref{ineq:estimate4} hold for all $a_0\in\mathcal{D}$.

Finally, summing up the estimates \eqref{ineq:estimate1}--\eqref{ineq:estimate4}, we arrive at the
inequality~\eqref{ineq:a-n-estimate} for $n=m+1$, with $N$ and $R$ defined as in \eqref{ineq:N-final} and \eqref{ineq:R-final},
respectively.
\end{proof}
\begin{corollary}\label{cor:convergence}
The series~\eqref{eq:Hat0-expansion} is uniformly convergent for $a_0\in\mathcal{D}$, where $\mathcal{D}$ is a compact domain
of $\mathbb{C}\setminus\{0\}$. The radius of convergence, $1/R$, is estimated in Lemma~\ref{lem:a-n-estimate}.
\end{corollary}
\subsection{Properties of the Coefficients $a_n$}\label{subsec:H0coeffs}
Using the recurrence relation~\eqref{eq:ak-recurrence}, we calculated, with the help of \textsc{Maple}, the first few
coefficients $a_k$:
\begin{equation}\label{eqs:a2-a6}
\begin{gathered}
a_2=-\frac{3}{4}(a_0^3+1),\quad
a_3=\frac{(a_0^3+1)(2a_0^3+1)}{4a_0^2},\quad
a_4=-\frac{(a_0^3+1)(20a_0^3+17)}{64a_0},\\
a_5=\frac{3(a_0^3+1)(100a_0^6+122a_0^3+25)}{1600a_0^3},\quad
a_6=-\frac{(a_0^3+1)(700a_0^6+1113a_0^3+416)}{6400a_0^2},\ldots.
\end{gathered}
\end{equation}
\begin{remark}\label{rem:constant-solution}
It is conspicuous that equation~\eqref{eq:hazzidakis} has three constant solutions, $\left(H(r)\right)^3=1$. In terms of the
solution $u(\tau)$ of equation~\eqref{eq:dp3u}, these solutions can be amalgamated as
the three branches of the algebraic solution $u(\tau)=\tau^{1/3}/2$. This fact can be reformulated, namely,
the numerators of the coefficients $a_k$, which are polynomials of $a_0^3$, are divisible by $a_0^3+1$.
Clearly, the last statement can be proved directly by induction with the help of the recurrence relation~\eqref{eq:ak-recurrence}.
\hfill $\blacksquare$\end{remark}
\begin{proposition}\label{prop:a-n-ansatz}
\begin{equation}\label{eq:a-n-ansatz}
a_n=(-1)^{n-1}\frac{\kappa_n(a_0^3+1)P_n(a_0^3)}
{(n!)^2a_0^{\frac{n}2-\frac14-(-1)^n\frac34}},\qquad
n\in\mathbb{N},
\end{equation}
where $\kappa_n\in\mathbb{N}$ and $P_n(x)\in\mathbb{Z}[x]$, with $\deg{P_n(x)}\leqslant\left\lfloor\frac{n-1}2\right\rfloor$,
and $\lfloor\ast\rfloor$ denotes the floor of a real number.
\end{proposition}
\begin{proof}
The proof is by mathematical induction (with the help of the recurrence relation~\eqref{eq:ak-recurrence}).
The base of the induction is a consequence of equations~\eqref{eq:a1} and \eqref{eqs:a2-a6}.
To take the inductive step from $a_m$ to $a_{m+1}$, change $n\to m$ in the recurrence relation and assume
that \eqref{eq:a-n-ansatz} is valid for all $n\leqslant m$, substitute it, in lieu of $a_n$, into \eqref{eq:ak-recurrence},
and divide both sides of the last equation by $(n+1)^2a_0$.
Finally, on the left-hand side of the obtained equation, one has $a_{m+1}$, whilst on the right-hand side, there are
sums of terms, each of which, in view of assumption~\eqref{eq:a-n-ansatz}, has the form that one expects in order
to get $a_{m+1}$. To see this, one has to divide and multiply each term by $(m!)^2$, and note that the numeric coefficient
in the denominators of the terms is precisely $((m+1)!)^2$, and their numerators can be presented as products of natural numbers
$\kappa_j$ with binomial coefficients in the double sums and trinomial coefficients in the triple sum. In order to establish
the functional dependence of the terms with respect to $a_0$, it is convenient to consider separately the case for odd and even
values of $m$.
The sums in equation~\eqref{eq:ak-recurrence} are also convenient to split into sums over odd and even indices.
Finally, summing all the terms, one arrives at the result stated in the proposition for $n=m+1$.
\end{proof}
\begin{remark}\label{rem:kappa-degree-Pn}
The numbers $\kappa_n$ are introduced because, later on, we present a conjecture that shows how to choose them in order
to keep the coefficients of the polynomial $P_n(x)$ coprime. In fact, we prove below that
$\deg{P_n(x)}=\lfloor\tfrac{n-1}2\rfloor$; however, to
confirm this with the help of the inductive procedure based on the recurrence relation~\eqref{eq:ak-recurrence} is a circuitous
matter, because the polynomial $P_{m+1}(a_0^3)$ is a linear combination of polynomials, each of the same degree, but
with positive and negative integer coefficients (see  Remark~\ref{rem:conjecture}, equation~\eqref{eq:ak-recurrence-variant}
below).
\hfill $\blacksquare$\end{remark}
\begin{remark}\label{rem:number-theory}
Before proceeding, recall some basic notions from Number Theory:
$\nu_3(\cdot)$ is the $3$-adic valuation of the corresponding natural number, i.e., the largest power of $3$ by which it is
divisible; and $|n!|_3=3^{-\nu_3(n!)}$ is the $3$-adic absolute value of $n!$ (thus $n!|n!|_3$ coincides with the decomposition
of $n!$ on primes where the entry corresponding to $3$ is omitted).

In this subsection, we use the notation $b_n$ to denote the sum of digits of $n$ in base $3$: it is the sequence A053735
in The On-Line Encyclopedia of Integer Sequences (OEIS) \cite{OEIS1}.
There is a useful formula for the large-$n$ calculation of $b_n$ that is due to Benoit Cloitre \cite{OEIS1}:
\begin{equation}\label{eq:Cloitre}
b_n=n-2\sum_{k=1}^{\infty}\left\lfloor\frac{n}{3^k}\right\rfloor.
\end{equation}
It is interesting to compare equation~\eqref{eq:Cloitre} with the famous Legendre formula~\cite{Legendre}
which, for the $3$-adic valuation of $n!$, reads
$$
\nu_3(n!)=\sum_{k=1}^{\infty}\left\lfloor\frac{n}{3^k}\right\rfloor.
$$
Introduce the following notation for the coefficients of the polynomials $P_n(x)$:
\begin{equation}\label{eq:Pn(x)general}
P_n(x)=\sum_{k=0}^{\lfloor\frac{n-1}{2}\rfloor}p^n_kx^k,\qquad
n\in\mathbb{N}.
\end{equation}
The {\it content} of a polynomial with integer coefficients is the greatest common divisor (g.c.d.) of its coefficients
~\cite{Prasolov}, i.e.,
\begin{equation}\label{eq:content-gcd}
\mathrm{cont}(P_n)=\textrm{g.c.d.}\left\{p^n_0,\ldots,p^n_{\lfloor\frac{n-1}{2}\rfloor}\right\}.
\end{equation}
A polynomial in $\mathbb{Z}[x]$ with coprime coefficients is called {\it primitive}; equivalently, $P_n$ is primitive iff
\begin{equation}\label{eq:content=1}
P_n\in\mathbb{Z}[x]\quad\text{and}\quad\mathrm{cont}(P_n)=1.
\end{equation}
\hfill $\blacksquare$\end{remark}
\begin{proposition}\label{prop:kappa-n}
\begin{equation}\label{eq:kappa-n}
\kappa_nP_n(-1)=(-1)^{\lfloor\frac{n-1}2\rfloor}3^{n-1},\quad
n\in\mathbb{N}.
\end{equation}
\end{proposition}
\begin{proof}
Let $f_m=(-1)^{\lfloor\frac{m-1}2\rfloor+1}\kappa_mP_m(-1)/(m!)^2$. Consider the recurrence
relation~\eqref{eq:ak-recurrence} for $n=m$, and divide it by $a_0^3+1$.
Then, having in mind \eqref{eq:a-n-ansatz}, substituting $a_0=-1$ and taking into account
\begin{equation}\label{eq:-1}
(m-1)-\left(\frac{m}2-\frac14-(-1)^m\frac34\right)-\left\lfloor\frac{m-1}2\right\rfloor-1=\left\{
\begin{matrix}
-2&m\;\text{is\;odd},\\
0&m\;\text{is\;even},
\end{matrix}\right.
\end{equation}
we get $(m+1)^2f_{m+1}=3f_m$. Multiplying the last equation for $m=1,\ldots,n-1$, one proves $(n!)^2f_n=3^{n-1}f_1$;
since $f_1=-1$, one obtains $(n!)^2f_n=-3^{n-1}$, which can be rewritten as equation~\eqref{eq:kappa-n}.
\end{proof}
\begin{corollary}\label{cor:PnNotDIVa3+1}
For all $n\in\mathbb{N}$, the polynomial $P_n(x)$ is not divisible by $x+1$.
\end{corollary}
\begin{proof}
By contradiction; otherwise, $P_n(-1)=0$ for some $n$, which contradicts \eqref{eq:kappa-n}.
\end{proof}
\begin{corollary}\label{cor:gcd3power}
$$
\mathrm{cont}(P_n)=3^{c_n},\qquad
\kappa_n=3^{d_n},
\qquad\textrm{where}\quad
c_n,d_n\in\{0\}\cup\mathbb{N}.
$$
\end{corollary}
Further studies of the ansatz~\eqref{eq:a-n-ansatz} with the help of the recurrence relation~\eqref{eq:ak-recurrence} seems to be
of no avail; however, additional experimentation using \textsc{Maple} allows one to formulate the following
\begin{conjecture}\label{con:structure-ak}
\begin{equation}\label{eq:a-n-conjecture}
a_n=(-1)^{n-1}\frac{3^{\nu_3(n+1)}(a_0^3+1)P_n(a_0^3)}
{(n!|n!|_3)^2a_0^{\frac{n}2-\frac14-(-1)^n\frac34}},\qquad
n\in\mathbb{N},
\end{equation}
where $P_n(x)$ is an irreducible polynomial over $\mathbb{Q}$ with positive integer coprime coefficients,
and $\deg{P_n}=\left\lfloor\frac{n-1}{2}\right\rfloor$.
\end{conjecture}
The first few polynomials $P_n(x)$ (cf. Conjecture~\ref{con:structure-ak}) read:
\begin{equation}\label{eq:Pn(x)}
\begin{gathered}
P_1(x)=P_2(x)=1,\;\;
P_3(x)=2x+1,\;\;
P_4(x)=20x+17,\;\;
P_5(x)=100x^2+122x+25,\\
P_6(x)=700x^2+1113x+416,\;\;
P_7(x)=19600x^3+38416x^2+21275x+2450,\\
P_8(x)=78400x^3+182672x^2+134227x+29952,\;\;\ldots.
\end{gathered}
\end{equation}
\begin{remark}\label{rem:conjecture}
In a subsequent part of this subsection, we briefly outline the technique of the generating functions developed in
\cite{KitSIGMA2019}, which allows one to derive explicit formulae for the coefficients $p^n_k$ and
$p^{n}_{\left\lfloor\frac{n-1}2\right\rfloor-k}$ of the polynomials $P_n(x)$ for any given
$k\in\mathbb{Z}_{\geqslant0}:=\{0\}\cup\mathbb{N}$ and for all $n\in\mathbb{N}$. Note that
this technique allows one to prove that $\kappa_np^n_{\left\lfloor\frac{n-1}2\right\rfloor}>0$ for all $n\in\mathbb{N}$,
which implies that $\deg{P_n}=\left\lfloor\frac{n-1}{2}\right\rfloor$. Furthermore, if $\kappa_n=3^{\nu_3(n+1)+2\nu_3(n!)}$
is fixed, as assumed in Conjecture~\ref{con:structure-ak}, then we prove that the polynomial $P_n(x)$ is primitive.

Therefore, Conjecture~\ref{con:structure-ak} contains three nontrivial statements:
(1) the value of the coefficient $\kappa_n$; (2) the statement that, for this choice of $\kappa_n$, the coefficients
$p^n_k\in\mathbb{N}$ for all $n\in\mathbb{N}$ and $k=0,\ldots,\left\lfloor\frac{n-1}{2}\right\rfloor$; and (3) the assertion
that $P_n(x)$ is an irreducible polynomial over $\mathbb{Q}$.
Let us explain, in particular, the problem related with the justification of item (2). The recurrence
relation~\eqref{eq:ak-recurrence} can be rewritten as follows:
\begin{equation}\label{eq:ak-recurrence-variant}
\begin{aligned}
(n+1)^2a_0a_{n+1}&=-3a_0^2a_n+3a_0\sum_{i=1}^{n-1}a_ia_{n-i}-\sum_{\substack{i+j+k=n\\i,j,k\geqslant1}}a_ia_ja_k\\
&+\sum_{i=1}^{\left\lfloor\frac{n}2\right\rfloor}(n-2i+1)^2a_ia_{n+1-i},
\qquad n=1,2,\ldots.
\end{aligned}
\end{equation}
If we assume the validity of Conjecture~\ref{con:structure-ak}, then the sign of each term in the first line of
~\eqref{eq:ak-recurrence-variant} is $(-1)^n$, while the sign of each term of the sum in the second line of
\eqref{eq:ak-recurrence-variant} is $(-1)^{n-1}$.

In this work, the technique of generating functions is not developed in complete detail; rather, some nontrivial results
that can be obtained with their utilization are outlined.
Further results obtainable for such generating functions can be found in \cite{KitSIGMA2019}.

Perusing equation~\eqref{eq:a-n-conjecture}, one notes that the coefficients $a_n$ have three singular points with respect to
the parameter $a_0$: $-1$, $0$, and $\infty$. To each of these singular points one can construct an infinite series of generating
functions for the coefficients of polynomials $P_n$. There are two other cube roots of $-1$, but the corresponding generating
functions can be obtained via symmetry from the one corresponding to $-1$.

The following propositions are proved under the assumption that Conjecture~\ref{con:structure-ak} is valid. Some results
can be obtained without reference to this conjecture: these cases are duely noted.
\hfill $\blacksquare$\end{remark}
\begin{proposition}\label{prop:Pn(-1)}
\begin{equation}\label{eq:Pn(-1)}
3^{\nu_3(n+1)}P_n(-1)=(-1)^{\left\lfloor\frac{n-1}2\right\rfloor}3^{b_n-1},\qquad
n=1,2,\ldots,
\end{equation}
where the sequence $b_n$ is defined in Remark~\ref{rem:number-theory}.
\end{proposition}
\begin{remark}\label{rem:Cloitre}
Consider, for example, $n=34=1021_3$, so that $b_{34}=1+0+2+1=4$, $\nu_3(35)=0$, thus $P_{34}(-1)=3^3$.
\hfill $\blacksquare$\end{remark}
\begin{proof}
This result is related to the first generating function for $a_0=-1$. In general, the generating functions associated
with $a_0=-1$ are constructed as follows. Define $\varepsilon$ via $a_0=-(1-\varepsilon)^{1/3}$; then,
the expansion~\eqref{eq:Hat0-expansion} can be rewritten as
\begin{equation}\label{eq:generating-1bunch}
H(r)=(1-\varepsilon)^{1/3}+\sum_{k=1}^{\infty}g_k(r)\varepsilon^k,
\end{equation}
where the coefficients $g_k(r)$, $k=1,2,\ldots$, are the generating functions. To achieve the result stated
in the proposition, one requires the generating function $g_1(r)$. Substituting the expansion~\eqref{eq:generating-1bunch} into
equation~\eqref{eq:hazzidakis}, expanding the corresponding expressions into power series in $\varepsilon$,
and equating coefficients of like powers of $\varepsilon$, one arrives at differential equations for $g_k(r)$'s;
in particular, this procedure for $\varepsilon^1$ gives rise to a differential equation for $g_1(r)$,
$$
(rg_1'(r))'=3g_1(r)-1,
$$
whose general solution is given in terms of modified Bessel functions,
$$
\frac13+C_1I_0\big(2\sqrt{3r}\big)+C_2K_0\big(2\sqrt{3r}\big),
$$
where the constant of integration $C_2=0$, because our solution $H(r)$  does not contain logarithmic terms in its small-$r$
expansion, and $C_1=-1/3$, since the small-$r$ expansion for $g_1(r)$ starts with the term $-r$. Thus,
\begin{equation}\label{eq:g1series}
g_1(r)=\frac13\left(1-I_0\big(2\sqrt{3r}\big)\right)=-\sum_{n=1}^{\infty}\frac{3^{n-1}r^n}{(n!)^2}.
\end{equation}
To complete the proof, we have to compare the series~\eqref{eq:g1series} with the part of the
expansion~\eqref{eq:Hat0-expansion} that is proportional to $\varepsilon$; in order to do so, one extricates  the factor
$a_0^3+1=\varepsilon$ from each coefficient $a_k$ and then sets $a_0=-1$.
After setting $a_0=-1$, one should take into account equation~\eqref{eq:-1}.
To verify the power of $3$ on the right-hand side of \eqref{eq:Pn(-1)}, one, using the Legendre formula, moves the factor
$3^{2\nu_3(n!)}$ from the denominators of the coefficients in the series~\eqref{eq:g1series} to corresponding numerators,
and denotes the numbers $b_n$ according to the Cloitre formula~\eqref{eq:Cloitre}. The interpretation of $b_n$ as the sum of
digits of $n$ in base $3$ is due to \cite{OEIS1}.
\end{proof}
\begin{remark}\label{prop:fence}
In $\mathbb{R}^2$, consider the points with co-ordinates $(n,b_n-1)$, $n\in\mathbb{N}$. Connect the neighbouring points
$(n,b_n-1)$ and $(n+1,b_{n+1}-1)$ with line segments. As a result, we get a semi-infinite figure located in the first
quadrant of the $(x,y)$-plane that is bounded from above by a broken line consisting of segments and from below
by the $x$-axis. For brevity, we call this figure `the fence'. For $n=3^l$, where $l\in\mathbb{Z}_{\geqslant0}$, $b_{n}=1$,
so that the fence consists of `parts'. The $l$th part of the fence is located on the segment $\left[3^l,3^{l+1}\right]$,
and at the end-points of the segment the fence has height $0$, so that the neighbouring parts have one common point lying on
the $x$-axis. Denote the area of the $l$th part of the fence by $S_l$; then,
$$
S_l=\sum_{l=3^l}^{3^{l+1}}(b_n-1)=(2l+1)3^l
\quad\text{and}\quad
\sum_{l=0}^LS_l=L3^{L+1}+1.
$$
Note that the natural formula for $S_l$ as the sum of the heights of the fence is valid only for those parts
between the points $[3^l,3^k]$ for integers $0\leqslant l<k$. In this case, the corresponding part of the fence can be transformed
into a part of a rectangular fence with the same area. Note that the sequence $S_l$ can be found in OEIS \cite{OEIS2}
as the sequence A124647. We were not able to locate in OEIS the relation between the sequences $b_n$ and $S_l$ indicated above.
\hfill $\blacksquare$\end{remark}
\begin{remark}
One might expect that the higher generating functions $g_k(r)$ for $k\geqslant2$ may be useful
for the proof that, in fact, $c_n=0$ in Corollary~\ref{cor:gcd3power}. It is straightforward to see that the functions $g_k(r)$
allow one to calculate $P_n^{(k-1)}(-1)$, the $(k-1)$th derivative of $P_n(x)$ with respect to $x$ at $x=-1$.

Using equations~\eqref{eq:Pn(x)}, one shows that $P_5^{'}(-1)=78$, $P_7^{'}(-1)=3243$, and $P_8^{'}(-1)=4083$. All of these
numbers are divisible by $3$, so that $g_2(r)$ can hardly help in establishing the hypothesis \eqref{eq:content=1}.
At the same time, it is not difficult
to see that the first nontrivial derivatives, $P_5^{''}(-1)$, $P_7^{''}(-1)$, and $P_8^{''}(-1)$, are not divisible by $3$,
so that the function $g_3(r)$ may have perspectives in proving the hypothesis~\eqref{eq:content=1}.
In Appendix~\ref{app:g2}, an explicit construction of the generating function $g_2(r)$, together with the explicit
formula for the coefficients of its expansion at $r=0$, are obtained. This case is of technical interest because, if one follows
the standard scheme for the construction of this expansion, which consists of an ODE for the generating function,
its explicit solution, and the corresponding expansion, then one encounters a cumbersome expression for the coefficients of
the expansion.
Furthermore, it is not clear whether it is possible to simplify this formula; however, in case the expansion is obtained directly
from the ODE, then the corresponding formula for the coefficients is much simpler. Therefore, it is evident that one can
explicitly continue this process of constructing the higher functions $g_k(r)$, $k=3,4,\ldots$. With the help of these functions,
one can calculate $P_n^{(k-1)}(-1)$ for any $n$ and $k$; however, to study the divisibility question with the help of the
formulae for $P_n^{(k-1)}(-1)$ may be problematic. Since the construction of the generating functions $g_k(r)$ is a recursive
process, we anticipate that the corresponding explicit expressions for the coefficients of $g_k(r)$ should be progressively
more complicated for increasing values of $k$. Hence, we do not expect that these functions will be beneficial towards
a proof of hypothesis~\eqref{eq:content=1}. Consequently, the other generating functions are considered below.
\hfill $\blacksquare$\end{remark}
\begin{proposition}\label{prop:gtnfunInfinity}
The higher coefficients of the polynomials $P_n$ (cf. \eqref{eq:Pn(x)general}$)$ are
\begin{equation}\label{eq:coefsPnSenior}
 p^n_{\lfloor\frac{n-1}2\rfloor}=\frac{(n+1)|n+1|_3(n!|n!|_3)^2}{2^n}
 =3^{b_n-n-\nu_3(n+1)}\prod_{k=0}^{n-1}\binom{n+1-k}{2}.
 \end{equation}
\end{proposition}
\begin{proof}
The proof is done with the help of the first generating function, $A_1(z)$, at the point $a_0=\infty$:
\begin{equation}\label{eq:generating-INFTYA1}
H(r)=a_0A_1(z)+\mathcal{O}\!\left(1/a_0^3\right),\quad
z=a_0r;\qquad
a_0\to\infty,\quad
|z|\leqslant\mathcal{O}(1).
\end{equation}
As a matter of fact, this expansion can be viewed as a \emph{double asymptotics} of $H(r)$. Substituting
the expansion~\eqref{eq:generating-INFTYA1} into equation~\eqref{eq:hazzidakis}, dividing the resulting equation by $a_0$,
and equating to zero the coefficient  independent of $a_0$, one arrives at a nonlinear ODE for the function $A_1(z)$:
\begin{equation}\label{eq:A1}
\left(z\frac{A_1^{'}(z)}{A_1(z)}\right)'=A_1(z).
\end{equation}
This ODE has the following general and special solutions,
\begin{equation}\label{eq:A1solutions-gen-spec}
{A_1}_{gen}=\frac{2C_2C_1^2z^{C1-1}}{(1-C_2z^{C1})^2},\qquad
{A_1}_{spec}=\frac{2}{z\ln^2(C_2z)},
\end{equation}
where $C_1$ and $C_2$ are constants of integration. Of interest is that solution in \eqref{eq:A1solutions-gen-spec}
which can be expanded into a power series in $z$,
$$
A_1(z)=\sum_{n=0}^\infty\alpha_nz^n,
$$
where $\alpha_n\in\mathbb{C}$. This expansion should be compared with the leading term of asymptotics as $a_0\to\infty$
of the function $H(r)$ in \eqref{eq:Hat0-expansion}; then, one obtains $\alpha_0=-1$, and
\begin{equation}\label{eq:alpha-n}
\alpha_n=(-1)^{n-1}\frac{3^{\nu_3(n+1)}p^n_{\left\lfloor\frac{n-1}2\right\rfloor}}{(n!|n!|_3)^2},\quad n\in\mathbb{N}.
\end{equation}
The fact that $\alpha_0=-1$ allows one to fix both constants of integration in \eqref{eq:A1solutions-gen-spec}, namely,
$C_1=1$ and $C_2=-1/2$; thus,
\begin{equation}\label{eq:A1rational}
A_1(z)=-\frac{1}{(1+z/2)^2}=\sum_{n=0}^{\infty}(-1)^{n-1}\frac{n+1}{2^n}z^n.
\end{equation}
Comparing the coefficients of the series~\eqref{eq:A1rational} with the coefficients $\alpha_n$, we arrive at
equation~\eqref{eq:coefsPnSenior}.
\end{proof}
\begin{remark}\label{rem:degPn}
In the proof above, instead of Conjecture~\ref{con:structure-ak}, we can use Proposition~\ref{prop:kappa-n}. In this case,
in lieu of equation~\eqref{eq:coefsPnSenior}, we get
$\kappa_np^n_{\lfloor\frac{n-1}2\rfloor}=(n+1)(n!)^2/2$. The last formula implies that
$p^n_{\lfloor\frac{n-1}2\rfloor}>0$ for all $n\in\mathbb{N}$; thus, we confirm that
$\deg{P_n}=\left\lfloor\frac{n-1}2\right\rfloor$.
\hfill $\blacksquare$\end{remark}
\begin{remark}\label{rem:genfunAnINFINITY}
The set of generating functions corresponding to
$a_0=\infty$, $\{A_k(z)\}_{k=1,2,\ldots}$, is defined via the expansion
\begin{equation}\label{eq:generating-INFTYbunch}
H(r)=a_0\sum_{k=1}^{\infty}\frac{A_k(z)}{a_0^{3(k-1)}},\quad
z=a_0r;\qquad
a_0\to\infty,\quad
|z|\leqslant\mathcal{O}(1).
\end{equation}
In fact, this expansion can be viewed as a double asymptotics of $H(r)$. Substituting the
expansion~\eqref{eq:generating-INFTYbunch} into equation~\eqref{eq:hazzidakis}, dividing the resulting equation by $a_0$, and
equating to zero the coefficients of successive powers of $a_0^{3(1-k)}$, $k=1, 2, \ldots$, we get, for $k=1$,
the nonlinear ODE \eqref{eq:A1} for $A_1(z)$,
and linear inhomogeneous ODEs for the determination of $A_k(z)$ for $k=2,3,\ldots$. The homogeneous part of these linear ODEs
is the same for all the functions $A_k(z)$ and can be viewed as a degenerate hypergeometric equation. The inhomogeneous part is
a rational function of $z$ with a single pole at $z=-2$. Since $z=-2$ is the only singular point of all the linear ODEs for the
functions $A_k(z)$, it then follows that the corresponding $z$-series for these functions have the same
radius of convergence, which equals $2$. According to the estimates presented in Lemma~\ref{lem:a-n-estimate},
the series~\eqref{eq:Hat0-expansion} for $H(r)$ is convergent at least
for $|a_0r|<1/16$, so that for these values of $r$ we can rearrange the series~\eqref{eq:Hat0-expansion} into
the series~\eqref{eq:generating-INFTYbunch} for the generating functions.

So, there is a recursive procedure allowing one to construct $A_k(z)$ in case all $A_l(z)$'s for $l<k$ are obtained.
The small-$z$ expansion of the function $A_k(z)$ generates the coefficients of $P_n(x)$ at the power
$x^{\left\lfloor\frac{n-1}2\right\rfloor+1-k}$.

Here, we limit our consideration only to the function $A_1(z)$. It is worth mentioning that the reader will find a very similar
construction for the higher generating functions in Section~3 of \cite{KitSIGMA2019}, where the first few generating functions
are explicitly obtained.
\hfill $\blacksquare$\end{remark}
\begin{corollary}\label{cor:gcd1}
For any $n\in\mathbb{N}$, the polynomial $P_n$ is primitive.
\end{corollary}
\begin{proof}
We have to prove that the coefficients of $P_n$ do not have a common divisor,
i.e., equation~\eqref{eq:content=1} is valid.
Proposition~\ref{prop:gtnfunInfinity} states that the highest coefficient of $P_n$ is not divisible by $3$; thus, the statement
follows from Corollary~\ref{cor:gcd3power}.
\end{proof}
\begin{proposition}\label{prop:P2k-1(0)}
For $k\in\mathbb{N}$,
\begin{align}
P_{2k-1}(0)&=p^{2k-1}_0=\frac{(2k)!|(2k)!|_3(2k-1)!|(2k-1)!|_3}{2^{3k-2}}=
\frac{3^{b_{2k-1}-(2k-1)-\nu_3(2k)}}{2^{k-1}}\prod_{l=0}^{2k-2}\binom{2k-l}{2}.\label{eq:P2k-1(0)}
\end{align}
\end{proposition}
\begin{proof}
In this case, we introduce the variable $z=r/\sqrt{a_0}$ and define the generating function $B_1(z)$:
\begin{equation}\label{eq:HinB1}
H(r)=B_1(z)/\sqrt{a_0}+O(a_0),\qquad
a_0\to0,\quad
|z|\leqslant\mathcal{O}(1);
\end{equation}
moreover, $B_1(z)$ is an odd function of $z$, and $B_1(z)=z+\mathcal{O}(z^3)$. Substituting the expansion~\eqref{eq:HinB1} into
the ODE \eqref{eq:hazzidakis} for $H(r)$, one obtains for $B_1(z)$ the same ODE \eqref{eq:A1} as for the function $A_1(z)$,
but for different choices of the constants of integration, $C_1=2$ and $C_2=1/8$ (cf. \eqref{eq:A1solutions-gen-spec});
thus, we get
\begin{equation}\label{eq:B1series}
B_1(z)=\frac{z}{(1-z^2/8)^2}=\sum_{k=1}^{\infty}\frac{k}{8^{k-1}}z^{2k-1}.
\end{equation}
On the other hand, we calculate the coefficients of the above series with the help of equation~\eqref{eq:a-n-conjecture};
by considering the expression $\sqrt{a_0}a_{2k-1}r^{2k-1}$ and letting $a_0\to0$, one finds the leading term of asymptotics:
$$
\frac{3^{\nu_3(2k)}P_{2k-1}(0)}{((2k-1)!|(2k-1)!|_3)^2}z^{2k-1}.
$$
Equating this expression to the corresponding term of the series~\eqref{eq:B1series}, one arrives at the result stated in the
proposition.
\end{proof}
\begin{proposition}
\begin{gather}
P_2(0)=p^2_0=1,\qquad
P_4(0)=p^4_0=17,\nonumber\\
P_{2k}(0)=p^{2k}_0=\frac{13}{(35)^2}\frac{((2k+1)!|(2k+1)!|_3)^2}{2^{3(k-2)}}3^{\nu_3(2k+1)}, \quad
k=3,4,5,\ldots.
\label{eq:P2k(0)}
\end{gather}
\end{proposition}
\begin{proof}
In order to calculate $P_{2k}(0)$, define the generating function $B_2(z)$ via
\begin{equation}\label{eq:HinB2}
H(r)=\frac{B_1(z)}{\sqrt{a_0}}+a_0B_2(z)+\mathcal{O}\big(a_0^{5/2}\big),\quad
z=\frac{r}{\sqrt{a_0}};\qquad
a_0\to0,\quad
|z|\leqslant\mathcal{O}(1),
\end{equation}
where $B_1(z)$ is given by the first equation in~\eqref{eq:B1series}. Now, substituting the expansion in \eqref{eq:HinB2}
into equation~\eqref{eq:hazzidakis},
dividing both sides of the resulting equation by $a_0$, expanding it in powers of $a_0^{3/2}$, and equating to zero the
coefficient of the highest term $a_0^{3/2}$, we arrive at the linear second-order inhomogeneous ODE for the
function $B_2(z)$:
$$
\left((1-z^2/8)^2B_2'(z)-(1+3z^2/8)(1-z^2/8)B_2(z)/z\right)'=B_2(z)-(1-z^2/8)^4/z^2.
$$
The homogeneous part of this ODE is a degenerate hypergeometric equation that is not complicated to solve explicitly:
\begin{align*}
B_2(z)&=C_1\frac{z(z^2+8)}{(z^2-8)^3}+C_2\frac{z(z^2\ln{z}+8\ln{z}+16)}{(z^2-8)^3}\\
&+\frac{40140800+15052800z^2+1254400z^4-56448z^6+1704z^8-25z^{10}}{78400(z^2-8)^3},
\end{align*}
where $C_1=C_2=0$, because $B_2(z)$ is a single-valued even function of $z$. Now, decompose $B_2(z)$ into partial fractions,
$$
B_2(z)=-\frac{1}{2^67^2}z^4+\frac{69}{70^2}z^2-\frac{393}{35^2}-\frac{13\cdot2^6}{35^2}\left(\frac{8}{(1-z^2/8)^3}
-\frac{8}{(1-z^2/8)^2}+\frac{1}{1-z^2/8}\right).
$$
Developing the quotients in the above equation into series in powers of $z^2$ and combining them into a unique series, we get
$$
B_2(z)=-\frac{393}{35^2}+\frac{69}{70^2}z^2-\frac{1}{2^67^2}z^4
-\frac{13}{35^2}\sum_{k=0}^{\infty}\frac{(2k+1)^2}{2^{3(k-2)}}z^{2k}.
$$
This series can be rewritten as
$$
B_2(z)=-1-\frac34z^2-\frac{17}{64}z^4
-\frac{13}{35^2}\sum_{k=3}^{\infty}\frac{(2k+1)^2}{2^{3(k-2)}}z^{2k}.
$$
Equate, now, the term $a_{2k}r^{2k}/a_0=a_{2k}a_0^{k-1}z^{2k}$ of the series $H(r)/a_0$ as $a_0\to0$
with the corresponding term of the above series for $B_2(z)$. Since
$$
\lim_{a_0\to0} a_{2k}a_0^{k-1}=-\frac{3^{\nu_3(2k+1)}P_{2k}(0)}{((2k)!|(2k)!|_3)^2},
$$
one arrives at the result asserted in the proposition.
\end{proof}
\begin{remark}
The justification for the introduction of the generating functions $B_k(z)$ is quite similar to that employed for the functions
$A_k(z)$. We define an infinite sequence of these functions via the expansion
\begin{equation}\label{eq:seriesU-Bk}
H(r)=\frac{1}{\sqrt{a_0}}\sum_{k=1}^{\infty}B_k(z)a_0^{\frac32(k-1)},\quad
z=\frac{r}{\sqrt{a_0}};\qquad
a_0\to0,\quad
|z|\leqslant\mathcal{O}(1).
\end{equation}
All the functions $B_k(z)$ are rational functions of $z$ with poles only at $z^2=8$; therefore, they can be developed into
power series in $z$ with the same radius of convergence $2\sqrt{2}$. The series~\eqref{eq:seriesU-Bk} is the rearrangement of
the series~\eqref{eq:Hat0-expansion} for $r=o(a_0^2)$ as $a_0\to0$ (see the estimates in Lemma~\ref{lem:a-n-estimate}).
The function $B_k(z)$ can be constructed explicitly provided all the functions $B_n(z)$ with $n<k$ are already obtained.
This inductive procedure
is quite analogous to the corresponding procedure for the functions $A_k(z)$. It is worth mentioning that the functions
$B_{2l+1}(z)$ define $p^{2n-1}_l$, whilst $B_{2l+2}(z)$ define $p^{2n}_l$, where $l=0,1,\ldots$ and $n=1,2,\ldots$.
\hfill $\blacksquare$\end{remark}
\begin{corollary}\label{cor:coefsRelation}
The highest and lowest coefficients of the polynomials $P_n$ are related by the following equations:
\begin{align}
2^{k-1}p_0^{2k-1}&=p_{k-1}^{2k-1},\qquad
k=1,2,\ldots,\label{eq:oddRelation}\\
2^{k-6}p_0^{2k}&=\frac{13}{(35)^2}(2k+1)p_{k-1}^{2k},\qquad
k=3,4,\ldots.\label{eq:evenRelation}
\end{align}
\end{corollary}
\begin{proof}
The formula ~\eqref{eq:oddRelation} (resp., \eqref{eq:evenRelation}) follows from the comparison of the explicit
formulae~\eqref{eq:coefsPnSenior} and \eqref{eq:P2k-1(0)} (resp., \eqref{eq:coefsPnSenior} and \eqref{eq:P2k(0)}).
\end{proof}
\begin{remark}\label{rem:relation-between-coefficients}
Corollary~\ref{cor:coefsRelation} is formally obtained using Conjecture~\ref{con:structure-ak}; however,
the relations~\eqref{eq:oddRelation} and \eqref{eq:evenRelation} are independent of the value of $\kappa_n$
(cf. equation~\eqref{eq:kappa-n}). Therefore, these relations can be proved via reference to Proposition~\ref{prop:kappa-n}.
\hfill $\blacksquare$\end{remark}

\section{Algebroid Solutions}\label{sec:algebroid}
In this section, we consider algebroid solutions of equation~\eqref{eq:dP3y}. It is convenient to rewrite
equation~\eqref{eq:dP3y} in the following form:
\begin{equation}\label{eq:dP3y-compact}
\left(\frac{ty'(t)}{y(t)}\right)'=y(t)-\frac{t}{(y(t))^2}.
\end{equation}
\begin{theorem}\label{th:algebroid}
If $y(t)$ is an algebroid solution of equation~\eqref{eq:dP3y-compact}, then there exist $n,m\in\mathbb{Z}_{\geqslant0}$
and $\alpha=(m+2n+3)/4$
such that
\begin{equation}\label{eq:y-w}
y(t)=x^{n+1-\alpha}w(x),\qquad
t=x^{\alpha},
\end{equation}
where the function $w(x)$, which is holomorphic at $x=0$ and $w(0)\neq0$, is the unique solution of the equation
\begin{equation}\label{eq:dP3w}
\left(\frac{xw'(x)}{w(x)}\right)'=\alpha^2\left(x^nw(x)-\frac{x^m}{(w(x))^2}\right).
\end{equation}
Conversely, for any $n,m\in\mathbb{Z}_{\geqslant0}$, $\alpha=(m+2n+3)/4$, and $a_0\in\mathbb{C}\setminus\{0\}$,
there exists a unique solution $w(x)$ of equation~\eqref{eq:dP3w} that is holomorphic at $x=0$ and $w(0)=-a_0$,
which defines, via~\eqref{eq:y-w}, an algebroid solution of equation~\eqref{eq:dP3y-compact}.
\end{theorem}
\begin{proof}
As a consequence of the Painlev\'e property, the only branching point of the solution is $t=0$. If the solution $y(t)$ is
algebroid, then there exists a natural number $\alpha$ such that $y(x^{\alpha})$ is a holomorphic function at $x=0$, or
it has a pole of finite order. It is convenient to make the transformation $y(t)=x^{1-\alpha}v(x)$, $t=x^{\alpha}$,
and to consider the function $v(x)$ which solves
\begin{equation}\label{eq:dP3v}
\left(\frac{xv'(x)}{v(x)}\right)'=\alpha^2\left(v(x)-\frac{x^{4\alpha-3}}{(v(x))^2}\right),\qquad
4\alpha-3=m_1,
\end{equation}
where $m_1\in\mathbb{Z}$, since $\alpha\in\mathbb{N}$. Now, assume that, for some $m_1\in\mathbb{Z}$, $v(x)$ is a solution
of \eqref{eq:dP3v} that is holomorphic or has a Laurent expansion at $x=0$; then, we see that $\alpha=(m_1+3)/4$ is a rational
number, and the solution $y(t)$ has an algebraic singularity at $x=0$. It is clear that $m_1\geqslant-2$, because,
otherwise, $\alpha\leqslant0$.
In that case, after substituting $x=t^{1/\alpha}$ into the Laurent expansion for $v(x)$, one gets an infinite number of terms
with negative powers of $t$ that are growing as $t\to0$. More precisely, the local analysis shows that the only possibility
to balance the leading term is to require that
\begin{equation}\label{eq:v(x)-pole}
v(x)=\frac{1}{x^l}\sum_{k=0}^{\infty}b_kx^k,\qquad
b_0\in\mathbb{C}\setminus\{0\},\quad
m_1=-3l,\quad
l-1\in\mathbb{N};
\end{equation}
otherwise,
the right-hand side of equation~\eqref{eq:dP3v} would have a pole whilst the left-hand side would not.
Even under the assumption~\eqref{eq:v(x)-pole}, however, one cannot construct an infinite Laurent expansion, because,
by induction, one proves that all the coefficients $b_k$, $k\geqslant1$, of such an expansion should vanish:
if, for $k\geqslant2$, $b_1=\dotsb=b_{k-1}=0$, then, on the left-hand side of equation~\eqref{eq:dP3v}, we have
the leading term $k^2b_kx^{k-1}/b_0$, and, on the right-hand side, the leading term is $3\alpha^2b_0b_kx^{k-l}$,
with $b_0=-a_0$ and $a_0^3+1=0$; so, the orders of terms are different for $l\geqslant2$. One proves, analogously, that
$b_1=0$. Therefore, the only solution for all $l\geqslant2$ is $v(x)=-a_0/x^l$. For all $l$, $v(x)$ generates the same explicit
solution $y(t)=-a_0t^{1/3}$. This observation does not work for $l=1$; in this case, however, $\alpha=0$,
and equation~\eqref{eq:dP3v} (even if, instead of $\alpha^2$, one uses a parameter) is not related to the
Painlev\'e equation~\eqref{eq:dP3y-compact}.

Thus, a solution of equation~\eqref{eq:dP3v} with a Laurent expansion at $x=0$ exists if $m_1\geqslant0$.
Section~\ref{sec:2} is devoted to the case $m_1=0$. The case $m_1\geqslant1$ can be studied similarly. Here, we only outline
some key points that are important for the following discussion. The function $v(x)$ cannot have a pole at $x=0$ because the
two other terms in equation~\eqref{eq:dP3v} are bounded; therefore, we can write $v(x)=x^nw(x)$ for some
$n\in\mathbb{Z}_{\geqslant0}$: by the sense of the introduction of the parameter $n$, we suppose that $w(0)\neq0$.
Making this substitution in equation~\eqref{eq:dP3v}, one arrives at equation~\eqref{eq:dP3w} with $m=m_1-2n$. By using
arguments similar to those employed in the previous paragraph for the proof $m_1\geqslant0$, one confirms that the necessary
condition for the existence of a holomorphic at $x=0$ solution of equation~\eqref{eq:dP3w} is $m\in\mathbb{Z}_{\geqslant0}$.
Thus, the direct statement of the theorem is proved.

Conversely, consider equation~\eqref{eq:dP3w} with $n,m\in\mathbb{Z}_{\geqslant0}$. In this case, the leading terms can
always be balanced: since we are looking for the solution with $w(0)\neq0$, the leading terms as $x\to0$ of the two expressions
on the right-hand side of equation~\eqref{eq:dP3w} are $-a_0x^n$ and $-x^m/a_0^2$, whilst the leading term as $x\to0$
of the term on the left-hand side of this equation is $-(q+1)^2a_{q+1}x^q/a_0$, where we assume that $a_{q+1}$ is the second
nonvanishing coefficient in the Taylor expansion of $w(x)$ (the first one is $-a_0$). Therefore, for any given $n$ and $m$,
one can always find an appropriate $q$ to balance the leading terms. (Note that the coefficients $a_1=\dotsb=a_q=0$). Hence,
we see that, for any $a_0\neq0$, we can balance the leading terms, and the subsequent coefficients $a_k$ for $k\geqslant q+1$
of the Taylor expansion of $w(x)$ can be uniquely determined with the help of a recurrence relation that can be deduced from
equation~\eqref{eq:dP3w}. The convergence of such an expansion can be established in a manner similar to that used for the
proof of Lemma~\ref{lem:a-n-estimate}.
\end{proof}
\begin{remark}\label{rem:ALG}
For any given pair $(n,m)\in\mathbb{Z}^2_{\geqslant0}$, Theorem~\ref{th:algebroid} presents the exact construction for
a family (class) of solutions to equation~\eqref{eq:dP3y}, $y(t)=y(t;a_0)$, where $a_0\in\mathbb{C}\setminus\{0\}$: the set
whose elements are such families is denoted by $\mathbb{ALG}(dP3_0)$:\footnote{The subscript $0$ represents the fact that
we consider a special case of equation~\eqref{eq:dp3u} for $a=0$.} moreover, for any algebroid solution of
equation~\eqref{eq:dP3y}, there exists a number $a_0\in\mathbb{C}\setminus\{0\}$ such that this solution belongs to one of
the elements of $\mathbb{ALG}(dP3_0)$.
\hfill$\blacksquare$\end{remark}
\begin{corollary}\label{cor:Q=algebroid}
There exists a one-to-one correspondence between the set of positive rational numbers $(\mathbb{Q}_{>0})$
and $\mathbb{ALG}(dP3_0)$:
\begin{equation}\label{eq:Q-ALG}
\mathbb{Q}_{>0}\rightarrow \mathbb{ALG}(dP3_0),\quad
q\mapsto
y(t)\underset{t\to0}{\sim} -a_0\,t^{1-4\rho},
\quad
a_0\in\mathbb{C}\setminus\{0\}\;\;\text{and}\;\;
2\rho=\frac{1}{1+2q},
\end{equation}
where $y(t)$ is a representative of the corresponding class.
\end{corollary}
\begin{proof}
Define a mapping $f:\;\mathbb{Q}_{>0}\rightarrow\mathbb{ALG}(dP3_0)$ as follows: if $q=\frac{n+1}{m+1}$, with coprime
$n+1$ and $m+1$, then $f(q)\rightarrow(n,m)\rightarrow y(t)$, where $y(t)\in\mathbb{ALG}(dP3_0)$ is
constructed in Theorem~\ref{th:algebroid}.\footnote{With abuse of notation, $y(t)$ is used to denote both a family
of solutions, $y(t)=y(t;a_0)$, to equation~\eqref{eq:dP3y} and the corresponding element of $\mathbb{ALG}(dP3_0)$.}

The mapping $f$ is injective. Consider the behaviour of $y(t)$ as $t\to0$, namely,
$y(t)\underset{t\to0}\sim t^{(n+1)/\alpha-1}w(0)$ (cf. Theorem~\ref{th:algebroid}). Since $\alpha=(m+1+2(n+1))/4$,
we get that the leading branching, $(n+1)/\alpha-1=q/(2q+1)-1$, is different for different $q$.

The mapping $f$ is surjective. According to the construction presented in Theorem~\ref{th:algebroid}, for any
$y(t)\in\mathbb{ALG}(dP3_0)$, one can find a pair of nonnegative integers $(n,m)$ so that a number
$q=(n+1)/(m+1)$ can be defined; the problem, however, is that the numbers $n+1$ and $m+1$ might not be coprime, so that
one can not claim that precisely this solution $y(t)$ corresponds to $f(q)$. We are going to prove that, for
a given $y(t)\in\mathbb{ALG}(dP3_0)$, any pair of nonnegative integers representing the same rational number $q$ is
suitable.

Assume that there exists $p\in\mathbb{N}$ such that $n+1=(p+1)(\tilde{n}+1)$ and $m+1=(p+1)(\tilde{m}+1)$,
where $\tilde{n}+1$ and $\tilde{m}+1$ are coprime.
Denote the solution of equation~\eqref{eq:dP3w} corresponding to the parameters $n$ and $m$ by $w_{n,m}(x)$. Now, making
the change of independent variable $x\to x^{p+1}$ and noting that
\begin{equation}\label{eq:alphatransformation}
\alpha=\frac14(m+1+2(n+1))=(p+1)\tilde{\alpha}=\frac{p+1}{4}(\tilde{m}+1+2(\tilde{n}+1)),
\end{equation}
one proves that $w_{n,m}(x)=w_{\tilde{n},\tilde{m}}(x^{p+1})$, assuming that $w_{n,m}(0)=w_{\tilde{n},\tilde{m}}(0)$.
Using the last equation and relation \eqref{eq:alphatransformation}, one proves that the functions $y(t)$ defined in
Theorem~\ref{th:algebroid} via the functions $w=w_{n,m}$ and $w=w_{\tilde{n},\tilde{m}}$ coincide exactly:
\begin{equation}\label{eq:y=ytilde}
y(t)=x^{n+1-\alpha}w_{n,m}(x)=
t^{\frac{n+1}{\alpha}-1}w_{n,m}\left(t^{\frac{1}{\alpha}}\right)=
t^{\frac{\tilde{n}+1}{\tilde{\alpha}}-1}w_{\tilde{n},\tilde{m}}\left(t^{\frac{p+1}{\alpha}}\right)=
t^{\frac{\tilde{n}+1}{\tilde{\alpha}}-1}w_{\tilde{n},\tilde{m}}\left(t^{\frac{1}{\tilde{\alpha}}}\right)=
\tilde{y}(t).
\end{equation}
Substituting for the function $w_{n,m}(t^{1/\alpha})$ appearing in the second equality  of equation~\eqref{eq:y=ytilde}
the first term of its Taylor expansion (cf. Theorem~\ref{th:algebroid}), one arrives at the asymptotics for $y(t)$ given
in \eqref{eq:Q-ALG}, with $1-4\rho=(n+1)/\alpha-1$. Finally, solving the latter equation for $2\rho$, one finds
\begin{equation}\label{eq:rhoALGgeberal}
2\rho=\frac{m+1}{m+1+2(n+1)}=\frac{1}{1+2q}.
\end{equation}
\end{proof}
\begin{remark}\label{rem:Alg(dP3-0)}
In the geometrical sense, Corollary~\ref{cor:Q=algebroid} states that the space of the algebroid solutions is isomorphic to the
trivial fiber bundle, $\mathbb{Q}_{>0}\times\mathbb{C}\setminus\{0\}$, where the base is $\mathbb{Q}_{>0}$, and the
cylinder, $\mathbb{C}\setminus\{0\}$, is the fiber defining the initial values of the solutions.
The constructed mapping allows one to pull back all structures to $\mathbb{ALG}(dP3_0)$; in particular, the ordering,
the topology, and the multiplicative Abelian group that are defined on $\mathbb{Q}_{>0}$. Consider, say, the group structure:
for $k=1, 2, 3$, let $y_k(t)\in\mathbb{ALG}(dP3_0)$, with the branching $\rho_k$.
We define the group multiplication $\ast$
in $\mathbb{ALG}(dP3_0)$ as follows: $y_1\ast y_2=y_3$ iff
\begin{equation}\label{eq:ALG(dP3-0)group}
\frac{1}{2\rho_3}-1=\frac12\left(\frac{1}{2\rho_1}-1\right)\left(\frac{1}{2\rho_2}-1\right).
\end{equation}
With the help of the last formula in equation~\eqref{eq:Q-ALG}, it is straightforward to check that the group
$\mathbb{Q}_{>0}$, with the usual multiplication of the rational numbers, and $\mathbb{ALG}(dP3_0)$, with the multiplication
defined above, are isomorphic. Note that the solution $y(t)$ which corresponds to the function $H(r)$
(cf. Section~\ref{sec:introduction}, equation~\eqref{eq:y-Hat0}) plays the role of the group unit in $\mathbb{ALG}(dP3_0)$.
A more interesting group that also acts in the fibers of the bundle is studied in Section~\ref{sec:Coxeter}.
\hfill $\blacksquare$\end{remark}
\begin{remark}\label{rem:rho-algebraic}
Algebroid solutions of equation~\eqref{eq:dp3u} have asymptotics as $\tau\to0$ that are similar to those of the algebroid
solutions $y(t)$:
\begin{equation*}\label{eq:u-algebroid0asympt}
u(\tau)\underset{\tau\to0}{\sim} c\,\tau^{1-4\rho},
\quad\text{with}\quad
c\in\mathbb{C}\setminus\{0\}\;\;\text{and}\;\;
2\rho=\frac{1}{1+2q}.
\end{equation*}
The notation $1-4\rho$ for the branching of the algebroid solutions is used to match with the result for asymptotics
of the general solution of equation~\eqref{eq:dp3u} stated in Appendix~\ref{app:asympt0}, Theorem~\ref{th:B1asympt0}.
\hfill $\blacksquare$\end{remark}
The remainder of this section is devoted to the study of two ``boundary'' sets of the algebroid solutions corresponding
to the pairs $(0,m)$ and $(n,0)$, respectively:
\begin{align*}
2\rho&=\frac{m+1}{m+3},\quad
m=0,1,2,\ldots,
\qquad
&m-\text{series},\\
2\rho&=\frac{1}{2n+3},\quad
n=1,2,\ldots,
\qquad
&n-\text{series}.
\end{align*}
We call them the algebroid solutions of the $m$- and $n$-series, respectively. Since
$1-4\rho=\tfrac{1-m}{m+3}$ for the $m$-series and $1-4\rho=\tfrac{2n+1}{2n+3}$ for the $n$-series, the corresponding
solutions $u(\tau)$ and $y(t)$ can be distinguished by the condition on the initial data, namely,
$y(0)=u(0)=\infty$ for $m>2$ and $y(0)=u(0)=0$ for the $n$-series. In this sense, the first two solutions of the $m$-series
are special: the one which corresponds to $m=0$ ($\rho=1/6$) has the same behaviour as the solutions of the $n$-series
for which $y(0)=u(0)=0$, and can, in principle, be treated as the only solution that belongs to both series;
the second solution corresponding to $m=1$ has a finite, nonvanishing initial value at $t=\tau=x=0$,
and is a meromorphic function in $\mathbb{C}$.
\begin{remark}\label{rem:defHnHm}
In the study of the coefficients of the Taylor expansion for the function $v(x)$, the parameter $\alpha^2$
in equation~\eqref{eq:dP3v} gives rise to slightly cumbersome expressions for the coefficients. It is convenient,
therefore, to rescale this equation, and to introduce, in lieu of $v(x)$, the normalized functions
$H_{-m}(\hat{x})$ and $H_n(\hat{x})$. In the notation of this section,
$H_0(\hat{x})=H(r)$, with $\hat{x}=r$, where $H(r)$ is the function studied in Section~\ref{sec:2}.
The definitions of the functions $H_p$ for $p\neq0$ read:
\begin{align*}
v(x)&=c_{-}^{\frac{m}{3}}H_{-m}(\hat{x}),\quad
x=c_{-}\hat{x},\quad
\left(\frac{4}{m+3}\right)^2c_{-}^{\frac{m+3}{3}}=2,\\
v(x)&=c_{+}^{\frac{2n}{3}}\hat{x}^nH_{n}(\hat{x}),\quad
x=c_{+}\hat{x},\quad
\left(\frac{4}{2n+3}\right)^2c_{+}^{\frac{2n+3}{3}}=1.
\end{align*}
Thus, $H_p(\hat{x})$, $p\in\mathbb{Z}$, are defined as meromorphic functions in $\mathbb{C}$ with $H_p(0)=-a_0\neq0$.
(These functions depend on the initial data, so that a more complete notation should be $H_p(\hat{x};a_0)$.) They
satisfy the following second-order ODEs:
\begin{align}
\left(\frac{\hat{x}H'_{-m}(\hat{x})}{H_{-m}(\hat{x})}\right)'&=
2\left(H_{-m}(\hat{x})-\frac{\hat{x}^m}{(H_{-m}(\hat{x}))^2}\right),
&m\geqslant1,\label{eq:Hm}\\
\left(\frac{\hat{x}H'_{n}(\hat{x})}{H_{n}(\hat{x})}\right)'&=\hat{x}^nH_{n}(\hat{x})-\frac{1}{(H_{n}(\hat{x}))^2},
&n\geqslant0.\label{eq:Hn}
\end{align}
Note that, according to our normalization, the function $H_0(\hat{x})$ satisfies equation~\eqref{eq:Hn}, as do equations of
the $n$-series.
\hfill $\blacksquare$\end{remark}
According to Theorem~\ref{th:algebroid}, in a neighbourhood of $\hat{x}=0$, the functions $H_p(\hat{x})$ can be developed
into Taylor series:
\begin{equation}\label{eq:HpTaylor}
H_p(\hat{x})=-a_0+\sum_{k=1}^{\infty}a_k^p\hat{x}^k.
\end{equation}
Note that the superscript $p$ in $a_k^p$ denotes the label of the corresponding function $H_p$, whilst
$H_p(0)=-a_0$ for all $p$; therefore, in the formulae below, $a_0^n=(a_0)^n$.
\begin{proposition}\label{prop:coeffsHm}
For $m\in\mathbb{N}$,
\begin{equation}\label{eq:coeffsHm}
a_{k}^{-m}=(-1)^{k+1}(k+1)a_0^{k+1}+\sum_{l=1}^{\lfloor\frac{k+1}{m+2}\rfloor}r_l^{-m}a_0^{k+1-(m+3)l},
\quad k\geqslant1,
\end{equation}
where the numbers $r_l^{-m}\in\mathbb{Q}\setminus\{0\}$, and $\lfloor\ast\rfloor$ denotes the floor of a real number.
\end{proposition}
\begin{remark}\label{rem:conjectureHmHncoeffs}
The proof of Proposition~\ref{prop:coeffsHm} is similar to the analogous one for the function $H(r)=H_0(\hat{x})$,
with $\hat{x}=r$, given in Subsection~\ref{subsec:H0coeffs},
Proposition~\ref{prop:a-n-ansatz}.
Here, manipulations with the sign of $a_0$ do not help to make all the numbers $r_l^{-m}>0$.
\hfill $\blacksquare$\end{remark}
\begin{proposition}\label{prop:coeffsHn}
For $n\in\mathbb{N}$,
\begin{equation}\label{eq:coeffsHn}
a_k^n=\sum_{(m_j,l_j)\in\mathcal{P}_k^n}\gamma_{m_j,l_j}^n a_0^{m_j+1-2l_j},
\end{equation}
where $\mathcal{P}_k^n$ is the set of pairs of nonnegative integers $(m_j,l_j)$ that represent all possible partitions of
\begin{equation}\label{eqPhk-partition}
k=(n+1)m_j+l_j,
\quad\text{where}\quad
l_j\in\{0,1,\ldots,m_j+1\},
\end{equation}
with the numbers $\gamma_{m_j,l_j}^n\in\mathbb{Q}\setminus\{0\}$, and $\gamma_{0,1}^n=1$.
\end{proposition}
\begin{remark}\label{rem:coeffsHn}
As a matter of fact, the set $\mathcal{P}_k^n$ contains very few elements:
\begin{equation}\label{eq:P-cardinality}
|\mathcal{P}_k^n|=\left\lfloor\frac{k+2n+2}{n+1}\right\rfloor-\left\lfloor\frac{k+2n+2}{n+2}\right\rfloor.
\end{equation}
If the set is empty, then the corresponding coefficient $a_k^n=0$.

As an application of equation~\eqref{eq:P-cardinality}, $|\mathcal{P}_1^n|=1$ for $n\geqslant1$;
in fact, $a_1^n=1/a_0$ for all $n$. On the other hand, for $n\geqslant2$,
$|\mathcal{P}_k^n|=0$, $k=2,\dotsc,n$; thus, $a_2^n=\dotsb=a_n^n=0$ for $n\geqslant2$.
Concurrently, for $n\geqslant1$, $|\mathcal{P}_{n+1}^n|=1$, so that $a_{n+1}^n\neq0$.
As another example, consider, say, $|\mathcal{P}_{37}^3|=\lfloor\tfrac{45}{4}\rfloor-\lfloor\tfrac{45}{5}\rfloor=11-9=2$;
in fact, $37=4\times9+1=4\times8+5$, thus $a_{37}^3=\gamma_{9,1}^3a_0^8+\gamma_{8,5}^3a_0^{-1}$.

For $n=1,2,3,4$, we found the sequences $|P_k^n|$ in OEIS \cite{OEIS4}. Actually, our sequences do not include the
first few members of the sequences in OEIS because these sequences have different combinatorial definitions.
For $n=5$, we did not find the corresponding sequence in OEIS. It seems that our combinatorial definition of the
sequences $|P_k^n|$ might be new.
\hfill $\blacksquare$\end{remark}
\begin{remark}\label{rem:ypDefinition-ypHp-a0-dependence}
For every function $H_p(\hat{x})$, there corresponds a solution to equation~\eqref{eq:dP3y} (cf.~\eqref{eq:dP3y-compact})
which is denoted by $y_p(t)$. Amalgamating the consecutive transformations relating equations~\eqref{eq:dP3y-compact},
\eqref{eq:Hm}, and ~\eqref{eq:Hn}, we find that
\begin{align}
ty_{-m}(t)&=c_{-}^{\frac{m}{3}}t^{\frac{4}{m+3}}H_{-m}\left(c_{-}^{-1}t^{\frac{4}{m+3}}\right),\label{eq:ym-definition}\\
ty_n(t)&=c_{+}^{-\frac{n}{3}}t^{\frac{4(n+1)}{2n+3}}H_n\left(c_{+}^{-1}t^{\frac{4}{2n+3}}\right),
\label{eq:yn-definition}
\end{align}
where $y_{-m}(t)$, $m\in\mathbb{Z}_{\geqslant0}$, and $y_{n}(t)$, $n\in\mathbb{N}$, denote the solutions of the $m$- and the
$n$-series, respectively.

Sometimes, it is imperative to explicitly indicate the dependence of our functions on the parameter $a_0$; in such cases,
we write
\begin{equation*}
y_p(t)=y_p(t;a_0),\quad
H_p(\hat{x})=H_p(\hat{x};a_0),\qquad
p\in\mathbb{Z}.
\end{equation*}
\hfill $\blacksquare$\end{remark}
\begin{corollary}\label{cor:symmetryHmym}
For $m\in\mathbb{N}$ and $q,l\in\mathbb{Z}$,
\begin{gather}
H_{-m}\left(\hat{x}\me^{-\frac{2\pi\mi q}{m+3}};a_0\me^{\frac{2\pi\mi q}{m+3}}\right)=
\me^{\frac{2\pi\mi q}{m+3}}H_{-m}\left(\hat{x};a_0\right),\label{eq:symmetry-Hm}\\
y_{-m}(t;a_0)=\me^{\mi\varphi_{m,q,l}}y_{-m}\left(t\me^{\mi\varphi_{m,q,l}};a_0\me^{\frac{2\pi\mi q}{m+3}}\right),\qquad
\varphi_{m,q,l}=\frac{\pi}{2}\big(l(m+3)-q\big).
\label{eq:symmetry-ym}
\end{gather}
\begin{proof}
The first symmetry ~\eqref{eq:symmetry-Hm} is proved via a straightforward calculation with the help of
equations~\eqref{eq:HpTaylor} and \eqref{eq:coeffsHm}. The second symmetry ~\eqref{eq:symmetry-ym}
also follows by means of a direct calculation using the definition of $y_{-m}(t)$ in \eqref{eq:ym-definition} and the
first symmetry for $H_{-m}(\hat{x})$.
\end{proof}
\end{corollary}
\begin{corollary}\label{cor:symmetryHnyn}
For $n\in\mathbb{Z}_{\geqslant0}$ and $q,l\in\mathbb{Z}$,
\begin{gather}
H_n\left(\hat{x}\me^{\frac{4\pi\mi q}{2n+3}};a_0\me^{\frac{2\pi\mi q}{2n+3}}\right)=
\me^{\frac{2\pi\mi q}{2n+3}}H_n(\hat{x};a_0),\label{eq:symmetry-Hn}\\
y_n(t;a_0)=\me^{\mi\psi_{n,q,l}}y_n\left(t\me^{\mi\psi_{n,q,l}};a_0\me^{\frac{2\pi\mi q}{2n+3}} \right),\qquad
\psi_{n,q,l}=\pi q+\frac{\pi}{2}l(2n+3).
\label{eq:symmetry-yn}
\end{gather}
\end{corollary}
\begin{proof}
The function $H_0(\hat{x})=H(r)$, with $r=\hat{x}$, formally belongs to $m$-series; however, its intermediate position
between the $m$- and the $n$-series diminishes its level of symmetry, so that it has the same type of
symmetry as the $n$-series. The formal proof of this fact follows the same line of reasoning as for the $n$-series
(see below); however, it requires another formula for the coefficients $a_k^0=a_k$ given in
equation~\eqref{eq:a-n-ansatz}. Here, we outline the proof for a generic member of the $n$-series.

Consider equation~\eqref{eqPhk-partition}. It can be rewritten in the following form:
\begin{equation*}
m_j+1-2l_j=(2n+3)m_j-(2k-1).
\end{equation*}
This equation implies (cf.~\eqref{eq:coeffsHn}) that, for all $k\geqslant0$,
\begin{equation*}
a_k^n\big(a_0\me^{\frac{2\pi\mi p}{2n+3}}\big)\left(\hat{x}\me^{\frac{4\pi\mi p}{2n+3}}\right)^k=
\me^{\frac{2\pi\mi p}{2n+3}}a_k^n(a_0)\hat{x}^k,
\end{equation*}
where we write $a_k^n=a_k^n(a_0)$. Now, equation~\eqref{eq:symmetry-Hn} follows from the Taylor
series for $H_n(\hat{x})$ (cf. \eqref{eq:HpTaylor}), and equation~\eqref{eq:symmetry-yn} is obtained from the
first one upon invoking the definition of $y_n(t)$ given in ~\eqref{eq:yn-definition}.
\end{proof}
\begin{remark}\label{rem:symmetry}
Applying Corollaries~\ref{cor:symmetryHmym} and \ref{cor:symmetryHnyn} and taking into consideration that
the functions $H_p(\hat{x};a_0)$ ($p=-m,n$) are defined via convergent series whose coefficients are rational functions
of $a_0$, it follows that $H_p(\hat{x};a_0\me^{2\pi\mi})= H_p(\hat{x};a_0)$. The dependence of the functions
$y_p(t;a_0)$ on $a_0$  is defined via the functions $H_p$; therefore, $y_p(t;a_0\me^{2\pi\mi})=y_p(t;a_0)$.
\hfill $\blacksquare$\end{remark}
\begin{proposition}\label{prop:yn-series}
\begin{equation}\label{eq:yn-series}
ty_{n}(t)=\left(\frac{2n+3}{4}\right)^2\,z^{\frac{n+1}{2n+3}}\,\sum_{q=0}^{2n+2}z^{\frac{q}{2n+3}}\,f_q^n(z)
=\left(\frac{2n+3}{4}\right)^2\,\sum_{q=0}^{2n+2}z^{\frac{q}{2n+3}}\,\hat{f}_q^n(z),\quad
z=\left(\frac{4}{2n+3}\right)^6t^4,
\end{equation}
where the functions $f_q^n(z)$, $q=0,1,\dotsc,2n+2$, are holomorphic at $z=0$,
\begin{equation*}
f_q^n(z)=\sum_{l=0}^{\infty}a_{l,q}^n\,z^l,\quad
a_{l,q}^n:= a_{l(2n+3)+q}^n\in\mathbb{C},\quad
q=0,1,\ldots,2n+2;
\end{equation*}
furthermore,
$f_q^n(0)\neq0$ iff $q\in\{0,1,n+1,n+2,n+3,2n+2\}$; moreover, $\hat{f}_k^n(z)=zf_{k+n+2}^n(z)$ for $k=0,1,\dotsc,n$,
$\hat{f}_k^n(z)=f_{k-n-1}^n(z)$ for $k=n+1,n+2,\ldots,2n+2$, and $\hat{f}_q^n(0)\neq0$ iff $q\in\{n+1,n+2,2n+2\}$.
\end{proposition}
\begin{proof}
The definition of the function $y_n(t)$ given in \eqref{eq:yn-definition} and the series for $H_n(\hat{x})$ in
\eqref{eq:HpTaylor} imply, after a rearrangement, the result presented in ~\eqref{eq:yn-series}.
The properties of the functions $f_q^n(z)$ follow from Proposition~\ref{prop:coeffsHn} and Remark~\ref{rem:coeffsHn}.
\end{proof}
\begin{remark}
It is important to note that, when using for the solution $y_n(t)$ the representation~\eqref{eq:yn-series}, and for
$y_{-m}(t)$ the analogous representation given in equation~\eqref{eq:ym-series} below, the following rule
for the consistency of the branches is assumed:
$z^{\frac{q}{p}}=|\big(\frac{4}{2n+3}\big)^6t|^{\frac{4q}{p}}\me^{\mi\frac{4q}{p}\arg{t}}$.
\hfill $\blacksquare$\end{remark}
\begin{proposition}\label{prop:ym-series}
Depending on the value of $(m+3)\!\!\!\mod\!4$, define the natural numbers $p_k$ for $k=1,2,4$ as follows:
\begin{equation*}
p_4:=m+3\equiv2n+1,\quad
2p_2:=m+3\equiv2(2n+1),\;\;
\text{and}\;\;
4p_1:=m+3\equiv4n,\quad
n\in\mathbb{N}.
\end{equation*}
Then, for $m=4p_k/k-3$, $y_{-m}(t)$ inherits the representation
\begin{equation}\label{eq:ym-series}
ty_{-m}(t)=2\left(\frac{p_k}{k}\right)^2(z_k)^{\frac{1}{p_k}}\,\sum_{q=0}^{p_k-1}(z_k)^{\frac{q}{p_k}}\,f_q^{-m}(z_k)
=2\left(\frac{p_k}{k}\right)^2\,\sum_{q=0}^{p_k-1}(z_k)^{\frac{q}{p_k}}\,\hat{f}_q^{-m}(z_k),\quad
z_k=(c_{-})^{-p_k}t^k,
\end{equation}
where the functions $f_q^{-m}(z_k)$, $q=0,1,\ldots,p_k-1$, are holomorphic at $z_k=0$,
\begin{equation*}
f_q^{-m}(z_k)=\sum_{l=0}^{\infty}a_{l,q}^{-m}\,(z_k)^l,\quad
a_{l,q}^{-m}:=a_{lp_k+q}^{-m}\in\mathbb{C},\quad
q=0,1,\ldots,p_k-1;
\end{equation*}
furthermore,
$f_q^{-m}(0)\neq0$ for all $q=0,1\ldots,p_k-1$; moreover, $\hat{f}_0^{-m}(z_k)=z_kf_{p_{k}-1}^{-m}(z_k)$,
$\hat{f}_j^{-m}(z_k)=f_{j-1}^{-m}(z_k)$ for $j=1,2,\ldots,p_k-1$,
$\hat{f}_q^{-m}(0)\neq0$ for $q=1,2,\ldots,p_k-1$, and $\hat{f}_0^{-m}(0)=0$.
\end{proposition}
\begin{proof}
The proof is very similar to the proof of Proposition~\ref{prop:yn-series}:
combination of the formulae \eqref{eq:ym-definition} and \eqref{eq:HpTaylor} followed by the rearrangement
presented in equation~\eqref{eq:ym-series}; the only difference between the proofs is that,
here, one has to take into account the divisibility of $m+3$ by $2$ and $4$. The properties of the functions
$f_q^n(z_k)$ are deduced from the properties of the coefficients $a_k^n$ formulated in Proposition~\ref{prop:coeffsHm}.
\end{proof}
\begin{remark}\label{rem:pk-zk}
The natural numbers $p_k$ and the variables $z_k$ defined in Proposition~\ref{prop:ym-series} can be explicitly
written as follows:
\begin{equation*}
\frac{p_k}{k}=\frac{m+3}{4},\qquad
z_4=\frac{2^9t^4}{(2n+1)^6},\quad
z_2=\frac{2^{3/2}t^2}{(2n+1)^3},\quad
z_1=\frac{t}{2^{3/4}n^{3/2}}.
\end{equation*}
\hfill $\blacksquare$\end{remark}
\begin{proposition}\label{prop:meromorphic-fpq}
The analytic continuations of the functions $f_q^n(z)$ and $f_q^{-m}(z_k)$ (cf. Propositions~\ref{prop:yn-series}
and \ref{prop:ym-series}, respectively) are meromorphic on $\mathbb{C}$.
\end{proposition}
\begin{proof}
Any one of the functions $y_{r}(t)$ introduced in Propositions~\ref{prop:yn-series} and \ref{prop:ym-series} admit the
following representation in a neighbourhood of $z=0$ ($z_k=0$):
\begin{equation}\label{eq:y-meromorphic}
ty(t)=c^2\sum_{q=0}^{p-1}z^{\frac{q}{p}}\hat{f}_q(z),\qquad
p\in\mathbb{N},\;\;
c\in\mathbb{Q}\setminus\{0\},
\end{equation}
where the functions $\hat{f}_q(z)$ are holomorphic in a neighbourhood of $z=0$.

For $p\geqslant1$ and $k=0,1,\ldots,p-1$,
define the functions $y_k(t):=y(te^{2\pi in_k})$, where the winding number $n_k=k(\text{mod}\,p)$, and the column vectors
\begin{equation*}\label{eq:vectorsYF}
Y(t)=(y_0(t),y_1(t),\ldots,y_{p-1}(t))^T\quad\text{and}\quad
F(z)=(\hat{f}_0(z),z^{1/p}\hat{f}_1(z),\ldots,z^{(p-1)/p}\hat{f}_{p-1}(z))^T,
\end{equation*}
where $T$ denotes transposition.
Then, equation~\eqref{eq:y-meromorphic} can be rewritten in the matrix form
\begin{equation}\label{eq:y(t)-matrix}
tY(t)=AF(z),
\end{equation}
where
\begin{equation*}
A \! = \!
\begin{pmatrix}
1 & 1 & 1& \hdotsfor{1} & 1 \\
1 & \varepsilon_p & \varepsilon_p^2 & \dots & \varepsilon_p^{p-1} \\
1 & \varepsilon_p^2 & \varepsilon_p^4& \hdotsfor[1]{1} & \varepsilon_p^{2(p-1)} \\
\vdots & \vdots & \vdots & \ddots & \vdots \\
1 & \varepsilon_p^{p-1} & \varepsilon_p^{2(p-1)} & \hdots & \varepsilon_p^{(p-1)^2}
\end{pmatrix},
\qquad \varepsilon_p \! = \! \me^{\frac{2 \pi \mi}{p}}, \quad p \! \in \! \mathbb{N}.
\end{equation*}
The $p\times p$ matrix $A$ is invertible because $\det{A}=p^{\frac{p}{2}}e^{-\frac{\pi\mi(p-1)(p-2)}{4}}\neq0$
(see \cite{FadSom1952}, Problem ${}^{*}299$); therefore,
\begin{equation}\label{eq:FinY}
F(z)=tA^{-1}Y(t),\qquad
t=\left(\frac{2n+3}{4}\right)^{3/2}z^{1/4}.
\end{equation}
For the functions $y_{-m}(t)$ (cf. Proposition~\ref{prop:yn-series}), $t$ is given via the inversion of the formulae for $z_k$
(cf. Remark~\ref{rem:pk-zk}).

Equation~\eqref{eq:FinY} defines the analytic continuation of the vector-valued function $F(z)$ on $\mathbb{C}$; therefore,
the only singularities of the components of $F(z)$ are poles, i.e., the only singular points of the functions $\hat{f}_q(z)$
on $\mathbb{C}$ are poles.
\end{proof}
\begin{remark}\label{rem:algebroid}
Proposition~\ref{prop:yn-series} (resp., Proposition~\ref{prop:ym-series}) implies that solutions of the $n$-series
(resp., $m$-series) are single-valued on the Riemann surface of $w^{2n+3}=z$ (resp., $w^{p_k}=z_k$). Below, we show
how to meet the formal definition of the algebroid function (see, for example, \cite{Steinmetz2017}).
\hfill $\blacksquare$\end{remark}
\begin{definition}\label{def:F(z)}
Consider the function
\begin{equation}\label{eq:g(z)Fconstruction}
g(z)=\sum_{q=0}^{\nu-1} \omega^q\tilde{f}_q(z),\qquad
\omega^{\nu}=z,
\end{equation}
where the functions $\tilde{f}_q(z)$ are holomorphic at $z=0$.
For $k=1,2,\ldots,\nu$, define the functions $f_q^k(z)$ holomorphic at $z=0$ via the identity
\begin{equation*}
(g(z))^k=\left(\sum_{q=0}^{\nu-1} \omega^q\tilde{f}_q(z)\right)^k=:
\sum_{q=0}^{\nu-1} \omega^qf_q^k(z).
\end{equation*}
Define the $\nu\times\nu$ matrix $F_{\nu}(z)=\{f_q^k(z)\}$, where  $q=0,1,\ldots,\nu-1$ enumerates the columns and
$k=1,2,\ldots,\nu$ enumerates the rows.
\end{definition}
\begin{remark}\label{rem:detF(x)calculation}
Although the definition of the matrix $F_{\nu}(z)$ looks simple enough, the exact calculation of its determinant
appears to be a rather complicated problem. The determinants of $F_k(z)$ for $k=1,\ldots,7$ can be calculated almost
immediately; however, the calculation of the determinant for the matrix $F_8(z)$ takes roughly $340\text{s}$:
in the factorized form over $\mathbb{Z}[\tilde{f}_0,\tilde{f}_1,\dotsc,\tilde{f}_{\nu-1}]$, the polynomial has four factors,
and the size of one of them, namely, a polynomial of degree $12$, exceeds $10^6$ symbols, and was not printable. We
did not succeed in calculating the determinant of $F_9(z)$ because \textsc{Maple}, after a few hours
of computations, was incapable of allocating enough memory on a computer equipped with 16GBs of RAM, not even for the calculation
of one $8\times8$ minor of $F_9(z)$ (almost the entirety of the RAM was occupied together with part of the hard drive).
We present, for example, the explicit formula for $\det(F_3(z))$:
\begin{equation}\label{eq:detF3}
\det(F_3(z))=\Big((\tilde{f}_1(z))^3-z(\tilde{f}_2(z))^3\Big)
\Big((\tilde{f}_2(z))^3z^2+\Big((\tilde{f}_1(z))^3-3\tilde{f}_0(z)\tilde{f}_1(z)\tilde{f}_2(z)\Big)z+(\tilde{f}_0(z))^3\Big).
\end{equation}
Using the fact that $\tilde{f}_k(z)$ are holomorphic at $z=0$, it is easy to prove that $\det(F_3(z))$ is identically
non-vanishing.
\hfill $\blacksquare$\end{remark}
\begin{proposition}\label{prop:F(x)degree}
By regarding the functions $\tilde{f}_q(z)$ as transcendental elements over $\mathbb{Z}$
rather than functions of $z$, denote them, in this sense, as the variables $\tilde{f}_q$.
Consider $\det(F_{\nu}(z))$
as a polynomial of $z$ over $\mathbb{Z}[\tilde{f}_0,\tilde{f}_1,\ldots,\tilde{f}_{\nu-1}]$. Then,
\begin{gather}
\deg{\det(F_{\nu}(z))}\leq\frac{\nu(\nu-1)}{2},\nonumber\\
\det(F_{\nu}(z))\underset{z\to\infty}{=}(\tilde{f}_{\nu-1})^{\frac{\nu(\nu+1)}{2}}(-z)^{\frac{\nu(\nu-1)}{2}}
+\mathcal{O}\left(z^{\frac{(\nu-2)(\nu+1)}{2}}\right).\label{eq:detF(z)infty}
\end{gather}
\end{proposition}
\begin{proof}
By the definition of the matrix $F_{\nu}(z)$, the elements of the $k$th row are polynomials in $z$ with degrees less
than or equal to $k-1$; therefore, the degree of the polynomial $\det(F_{\nu}(z))$ cannot be greater than
$0+1+\dotsb+\nu-1$.
In fact, the highest degree can only be attained by one product of the elements forming the determinant, that is,
the product of the leading terms of the polynomials on the main off-diagonal (listed, successively, from the
upper-right corner to the bottom-left),
\begin{equation*}\label{eq:highestOFFdiagonalsF(z)}
\tilde{f}_{\nu-1},\; (\tilde{f}_{\nu-1})^2z,\;\ldots,\; (\tilde{f}_{\nu-1})^{\nu}z^{\nu-1},
\end{equation*}
the product of which, with the corresponding sign $(-1)^{\frac{\nu(\nu-1)}{2}}$, represents the leading term of the polynomial
$\det(F_{\nu}(z))$.
\end{proof}
\begin{conjecture}\label{con:F-reducibility}
The polynomial $\det(F_{\nu}(z))$, $\nu\geqslant3$, is always reducible over
$\mathbb{Z}[\tilde{f}_0,\tilde{f}_1,\ldots,\tilde{f}_{\nu-1}]$, with the number of factors equal to the number of divisors
of $\nu$ (including $1$ and the number itself). One of the factors is a polynomial of $z$ with degree $\nu-1$.
\end{conjecture}
Equation~\eqref{eq:detF3} in an illustration of Conjecture~\ref{con:F-reducibility}.
\begin{lemma}\label{lem:detF(z)small-z}
\begin{align}
&\det(F_{\nu}(0))=(\tilde{f}_0(0))^{\nu}(\tilde{f}_1(0))^{\frac{\nu(\nu-1)}{2}},\label{eq:detF(0)}\\
&\det(F_{\nu}(z))\underset{z\to0}{=}(-1)^{\nu+1}(\tilde{f}_1(0))^{\frac{\nu(\nu+1)}{2}}z+\mathcal{O}\left(z^2\right),\quad
\tilde{f}_0(0)=0,\label{eq:as-detF(z)|f0=0}
\end{align}
\begin{multline}\label{eq:as-detF(z)|f0-fn=0}
\det(F_{2n+3}(z))\underset{z\to0}{=}(-1)^{\frac{(n+1)(3n+4)}{2}}(\tilde{f}_{n+1}(0))^{(n+2)(2n+3)}z^{(n+1)^2}
+\mathcal{O}\left(z^{(n+1)^2+1}\right),\quad
n\in\mathbb{Z}_{\geqslant0},\\
\tilde{f}_0(0)=\dotsb=\tilde{f}_n(0)=0.
\end{multline}
\end{lemma}
\begin{proof}
Consider the matrix $F_\nu(0)$. Its first column consists of powers of $\tilde{f}_0(0)$, that is, $(\tilde{f}_0(0))^k$,
$k=1,\dotsc,\nu$.
Remove $\tilde{f}_0(0)$ from the first column so that it appears as a factor of the determinant;
then, the $(1,1)$-element of the resulting matrix is equal to $1$.
Multiplying the first row by proper powers of $\tilde{f}_0(0)$ and subtracting them, successively, from the other rows,
we get a first column consisting of zeros, with the exception of the $(1,1)$-element which equals $1$. It is clear that
the resulting determinant is equal to its minor obtained by deleting the first column and the first row:
this minor is equal to the determinant of the derived $(\nu-1)\times(\nu-1)$ matrix.
The first column of this newly-obtained matrix consists of the elements $\binom{k}{1}(\tilde{f}_0(0))^k\tilde{f}_1(0)$,
$k=1,\dotsc,\nu-1$; in particular, the first element is $\tilde{f}_0(0)\tilde{f}_1(0)$. Remove this factor from the column
and obtain a determinant whose first column consists of the terms $\binom{k}{1}(\tilde{f}_0(0))^{k-1}$: the first element
is equal to $1$.
Multiplying the rows of this determinant by proper powers of $\tilde{f}_0(0)$ and subtracting them successively
from the subsequent rows, one obtains a first column with $1$ as its first element and whose remaining elements are all
equal to $0$; thus, the transformed determinant is equal to the $(\nu-2)\times(\nu-2)$ minor that is obtained by
deleting the first row and the first column.
The first column of the $(\nu-2)\times(\nu-2)$ determinant derived in the previous step consists of the elements
$\binom{k}{2}(\tilde{f}_0(0))^{k-1}(\tilde{f}_1(0))^2$, $k=2,\dotsc,\nu-1$. All the terms containing $\tilde{f}_2(0)$
that were in the third column of the original determinant are now cancelled as a result of the previous subtractions
and certain identities for the binomial coefficients. The first element of this column, $\tilde{f}_0(0)\tilde{f}_1(0)^2$,
is now removed from the determinant, and it combines with the factors $\tilde{f}_0(0)$ and $\tilde{f}_0(0)\tilde{f}_1(0)$
obtained in the previous two steps. Hence, this procedure undergoes $\nu$ steps, and it results in an overall
multiplicative factor equal to
$\tilde{f}_0(0)\pmb{\cdot}\tilde{f}_0(0)\tilde{f}_1(0)\pmb{\cdot}\dotsb\pmb{\cdot}\tilde{f}_0(0)(\tilde{f}_1(0))^{\nu-1}$.

For the case $\tilde{f}_0(0)=0$, let $\tilde{f}_0(z)=zg_0(z)$ for some function $g_0(z)$ that is holomorphic at $z=0$.
If we recall
the construction of the matrix $F_{\nu}(z)$, then it becomes clear that successive powers of $z$ appear because in products
of the type $\prod_{j\leqslant k} \omega^j\tilde{f}_j$ the sums of indices over $j$ become greater than $\nu$, $2\nu$, etc.
Therefore, expanding
the holomorphic functions $\tilde{f}_j(z)$ into Taylor series, it is apparent that the smallest power of $z$ is generated by
the products with the smallest sums of indices. It is evident that there is only one term in $\det(F_{\nu}(z))$ with
this property, namely, it is the term that appears as the product of the successive powers of $\tilde{f}_1(0)$,
which are contained in the matrix elements that lie on the next line above the main diagonal, and the term
$z(\tilde{f}_1(0))^{\nu}$, which is the only term containing the first power of $z$ that is located at the bottom-left corner
of the matrix $F_{\nu}(z)$:
$\tilde{f}_1(0)\pmb{\cdot}(\tilde{f}_1(0))^2\pmb{\cdot}\dotsb\pmb{\cdot}(\tilde{f}_1(0))^{\nu-1}\pmb{\cdot}
z(\tilde{f}_1(0))^{\nu}$. The parity of this term is equal to the parity of the permutation $\nu,1,2,\dotsc,\nu-1$.

The proof of the asymptotics~\eqref{eq:as-detF(z)|f0-fn=0} is quite similar to the previous proof for
\eqref{eq:as-detF(z)|f0=0}. In this case, $\nu=2n+3$.
Set $\tilde{f}_0(z)=zg_0(z),\ldots, \tilde{f}_n(z)=zg_n(z)$, and employ a \emph{gedankenexperiment} by associating the
remaining functions $\tilde{f}_k(z)$ as corresponding to power of $\omega$, that is, $\tilde{f}_k(z)\to\omega^k\tilde{f}_k(z)$.
In this manner, we understand that the minimal power of $z$ is given by only one entry of the determinant, which consists of
the product of the terms
\begin{equation*}
\tilde{f}_{n+1}(0)\pmb{\cdot}(\tilde{f}_{n+1}(0))^2\pmb{\cdot}(\tilde{f}_{n+1}(0))^3z\pmb{\cdot}(\tilde{f}_{n+1}(0))^4z
\pmb{\cdot}\dotsb\pmb{\cdot}(\tilde{f}_{n+1}(0))^{2n+1}z^{n}\pmb{\cdot}(\tilde{f}_{n+1}(0))^{2n+2}z^{n}
\pmb{\cdot}(\tilde{f}_{n+1}(0))^{2n+3}z^{n+1}.
\end{equation*}
These terms
are the entries of the matrix elements in the successive rows $1,2,\ldots,2n+3$, but in the `mixed'
columns $n+2,2n+3,n+1,2n+2,n,2n+1,\ldots,2,n+3,1$. This permutation consists of $(n+1)(3n+4)/2$ transpositions.
The product equals $(\tilde{f}_{n+1}(0))^{1+2+\dotsb+2n+3}z^{2(0+1+\dotsb+n)+n+1}$.
\end{proof}
\begin{proposition}\label{prop:polynomialEquation}
There exist meromorphic functions $g_k^p(z)$, $z\in\mathbb{C}$, such that the functions $y_{p}(t)$ satisfy the
polynomial equations
\begin{equation}\label{eq:polynom-yp}
\sum_{k=1}^{2n+2} g_k^n(z)(ty_n(t))^k-1=0,\quad
n\in\mathbb{N},\qquad
\sum_{k=1}^{p_k-1} g_k^{-m}(z_k)(ty_{-m}(t))^k-1=0,\quad
m\in\mathbb{Z}_{\geqslant0},
\end{equation}
where $z$ is defined in ~\eqref{eq:yn-series}, $p_k$ and $z_k$ are given in Remark~\ref{rem:pk-zk}, and the polynomials
in equations~\eqref{eq:polynom-yp} are irreducible over the field of meromorphic functions.
\end{proposition}
\begin{proof}
Consider the construction of the matrix $F_{\nu}(z)$ (cf. Definition~\ref{def:F(z)}) by taking
$g(z)=ty_p(t)$, where $t$ is defined via $z$ or $z_k$ depending on whether $p=n$ or $p=-m$
(cf. Propositions~\ref{prop:yn-series} or \ref{prop:ym-series}, respectively, and Remark~\ref{rem:pk-zk}). Note that, for $p=n$,
the parameter $\nu=2n+3$, whilst for $p=-m$, $\nu=p_k$, $k=1,2,4$. Next, let $\tilde{f}_q(z)=\hat{f}_q(z)$, and, for
the $m$-series, put $z=z_k$. Since the proof for the $n$- and the $m$-series are literally the same, with only the slight
change of the notation delineated above, we present it for the $n$-series.

Introduce two column vectors: $Y_n(t)=(y_n(t),(y_n(t))^2,\ldots,(y_n(t))^{\nu})^T$ and
$\Omega(z)=(1,\omega,\ldots,\omega^{\nu})^T$.
Now, using the construction for the matrix $F_{\nu}(z)$ given in Definition~\ref{def:F(z)}, one writes
\begin{equation}\label{eq:Yn=FOmega}
tY_{n}(t)=F_{\nu}(z)\Omega(z).
\end{equation}
According to Proposition~\ref{prop:yn-series}, $\hat{f}_{n+1}^n(0)\neq0$; therefore,
Lemma~\ref{lem:detF(z)small-z} (the asymptotics~\eqref{eq:as-detF(z)|f0-fn=0}) implies that $F_{\nu}(z)$ does not vanish
identically, so that one can invert equation~\eqref{eq:Yn=FOmega} to arrive $(F_{\nu}(z))^{-1}tY_n(t)=\Omega(z)$.
Consequently, the first polynomial equation in \eqref{eq:polynom-yp} is none other than the equation for the first component of
$\Omega(z)$. In the case of the $m$-series, the invertibility of matrix $F_{\nu}(z_k)$ is justified via the
asymptotics~\eqref{eq:as-detF(z)|f0=0}, because, according to Proposition~\ref{prop:ym-series}, $\hat{f}_1^{-m}(0)\neq0$.

The irreducibility of the polynomials in~\eqref{eq:polynom-yp} follows from the fact that, otherwise, the functions
$y_n(t)$ and $y_{-m}(t)$ would have fewer than $2n+3$ and $p_k$ branches, respectively. This fact, however,
contradicts the small-$t$ expansions of these functions given in ~\eqref{eq:ym-definition} and \eqref{eq:yn-definition}.
\end{proof}

In Proposition~\ref{prop:polynomialEquation}, the polynomial equations, as well as their solutions, are given in terms of the
functions $\hat{f}_q(z)$. Below, we show that these functions can be characterized as meromorphic solutions of some
special nonlinear systems of polynomial differential equations.

\begin{proposition}\label{prop:Esystem}
For any algebroid solution, $y(t)$, of equation~\eqref{eq:dP3y} (cf. \eqref{eq:dP3y-compact}$)$ with $p$ branches,
there exists a system $\mathcal{E}_p$ of $p$ second-order polynomial {\rm ODEs}, $E_p^q=0$, where
\begin{equation}\label{eq:sysEp}
E_p^q=E_p^q\Big(z,\{\hat{f}_0,\hat{f}_1,\dotsc,\hat{f}_{p-1}\};\{\hat{f}_0^{\prime},\hat{f}_1^{\prime},\dotsc,
\hat{f}_{p-1}^{\prime}\};
\{\hat{f}_0^{\prime\prime},\hat{f}_1^{\prime\prime},\dotsc,\hat{f}_{p-1}^{\prime\prime}\}\Big),\quad
q=0,1,\ldots,p-1,
\end{equation}
which has a meromorphic (in $\mathbb{C})$ solution $\{\hat{f}_0(z),\hat{f}_1(z),\dotsc,\hat{f}_{p-1}(z)\}$
that defines $y(t)$ via the formulae given in Propositions~\ref{prop:yn-series} or \ref{prop:ym-series}.
Conversely, any meromorphic
solution of the system $\mathcal{E}_p$ defines, via Propositions~\ref{prop:yn-series} or \ref{prop:ym-series},
an algebroid solution of equation~\eqref{eq:dP3y} (cf. \eqref{eq:dP3y-compact}$)$ with $p$-branches.
\end{proposition}
\begin{proof}
The proof is constructive. Consider, for example, the case of the $n$- and the $m$-series for for even $m$. According to
Propositions~\ref{prop:yn-series} and \ref{prop:ym-series}, the solution $y(t)$, in this case, can be presented in
the following form:
\begin{equation}\label{eq:yEsystem}
ty(t)=v(z)=(c_1)^2\sum_{q=0}^{p-1}z^{\frac{q}{p}}\tilde{f}_q(z),
\end{equation}
where $p$ is an odd positive integer, and $c_1$ is a parameter. The equation for the function $v(z)$ can be written as
\begin{equation}\label{eq:vEsystem}
v(z)\delta^2v(z)-\left(\delta v(z)\right)^2=(c_2)^2\left((v(z))^3-z\right),\qquad
\delta=z\frac{\md}{\md z},
\end{equation}
where $c_2$ is some parameter. Since, at this stage, the functions $\tilde{f}_q$ are defined modulo multiplication by a parameter,
we can, upon rescaling  $v(z)$ and $z$, always fix $c_1=c_2=1$.

Substituting $v(z)$ given by \eqref{eq:yEsystem} into equation~\eqref{eq:vEsystem} we arrive at, after straightforward
calculations, the equation of the form
\begin{equation}\label{eq:defE}
\sum_{q=0}^{p-1}z^{\frac{q}{p}}E_p^q=0,
\end{equation}
where $E_p^q$ are meromorphic functions of the form \eqref{eq:sysEp}. Since the functions $E_p^q$ are single-valued,
we, after repeating the arguments used in the proof of Proposition~\ref{prop:meromorphic-fpq}, arrive at the equation
$A\vec{\mathcal{E}}_p=\vec{0}$ for the vector $\vec{\mathcal{E}}_p=(E_p^0, z^{1/p}E_p^1,\ldots,z^{(p-1)/p}E_p^{p-1})^T$,
where the matrix $A$ is defined in Proposition~\ref{prop:meromorphic-fpq}.

Conversely, if we have a meromorphic solution of the system $\mathcal{E}_p$, we construct the functions $v(z)$ and $y(t)$
via the formulae~\eqref{eq:yEsystem}; after substituting $y(t)$ into equation~\eqref{eq:dP3y}, one arrives at
equation~\eqref{eq:defE}, which is valid by virtue of the fact that $E_p^q=0$, $q=0,1,\dotsc,p-1$.
\end{proof}
\begin{remark}\label{eq:Eexplicit}
For the system $\mathcal{E}_p$ constructed in the proof of Proposition~\ref{prop:Esystem}, we make some additional remarks.

All meromorphic solutions of the system $\mathcal{E}_p$ are holomorphic at the origin. For even values of $m=2n\geqslant2$,
the systems $\mathcal{E}_{m+3}$ and $\mathcal{E}_{2n+3}$ coincide modulo scaling ($c_1=c_2=1$). This last fact implies that,
for any $n\geqslant1$, system $\mathcal{E}_{2n+3}$ has exactly two meromorphic solutions: these solutions can be distinguished with
the help of the initial data given in Propositions~\ref{prop:yn-series} and ~\ref{prop:ym-series}.

This is not the case for $m=0$ (see
Remark~\ref{rem:exampleE3} below). The other solutions of the $m$-series for even $m$ (cf. Proposition~\ref{prop:ym-series})
may also have the same branching number when $p_2=2n+3$ or $p_4=2n+3$; however, the systems $\mathcal{E}_{p_2}$
and $\mathcal{E}_{p_4}$ are different, because they are obtained from an equation like~\eqref{eq:vEsystem} where
the variable $z$ on the right-hand side is changed to $(z_2)^2$ or $(z_4)^4$, respectively. Certainly, we can map them into
the corresponding system $\mathcal{E}_{2n+3}$ via the change of variable $(z_2)^2=z$ or $(z_4)^4=z$; but, in this case,
solutions that are holomorphic at $z_2=0$ (resp., $z_4=0$) in the variable $z_2$ (resp., $z_4$) will have an expansion over
$\sqrt{z}$ (resp.,$\sqrt[4]{z}$) at $z=0$.

The explicit form of the system $\mathcal{E}_{2n+3}$, whose derivation is described in the proof of
Proposition~\ref{prop:Esystem}, reads:
\begin{align*}
&E_p^q:&
&\underset{q_i\geqslant0,\,q_j\geqslant0}{\sum_{q_i+q_j=q(\text{mod}\,p)}}z^{\frac{q_i+q_j-q}{p}}
\left(\hat{f}_{q_i}\bigg(\frac{q_j^2}{p^2}\hat{f}_{q_j}+\frac{2q_j}{p}\delta(\hat{f}_{q_j})+\delta^2(\hat{f}_{q_j})\bigg)
-\left(\frac{q_i}{p}\hat{f}_{q_i}+\delta(\hat{f}_{q_i})\!\right)\!\!
\left(\frac{q_j}{p}\hat{f}_{q_j}+\delta(\hat{f}_{q_j})\!\right)\!\!\right)\\
&&
&=(c_1)^4(c_2)^2\underset{q_i\geqslant0,\,q_j\geqslant0,\,q_k\geqslant0}{\sum_{q_i+q_j+q_k=q(\text{mod}\,p)}}z^{\frac{q_i+q_j+q_k-q}{p}}
\hat{f}_{q_i}\hat{f}_{q_j}\hat{f}_{q_k}\quad -\quad\frac{z}{(c_1)^6}\delta_{0,q},\qquad
q=0,1,\ldots,p-1,
\end{align*}
where $\hat{f}_{q_l}=\hat{f}_{q_l}(z)$, $\delta(\hat{f}_{q_l})=z\frac{\md}{\md z}\hat{f}_{q_l}(z)$, $l=i,j,k$, and
$\delta_{0,q}$ is the Kronecker delta.
\hfill $\blacksquare$\end{remark}
\begin{remark}\label{rem:exampleE3}
Here, we consider the example of $\mathcal{E}_3$ system associated with the solution $H(r)$ considered in Section~\ref{sec:2}.
Recall that the corresponding  solution of equation~\eqref{eq:dP3y-compact} (cf. \eqref{eq:dP3y}) is denoted by
$y_0(t)$. Define
\begin{equation}\label{eq:y0m=0example}
v_0(z):=ty_0(t)=\hat{f}_0(z)+z^{1/3}\hat{f}_1(z)+z^{2/3}\hat{f}_2(z),\qquad
z=t^4.
\end{equation}
Comparing this formula with the one given in Proposition~\ref{prop:ym-series}, one notes that the scaling coefficient
of $z$ has been modified because we want to arrive exactly at the series introduced in Section~\ref{sec:2}.
Substituting into equation~\eqref{eq:dP3y-compact} the function $v_0(z)$, we, after straightforward transformations, arrive
at the following ODE for $v_0(z)$:
\begin{equation}\label{eq:v3-compact}
9\left(\frac{zv_0'(z)}{v_0(z)}\right)'=v_0(z)-\frac{z}{(v_0(z))^2}.
\end{equation}
Recall that, in the notation of Section~\ref{sec:2}, $z=r$. To simplify the notation in the ensuing system for the
functions $\hat{f}_q$, $q=0,1,2$, we omit their $z$-dependence, and the primes denote differentiation with respect to $z$:
\begin{align*}
&E_3^0:&
&
z^2\big(\hat{f}_2\hat{f}_1^{\prime\prime}-2\hat{f}_1^{'}\hat{f}_2^{'}+\hat{f}_1\hat{f}_2^{''}\big)
+z\big(\hat{f}_0\hat{f}_0^{''}-(\hat{f}_0^{'})^2\big)
+(1/3)z\hat{f}_2\hat{f}_1^{'}+(5/3)z\hat{f}_1\hat{f}_2^{'}+\hat{f}_0\hat{f}_0^{'}\\
&&
&=(1/9)\big(z(\hat{f}_2)^3+(\hat{f}_0)^3/z+(\hat{f}_1)^3-\hat{f}_1\hat{f}_2-1\big)+(2/3)\hat{f}_0\hat{f}_1\hat{f}_2,\\
&E_3^1:&
&
z^2\big(\hat{f}_2\hat{f}_2^{''}-(\hat{f}_2^{'})^2\big)
+z\big(\hat{f}_1\hat{f}_0^{''}-2\hat{f}_0^{'}\hat{f}_1^{'}+\hat{f}_0\hat{f}_1^{''}\big)
+(1/3)\hat{f}_1\hat{f}_0^{'}+(5/3)\hat{f}_0\hat{f}_1^{'}+z\hat{f}_2\hat{f}_2^{'}\\
&&
&=(1/3)\big((\hat{f}_0)^2\hat{f}_1/z+\hat{f}_0(\hat{f}_2)^2+(\hat{f}_1)^2\hat{f}_2\big)-(1/9)\hat{f}_0\hat{f}_1/z,\\
&E_3^2:&
&z^2\big(\hat{f}_0\hat{f}_2^{''}-2\hat{f}_0^{'}\hat{f}_2^{'}+\hat{f}_2\hat{f}_0^{''}
+\hat{f}_1\hat{f}_1^{''}-(\hat{f}_1^{'})^2\big)
+z\big((7/3)\hat{f}_0\hat{f}_2^{'}-(1/3)\hat{f}_2\hat{f}_0^{'}+\hat{f}_1\hat{f}_1^{'}\big)\\
&&
&=(1/3)\big((\hat{f}_1)^2\hat{f}_0+(\hat{f}_0)^2\hat{f}_2+z(\hat{f}_2)^2\hat{f}_1\big)-(4/9)\hat{f}_0\hat{f}_2.
\end{align*}
Analysing the order of the poles in the $E_3^0$ equation, one proves that a meromorphic solution of the $\mathcal{E}_3$
system cannot have a pole at $z=0$: assume, to the contrary, that $\hat{f}_0$ has a pole at $z=0$ of order higher than the
orders of the poles of $\hat{f}_1$ and $\hat{f}_2$; then, it is easy to arrive at a contradiction, namely, that one of the
functions $\hat{f}_1$ or $\hat{f}_2$  would have to have a pole of higher order (by at least $1$).
An analogous contradiction appears if one assumes that either $\hat{f}_1$ or  $\hat{f}_2$ has the highest-order pole.
If, on the other hand, all the poles are of the same order, then the term $\hat{f}_0^3/z$ has a pole at the origin that
cannot be cancelled by the pole of any other term of the $E_3^0$ equation.
Thus, any meromorphic solution of $\mathcal{E}_3$ is regular at $z=0$.

It is easy to establish that there is only one
solution of the $\mathcal{E}_3$ system that can be expanded in a Taylor series at $z=0$; its first few coefficients
can be found with the help of \textsc{Maple}:
$$
\hat{f}_0(z)=z(a_2+a_5z+a_8z^2+\dotsb),\quad
\hat{f}_1(z)=a_0+a_3z+a_6z^2+\dotsb,\quad
\hat{f}_2(z)=a_1+a_4z+a_7z^2+\dotsb,
$$
where the numbers $a_k$, $k=0, 1,2,\ldots$, are defined in Section~\ref{sec:2} (see equations~\eqref{eq:Hat0-expansion},
\eqref{eq:a1}, \eqref{eq:ak-recurrence}, \eqref{eqs:a2-a6}, and \eqref{eq:a-n-conjecture}). This structure of the series
for the functions $\hat{f}_q(z)$ coincides with the one for $m=0$ in Proposition~\ref{prop:ym-series}; in particular,
$\hat{f}_0(z)=zf_0(z)$, with $f(0)\neq0$.
\hfill $\blacksquare$\end{remark}
\section{The Monodromy Data} \label{sec:mondata}
The space of solutions of the Painlev\'e equations can be characterized by the manifold of the monodromy data;
in fact, this manifold is an algebraic variety defined by a set of polynomial equations in $\mathbb{C}^n$.
The co-ordinates of the points of this manifold are called the monodromy data of the solution;
in particular, the manifold of the monodromy data for equation~\eqref{eq:dp3u} is defined in \cite{KitVar2004}.
Below, we present a reduced version of this manifold corresponding to the case $a=0$.

Consider $\mathbb{C}^{7}$ with co-ordinates $(s_{0}^{0},s_{0}^{\infty},s_{1}^{\infty},g_{11},g_{12}, g_{21},g_{22})$.
The monodromy manifold for equation~\eqref{eq:dp3u} with $a=0$, denoted by $\mathcal{M}$, is defined via the following system
of algebraic equations:
\begin{gather}
s_{0}^{\infty}s_{1}^{\infty} \! = \! -2 \! - \! \mi s_{0}^{0}, \label{eq:mdta2} \\
g_{21}g_{22} \! - \! g_{11}g_{12} \! + \! s_{0}^{0}g_{11}g_{22} \! = \! \mi, \label{eq:mdta3} \\
g_{11}^{2} \! - \! g_{21}^{2} \! - \! s_{0}^{0} g_{11} g_{21} \! = \! \mi s_{0}^{\infty}, \label{eq:mdta4} \\
g_{22}^{2} \! - \! g_{12}^{2} \! + \! s_{0}^{0} g_{12} g_{22} \! = \! \mi s_{1}^{\infty}, \label{eq:mdta5} \\
g_{11}g_{22} \! - \! g_{12} g_{21} \! = \! 1. \label{eq:mdta6}
\end{gather}
\begin{remark}
Multiplying equations~\eqref{eq:mdta4} and \eqref{eq:mdta5}, one proves, with the help of the three remaining equations, that
this product is an identity; therefore, $\dim_{\mathbb{C}}{\mathcal{M}}=3$.
As a matter of fact, this monodromy manifold uniquely characterizes a solution of a slightly extended system rather than just
solutions, $u(\tau)$, to equation~\eqref{eq:dp3u}, namely, $\mathcal{M}$ uniquely characterizes the pair of functions
$(u(\tau),\varphi(\tau))$, where $\varphi(\tau)$ can be written as the indefinite integral $\smallint^{\tau}\md\xi/u(\xi)$,
and the additional parameter is needed in order to fix a particular primitive function. The function $\varphi(\tau)$
is addressed in Appendices~\ref{app:asympt0} and \ref{app:infty}.

The unique parametrization of solutions $u(\tau)$ is achieved via a quadratic contraction of $\mathcal{M}$. Define
the following contraction variables:
\begin{equation}\label{eq:contraction-variables}
\tilde{g}_1=\mi g_{12}g_{11},\quad
\tilde{g}_2=\mi g_{21}g_{22},\quad
\tilde{g}_3=g_{11}g_{22},\quad
\tilde{g}_4=g_{12}g_{21},\quad
\tilde{s}=1+\mi s_0^0=-(1+s_{0}^{\infty}s_{1}^{\infty}).
\end{equation}
The parameter $\tilde{g}_4$ is introduced merely for convenience; it plays an auxiliary role, and is formally not required
for the definition of the contracted monodromy manifold. Note that, by definition, one has
$\tilde{g}_3\tilde{g}_4=-\tilde{g}_1\tilde{g}_2$, and equation~\eqref{eq:mdta6} allows one to remove $\tilde{g}_4$, that is,
$\tilde{g}_3-\tilde{g}_4=1$. Finally, multiplying equation~\eqref{eq:mdta3} by $-\mi$, we arrive at algebraic equations
defining, in $\mathbb{C}^4$ with co-ordinates $(\tilde{g}_1,\tilde{g}_2,\tilde{g}_3,\tilde{s})$, the contracted monodromy
manifold:
\begin{equation}\label{eqs:ContrManifold}
\tilde{g}_1-\tilde{g}_2+\tilde{g}_3(1-\tilde{s})=1,\qquad
\tilde{g}_3(\tilde{g}_3-1)=-\tilde{g}_2\tilde{g}_1.
\end{equation}
Either one of the co-ordinates $\tilde{g}_1$ or $\tilde{g}_2$ can be further excluded from the system~\eqref{eqs:ContrManifold},
so that the contracted monodromy manifold can be presented as a single equation in $\mathbb{C}^3$:
\begin{equation}\label{eq:ContrManReduced}
(\tilde{g}_1)^2+\tilde{g}_1\tilde{g}_3(1-\tilde{s})+(\tilde{g}_3)^2=\tilde{g}_1+\tilde{g}_3
\end{equation}

For an equation equivalent to \eqref{eq:dp3u} with $a=0$ (see Appendix~\ref{app:Kit87},
equation~\eqref{eq:tzieica-sym}),
two (one for each value of $\varepsilon=\pm1$) equivalent (related by a birational transformation) monodromy manifolds were
introduced in \cite{Kit87}. For $\varepsilon=+1$, say, the corresponding manifold is described by the following system of equations:
\begin{equation}\label{eqs:manifold87}
g_1+g_2(1-s)+g_3=1,\qquad
g_2(g_2-1)=g_1g_3.
\end{equation}
The two manifolds~\eqref{eqs:ContrManifold} and \eqref{eqs:manifold87} should be birationally equivalent. There are two
apparent permutation transformations:
\begin{equation*}\label{eqs:permutation}
\tilde{s}=s,\quad
\tilde{g}_3=g_2,\quad
\tilde{g}_2=-g_3,\quad
\tilde{g}_1=g_1\qquad
\text{or}\qquad
\tilde{s}=s,\quad
\tilde{g}_3=g_2,\quad
\tilde{g}_2=-g_1,\quad
\tilde{g}_1=g_3.\qquad
\end{equation*}
These transformations, however, do not correlate with the parametrization of the solutions that are obtained in
the papers \cite{Kit87,KitVar2004}; in fact, comparing the amplitude of the oscillation terms of asymptotics,
we see that $\tilde{g}_3=g_3$, implying that there should be some other birational transformation connecting the contraction
manifolds:
\begin{equation}\label{eqs:contraction-symmetries-g2=-g2}
s=\tilde{s},\;\;
g_3=\tilde{g}_3,\;\;
g_2=-\tilde{g}_2,\;\;
g_1=\tilde{g}_1-\tilde{s}(\tilde{g}_2+\tilde{g}_3),
\end{equation}
or
\begin{equation}\label{eqs:contraction-symmetries-g2=g1}
s=\tilde{s},\;\;
g_3=\tilde{g}_3,\;\;
g_2=\tilde{g}_1,\;\;
g_1=-\tilde{g}_2+\tilde{s}(\tilde{g}_1-\tilde{g}_3).
\end{equation}
The proof of the above transformations is slightly more complicated than that for the permutated ones; therefore, we outline
the proof by taking, as an example, the transformation~\eqref{eqs:contraction-symmetries-g2=g1}.
Substituting the formulae for the variables without tildes into the first equation of \eqref{eqs:manifold87}, we immediately
confirm the validity of the first equation of \eqref{eqs:ContrManifold} for the tilde variables. To confirm the second
equation of \eqref{eqs:ContrManifold}, we have to use the first one twice; more precisely, substituting variables
without tildes into the second equation of \eqref{eqs:manifold87}, we get
\begin{align*}\label{eq:prooftransformations}
\tilde{g}_1(\tilde{g}_1-1)=\tilde{g}_3(-\tilde{g}_2+s\tilde{g}_1-s\tilde{g}_3)&\quad\Rightarrow\quad
\tilde{g}_1(\tilde{g}_2-\tilde{g}_3+\tilde{s}\tilde{g}_3)=\tilde{g}_3(-\tilde{g}_2+\tilde{s}\tilde{g}_1-
\tilde{s}\tilde{g}_3)\Rightarrow\\
\tilde{g}_1\tilde{g}_2=\tilde{g}_3(\tilde{g}_1-\tilde{g}_2+\tilde{s}\tilde{g}_3)&\quad\Rightarrow\quad
\tilde{g}_1\tilde{g}_2=\tilde{g}_3(1-(1-\tilde{s})\tilde{g}_3+\tilde{s}\tilde{g}_3).
\end{align*}
\hfill $\blacksquare$\end{remark}
\begin{lemma} \label{lem:mondataH0}
Let $u(\tau)$ be the solution of equation~\eqref{eq:dp3u} with $a=0$ defined by the asymptotics~\eqref{eq:u(tau)asympt0}.
Then, the monodromy data characterizing $u(\tau)$ reads:
\begin{equation} \label{eqs:mondataH0}
\begin{gathered}
\tilde{s}=0,\quad
s_{0}^{0}=\mi,\quad
s_{0}^{\infty}s_{1}^{\infty}=-1,\quad
s_{0}^{\infty}g_{12}^{2}=s_{1}^{\infty}g_{21}^2=-\frac{\mi(H(0)-1)^{2}}{3H(0)},\\
\tilde{g}_4=g_{12}g_{21}=\frac{(H(0)-1)^2}{3H(0)},\quad
\tilde{g}_3=g_{11}g_{22}=\frac{1}{3H(0)}\Big(H(0)-\me^{\frac{2\pi\mi}{3}}\Big)\Big(H(0)-\me^{-\frac{2\pi\mi}{3}}\Big),\\
\tilde{g}_2=\mi g_{21}g_{22}=-\frac{\me^{-\frac{2\pi\mi}{3}}}{3H(0)}\Big(H(0)-1\Big)\Big(H(0)-\me^{-\frac{2\pi\mi}{3}}\Big),\quad
\tilde{g}_1=\mi g_{12}g_{11}=\frac{\me^{\frac{2\pi\mi}{3}}}{3H(0)}\Big(H(0)-1\Big)\Big(H(0)-\me^{\frac{2\pi\mi}{3}}\Big).\\
\end{gathered}
\end{equation}
\end{lemma}
\begin{proof}
The asymptotics as $\tau\to0$ of the solution $u(\tau)$ does not contain logarithmic terms (cf. \eqref{eq:u(tau)asympt0});
therefore, its parametrization via the monodromy data is given in Appendix~\ref{app:asympt0}, Theorem~\ref{th:B1asympt0}
(where we substitute $a=0$):
\begin{align} \label{eq:u(tau)omegaProof}
u(\tau)\underset{\tau\to+0}{=}&\,\frac{\tau}{16 \pi}\!\left(\varpi_{1}(-\rho)\varpi_{2}(-\rho)\tau^{-4\rho}
+\varpi_{1}(-\rho)\varpi_{2}(\rho)+\varpi_{1}(\rho) \varpi_{2}(-\rho)+\varpi_{1}(\rho)\varpi_{2}(\rho)\tau^{4\rho}\right)
\left(1+o(\tau^{\delta})\right),
\end{align}
where $\varpi_{k}(\lambda)$ for $k=1,2$ and $\lambda=\pm\rho$ are given in
equations~\eqref{eq:app:varpi-pk-chi}--\eqref{eq:app:pk-gamma}, and $\delta>0$.

To match with the asymptotics~\eqref{eq:u(tau)asympt0}, one has to assume that
$\varpi_{1}(-\rho) \varpi_{2}(-\rho)\neq0$ and $1-4\rho=1/3$:
the last equality implies $\rho=1/6$. In the conclusions above, we used the fact that the leading term of
asymptotics~\eqref{eq:u(tau)omegaProof} is symmetric with respect to the change $\rho\to-\rho$,
so that, in case $\rho\neq0$, one can always assume that $\rho>0$. Equations~\eqref{eq:appAsympt0:rho-s} now read
\begin{equation*} \label{eq:Asympt0:rho-sProof}
\cos\Big(\frac{\pi}3\Big)=\frac12=-\frac{\mi s_{0}^{0}}{2}=1+\frac{1}{2}s_{0}^{\infty}s_{1}^{\infty},
\end{equation*}
which confirms the fist three relations in \eqref{eqs:mondataH0}.

Comparing the coefficients of the leading terms in equations~\eqref{eq:u(tau)omegaProof} and ~\eqref{eq:u(tau)asympt0},
we get
\begin{equation}\label{eq:omegaH0Proof}
\frac{\varpi_{1}(-1/6)\varpi_{2}(-1/6)}{16\pi}=\frac{1}{2}H(0).
\end{equation}
For $\rho=1/6$, equations~\eqref{eq:app:varpi-pk-chi}--\eqref{eq:app:pk-gamma} in Appendix~\ref{app:asympt0} read:
\begin{gather}
\varpi_{1}(\pm 1/6)=\mathbf{p}_1(\pm 1/6)\chi_{1}(\pm1/6)=\mathbf{p}_1(\pm 1/6)
\left(g_{11}\me^{\mi\pi/4}\me^{\pm\mi\pi/6}+g_{21}\me^{-\mi\pi/4}\me^{\mp\mi\pi/6}\right),\label{eq:omega1Proof}\\
\varpi_{2}(\pm 1/6)=\mathbf{p}_2(\pm 1/6) \chi_{2}(\pm 1/6)=\mathbf{p}_2(\pm 1/6)
\left(g_{12}\me^{\mi\pi/4}\me^{\pm\mi\pi/6}+g_{22}\me^{-\mi\pi/4}\me^{\mp\mi\pi/6}\right),\label{eq:omega2Proof}
\end{gather}
with
\begin{equation}\label{eq:p1pm1/6GAMMA:p2p1}
\mathbf{p}_1(\pm 1/6)=\pm6\me^{\pm\mi\pi/12}\left(\frac{1}{2}\right)^{\pm 1/6}
\frac{\Gamma(1\mp\frac{1}{3})\Gamma (1\pm\frac{1}{6})}{\Gamma(1\pm\frac{1}{3})},\qquad
\mathbf{p}_2(\pm 1/6)=\mathbf{p}_1(\pm 1/6)\me^{\mp\mi\pi/6},
\end{equation}
and $\Gamma (\ast)$ is the (Euler) gamma function \cite{BE1}. The numbers $\mathbf{p}(\pm 1/6)$ can be calculated
explicitly,
\begin{equation}\label{eq:p1pm1/6numbers}
\mathbf{p}_1(1/6)=3\sqrt{2\pi}\,\me^{\frac{\pi\mi}{12}},\quad
\mathbf{p}_1(-1/6)=-2\sqrt{2\pi}\,\me^{-\frac{\pi\mi}{12}},\quad
\mathbf{p}_1(1/6)\,\mathbf{p}_1(-1/6)=-12\pi.
\end{equation}
Combining equations~\eqref{eq:omegaH0Proof}--\eqref{eq:omega2Proof}, one obtains
\begin{equation}\label{eq:H0-gProof}
(\mathbf{p}_1(-1/6))^{2}\left(g_{11}\me^{\mi\pi/4}\me^{-\mi\pi/6}+g_{21}\me^{-\mi\pi/4}\me^{\mi\pi/6}\right)\!
\left(g_{12}\me^{\mi\pi/4}\me^{-\mi\pi/6}+g_{22}\me^{-\mi\pi/4}\me^{\mi\pi/6}\right)
\me^{\mi\pi/6}=8\pi H(0).
\end{equation}
Using the expression for $\mathbf{p}_{1}(-1/6)$ given in~\eqref{eq:p1pm1/6numbers}, multiplying out
equation~\eqref{eq:H0-gProof} and  exploiting \eqref{eq:mdta6} in order to remove the term $g_{12}g_{21}$,
and introducing the $\tilde{g}_k$ variables, one arrives at the following equation for the monodromy data:
\begin{equation}\label{eq:H0-g-linearProof}
\tilde{g}_1\me^{-\frac{\pi\mi}{3}}+2(\tilde{g}_3-1)-\tilde{g}_2\me^{\frac{\pi\mi}{3}}=H(0)-1.
\end{equation}
Equation~\eqref{eq:H0-g-linearProof} should be supplemented with two equations defining the monodromy
manifold~\eqref{eqs:ContrManifold}; therefore, we obtain three equations for the three variables $\tilde{g}_k$, $k=1,2,3$.
In order to solve these equations, express $\tilde{g}_1$ and $\tilde{g}_2$ from the linear equations
as linear combinations of $H(0)-1$ and $\tilde{g}_3-1$; then, substituting these expressions into the quadratic equation
in \eqref{eqs:ContrManifold}, one finds that the quadratic term $(\tilde{g}_3-1)^2$ cancels, so that we get a linear
equation in $\tilde{g}_3-1=\tilde{g}_4$ which contains $H(0)$. Solving the last equation for $\tilde{g}_4$,
we arrive at the expression for $\tilde{g}_4$ stated in \eqref{eqs:mondataH0}.
Then, the formula for $\tilde{g}_3$ follows from the equation  $\tilde{g}_3-1=\tilde{g}_4$, and $\tilde{g}_1$ and
$\tilde{g}_2$ are obtained from the linear equations mentioned in the proof above.
\end{proof}
\begin{remark}\label{rem:explorationAsymptotics}
Note that the error estimate in the asymptotics~\eqref{eq:u(tau)omegaProof} contains an undetermined positive parameter
$\delta$. The value of this parameter in many questions, as, in particular, demonstrated in the proof above, is not
important. For some very special cases, though, this parameter may turn out to be an impediment to the direct application
of asymptotics for the calculation of the monodromy data; in such cases, however, it is the special properties of the solution
that may, nevertheless, help to circumvent this problem (see, for example, \cite{KitSIGMA2019}).
The value of $\delta$ is not universal, and it depends on the solution;
of course, the local analysis of the solution allows one to find the value of the parameter $\delta$ for particular
solutions (see \eqref{eq:u-H-expansionCONFIRM} below).

Below, we show how our asymptotic formula~\eqref{eq:u(tau)omegaProof} is consistent with the expansion for the function
$H(r)$ studied in Section~\ref{sec:2}. As a matter of fact, we present an alternative calculation for the monodromy
data that were obtained in Lemma~\ref{lem:mondataH0} for the purpose of demonstrating the applicability of those formulae
in Appendix~\ref{app:asympt0}, Theorem~\ref{th:B1asympt0} that were not used for the calculation of the monodromy data.
Note, however, that this latter calculation does, in fact, use the value of $\delta$.

We commence with the proof of two identities for $\varpi_{k}(\pm1/6)$, $k=1,2$:
\begin{equation}\label{eqs:varpi-sum-relation}
\begin{gathered}
\varpi_{1}(-1/6)\,\varpi_{2}(1/6)+\varpi_{1}(1/6)\,\varpi_{2}(-1/6)\\
=\mathbf{p}_1(1/6)\,\mathbf{p}_1(-1/6)\left(\me^{\mi\pi/6}\chi_{1}(1/6)\,\chi_{2}(-1/6)+\me^{-\mi\pi/6}
\chi_{1}(-1/6)\,\chi_{2}(1/6)\right)\\
=-12\pi\sqrt{3}\,(\mi(g_{11}g_{12}-g_{21}g_{22})+g_{12}g_{21})=0;
\end{gathered}
\end{equation}
the first equality follows from equation~\eqref{eq:app:varpi-pk-chi} and the second equation in \eqref{eq:p1pm1/6GAMMA:p2p1};
the second equality uses the definition~\eqref{eq:app:chi-g} and the third equation in \eqref{eq:p1pm1/6numbers};
and the last relation is equivalent to equations~\eqref{eq:mdta3} (with $s_0^0=\mi$) and \eqref{eq:mdta6}.
In an analogous manner, one proves the following identity:
\begin{equation}\label{eq:varphiProduct}
\begin{aligned}
\varpi_{1}(1/6)\,\varpi_{1}(-1/6)\,\varpi_{2}(1/6)\,\varpi_{2}(-1/6)&=\\
\underbrace{\big(\mathbf{p}_1(-1/6)\mathbf{p}_1(1/6)\big)^{2}}_{=(-12\pi)^{2}}
\underbrace{\big(g_{11}^{2}-g_{21}^{2}-\mi g_{11}g_{21}\big)}_{=\mi s_{0}^{\infty}}
\underbrace{\big(g_{22}^{2}-g_{12}^{2}+\mi g_{12}g_{22}\big)}_{=\mi s_{1}^{\infty}}&=144\,\pi^{2};
\end{aligned}
\end{equation}
the first equality is derived with the help of the definitions in Theorem~\ref{th:B1asympt0}; the underbraced relations
follow from the third relation in equation~\eqref{eq:p1pm1/6numbers} and the equations~\eqref{eq:mdta4} and \eqref{eq:mdta5}
defining the monodromy manifold; and the last equality is a consequence of equation~\eqref{eq:mdta2} with $s_0^0=\mi$.

Equations~\eqref{eq:varphiProduct} and ~\eqref{eq:omegaH0Proof} imply that
\begin{equation} \label{eq:varpi1/H0}
\varpi_{1}(1/6)\,\varpi_{2}(1/6)=\frac{144\,\pi^{2}}
{\varpi_{1}(-1/6)\,\varpi_{2}(-1/6)}=\frac{18\pi}{H(0)}.
\end{equation}
Now, with the help of ~\eqref{eq:omega1Proof} and ~\eqref{eq:omega2Proof}, equation~\eqref{eq:varpi1/H0} can be presented as
follows:
\begin{equation}\label{eq:1/H0-Proof}
\big(\mathbf{p}_1(1/6)\big)^{2}\left(g_{11}\me^{\mi\pi/4}\me^{\mi\pi/6}+
g_{21}\me^{-\mi\pi/4}\me^{-\mi\pi/6}\right)\left(g_{12}\me^{\mi\pi/4}
\me^{\mi \pi/6}+g_{22}\me^{-\mi \pi/4}\me^{-\mi\pi/6}\right)
\me^{-\mi \pi/6}=\frac{18\pi}{H(0)}.
\end{equation}
Using the formula for $\mathbf{p}(1/6)$ given in \eqref{eq:p1pm1/6numbers}, multiplying out the expressions in
parentheses, and introducing the contraction variables, we rewrite equation~\eqref{eq:1/H0-Proof} as
\begin{equation}\label{eq:1/H0-g-linearProof}
\tilde{g}_1\me^{\frac{\pi\mi}{3}}+2(\tilde{g}_3-1)-\tilde{g}_2\me^{-\frac{\pi\mi}{3}}=\frac{1}{H(0)}-1.
\end{equation}
Equation~\eqref{eq:1/H0-g-linearProof} is consistent with equation~\eqref{eq:H0-g-linearProof}, and, together with the
equations defining the contraction manifold, are equivalent to equations~\eqref{eqs:mondataH0} stated in
Lemma~\ref{lem:mondataH0}.

Now, we show that the asymptotics~\eqref{eq:u(tau)omegaProof} is consistent with the local expansion of the function $u(\tau)$
that follows from equations~\eqref{eq:hazzidakis-dP3y}--\eqref{eq:u-y-transformation} and \eqref{eq:Hat0-expansion}.
For this purpose, substitute into the asymptotic expansion~\eqref{eq:u(tau)omegaProof} the relations for
$\varpi_{k}(\pm1/6)$, $k=1,2$, given in equations~\eqref{eq:omegaH0Proof}, \eqref{eqs:varpi-sum-relation},
and \eqref{eq:varpi1/H0}. The parameter $\delta$ is not yet specified; but, we know the expansion for $H(r)$ given in
Subsection~\ref{subsec:existence}: the latter implies, in fact, that $\delta=3/4$. Thus, replacing
$1+o\big(\tau^{\delta}\big)$ by a corresponding series expansion, one arrives at
\begin{equation}\label{eq:u-H-expansionCONFIRM}
u(\tau)\underset{\tau\to+0}{=}\frac12\tau^{1/3}\left(H(0)+\frac{9}{4H(0)}\tau^{4/3}\right)
\left(1+\sum_{m=1}^{\infty}\mu_{m}\tau^{4m/3}\right)=\frac12\tau^{1/3}H\left(-\left(\frac32\right)^2\tau^{4/3}\right),
\end{equation}
where $\mu_m$ are $\tau$-independent coefficients, and the function $H(r)$ is defined via the series~\eqref{eq:Hat0-expansion}.
Multiplying out the expressions in parentheses in equation~\eqref{eq:u-H-expansionCONFIRM} one shows that
\begin{align*}
&H(0)\mu_1+\frac{9}{4H(0)}=-\frac94H'(0)=-\frac94\left((H(0))^2-\frac{1}{H(0)}\right)
\Rightarrow\mu_1=-\frac94H(0),\\
&H(0)\mu_m+\frac{9}{4H(0)}\mu_{m-1}=(-1)^m\left(\frac32\right)^{2m}a_m,\quad
m=2,3,\dotsc,
\end{align*}
where the numbers $a_m$ are defined in Section~\ref{sec:2}. Obviously, the series $\sum_{m=1}^{\infty}\mu_{m}\tau^{4m/3}$
is uniquely defined and convergent in some neighbourhood of $\tau=0$.
\hfill $\blacksquare$\end{remark}

\begin{corollary}\label{cor:algebraicMONDATA}
There are three algebraic solutions of equation~\eqref{eq:dp3u} (cf. \eqref{eq:dP3y}$)$ with $a=0$  that correspond
to three constant solutions of equation~\eqref{eq:hazzidakis}. These solutions and the corresponding
monodromy data of equation~\eqref{eq:dp3u} read:
\begin{enumerate}
\item[{\rm\pmb{(1)}}]\label{item:H0=1}
\begin{equation*}
u(\tau)=\frac{1}{2}\tau^{1/3},\qquad\qquad
y(t)=t^{1/3},\qquad\qquad
H(r)=1,
\end{equation*}
\begin{equation*}
\tilde{s}=\tilde{g}_1=\tilde{g}_2=\tilde{g}_4=0,\qquad
\tilde{g}_3=1,
\end{equation*}
\begin{equation*}
s_0^0=\mi,\quad
s_0^{\infty}=-\mi g_{11}^2,\quad
s_1^{\infty}=-\mi g_{22}^2,\quad
g_{12}=g_{21}=0,\quad
g_{11}g_{22}=1;
\end{equation*}
\item[{\rm\pmb{(2)}}]\label{item:H0=e+2ipi/3}
\begin{equation*}
u(\tau)=\frac{1}{2}\me^{\frac{2\pi\mi}{3}}\tau^{1/3},\qquad\qquad
y(t)=\me^{\frac{2\pi\mi}{3}}t^{1/3},\qquad\qquad
H(r)=\me^{\frac{2\pi\mi}{3}},
\end{equation*}
\begin{equation*}
\tilde{s}=\tilde{g}_1=\tilde{g}_3=0,\qquad
\tilde{g}_2=-1,\quad
\tilde{g}_4=-\mi,
\end{equation*}
\begin{equation*}
s_0^0=\mi,\quad
s_0^{\infty}=\mi g_{21}^2,\quad
s_1^{\infty}=\mi g_{12}^2,\quad
g_{11}=0,\quad
g_{12}g_{21}=-1,\quad
g_{21}g_{22}=\mi;
\end{equation*}
\item[{\rm\pmb{(3)}}]\label{item:H0=e-2ipi/3}
\begin{equation*}
u(\tau)=\frac{1}{2}\me^{-\frac{2\pi\mi}{3}}\tau^{1/3},\qquad\qquad
y(t)=\me^{-\frac{2\pi\mi}{3}}t^{1/3},\qquad\qquad
H(r)=\me^{-\frac{2\pi\mi}{3}},
\end{equation*}
\begin{equation*}
\tilde{s}=\tilde{g}_2=\tilde{g}_3=0,\qquad
\tilde{g}_1=1,\quad
\tilde{g}_4=-\mi,
\end{equation*}
\begin{equation*}
s_0^0=\mi,\quad
s_0^{\infty}=\mi g_{21}^2,\quad
s_1^{\infty}=\mi g_{12}^2,\quad
g_{22}=0,\quad
g_{12}g_{21}=-1,\quad
g_{11}g_{12}=-\mi.
\end{equation*}
\end{enumerate}
\end{corollary}
\begin{remark}\label{rem:algebraicFunction}
The branch of $\tau^{1/3}$ in Corollary~\ref{cor:algebraicMONDATA} is fixed such that it is positive for $\tau>0$.
The function $y(t)$ is calculated with the help of relation~\eqref{eq:u-y-transformation}. These three different
solutions of equation~\eqref{eq:dp3u} (cf. \eqref{eq:dP3y}) coincide, of course, with the pullback
of the three branches of the algebraic function $\tfrac12\tau^{1/3}$ (resp., $t^{1/3}$); however, from the
point of view of solutions to the ODE, they represent three different solutions, since, for the same value of $\tau$
(resp., $t$), they have different initial values.
\hfill $\blacksquare$\end{remark}
\begin{remark}\label{rem:algebraicUniqueness}
According to ~\cite{Gromak1979}, the solutions enumerated in Corollary~\ref{cor:algebraicMONDATA} are the only
algebraic solutions of equation~\eqref{eq:dp3u} with $a=0$ (resp., equation~\eqref{eq:dP3y}). Here, we show
how this fact can be deduced from our asymptotic results. The essential singular point at infinity imposes severe
restrictions on the algebraic behaviour of solutions at this point. Our main results concerning  the asymptotic
behaviour of solutions at the point at infinity are presented in Appendix~\ref{app:infty}. These asymptotic results imply,
in particular, that in case of algebraic behaviour of solutions at the point at infinity, the corresponding monodromy data
necessarily satisfy the condition $g_{12}g_{21}g_{11}g_{22}=0$ (cf. equation~\eqref{eq:app:u-reg-as-conditions} with
~\eqref{eq:app:u-reg-as}, and equation~\eqref{eq:app:u-sing-as-conditions} with ~\eqref{eq:app:u-sing-as}).
There is, seemingly, another possibility, namely, $g_{11}g_{22}=1$ (cf. equation~\eqref{eq:app:nu+1}
with ~\eqref{eq:app:u-reg-as}). The latter case, in conjunction with equation~\eqref{eq:mdta6}, implies that $g_{12}g_{21}=0$,
so that the condition $g_{12}g_{21}g_{11}g_{22}=0$ holds. This last condition, with the help of the equations defining the
monodromy manifold (cf. equations~\eqref{eq:mdta3}--\eqref{eq:mdta6}), can be subdivided into three sub-cases:
(1) $g_{12}g_{21}=0$; (2) $g_{11}=0$; and (3) $g_{22}=0$.

Consider sub-case (1) for which the large-$\tau$ asymptotics is stated in Theorem~3.2 of
\cite{KitVar2004}.\footnote{In Theorem~3.2 of \cite{KitVar2004}, set $a=0$, $\varepsilon b=+1$,
$(\varepsilon_{1},\varepsilon_{2})=(0,0)$, $s_{0}^{0}(0,0):=s_{0}^{0}$, $s_{0}^{\infty}(0,0):=s_{0}^{\infty}$,
$s_{1}^{\infty}(0,0):=s_{1}^{\infty}$, and $g_{ij}(0,0):= g_{ij}$, $i,j\in\lbrace1,2\rbrace$.}
This theorem implies that the algebraic behaviour is possible only if $s_0^0=\mi$. Now, equations~\eqref{eq:mdta3}
and ~\eqref{eq:mdta6} imply that $g_{12}=g_{21}=0$. This supplies the necessary conditions for the
existence of algebraic solutions in case the monodromy data satisfy $g_{11}g_{22}=1$. The first case
in Corollary~\ref{cor:algebraicMONDATA} supplies the sufficiency conditions.

With the help of equations~\eqref{eq:mdta3}--\eqref{eq:mdta6}, sub-case (2) gives rise to the following
values for the monodromy data: $g_{11}=0$, $g_{12}g_{21}=-1$,  $g_{21}g_{22}=\mi$, and $s_0^{\infty}=\mi g_{21}^2$.
Even though the values of $s_0^0$ and $s_1^{\infty}$ cannot be determined directly, there is, however, a simple
ruse related to symmetries.
Clearly, if $u(\tau)$ is an algebraic solution, then $\hat{u}(\tau)=u\big(\tau\me^{2\pi\mi}\big)$ is also an
algebraic solution; consequently, one can deduce the action of the transformation $\tau\to\tau\me^{2\pi\mi}$ on the
monodromy manifold. In our case, that is, $a=0$, this action, in terms of the Stokes matrices, $S_k$, and the connection matrix,
$G$, defined in \cite{KitVar2004}, reads: $\hat{S}_k^{\infty}=S_k^{\infty}$ and $-\mi S_0^0\sigma_1\hat{G}=G$, where
$\sigma_1=\left(\begin{smallmatrix}0&1\\1&0\end{smallmatrix}\right)$, and where the `hat' denotes the monodromy matrices
corresponding to the solution $\hat{u}(\tau)$, whereas the monodromy matrices without the `hat' correspond to the solution
$u(\tau)$. Taking into account the definition of these matrices \cite{KitVar2004} and
assuming that the hat variables correspond to sub-case (2), the matrix relations can be rewritten in terms of the corresponding
scalar variables as follows: $\hat{s}_k^{\infty}=s_k^{\infty}$, $k=0,1$, $\hat{s}_0^0=s_0^0$, $g_{12}=g_{21}=0$,
$g_{11}=-\mi\hat{g}_{21}$, and $g_{22}=-\mi\hat{g}_{12}$. Thus, we map sub-case (2) to sub-case (1). For sub-case (1),
it is proved that algebraic solutions exist iff $s_0^0=\mi$, which means that the same is true for sub-case (2).
Clearly, the same matrix action links the monodromy data corresponding to sub-cases (3) and (2), where, now, the hat variables
correspond to sub-case (3). Since sub-case (2) is studied,
one can repeat the aforementioned arguments in order to confirm the uniqueness of the algebraic solution for sub-case (3).
\hfill $\blacksquare$\end{remark}
\section{The Coxeter Group}\label{sec:Coxeter}
In this section, we consider some group actions on the monodromy manifold.
\begin{proposition}\label{prop:cubicA3}
The projectivization of the contracted monodromy manifold for equation~\eqref{eq:dp3u} is a singular cubic surface of type $A_3$.
\end{proposition}
\begin{proof}
The monodromy manifold for equation~\eqref{eq:dp3u} in the generic case $a\in\mathbb{C}$ is defined in \cite{KitVar2004}
(see \cite{KitVar2004}, p. 1172, the system (33)). We introduce the same change of variables \eqref{eq:contraction-variables}
as for the case $a=0$ and arrive at the following system for the contracted variables:
\begin{equation}\label{eqs:ContrManifold-aGEN}
\tilde{g}_1-\tilde{g}_2+\tilde{g}_3(1-\tilde{s})=\me^{-\pi a},\qquad
\tilde{g}_3(\tilde{g}_3-1)=-\tilde{g}_1\tilde{g}_2.
\end{equation}
Solving the first equation of the system~\eqref{eqs:ContrManifold-aGEN} for $\tilde{g}_2$ and substituting the result into the
second equation, we obtain the following cubic equation in $\mathbb{C}^3$ with the parameter $\me^{-\pi a}$,
\begin{equation}\label{eq:ContrManModified-aGen}
\tilde{g}_1\tilde{g}_3(\tilde{s}-1)=(\tilde{g}_1)^2+(\tilde{g}_3)^2 -(\tilde{g}_1\me^{-\pi a}+\tilde{g}_3).
\end{equation}
Introducing local co-ordinates in $\mathbb{CP}^3$, $\{x_0:x_1:x_2:x_3\}$, according to the formulae,
\begin{equation}\label{eq:projectivizationContrMan-aGen}
\tilde{g}_3=-\frac{x_0}{x_2},\qquad \tilde{g}_1=-\frac{x_1}{x_2}\me^{-\pi a},\qquad
\tilde{s}-1+2\cosh(\pi a)=\frac{x_3}{x_2}\,\me^{\pi a},
\end{equation}
one rewrites equation~\eqref{eq:ContrManModified-aGen} as follows:
\begin{equation}\label{eq:ContrManProjective-aGen}
x_0x_1x_3=x_2(x_0+x_1+x_2)(x_0-ux_1),\quad\mathrm{where}\quad
\sqrt{-u}=\me^{-\pi a}.
\end{equation}
This surface has a singularity of type $A_3$  \cite{BruceWall1979,Sakamaki2010}.
\end{proof}
Henceforth, till the end of this section, we proceed with the study of the case $a=0$ $(u=-1)$; in this case,
equation~\eqref{eq:ContrManProjective-aGen} reads
\begin{equation}\label{eq:ContrManProjective}
x_0x_1x_3=x_2(x_0+x_1+x_2)(x_0+x_1).
\end{equation}
Equation~\eqref{eq:ContrManProjective} contains ten $\mathbb{CP}^3$-lines \cite{BruceWall1979,Sakamaki2010}.
These lines can be presented as the intersection of two hyperplanes.
Three of these lines belong to the hyperplane $x_2=0$, which is located at ``infinity'', namely, they can be
presented as the intersection of the plane $x_2=0$ with the planes $x_0=0$, $x_1=0$, and $x_3=0$.
The monodromy co-ordinates cannot take on infinite values; therefore, we cannot give, at least not directly,
an interpretation for these lines in terms of the monodromy data and the corresponding solutions of equation~\eqref{eq:dp3u}.
Consequently, for our purposes, we resort back to $\mathbb{C}^4$,
and denote the co-ordinates in this space as $(x,y,z,s)$. We identify these co-ordinates with our monodromy data
as follows:
\begin{equation*}\label{eqs:coordinates:mon-C4}
x=\tilde{g}_1,\quad
y=-\tilde{g}_2,\quad
z=\tilde{g}_3,\quad
s=\tilde{s}.
\end{equation*}
In these co-ordinates, the system of equations~\eqref{eqs:ContrManifold} defining the contracted monodromy
manifold reads
\begin{equation}\label{eqs:contrmanifold-xyz}
x+y+z(1-s)=1,\qquad
z(z-1)=xy.
\end{equation}
The remaining seven lines of the surface~\eqref{eq:ContrManProjective} and the corresponding monodromy co-ordinates are:
\begin{enumerate}
\item
$x_0=0$ and $x_1=0$, $(0,1,0,s)$;
\item
$x_0=0$ and $x_1+x_2=0$, $(1,0,0,s)$;
\item
$x_1=0$ and $x_0+x_2=0$, $(0,s,1,s)$;
\item
$x_3=0$ and $x_0+x_1=0$, $(x,x+1,-x,-1)$;
\item
$x_3=0$ and $x_0+x_1+x_2=0$, $(x,x-1,1-x,-1)$;
\item\label{item:9}
$x_1+x_2=0$ and $x_0+x_1+x_3=0$, $(1,s(s-1),s,s)$;
\item
$x_0+x_2=0$ and $x_0+x_1+x_3=0$, $(s,0,1,s)$;
\end{enumerate}
where $s,x\in\mathbb{C}$. Note that, in item~\ref{item:9} above, the dependence of $y$ with respect to $s$ is quadratic,
and it remains a straight line due to the fact that the surface~\eqref{eq:ContrManProjective} is written in terms of
co-ordinates that do not depend on $y$. These lines appear again later in this section whilst studying the action of
the Coxeter group on the contracted monodromy manifold.

Letting $z=\kappa x$ in the system~\eqref{eqs:contrmanifold-xyz}, one finds a rational parametrization of the contracted
monodromy manifold in terms of the parameters $\kappa,s\in\mathbb{C}$. (The second equation in~\eqref{eqs:contrmanifold-xyz}
suggests three other rational parametrizations; however, this fact is not important for our considerations.) With the
help of this rational parametrization, we find the following transformations of the contracted monodromy manifold:
\begin{align}
&r_1:\;\;(x,y,z,s)\longrightarrow(y,x,z,s),\label{eq:r1tr}\\
&r_2:\;\;(x,y,z,s)\longrightarrow(z,y+(x-z)s,x,s),\label{eq:r2tr}\\
&r_3:\;\;(x,y,z,s)\longrightarrow\left(\frac{(2-x)z+xy}{z+y},-\frac{z-y}{z+y}y,\frac{z-y}{z+y}z,2-s\right).\label{eq:r3tr}
\end{align}
One can consider these transformations as acting in $\mathbb{C}^4$. Straightforward calculations show that these
transformations are of order $2$:
\begin{equation}\label{eqs:r1r2r3Reflections}
r_1^2=r_2^2=r_3^2=1,
\end{equation}
where $1$ in ~\eqref{eqs:r1r2r3Reflections} above and in \eqref{eqs:reflections-relations} below denotes the transformation
corresponding to the identity map in $\mathbb{C}^4$. While the transformations $r_1$ and $r_2$ act in $\mathbb{C}^4$,
the transformation $r_3$ is not defined on the hyperplane $z+y=0$; moreover, if one desires to apply it twice in order to prove
the last relation in~\eqref{eqs:r1r2r3Reflections}, then one has to exclude the hyperplane $z-y=0$. If one wants to consider
the action of the group generated by the three transformations~\eqref{eq:r1tr}--\eqref{eq:r3tr}, then one has to remove
a countable number of surfaces from $\mathbb{C}^4$. We will not discuss this question further because we are primarily
interested in the action of these transformations on the surface~\eqref{eqs:contrmanifold-xyz}.

Restricted to the surface~\eqref{eqs:contrmanifold-xyz}, these reflections satisfy the relations
\begin{equation}\label{eqs:reflections-relations}
(r_3r_1)^4=(r_1r_3)^4=1\qquad
\text{and}\qquad
(r_3r_2)^2=(r_2r_3)^2=1.
\end{equation}
Of course, the action of $r_3$ is not defined on the entirety of the surface~\eqref{eqs:contrmanifold-xyz}, so that the
relations~\eqref{eqs:reflections-relations} are proved only for those points of \eqref{eqs:contrmanifold-xyz}
where the corresponding transformations are defined. In fact, as we show at the end of this section, one can
regularize the definition of $r_3$ on the surface~\eqref{eqs:contrmanifold-xyz} so that after the excision of a few
lines from the surface it is well defined.

We commence our considerations with the dihedral group $\mathcal{D}_{1,2}$ generated by $\{r_1,r_2\}$.
Let $\mathcal{N}_{1,2}$ be its normal subgroup with generator $r_1r_2$.
\begin{proposition}\label{prop:finiteOrbits}
There is a one-to-one correspondence between algebroid solutions of equation~\eqref{eq:dp3u}
and the finite orbits of the action of $\mathcal{N}_{1,2}$ on the monodromy manifold~\eqref{eqs:contrmanifold-xyz}.
The length of the finite orbits coincides with the order of a generator of the symmetry transformations for the
corresponding algebroid solutions.
\end{proposition}
\begin{proof}
Let $u(\tau,c)$ be a solution corresponding to the branching parameter $\rho\in\mathbb{C}$
($u(\tau,c)\underset{\tau\to0}{\sim} c\,\tau^{1-4\rho}$, where $c\in\mathbb{C}\setminus\{0\}$:
see Appendix~\ref{app:asympt0}); then the transformation corresponding
to the generator $r_1r_2$ is $u(\tau,c)\to-\mi u(\tau\me^{\pi\mi/2},-c)=u(\tau,-c\me^{-2\pi\mi\rho})$.
If $\rho\notin\mathbb{Q}$, then any finite number of such transformations give different solutions. On the other hand,
all algebroid solutions after a finite number of such iterations are mapped to themselves. For the $m$- and the $n$-series,
these symmetries are defined explicitly in Corollaries~\ref{cor:symmetryHmym} and \ref{cor:symmetryHnyn}.

The only solutions that remain are those possessing logarithmic behavior as $\tau\to0$ \cite{Kit87,KitVar2004}:
they have infinite orbits that are explicitly presented in Proposition~\ref{prop:s=-1s=3orbits} below.
\end{proof}

The action of $\mathcal{D}_{1,2}$ on $\mathbb{C}^4$ does not change the fourth coordinate. This fact allows us to treat
$s$ as a parameter, and to consider the action of $\mathcal{D}_{1,2}$ in $\mathbb{C}^3$ by regarding it as
the hyperplane $s=s_0$ in $\mathbb{C}^4$; in this case, we denote  this action by $\mathcal{D}_{1,2}(s_0)$.
\begin{proposition}\label{prop:dihedral-finite}
Define $\rho_1=1/2-\rho$, where $\rho$ is the branching parameter of the algebroid solution (see
Corollary~\ref{cor:Q=algebroid} and Remark~\ref{rem:rho-algebraic}$)$.
The group $\mathcal{D}_{1,2}(s_0)$ is finite iff $s_0$ is an algebraic number that can be written in the form
$s_0=1+2\cos(2\pi\rho_1)$, with $\rho_1\in\mathbb{Q}$ such that $0<2\rho_1<1$. In this case, the length of the orbit
of the normal subgroup $\mathcal{N}_{1,2}$ coincides the denominator of $\rho_1$ in its representation as an irreducible
fraction.
\end{proposition}
\begin{proof}
We proved in Proposition~\ref{prop:finiteOrbits} that, when acting on the contracted monodromy manifold, $\mathcal{D}_{1,2}$
has finite orbits for the points corresponding to the algebroid solutions. Here, we consider the action of
$\mathcal{D}_{1,2}(s_0)$ in $\mathbb{C}^3$. Consider the column-vector $\mathbf{R}=(x_0,y_0,z_0)^T\in\mathbb{C}^3$; then,
after $n$ iterations via $r_1r_2$, we arrive at the point $\widehat{P}_n(s_0)\mathbf{R}$, where $\widehat{P}_n(s_0)$ is a
$3\times3$ matrix with polynomial entries in $\mathbb{Z}[s_0]$. Assume that $s_0$ corresponds to an algebroid solution whose
orbit on the contracted monodromy manifold has length $n$. In this case, $\widehat{P}_n(s_0)\mathbf{R}=\mathbf{R}$, for
$\mathbf{R}$ defining a point on the monodromy manifold. We choose the following three points of the manifold,
$\mathbf{R}_1=(1,0,0)^T$, $\mathbf{R}_2=(0,1,0)^T$, and $\mathbf{R}_3=(s_0,0,1)^T$; then, the matrix
$\widehat{R}:=(\mathbf{R}_1,\mathbf{R}_2,\mathbf{R}_3)$ has unit determinant and satisfies the equation
$(\widehat{P}_n(s_0)-I)\widehat{R}=0$, where $I$ is the $3\times3$ identity matrix . Thus, for this value of $s_0$,
$\widehat{P}_n(s_0)=I$.
Since the matrix $\widehat{P}_n(s_0)$ does not depend on the initial point $\mathbf{R}$, it means that the condition
$\widehat{P}_n(s_0)\mathbf{R}=\mathbf{R}$ is true for any point $\mathbf{R}\in\mathbb{C}^3$.

Reverting to the proof of Proposition~\ref{prop:finiteOrbits}, since we know that the generator of the transformation
$r_1r_2$ is equivalent to the change of the parameters defining the solution $(\rho,c)\to(\rho,-c\me^{-2\pi\mi\rho})$
because $-c\me^{-2\pi\mi\rho}=c\me^{2\pi\mi(1/2-\rho)}=c\me^{2\pi\mi\rho_1}$, we see that the finite orbits are possible
only for rational $\rho_1$, and the lengths of these orbits coincide with the denominators of the irreducible representation
of the numbers $\rho_1$ as ratios of integers.

According to Theorem~\ref{th:B1asympt0} (see Appendix~\ref{app:asympt0}), the contracted Stokes multiplier $s=1+\mi s_0^0$
(cf.~\eqref{eq:contraction-variables}) is related to the branching parameter $\rho$ of the solution $u(\tau)$ as
$s=1-2\cos(2\pi\rho)$, with $0<2\rho<1$; hence, $s=1+2\cos(2\pi\rho_1)$, where $0<2\rho_1<1$. Since, for the algebroid
solutions, $\rho_1$ is rational, the corresponding numbers $s$ are algebraic \cite{Lehmer1933}, and the dihedral group
$\mathcal{D}_{1,2}(s_0)$ is finite for this $s_0=s$.
\end{proof}
\begin{remark}\label{rem:example-finiteorbits}
The length of the orbit of the normal subgroup $\mathcal{N}_{1,2}$ corresponding to solution $u(\tau)$ defined via
$H(r)$ (cf. Section~\ref{sec:algebroid}, Remark~\ref{rem:rho-algebraic} and Section~\ref{sec:mondata}, Lemma~\ref{lem:mondataH0})
equals $3$ because $s=0$, $\rho=1/6$, and $\rho_1=1/3$. The length of the orbit corresponding
to the solution $u(\tau)$ holomorphic at $\tau=0$ (cf. Section~\ref{sec:algebroid}, Remark~\ref{rem:rho-algebraic}) equals $4$
because $s=1$ and $\rho=\rho_1=1/4$.

Below, we define the set of minimal polynomials $q_k(s)$, $k\in\mathbb{N}$; for $k=3,4,\dotsc$, these polynomials
define the algebraic numbers $s=1+2\cos(2\pi\rho_1)$, $0<2\rho_1<1$, that coincide with the contracted Stokes multipliers
corresponding to the algebroid solutions. This set is defined so that the subscript $k$ of the polynomial $q_k(s)$ coincides
with the denominator of $\rho_1$ in its representation as an irreducible fraction, and thus with the length of the
corresponding orbit of $\mathcal{N}_{1,2}$.
\hfill $\blacksquare$\end{remark}
Since the minimal polynomials $q_k(s)$ defining the algebraic numbers $1+2\cos(2\pi\rho_1)$ for $\rho_1\in\mathbb{Q}$ and
$0<2\rho_1<1$ play an important role in the description of the algebroid solutions, we briefly recall the corresponding
construction in the notation adopted in this paper. The subject is well known \cite{Lehmer1933}, so that some
details of the proofs are omitted.

Consider the cyclotomic equation $\me^{2\pi\mi\rho_1 n}=1$. Use the Euler formula
$\me^{2\pi\mi\rho_1}=\cos(2\pi\rho_1)+\mi\sin(2\pi\rho_1)$ to find that the cyclotomic equation is
equivalent to $T_n(\cos(2\pi\rho_1))=\cos(2\pi\rho_1 n)=1$, where $T_n(x)$ is the $n$th Chebyshev polynomial of the first kind;
the explicit formulae for it can be found in \cite{BE2}:
$$
T_n(x)=\frac{n}{2}\sum_{m=0}^{\lfloor\frac{n}{2}\rfloor}(-1)^m\frac{(n-m-1)!}{m!(n-2m)!}(2x)^{n-2m}.
$$
With the aid of the Euler formula, it is easy to establish that the polynomial $T_n(x)-1$ for $n\geqslant2$ is always reducible,
and, moreover, the roots of the polynomials on the right-hand sides of the following identities are of order two,
\begin{gather}
2q_1(s)\left(T_{2n+1}\left(\frac{s-1}{2}\right)-1\right)=\left(\prod_{d\smallsetminus(2n+1)}q_d(s)\right)^2, \label{eq:T2n+1Q2}\\
2q_1(s)q_2(s)\left(T_{2n+2}\left(\frac{s-1}{2}\right)-1\right)=\left(\prod_{d\smallsetminus(2n+2)}q_d(s)\right)^2,\quad
n\in\mathbb{N}, \label{eq:T2n+2Q2}
\end{gather}
where $q_1(s)=2(\cos(2\pi\rho_1)-1)=2(s-1)/2-2=s-3$, and $q_1(s)q_2(s)=2(\cos(4\pi\rho_1)-1)=4\cos^2(2\pi\rho_1)-4=(s-3)(s+1)$,
so that $q_2(s)=s+1$. The polynomials $q_d(s)$ are assumed to be irreducible over $\mathbb{Z}$. Equations~\eqref{eq:T2n+1Q2}
and \eqref{eq:T2n+2Q2} allow one to recursively derive the polynomials $q_k(s)\in\mathbb{Z}[s]$ for all $k\in\mathbb{N}$.
If we assume that the polynomials are monic, then the sequence $q_k(s)$ satisfying the system~\eqref{eq:T2n+1Q2}
and \eqref{eq:T2n+2Q2} is unique.

It follows (by mathematical induction) from the Gau\ss\; identity for the Euler totient function, $\varphi(n)$,
\begin{equation*}
\sum_{d\smallsetminus n}\varphi(d)=n,
\end{equation*}
that $\deg q_k(s)=\varphi(k)/2$ for $k>2$. The set of roots of the polynomials $q_k(s)$, $k\in\mathbb{N}$, are,
by construction, real algebraic numbers that are dense on the segment $[-1,3]$.
The Galois group of the polynomials $q_k(s)$ is solvable, so that all its roots can be presented in terms of radicals.
Thus, the Stokes multipliers corresponding to the algebroid solutions can be expressed in terms of radicals.

In \cite{WatkinsZeitlin1993}, the authors, using identities for the Chebyshev  polynomials,
derive, from the system~\eqref{eq:T2n+1Q2} and \eqref{eq:T2n+2Q2}, a more convenient system that allows one to
recursively obtain the polynomials $q_k(s)$:
\begin{align}
2(T_{n+1}((s-1)/2)-T_n((s-1)/2))&=\prod_{d\smallsetminus(2n+1)}q_d(s),\label{eq:CebyshevOddQ}\\
2(T_{n+1}((s-1)/2)-T_{n-1}((s-1)/2))&=\prod_{d\smallsetminus(2n+2)}q_d(s),\qquad
n\in\mathbb{N}.\label{eq:ChebyshevEvenQ}
\end{align}
We list below the first $18$ polynomials $q_k(s):=q_k$ derived with the help of equations~\eqref{eq:CebyshevOddQ} and
\eqref{eq:ChebyshevEvenQ}:
\begin{align*}
q_1&=s-3,\;
q_2=s+1,\;
q_3=s,\;
q_4=s-1,\;
q_5=s^2-s-1,\;
q_6=s-2,\;
q_7=s^3-2s^2-s+1\;\\
q_8&=s^2-2s-1,\;
q_9=s^3-3s^2+3,\;
q_{10}=s^2-3s+1,\;
q_{11}=s^5-4s^4+2s^3+5s^2-2s-1,\\
q_{12}&=s^2-2s-2,\;
q_{13}=s^6-5s^5+5s^4+6s^3-7s^2-2s+1,\;
q_{14}=s^3-4s^2+3s+1,\\
q_{15}&=s^4-5s^3+5s^2+5s-5,\;
q_{16}=s^4-4s^3+2s^2+4s-1,\\
q_{17}&=s^8-7s^7+14s^6+s^5-25s^4+9s^3+12s^2-3s-1,\;
q_{18}=s^3-3s^2+1.
\end{align*}
As follows from Corollary~\ref{cor:Q=algebroid}, the boundary values of $s$, i.e., $s=-1$ ($2\rho=1$) and $s=3$ ($2\rho=0$),
are the roots of the polynomials $q_1(s)$ and $q_2(s)$. In fact, we know that for these values of $s$ there correspond
solutions of equation~\eqref{eq:dp3u} for $a=0$ that have logarithmic behaviour as $\tau\to0$ \cite{Kit87,KitVar2004}.
Thus, according to Proposition~\ref{prop:dihedral-finite}, the corresponding dihedral group $\mathcal{D}_{1,2}(s_0)$,
$s_0=-1,3$, is infinite.
\begin{proposition}\label{prop:s=-1s=3orbits}
Let $s=-1$ or $s=3$ and the point $(x_0,y_0,z_0)$ belong to $\mathbb{C}^3$ or to the manifold of the monodromy
data~\eqref{eqs:contrmanifold-xyz}; then, the orbits of the normal subgroup $\mathcal{N}_{1,2}(s)$ in $\mathbb{C}^3$
or on the manifold of the monodromy data are infinite. The points, after $n\in\mathbb{Z}_{\geqslant0}$
iterations, have the following co-ordinates:
\begin{align*}
&s=-1;&
x_n&=(-1)^n\left\lfloor1+\frac{n}{2}\right\rfloor x_0+(-1)^{n+1}\left\lfloor\frac{n+1}{2}\right\rfloor y_0
+n(\mathrm{mod}(2))\;z_0,\\
&&
y_n&=(-1)^n\left\lfloor\frac{n}{2}\right\rfloor x_0+(-1)^{n-1}\left\lfloor\frac{n-1}{2}\right\rfloor y_0
+n(\mathrm{mod}(2))\;z_0,\\
&&
z_n&=(-1)^{n+1}\left\lfloor\frac{n+1}{2}\right\rfloor x_0+(-1)^{n}\left\lfloor\frac{n}{2}\right\rfloor y_0
+(n+1)(\mathrm{mod}(2))\;z_0,\\
&s=3;&
x_n&=\frac{(n+1)(n+2)}{2}\,x_0+\frac{n(n+1)}{2}\,y_0+\big(1-(n+1)^2\big) z_0,\\
&&
y_n&=\frac{n(n-1)}{2}\,x_0+\frac{(n-2)(n-1)}{2}\,y_0+\big(1-(n+1)^2\big) z_0,\\
&&
z_n&=\frac{n(n+1)}{2}\,x_0+\frac{(n-1)n}{2}\,y_0+(1-n^2) z_0.\\
\end{align*}
\end{proposition}
\begin{proof}
By mathematical induction:
the base of the induction, $n=0$, can be verified immediately, and the inductive step is straightforward to make
with the help of the explicit formula for the transformation $r_1r_2$ (cf. \eqref{eq:r1tr} and \eqref{eq:r2tr}).
\end{proof}

Now, consider the transformation $r_3$ (cf. \eqref{eq:r3tr}). This transformation is interesting for us provided that it acts
on the monodromy manifold~\eqref{eqs:contrmanifold-xyz}, and thus its action can be extended to the space of solutions;
therefore, we consider its action on the monodromy manifold rather than on $\mathbb{C}^4$.

If we want to apply this transformation once to a point $\mathcal{P}$ of the monodromy
manifold~\eqref{eqs:contrmanifold-xyz}, then the co-ordinates of $\mathcal{P}$ should satisfy the condition
$\mathcal{P}\neq(1,0,0,s)$ or $\mathcal{P}\neq(x,x-1,1-x,-1)$, $x,s\in\mathbb{C}$.

The first condition can, however, be regularized; in this case, both the numerators and denominators of the proposed image of
$r_3$ (cf. \eqref{eq:r3tr}) are zeros. Using this fact, one can set $z+y=\varepsilon$,\footnote{Not be confused with $\varepsilon$
in the Introduction.} and rewrite the equations
defining the monodromy manifold~\eqref{eqs:contrmanifold-xyz} in the following manner:
\begin{equation}\label{eqs:xyz-epsilon}
z=\frac{\varepsilon}{2}+\sqrt{\frac{\varepsilon}{s+1}+\frac{\varepsilon^2(s-3)}{4(s+1)}},\qquad
y=\frac{\varepsilon}{2}-\sqrt{\frac{\varepsilon}{s+1}+\frac{\varepsilon^2(s-3)}{4(s+1)}},\qquad
x=zs+1-\varepsilon.
\end{equation}
Substituting these equations into \eqref{eq:r3tr} and considering the limit $\varepsilon\to0$, one finds that
\begin{equation*}\label{eq:r3-100s}
r_3(1,0,0,s)=\left(\frac{1-s}{1+s},\frac{2}{1+s},\frac{2}{1+s},2-s\right),\quad
s\in\mathbb{C}\setminus\{-1\}.
\end{equation*}
One can readily verify that this definition implies $r_3^2(1,0,0,s)=(1,0,0,s)$, so that the third relation
in ~\eqref{eqs:r1r2r3Reflections} holds.
We see that the only problem occurs when $s=-1$. For $s=-1$, the monodromy manifold consists of two lines,
$(x,x-1,1-x,-1)$ and $(x,x+1,-x,-1)$, where $x\in\mathbb{C}$. They are two different lines which are related by the
symmetry $r_1(x,x-1,1-x,-1)=(x_1,x_1+1,-x_1,-1)$, where $x_1=x-1$.
On the other hand, $r_3(x,x+1,-x,-1)=((2x-1)x,(2x+1)(x+1),(2x+1)x,3)$: the last quadric curve provides a rational
parametrization for the monodromy manifold when $s=3$. Therefore, we can correctly define the action of $r_3$ on
the first curve $(x,x-1,1-x,-1)$ only in the sense of projective geometry; however, having in mind an application
to the theory of the degenerate third Painlev\'e equation, we do not consider this option.
Note that, for all other points of the monodromy manifold, any transformation $r_3w(r_1,r_2)$, where $w$ is any word
consisting of two letters $r_1$ and $r_2$, can be
regularized in a natural way; e.g., $r_3r_1(0,1,0,s)=r_3(1,0,0,s)$ and $r_3r_2(0,s,1,s)=r_3(1,0,0,s)$, or the more complicated
examples, $r_3(r_1r_2)^3(1,s(s-1),s,s)=r_3(1,0,0,s)$ and $r_3(r_1r_2)^4=(s,(s-1)(s^2-s-1),s(s-1))=r_3(1,0,0,s)$. In the last
two examples, one can, upon using equations~\eqref{eqs:r1r2r3Reflections}, certainly find a general formula for the
transformations $r_3(r_1r_2)^3$ and $r_3(r_1r_2)^4$, and then apply the limiting procedure of the type delineated above.
The latter limiting procedure, however, is significantly more elaborate than that described by equations~\eqref{eqs:xyz-epsilon}.

Thus, in case one would like to consider the action of the complete group $\mathcal{G}$ (with generators
$r_1$, $r_2$, and $r_3$) on the monodromy manifold, one has to remove from it
those points with the fourth co-ordinate $s=-1$ and $s=3$. The previous considerations suggest the following construction:
for any $s\in\mathbb{C}\setminus\{-1,3\}$ and $s\neq1$, consider in $\mathbb{C}^3$, with co-ordinates $x$, $y$, and $z$,
the following two planes:
\begin{gather}\label{eqs:hyperplane}
\mathcal{H}_s:\qquad
x+y+z(1-s)=1,\\
\mathcal{H}_{2-s}:\qquad
x+y-z(1-s)=1.
\end{gather}
The intersection of these planes is the line $x+y=1$, which lies on the plane $z=0$. Consider the
quadric $z(z-1)=xy$: each plane $\mathcal{H}_k$, $k=s,2-s$, intersects it by a conic $\mathcal{C}_k$. These two conics
have two common points, $(1,0,0)$ and $(0,1,0)$, in $\mathbb{C}^3$; however, in $\mathbb{C}^4$, instead of these points, we have
two pairs of points: $(1,0,0,s)$ and $(1,0,0,2-s)$, and  $(0,1,0,s)$ and $(0,1,0,2-s)$. Therefore, the correct geometric object
for the action of the group $\mathcal{G}$ in $\mathbb{C}^3$ is its restriction on the disjoint sum of two conics
$\mathcal{C}_s\sqcup\mathcal{C}_{2-s}$. Denote this restriction as $\mathcal{G}(s)$: this notation assumes that
$\mathcal{G}(s)\equiv\mathcal{G}(2-s)$.
In the case $s=1$, the planes $\mathcal{H}_s$ and $\mathcal{H}_{2-s}$ coincide and,
instead of the disjoint sum of conics, we have one conic $\mathcal{C}_1$.

The dihedral group $\mathcal{D}_{1,2}(s_0)$ for $s_0=s$ and $s_0=2-s$ acts on the conics $\mathcal{C}_s$ and $\mathcal{C}_{2-s}$,
respectively, whilst the transformation $r_3$ maps the points of one conic to another, e.g., $r_3(0,1,0,s)=(0,1,0,2-s)$.
\begin{proposition}\label{prop:CoxeterAction}
The Coxeter group $\mathcal{G}(s)$ is finite iff $s$ is the Stokes multiplier corresponding to algebroid solutions
of equation~\eqref{eq:dp3u} for $a=0$, or, in other words, it is a root of some polynomial $q_m(s)$. In this case,
$2-s$ is also a root of some polynomial $q_n(s)$ and $\text{ord}\,\mathcal{G}(s)=4\max\{m,n\}$.
\end{proposition}

\section{Large-$r$ Asymptotics of $H(r)$ and Numerical Aspects}\label{sec:asymptnumerics}
In Section~\ref{sec:mondata}, the $\tau\to+0$ asymptotics of the general solution $u(\tau)$ of equation~\eqref{eq:dp3u} is used
in order to determine the monodromy data corresponding to the function $H(r)$.
This data constitutes the set of parameters that enables one to determine, with the help of the results
derived in \cite{KitVar2004,KitVar2010}, the asymptotics as $r\to-\infty$ of the function $H(r)$ and the corresponding integral
$I(r):=\smallint_r^0\frac{1}{\sqrt{-r}H(r)}\,\md r$.
For the convenience of the reader, all the necessary asymptotic results from \cite{KitVar2004,KitVar2010}
are collected in Appendices~\ref{app:asympt0} and ~\ref{app:infty} below. In this section, we present and
compare the asymptotic and numerical results for several solutions corresponding to different choices of the initial
value $H(0)$.

Before we present the corresponding asymptotic formulae,
let us comment on the numerical calculations. The function $H(r)$ and the corresponding monodromy data are defined
via the initial value $H(0)$; on the other hand, equation~\eqref{eq:hazzidakis} defining $H(r)$ is singular at $r=0$.
Strictly speaking, one has to take a step from $r=0$ to $r=r_1$: this step should be smaller than the radius of convergence
of the series (cf. \eqref{ineq:N-final} and \eqref{ineq:R-final}) representing $H(r)$, and then calculate, with the help of this
series, the initial data for $H(r)$ at $r=r_1$. Theoretically, this calculation can be executed with arbitrary precision.
Our calculations are performed via \textsc{Maple} 16 and 17. We found that, in case we want to calculate only the function $H(r)$ for
initial data at $r=0$ given by Gaussian rationals, then, in many (but not all!) cases, the standard \textsc{Maple} procedure for
the numerical solution of ODEs was able to correctly calculate the corresponding solution; at least visually the plots obtained
by the `simplified' procedure and the `correct' method coincide.
This might be occurring because (cf. equation~\eqref{eq:Hprime0-H0}) $H'(0)=H(0)^2-1/H(0)$
appears to be a much more complicated Gaussian rational than $H(0)$, thus \textsc{Maple} treats it in floating-point arithmetic,
hence making this relation only approximately valid, and applies general numerical procedures. Clearly, for generic Cauchy
data specified at singular points of an ODE, the regular solution does not exist.

Theorem~\ref{th:asympt-infty2004} (see Appendix~\ref{app:infty}) implies the following asymptotics for $H(r)$:
\begin{equation}\label{eq:H-asympt-Large-regular}
H(r)\underset{r\to-\infty}{=}1-\sqrt{6\nu_1}\cdot
\frac{\cos\big(\psi(r)+o\big(r^{-\delta}\big)\big)}{\sqrt[4]{-3r}},\qquad
\delta>0,
\end{equation}
where
\begin{gather}
\psi(r)=2\sqrt{-3r}+\frac{\nu_1}{2}\ln{(-3r)}+\nu_1\ln(24)+\frac{3\pi}{4}
-\frac{3\pi\mi}{2}\nu_1-\frac{\mi}{2}\ln(2\pi)+\mi\ln\left(\tilde{g}_1\sqrt{\nu_1}\,\Gamma(\mi\nu_1)\right),
\label{eq:psi-for-H-regular}\\
\nu_1:=\frac{\ln{\tilde{g}_3}}{2\pi},\qquad
\left|\text{Im}\,\nu_1\right|<\frac16=0.1666\ldots.\label{eq:cond:nu+1regular}
\end{gather}
Here and below, the following natural conventions for the branches of multi-valued functions are assumed: (i)
the branches of all roots of positive numbers are positive; (ii) the branch of a root of a parameter, which may take positive
values, is fixed to be positive, and then further defined via analytic continuation; and (iii) the branches of logarithms of
positive numbers are real.

The parameter $\nu_1$ is uniquely defined via equation~\eqref{eq:cond:nu+1regular}. It is related to the parameter
$\tilde\nu$ in Appendix~\ref{app:infty} via the relation $\mi\nu_1=\tilde\nu+1$.
The branch of $\sqrt{\nu_1}$ can be chosen as per the above conventions; however, the particular choice for
the branch of $\sqrt{\nu_1}$ is not important, provided that the following natural branch matching is assumed:
$\sqrt{6\nu_1}=\sqrt{6}\cdot\sqrt{\phantom{!}\!\nu_1}$ (cf.~\eqref{eq:H-asympt-Large-regular} and \eqref{eq:psi-for-H-regular}).

The asymptotics~\eqref{eq:H-asympt-Large-regular} is not valid for $H(0)=1$, $\me^{\pm\frac{2\pi\mi}{3}}$
because the monodromy parameters $\tilde{g}_1$ and $\tilde{g}_3$ are not defined for these values of $H(0)$. We can, however,
formally consider that, for $H(0)=1$, the asymptotics remains valid, since it is known
(cf. Corollary~\ref{cor:algebraicMONDATA}, item (1)) that
$\tilde{g}_3=1$, which implies that $\nu_1=0$ if we assume that the cosine function remains finite, in which case
$H(r)\to1$ as $r\to-\infty$, which is consistent with the fact that $H(r)=1$ for all $r$.

An additional restriction on the initial value $H(0)$ is provided by equation~\eqref{eq:cond:nu+1regular}: this condition not
only fixes the branch of the logarithm, but also imposes a condition on $H(0)$.

For the same conditions as for asymptotics~\eqref{eq:H-asympt-Large-regular}, the following asymptotic formula is valid:
\begin{equation}\label{eq:asympt-reg-int}
\int_{r}^0\frac{\md r}{\sqrt{-r}H(r)}\underset{r\to-\infty}{=}
2\sqrt{-r}+2\nu_1\ln(2+\sqrt{3})+\mi\ln\left(
\frac{\me^{\frac{2\pi\mi}{3}}H(0)-\me^{-\frac{2\pi\mi}{3}}}{\me^{\frac{2\pi\mi}{3}}-H(0)\me^{-\frac{2\pi\mi}{3}}}\right)
+\mathrm{E}(r),
\end{equation}
where
\begin{equation}\label{eq:corr-asympt-reg-int}
\mathrm{E}(r)=\frac{\sqrt{6\nu_1}}{2}
\cdot\frac{\sin\big(\psi(r)+o\big(r^{-\delta}\big)\big)}{\sqrt[4]{-3r}}.
\end{equation}
While the branch of the right-most logarithmic term in equation~\eqref{eq:psi-for-H-regular} is not important, the branch of
the right-most logarithmic term in equation~\eqref{eq:asympt-reg-int} is fixed by the condition $|\text{Im}\,\ln(\cdot)|<\pi$;
the last condition does not impose any additional restrictions on the initial value $H(0)$ since
$|\text{Im}\,\ln(\cdot)|=\pi$ only for $H(0)=\infty$. In terms of the initial value $H(0)$,
the important condition~\eqref{eq:cond:nu+1regular} for the validity of the asymptotic
formulae~\eqref{eq:H-asympt-Large-regular} and \eqref{eq:asympt-reg-int} reads:
\begin{equation*}\label{eq:asympt-regular-cond-H0}
\left|\text{arg}\left(H(0)+\frac{1}{H(0)}+1\right)\right|<\frac{\pi}{3}.
\end{equation*}

In the following, we present two examples of the application of the asymptotic formulae~\eqref{eq:H-asympt-Large-regular} and
\eqref{eq:asympt-reg-int}, and compare them with the corresponding numerical plots for $H(r)$ and $I(r)$. These functions are
finite at $r=0$, whilst the denominators of their asymptotics contain the factor $\sqrt[4]{-3r}$; therefore, the plots of the
functions $H(r)$ and $I(r)$ and their asymptotics are compared outside of some small neighbourhoods of the origin, which are
specified in the figure captions.

\subsection{Example 1: $H(0)=-\frac{1}{30}-\mi$}\label{subsec:example1}
For this value of $H(0)$, $\nu_1=-0.185823\ldots-\mi0.0001892\ldots$, so that the
condition~\eqref{eq:cond:nu+1regular} is satisfied. For the \textsc{Maple} calculations we choose the
parameter \texttt{Digits}=50, and the procedure that provides maximum precision for this value of \texttt{Digits}.
Each plot of the numeric solutions and the corresponding integrals presented below is based on the calculation of $500$ points.
We also choose the initial point $r_1=-10^{-6}$. In principle, for the calculation of $H(r)$, one can choose a smaller value
for the number of digits and larger values for $r_1$, \texttt{Digits}=20 and $r_1=-0.001$, say; however, these parameters
do not produce `correct' numerical precision for the integral $\smallint_{r}^0\tfrac{1}{\sqrt{-r}H(r)}\md r$. Our choice for
these parameters is close to the optimal values, i.e., an increase in the accuracy of calculations is not noticeable on the
plots of the corresponding functions; for example, using, say, the parameter values \texttt{Digits}=100 and $r_1=-10^{-9}$,
the \textsc{Maple} calculations produce plots that are visually indistinguishable from those presented in
Figs.~\ref{fig:H0=-1over30-i+ReH}--\ref{fig:H0=-1over30-i+IntImH}.\footnote{\label{foot:notebooksCalc1}
\textsc{Maple} $16$, on a notebook with 4Gb RAM and processor Intel(R) Core(TM) i7-3517U (3rd Generation),
makes $50$-digits calculation in $84\mathrm{s}$, whilst $100$-digits calculation takes $103\mathrm{s}$. In preparing this work,
we redid the calculations with the help \textsc{Maple} $2017$ on a more modern notebook with 16Gb RAM and processor
Intel(R) Core(TM) i7-12700H (12th Generation): these calculations were executed in $60\mathrm{s}$ and $65\mathrm{s}$,
respectively.}
\begin{figure}[htpb]
\begin{center}
\includegraphics[height=50mm,width=100mm]{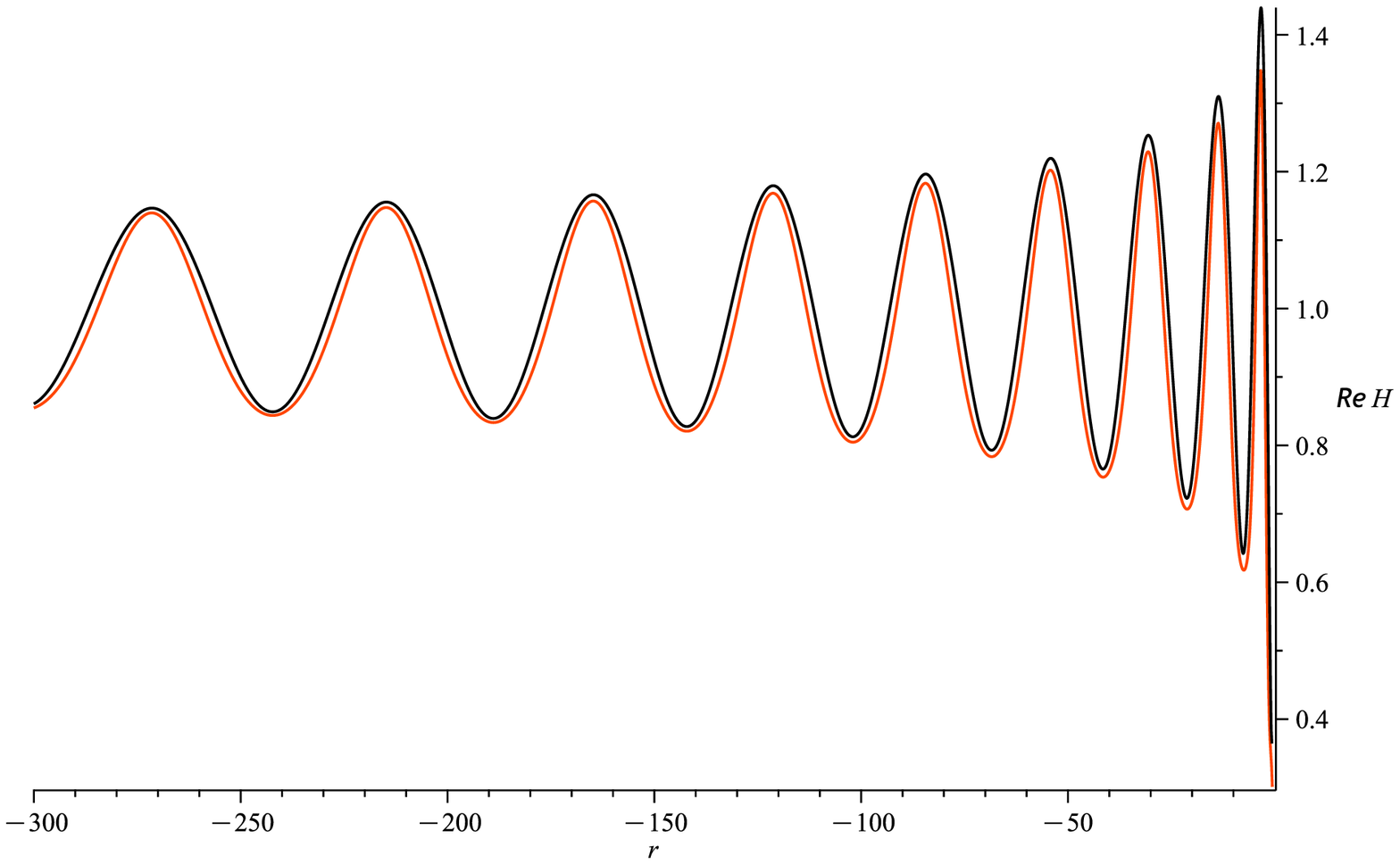}
\caption{The red and black plots are, respectively, the real parts of the numeric and large-$r$ asymptotic values
of the function $H(r)$ for $r\leqslant-0.6$ corresponding to the initial value $H(0)=-1/30-\mi$.}
\label{fig:H0=-1over30-i+ReH}
\end{center}
\end{figure}
\begin{figure}[htpb]
\begin{center}
\includegraphics[height=50mm,width=100mm]{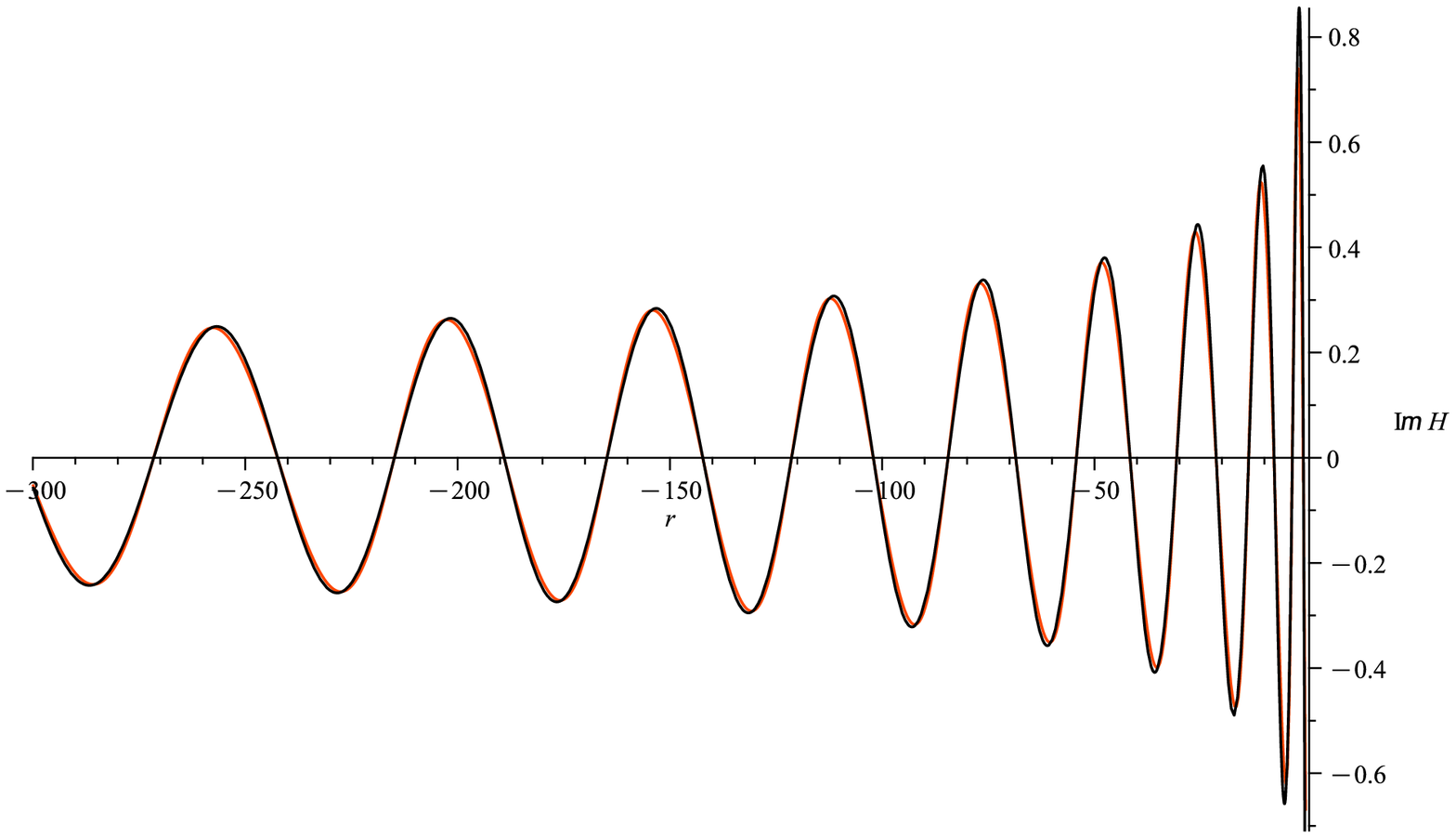}
\caption{The red and black plots are, respectively, the imaginary parts of the numeric and large-$r$ asymptotic values
of the function $H(r)$ for $r\leqslant-0.5$ corresponding to the initial value $H(0)=-1/30-\mi$.}
\label{fig:H0=-1over30-i+ImH}
\end{center}
\end{figure}
\begin{figure}[htpb]
\begin{center}
\includegraphics[height=50mm,width=100mm]{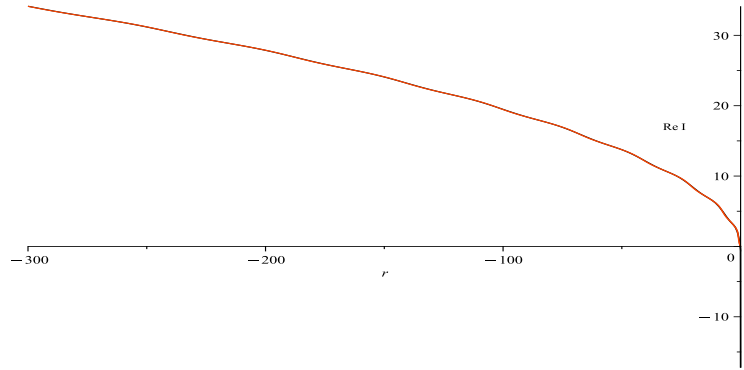}
\caption{The red and black plots are, respectively, the real parts of the numeric and large-$r$ asymptotic values
of $I=\smallint_{r}^0\tfrac{1}{\sqrt{-r}H(r)}\,\md r$ for $r\leqslant-10^{-7}$ corresponding to the initial value $H(0)=-1/30-\mi$.
On this scale, both plots almost coincide. In Figure~\ref{fig:H0=-1over30-i+IntCorrReH}, the reader will see a more detailed
comparison of these plots.}
\label{fig:H0=-1over30-i+IntReH}
\end{center}
\end{figure}
\begin{figure}[htpb]
\begin{center}
\includegraphics[height=50mm,width=100mm]{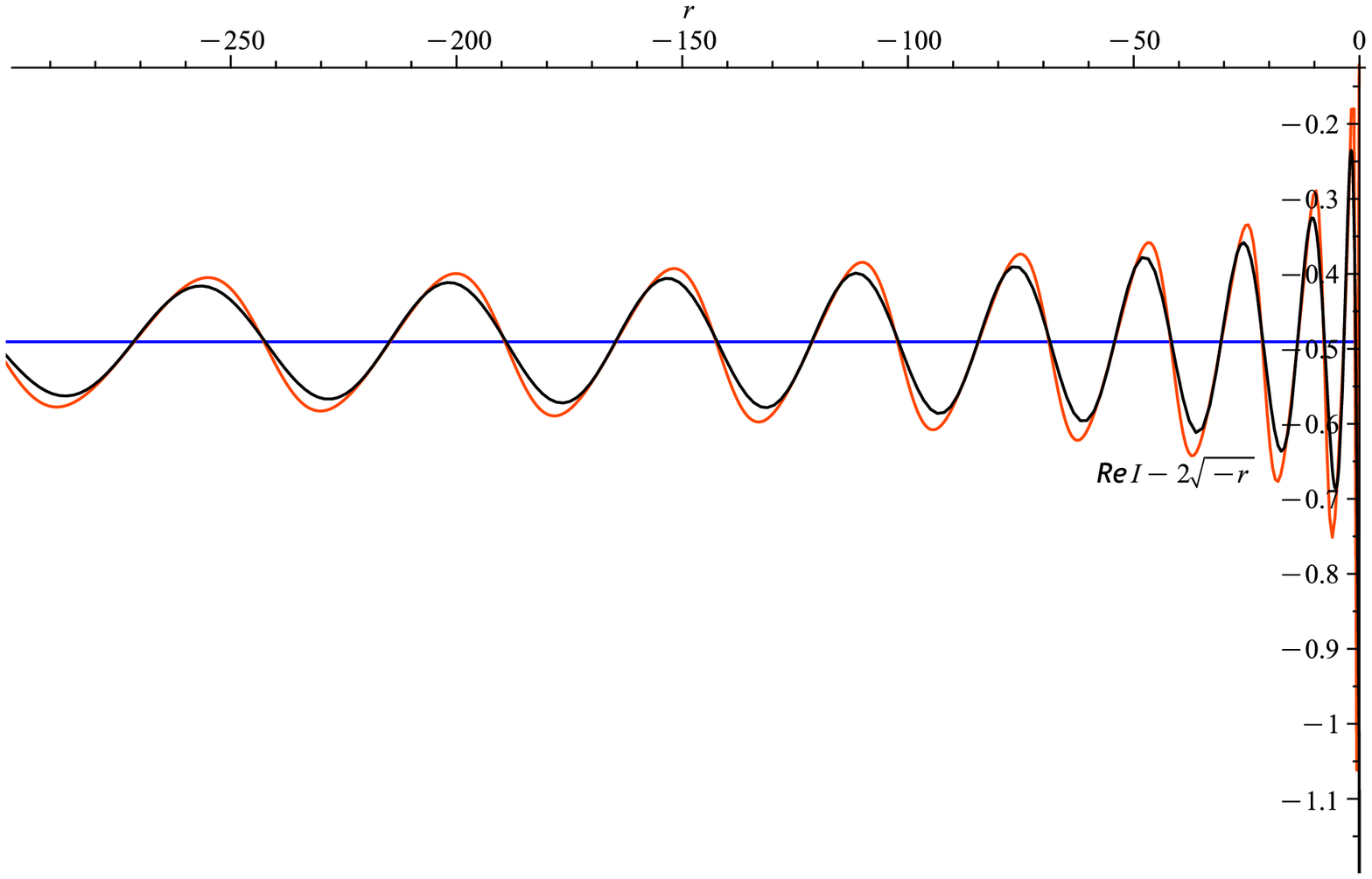}
\caption{The red and black plots are, respectively, the real parts of the numeric and large-$r$ asymptotic values
of $\smallint_{r}^0\tfrac{1}{\sqrt{-r}H(r)}\,\md r-2\sqrt{-r}$ for $r\leqslant-10^{-7}$ corresponding to the initial value
$H(0)=-1/30-\mi$. The blue line is the real part of the asymptotics~\eqref{eq:asympt-reg-int} modulo the
parabola $2\sqrt{-r}$ and $\mathrm{E}(r)$.}
\label{fig:H0=-1over30-i+IntCorrReH}
\end{center}
\end{figure}
\begin{figure}[htpb]
\begin{center}
\includegraphics[height=50mm,width=100mm]{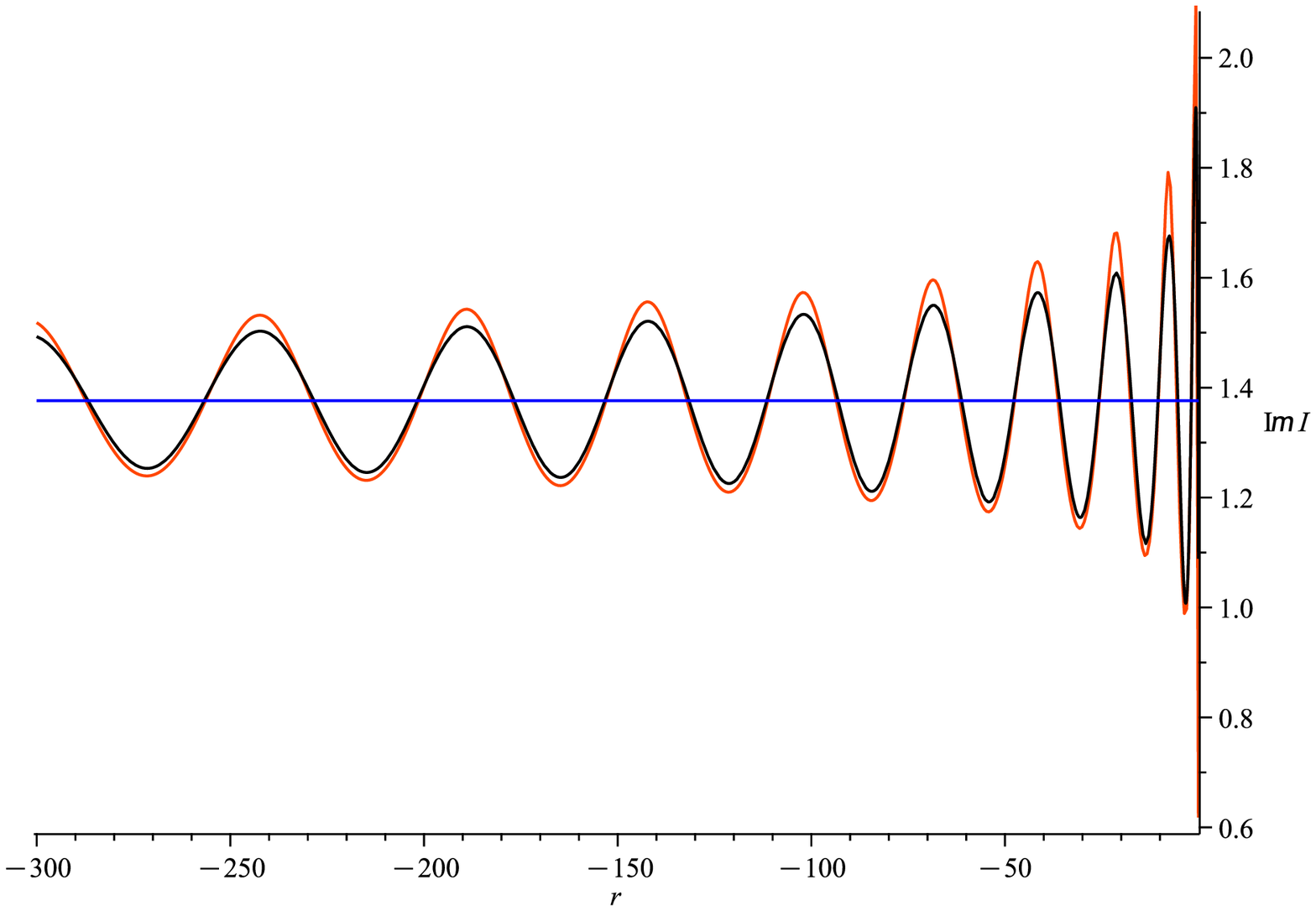}
\caption{The red and black plots are, respectively, the imaginary parts of the numeric and large-$r$ asymptotic values
of $\smallint_{r}^0\frac{1}{\sqrt{-r}H(r)}\,\md r$ for $r\leqslant-0.1$ corresponding to the initial value $H(0)=-1/30-\mi$.
The blue line is the imaginary part of the asymptotics~\eqref{eq:asympt-reg-int} modulo $\mathrm{E}(r)$.}
\label{fig:H0=-1over30-i+IntImH}
\end{center}
\end{figure}

\subsection{Example 2: $H(0)=60-\mi100$}\label{subsec:example2}
For this value of $H(0)$, $\nu_1=0.583250\ldots-\mi0.162814\ldots$, so that the
condition~\eqref{eq:cond:nu+1regular} is satisfied. Here, however, we are closer to the boundary of the interval of validity
\eqref{eq:cond:nu+1regular} for the asymptotics~\eqref{eq:H-asympt-Large-regular}.
For the \textsc{Maple} calculations, we choose the same settings as in Example 1. The results of these calculations are
presented in Figs.~\ref{fig:H0=60-i100+ReH}--\ref{fig:H0=60-i100+IntImH}: these $50$-digits calculations take approximately
$133\text{s}$ on an older notebook (see footnote~\ref{foot:notebooksCalc1}).
\begin{figure}[htpb]
\begin{center}
\includegraphics[height=50mm,width=100mm]{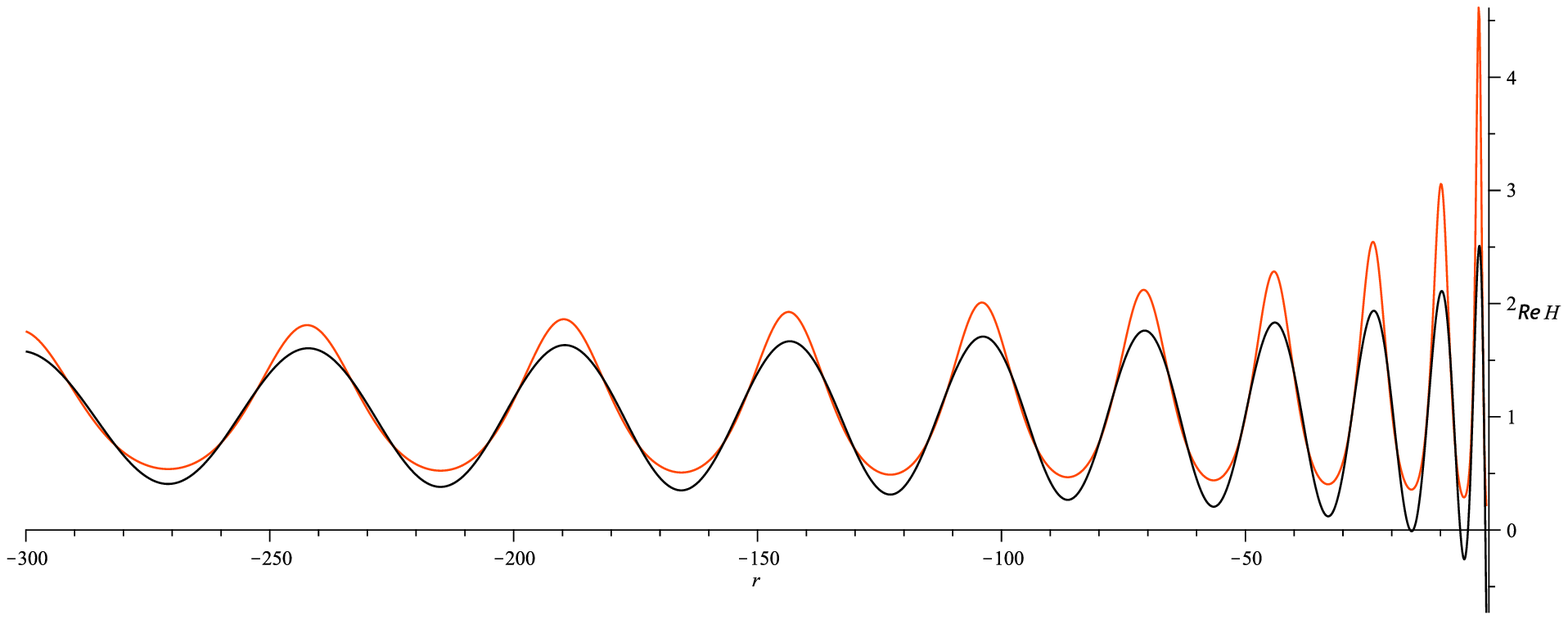}
\caption{The red and black plots are, respectively, the real parts of the numeric and large-$r$ asymptotic values
of the function $H(r)$ for $r\leqslant-0.6$ corresponding to the initial value $H(0)=60-\mi100$}
\label{fig:H0=60-i100+ReH}
\end{center}
\end{figure}
\begin{figure}[htpb]
\begin{center}
\includegraphics[height=50mm,width=100mm]{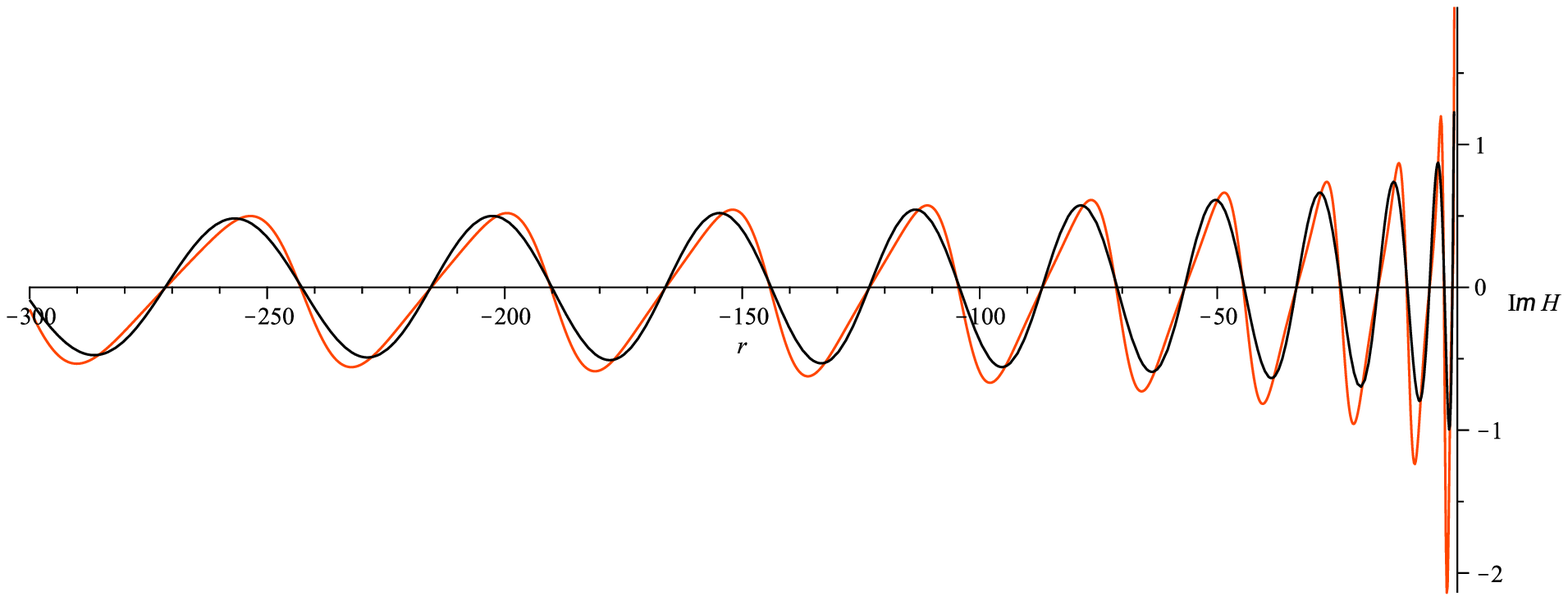}
\caption{The red and black plots are, respectively, the imaginary parts of the numeric and large-$r$ asymptotic values
of the function $H(r)$ for $r\leqslant-0.1$ corresponding to the initial value $H(0)=60-\mi100$.}
\label{fig:H0=60-i100+ImH}
\end{center}
\end{figure}
\begin{figure}[htpb]
\begin{center}
\includegraphics[height=50mm,width=100mm]{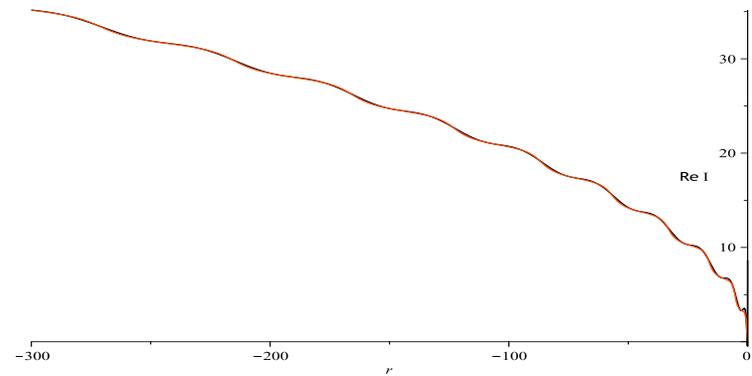}
\caption{The red and black plots are, respectively, the real parts of the numeric and large-$r$ asymptotic values
of $\smallint_{r}^0\tfrac{1}{\sqrt{-r}H(r)}\,\md r$ for $r\leqslant-10^{-6}$ corresponding to the initial value
$H(0)=60-\mi100$. Both plots virtually coalesce into one curve; however, one notices that this curve has some thick segments
coloured in red from one side and in black from the other. A more detailed comparison of these plots is presented in
Figure~\ref{fig:H0=60-i100+IntCorrReH}.}
\label{fig:H0=60-i100+IntReH}
\end{center}
\end{figure}
\begin{figure}[htpb]
\begin{center}
\includegraphics[height=50mm,width=100mm]{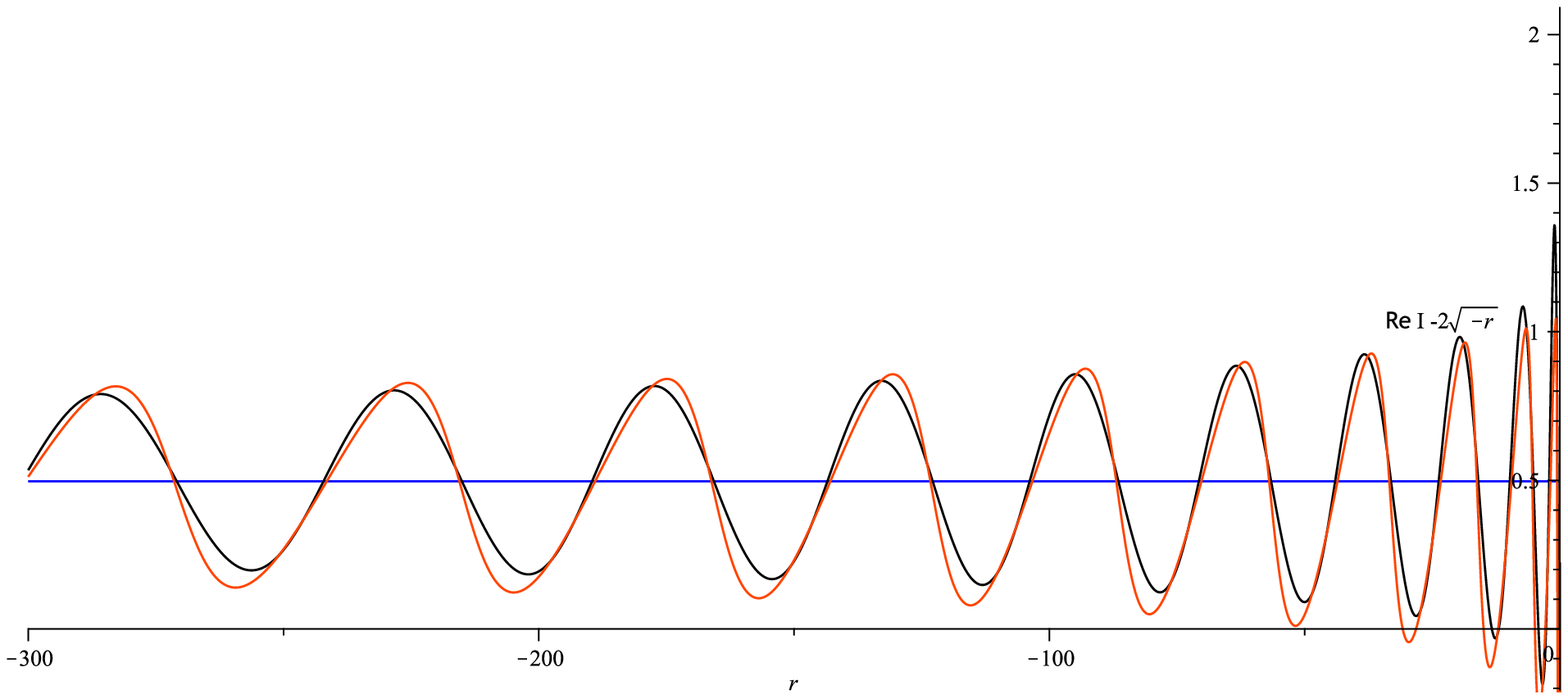}
\caption{The red and black plots are, respectively, the real parts of the numeric and large-$r$ asymptotic values
of $\smallint_{r}^0\frac{1}{\sqrt{-r}H(r)}\,\md r-2\sqrt{-r}$ for $r\leqslant-10^{-2}$ corresponding to the initial value
$H(0)=60-\mi100$. The blue line is the real part of the asymptotics~\eqref{eq:asympt-reg-int} modulo the
parabola $2\sqrt{-r}$ and $\mathrm{E}(r)$.}
\label{fig:H0=60-i100+IntCorrReH}
\end{center}
\end{figure}
\begin{figure}[htpb]
\begin{center}
\includegraphics[height=50mm,width=100mm]{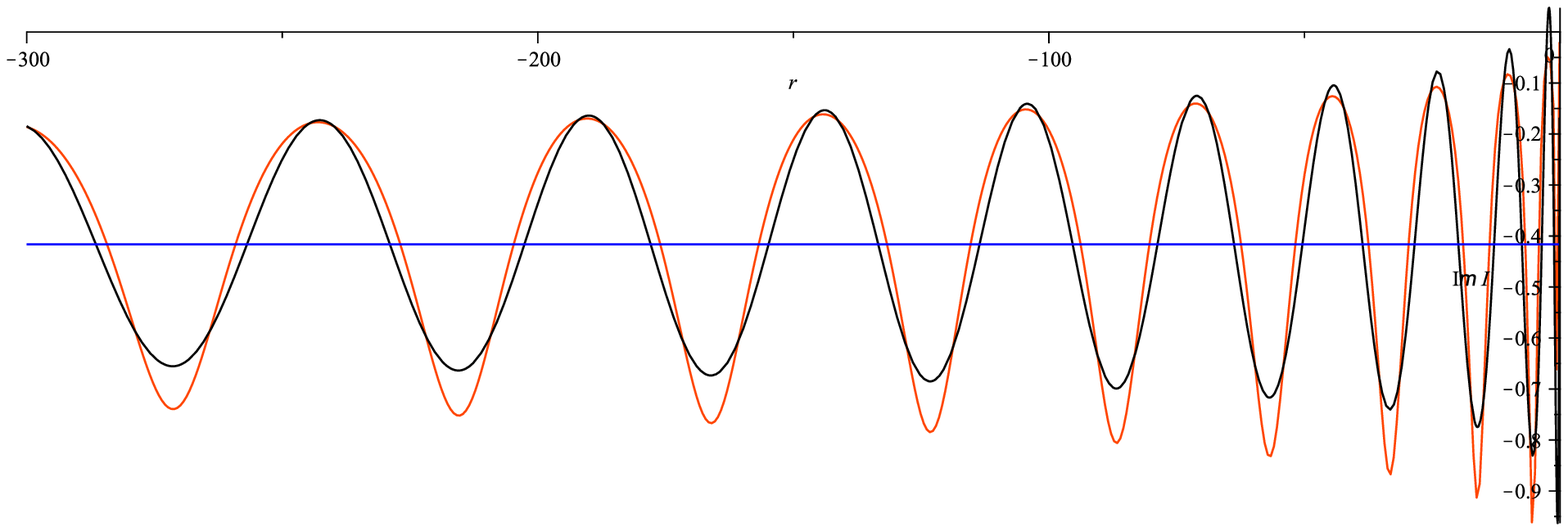}
\caption{The red and black plots are, respectively, the imaginary parts of the numeric and large-$r$ asymptotic values
of $\smallint_{r}^0\frac{1}{\sqrt{-r}H(r)}\,\md r$ for $r\leqslant-0.1$  corresponding to the initial value $H(0)=60-\mi100$.
The blue line is the imaginary part of the asymptotics~\eqref{eq:asympt-reg-int} modulo $\mathrm{E}(r)$.}
\label{fig:H0=60-i100+IntImH}
\end{center}
\end{figure}
\begin{remark}\label{rem:regular-asymptotics}
In principle, the asymptotics~\eqref{eq:H-asympt-Large-regular}  and  \eqref{eq:asympt-reg-int} are valid
even  beyond the condition~\eqref{eq:cond:nu+1regular}, that is, $|\text{Im}\,\nu_1|<1/2$; however, going farther and farther
beyond the condition~\eqref{eq:cond:nu+1regular}, the visual correspondence between the solution and its asymptotics
is attained at larger and larger values of $|r|$. The correction terms to the asymptotics~\eqref{eq:H-asympt-Large-regular}
and  \eqref{eq:asympt-reg-int} (see Appendix~\ref{app:infty}) improve the correspondence, but do not alter the general
tendency: the correspondence between the numeric and asymptotic values for $H(r)$ and $I(r)$ deteriorates for finite,
though quite large, values of $r$. As one approaches $|\text{Im}\,\nu_1|\approx 2/5$, the values of $r$ for which the
visual correspondence is expected appear to be so large that it becomes unfeasible to calculate the solution for such values
of $r$.  This tendency is partially illustrated upon comparing
Figs.~\ref{fig:H0=-1over30-i+ReH} and \ref{fig:H0=-1over30-i+ImH} with Figs.~\ref{fig:H0=60-i100+ReH} and
\ref{fig:H0=60-i100+ImH}, respectively, and the corresponding figures with the plots for the integrals: the visual
correspondence worsens when the boundary value of the condition~\eqref{eq:cond:nu+1regular} is approached. Example~6 below
illustrates at which values of $r$ one may expect to observe that the asymptotics~\eqref{eq:H-asympt-Large-regular}  and
\eqref{eq:asympt-reg-int} faithfully describe the large-$r$ behaviour of $H(r)$ and $I(r)$.

In order to cope with the problem indicated above, we derive in \cite{KitVar2010} yet another asymptotic formula
which will be discussed below.
\hfill$\blacksquare$\end{remark}

We now turn our attention to the numerical illustration of the asymptotic results for the function $H(r)$ that follow from
Theorem~\ref{th:asympt-infty2010} (see Appendix~\ref{app:infty}):
\begin{equation}\label{eq:H-asympt-Large-singular}
H(r)\underset{r\to-\infty}{=}1-\frac3{2\sin^2\left(\frac12\hat\psi(r)+o\big(r^{-\delta}\big)\right)},
\qquad
\delta>0,
\end{equation}
where
\begin{gather}
\hat\psi(r)=2\sqrt{-3r}+\left(\nu_1+\frac{\mi}{2}\right)\ln\big(24\sqrt{-3r}\big)-\frac{\pi}4-\frac{3\pi\mi}{2}\nu_1
-\frac{\mi}{2}\ln(2\pi)+\mi\ln\big(\tilde{g}_1\Gamma(\mi\nu_1)\big),\label{eq:H-asympt-Large-singular-phase}\\
\nu_1:=\frac{\ln{\tilde{g}_3}}{2\pi},\qquad
\text{Im}\,\nu_1\in(-1,0)\setminus\lbrace -1/2 \rbrace.\label{eq:nu1-singular}
\end{gather}
We draw the attention of the reader to the fact that, even though the parameter $\nu_1$ in \eqref{eq:nu1-singular} is defined
by the same formula as $\nu_1$ in \eqref{eq:cond:nu+1regular}, its imaginary part is fixed in another interval. Clearly,
both asymptotics \eqref{eq:H-asympt-Large-regular} and \eqref{eq:H-asympt-Large-singular} have intersecting domains of validity.
The asymptotics for the integral $I(r)$ reads:
\begin{equation}\label{eq:asympt-sing-int}
\int_{r}^0\frac{\md r}{\sqrt{-r}H(r)}\underset{r\to-\infty}{=}
2\sqrt{-r}+(2\nu_1+\mi)\ln(2+\sqrt{3})+\pi(2k-1)+\mi\ln\left(
\frac{{H(0)\me^\frac{2\pi\mi}{3}}-\me^{-\frac{2\pi\mi}{3}}}{\me^{\frac{2\pi\mi}{3}}-H(0)\me^{-\frac{2\pi\mi}{3}}}\right)
+\mathcal{E}(r),
\end{equation}
where $k\in\mathbb{Z}$,
\begin{equation}\label{eq:corr-asympt-sing-int}
\mathcal{E}(r)=-\mi\ln\left(\frac{\sin\big(\frac12\hat\psi(r)+\theta_0+o\big(r^{-\delta}\big)\big)}
{\sin\big(\frac12\hat\psi(r)-\theta_0+o\big(r^{-\delta}\big)\big)}\right),\qquad
\theta_0=-\frac{\pi}{2}+\frac{\mi}{2}\ln\big(2+\sqrt{3}\big).
\end{equation}
\begin{remark}\label{rem:integerK}
The number $k=k(H(0))\in\mathbb{Z}$ cannot be determined via the isomonodromy technique developed in \cite{KitVar2010} because,
in fact, asymptotics of the function $\exp\left(\smallint_{r}^0\frac{1}{\sqrt{-r}H(r)}\,\md r\right)$ is studied there.
It looks like an interesting problem to determine the integer $k$ in terms of the initial value $H(0)$. In the numerical examples considered
below, we find $k$ by comparing the numerical solution for $H(r)$ with its asymptotics~\eqref{eq:asympt-sing-int}. Since
$k$ is an integer, its experimental determination in this situation does not represent a problem. The values
of $k$ are given in the corresponding figure captions.
\hfill $\blacksquare$\end{remark}
\begin{remark}\label{rem:best-intervals-asympt2004}
If $\text{Im}\,\nu_1\in(-1,-5/6)$, which corresponds to the range of validity of the
asymptotics~\eqref{eq:H-asympt-Large-singular} and \eqref{eq:asympt-sing-int} (cf.  condition~\eqref{eq:nu1-singular}), then
one can fix the branch of $\nu_1$ such that $\text{Im}\,\nu_1\in(0,1/6)$, so that the
asymptotics~\eqref{eq:H-asympt-Large-regular} and \eqref{eq:asympt-reg-int} can be used; moreover, the closer
$\text{Im}\,\nu_1$ is to $-1$ or to $0$, the better (in the sense of approximating $H(r)$ at finite values) work
the asymptotics~\eqref{eq:H-asympt-Large-regular} and \eqref{eq:asympt-reg-int}.
\hfill $\blacksquare$\end{remark}
\begin{remark}\label{rem:correctionFORsingularINTEGRAL}
This remark concerns the term $\mathcal{E}(r)$ in the asymptotic formula~\eqref{eq:asympt-sing-int}
(cf. \eqref{eq:corr-asympt-sing-int}). It is assumed that one has to take a continuous branch of the logarithmic function
as $r$ varies from $0$ to $-\infty$. This problem is important in numerical calculations because {\sc Maple} calculates
only the principal branch of the logarithmic function; therefore, plotting the asymptotic formula ~\eqref{eq:corr-asympt-sing-int}
on a large-$r$ interval almost always produces a saw-like plot. To cope with this problem, one can present the logarithm
in ~\eqref{eq:corr-asympt-sing-int} as an integral of the difference of cotangent functions. This idea works on
small-length intervals; however, the calculation of this integral on the segment $[-10,0]$, say, already consumes a considerable
amount of time, several times longer than the numerical calculation of the solution of the Painlev\'e equation itself!
Consequently, in lieu of the integral, we define the continuous branch of the logarithm as the solution of the differential
equation for the above-mentioned integral.
It appears that {\sc Maple} numerically evaluates this solution much faster than its explicit form represented by the integral
in terms of cotangents. The execution time of this fast calculation, however, where only $20$-digits of accuracy
were maintained, appears to be $5$ to $8$ times longer than the numerical calculation of the integral on the left-hand side of
equation~\eqref{eq:asympt-sing-int} based on its direct calculation  (with $120$-digits of accuracy) via the Painlev\'e
transcendent!
There is, thus, a significant difference between the numerical calculations using the asymptotic
formulae~\eqref{eq:asympt-reg-int} and \eqref{eq:asympt-sing-int}:
plots of the first asymptotic formulae~\eqref{eq:asympt-reg-int} are calculated almost immediately, whilst
the second asymptotics~\eqref{eq:asympt-sing-int} is interesting for theoretical
studies until a fast algorithm for the calculation of continuous branches of the logarithmic functions will be implemented into
{\sc Maple} (and analogous codes).
\hfill $\blacksquare$\end{remark}
Below, we consider several examples that illustrate some features of the asymptotic
formulae~\eqref{eq:H-asympt-Large-singular} and \eqref{eq:asympt-sing-int}.

\subsection{Example 3: $H(0)=-0.148+\mi0.191$}\label{subsec:example3}
For this value of $H(0)$, $\nu_1=0.0249933\ldots-\mi0.329580\ldots$, so that the
condition~\eqref{eq:nu1-singular} is satisfied, and we can use the asymptotic formulae~\eqref{eq:H-asympt-Large-singular}
and \eqref{eq:asympt-sing-int}.
For the numerical calculation of the solution $H(r)$ and its related integral $I(r)$ via \textsc{Maple}, we consider the
initial values for $H(r)$ at $r_1=10^{-8}$ and use 4 terms of the corresponding Taylor series expansions for the calculation
of the initial values, $H(r_1)$ and $H^\prime(r_1)$, and the parameter \texttt{Digits} $= 120$.
For the calculation of the corresponding asymptotics, \texttt{Digits} is set equal to $20$.
In principle, for the numerical calculation of $H(r)$ and $I(r)$, one can take a smaller number of digits, $60$, say, and the
resulting calculation with $60$ digits produces visually the same pictures.

Figs.~\ref{fig:H0=-0148+i0191+ReH20}, \ref{fig:H0=-0148+i0191+ImH20}, \ref{fig:H0=-0148+i0191+IntReH20}, and
\ref{fig:H0=-0148+i0191+IntImH20}
demonstrate that the large-$r$ asymptotics also give quite good approximations for the corresponding functions
even for small values of $r$. We took $800$ points for the calculation of these plots.
The calculation from scratch of the numerical solution presented in Figure~\ref{fig:H0=-0148+i0191+IntReH20} with
an older notebook (see footnote~\ref{foot:notebooksCalc1}) takes $20\text{s}$, whilst the calculation of its asymptotics
represented by the blue line of the same figure takes $170\text{s}$ (cf. Remark~\ref{rem:correctionFORsingularINTEGRAL}).
In order to produce the numerical solution and the corresponding integral shown in
Figs.~\ref{fig:H0=-0148+i0191+ReH}, \ref{fig:H0=-0148+i0191+ImH}, \ref{fig:H0=-0148+i0191+IntReH}, and
\ref{fig:H0=-0148+i0191+IntImH} by the red curves, which are virtually concealed under the blue curves, a newer
notebook (cf. footnote~\ref{foot:notebooksCalc1}) requires $188\text{s}$ of execution time. The blue curves representing
the asymptotic values were generated in approximately $1015\text{s}$ (cf. Remark~\ref{rem:correctionFORsingularINTEGRAL}).
For the generation of these plots, we increased the number of points from $800$ to $1837$ because the
plots have sharp peaks and the heights of the peaks depend substantially on the number and position of the points inside the
peaks.

Figure~\ref{fig:H0=-0148+i0191+IntReH} resembles a mole's dwelling that can be reached with the aid of two
ladders.\footnote{The mole seems to be the main personnage in the Carlo Goldoni comedy \emph{Il servitore di due padroni}, or,
perhaps, works as a double agent.}
The slopes of all the steps and risers of the left staircase are negative. The right staircase possesses steps with negative
slopes and risers with positive slopes. Obviously, the sign of the slopes coincides with the sign of $-\mathrm{Re}\,(H(r))$.
If we look at Figure~\ref{fig:H0=-0148+i0191+ReH}, we observe, in fact, that the function $\mathrm{Re}\,(H(r))$ is negative
on the segments near its local minima located near the origin; however, these minima are growing monotonically and slowly,
as long as one moves in the negative direction, farther and farther away from the origin. What is not clearly seen from the
figure, however, is the fact that the
calculations show that the last segment where $H(r)<0$ occurs for $r\in(-414.68789\ldots,-398.07403\ldots)$, which corresponds
to the right wall (first riser) of the mole's dwelling. The left wall of the mole's dwelling is the image of a segment around
a local minimum at $r_{\mathrm{min}}=-482.3271\ldots$ where $H(r_{\mathrm{min}})=0.0022342881\ldots$. Due to the very small
values of $|H(r)|$ on these segments, the walls look almost vertical, although, in fact, they have negative and positive slopes
for the left and right walls, respectively.
\begin{figure}[htpb]
\begin{center}
\includegraphics[height=50mm,width=100mm]{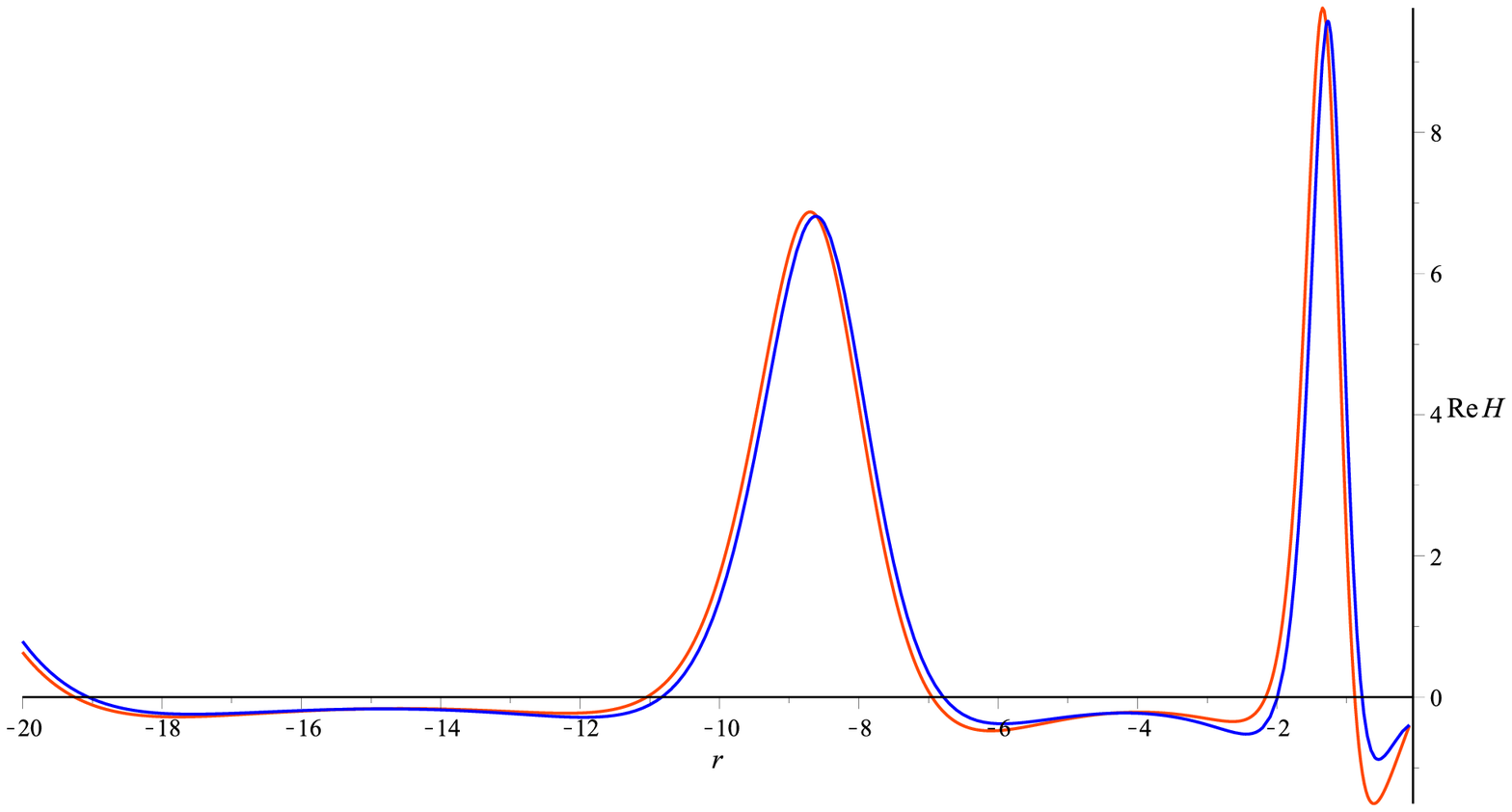}
\caption{The red and blue plots are, respectively, the real parts of the numeric and large-$r$ asymptotic
(cf. \eqref{eq:H-asympt-Large-singular}) values of the function $H(r)$ for $r\leqslant-0.1$ corresponding to
the initial value $H(0)=-0.148+\mi0.191$.}
\label{fig:H0=-0148+i0191+ReH20}
\end{center}
\end{figure}
\begin{figure}[htpb]
\begin{center}
\includegraphics[height=70mm,width=100mm]{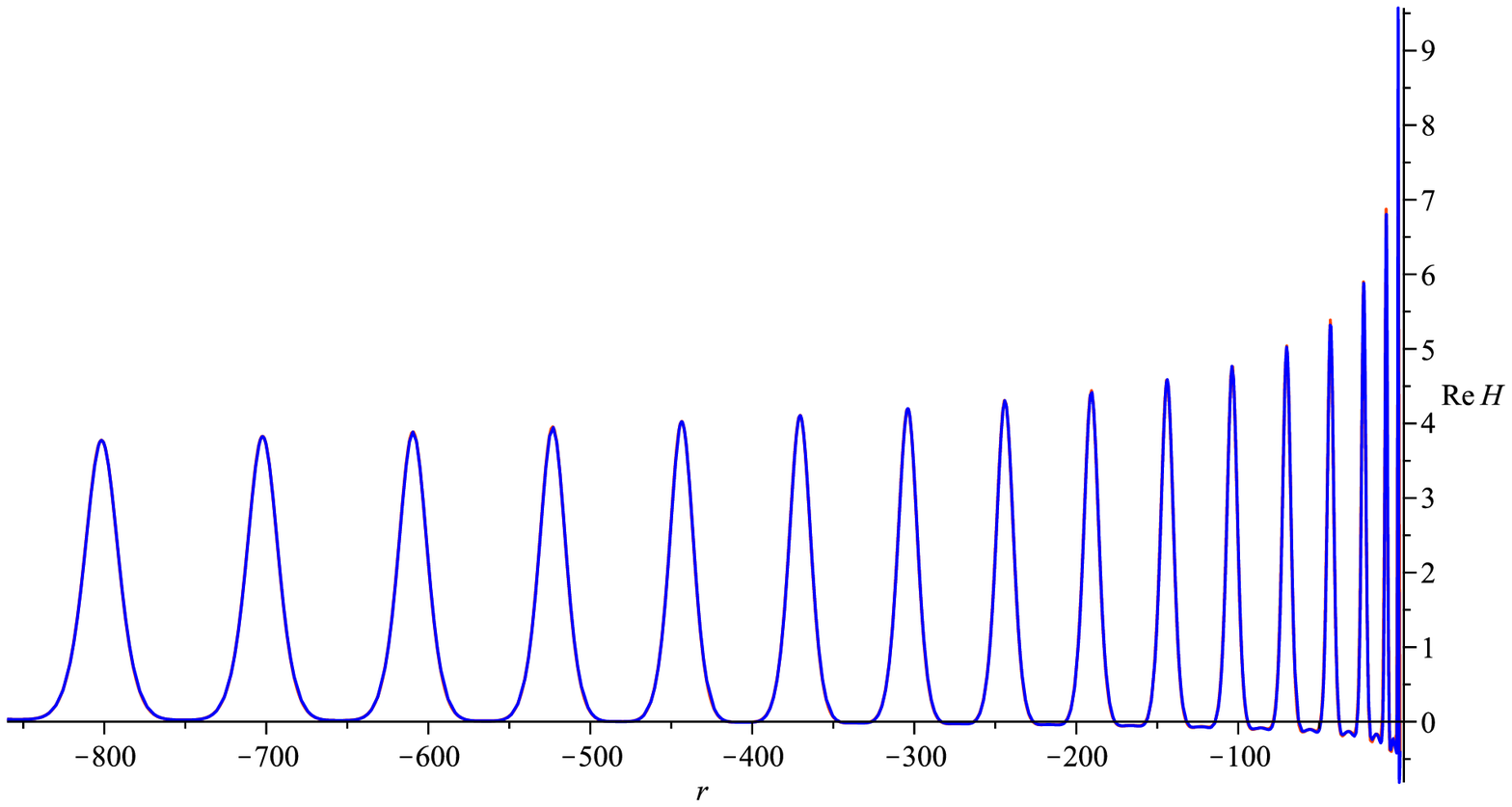}
\caption{Extended version of Figure~\ref{fig:H0=-0148+i0191+ReH20} (both plots almost coincide).}
\label{fig:H0=-0148+i0191+ReH}
\end{center}
\end{figure}
\begin{figure}[htpb]
\begin{center}
\includegraphics[height=50mm,width=100mm]{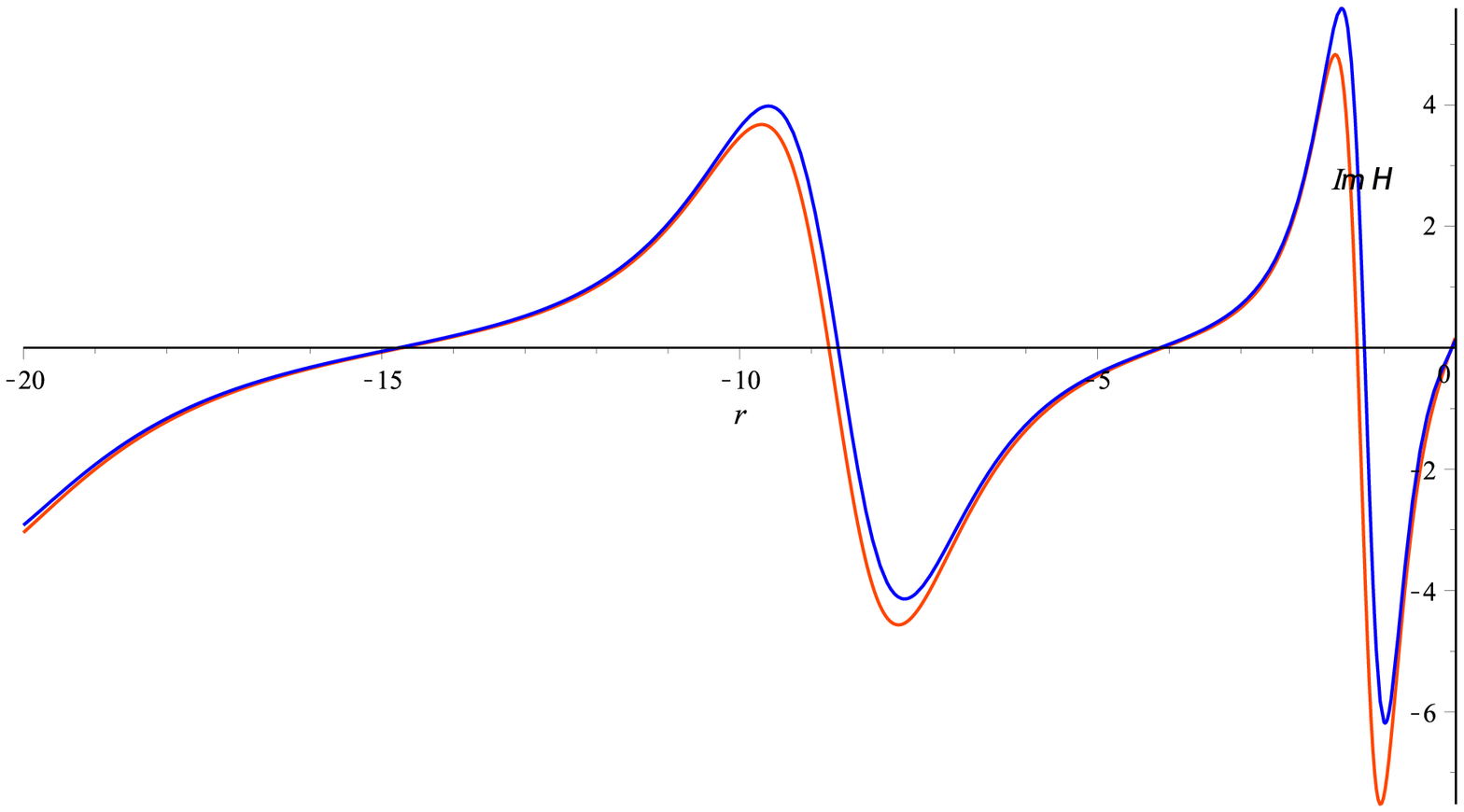}
\caption{The red and blue plots are, respectively, the imaginary parts of the numeric and large-$r$ asymptotic
(cf. \eqref{eq:H-asympt-Large-singular}) values of the function $H(r)$  for $r\leqslant-0.01$ corresponding to the initial value
$H(0)=-0.148+\mi0.191$.}
\label{fig:H0=-0148+i0191+ImH20}
\end{center}
\end{figure}
\begin{figure}[htpb]
\begin{center}
\includegraphics[height=70mm,width=100mm]{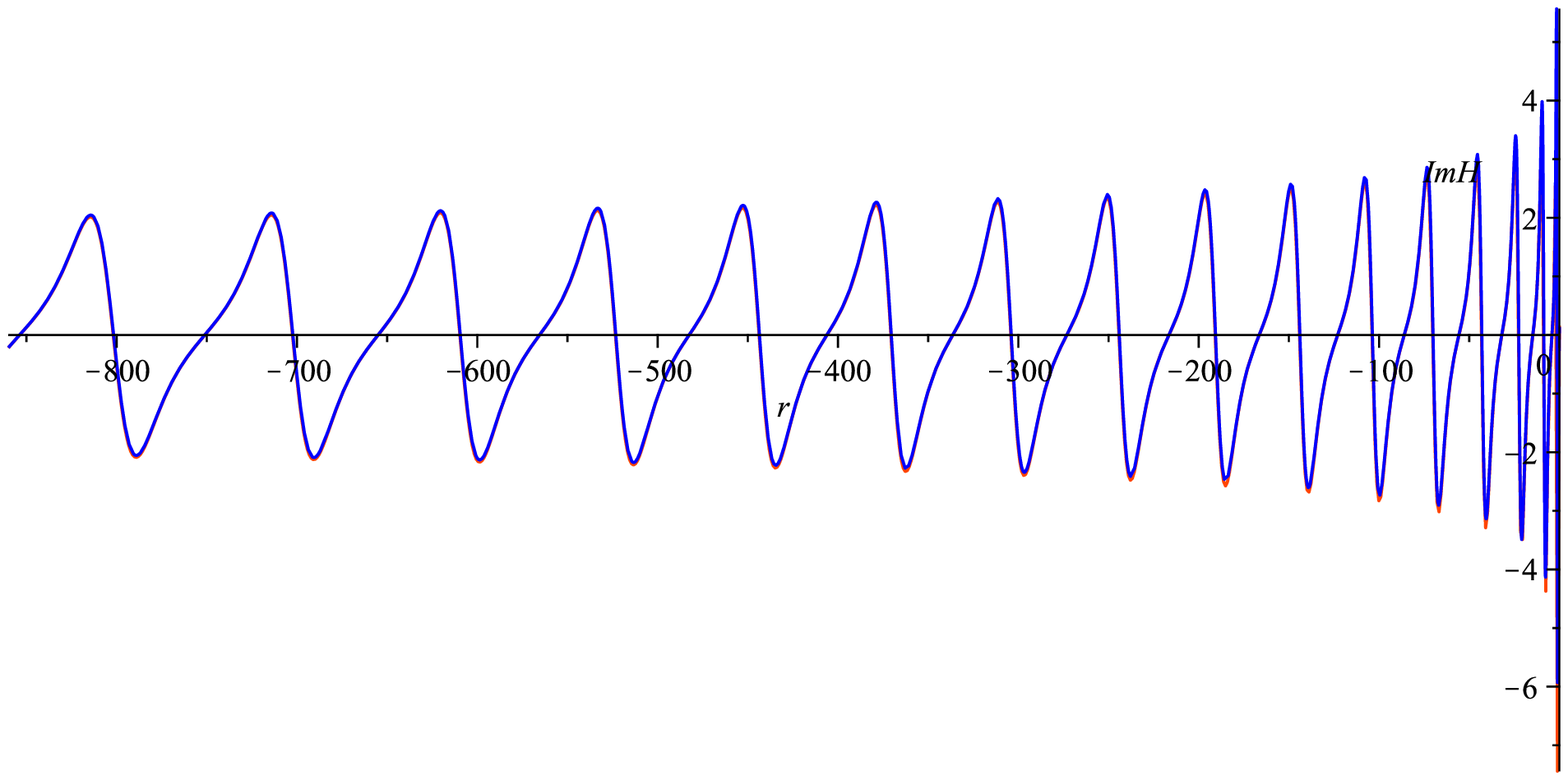}
\caption{Extended version of Figure~\ref{fig:H0=-0148+i0191+ImH20} (both plots almost coincide).}
\label{fig:H0=-0148+i0191+ImH}
\end{center}
\end{figure}
\begin{figure}[htpb]
\begin{center}
\includegraphics[height=60mm,width=100mm]{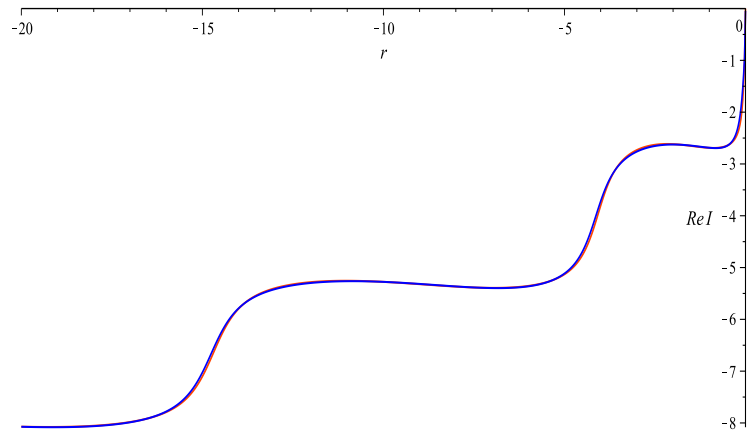}
\caption{The red and blue plots are, respectively, the real parts of the numeric and large-$r$ asymptotic
(cf. \eqref{eq:asympt-sing-int} with $k=0$) values
of $I=\smallint_{r}^0\frac{1}{\sqrt{-r}H(r)}\,\md r$ for $r\leqslant-10^{-8}$, where $H(r)$ is the solution with initial value
$H(0)=-0.148+\mi0.191$. In black-and-white  the curves are indistinguishable; in colour, though, one notices minor
discrepancies between the curves.}
\label{fig:H0=-0148+i0191+IntReH20}
\end{center}
\end{figure}
\begin{figure}[htpb]
\begin{center}
\includegraphics[height=70mm,width=100mm]{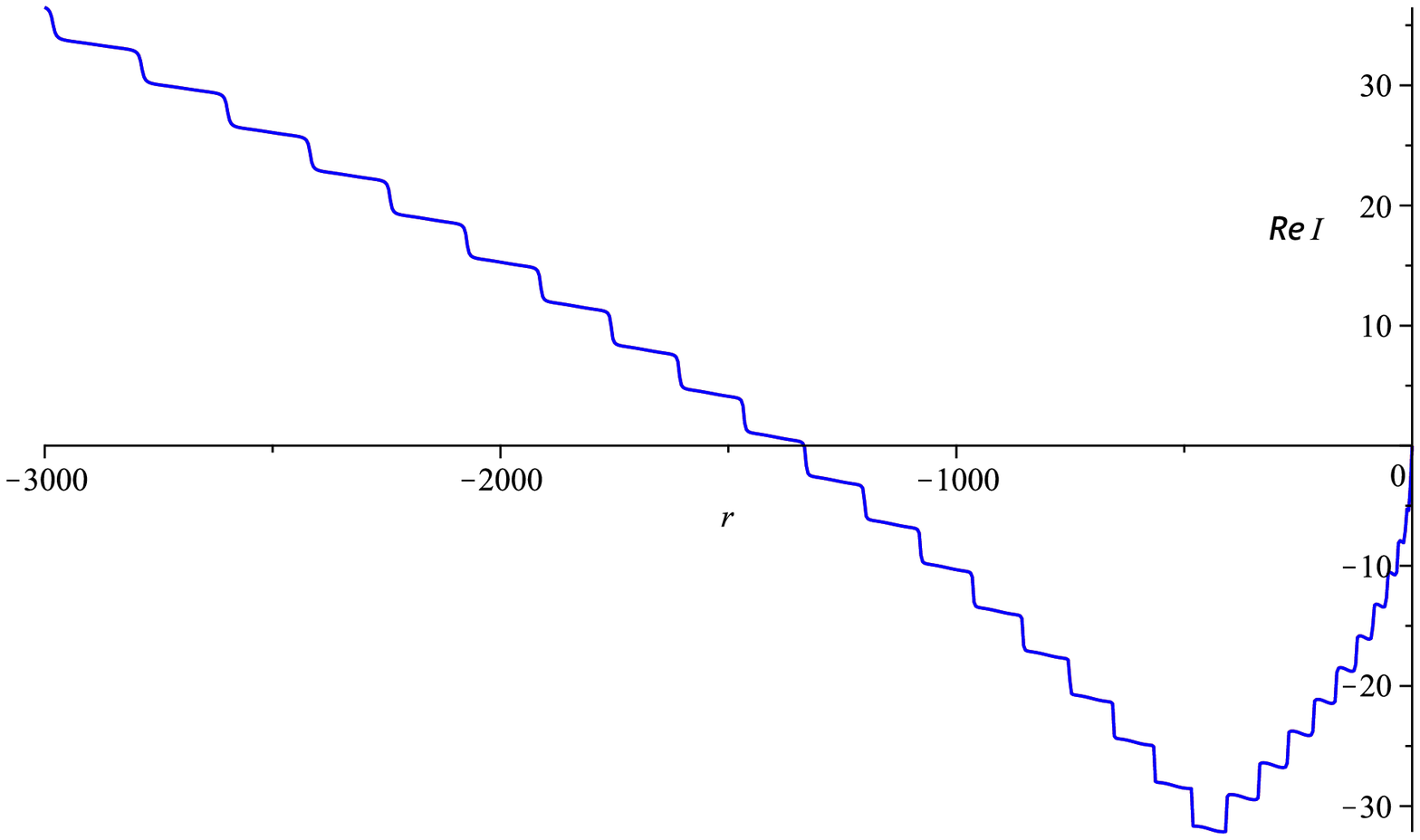}
\caption{Extended version of Figure~\ref{fig:H0=-0148+i0191+IntReH20} (both plots almost coincide).}
\label{fig:H0=-0148+i0191+IntReH}
\end{center}
\end{figure}
\begin{figure}[htpb]
\begin{center}
\includegraphics[height=50mm,width=100mm]{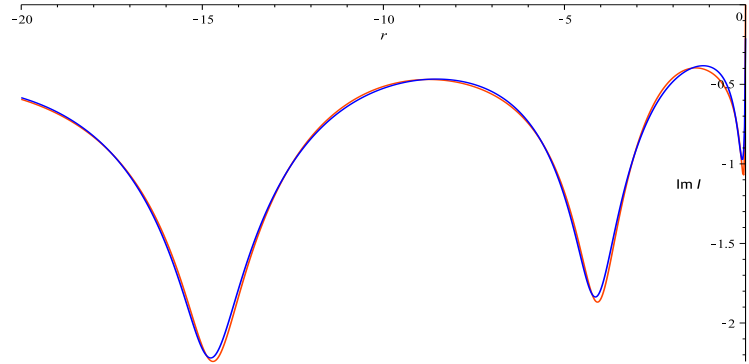}
\caption{The red and blue plots are, respectively, the imaginary parts of the numeric and large-$r$ asymptotic
(cf. \eqref{eq:asympt-sing-int} with $k=0$) values
of $I=\smallint_{r}^0\frac{1}{\sqrt{-r}H(r)}\,\md r$ for $r\leqslant-10^{-3}$, where $H(r)$ is the solution with initial
value $H(0)=-0.148+\mi0.191$. In black-and-white some minor discrepancies between the curves are noticeable.}
\label{fig:H0=-0148+i0191+IntImH20}
\end{center}
\end{figure}
\begin{figure}[htpb]
\begin{center}
\includegraphics[height=50mm,width=100mm]{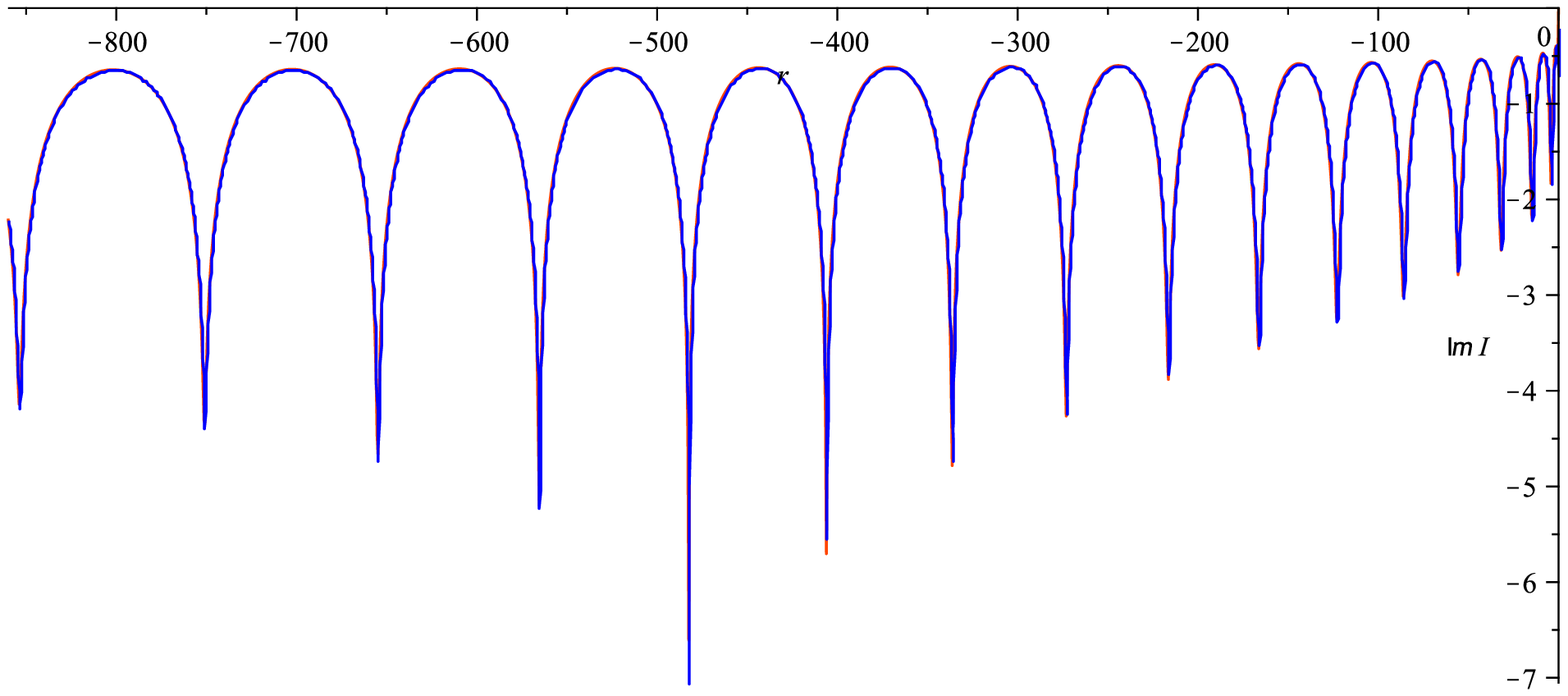}
\caption{The red and blue plots are, respectively, the imaginary parts of the numeric and large-$r$ asymptotic
(cf. \eqref{eq:asympt-sing-int}) values
of $I=\smallint_{r}^0\frac{1}{\sqrt{-r}H(r)}\,\md r$ for $r\leqslant-10^{-3}$ corresponding to the solution $H(r)$ with initial
value $H(0)=-0.148+\mi0.191$. On a coloured figure one can see that the three red peaks on the segment $[-600,-400]$  are
slightly (less than $5$\%) longer than the blue ones, whilst the differences between the lengths of the other corresponding red
and blue peaks are virtually indistinguishable.}
\label{fig:H0=-0148+i0191+IntImH}
\end{center}
\end{figure}

In Example~3, we demonstrated that, for some initial vales of $H(r)$, the leading terms of the large-$r$ asymptotics
approximate very closely the corresponding Painlev\'e function and its related integral starting from very small values
of $r$ ($r<1$). In Example~4 below, we show that there are markedly different initial values of $H(r)$ for which $H(r)$
and its related integral have behaviour similar to that  seen in Example~3; however, an approximation for small values of
$r$ becomes considerably worse, even though it remains satisfactory in the qualitative sense; moreover, in order to achieve
the correct plots, one has to be much more cautious with the numerical settings.

\subsection{Example 4: $H(0)=-100-\mi300$}\label{subsec:example4}
For this value of $H(0)$, $\nu_1=0.741160\ldots-\mi0.300731\ldots$, so that the
condition~\eqref{eq:nu1-singular} is satisfied, and we can use the asymptotic formulae~\eqref{eq:H-asympt-Large-singular}
and \eqref{eq:asympt-sing-int}.
For the numerical calculation of the solution and its related integral via \textsc{Maple}, we consider the initial values for
$H(r)$ at $r_1=10^{-9}$. For the calculation of the initial values of $H(r)$ and its corresponding integral, we use 6 terms of
the Taylor expansion for $H(r)$ (cf. Section~\ref{sec:2}, equations~\eqref{eq:Hat0-expansion}, \eqref{eq:a1}, and
\eqref{eqs:a2-a6}).
For visualisation purposes, 4 terms for $H(r)$ and 2 terms for $I(r)$ of the Taylor series expansion are, in principle,
sufficient. We set the parameter \texttt{Digits} $=180$, and the same value $180$ was set for the calculation of asymptotics.
This large value $180$ for the parameter \texttt{Digits} is important for the calculation of the imaginary part of the
asymptotic formulae, because, as explained above, the logarithmic term in the asymptotics~\eqref{eq:corr-asympt-sing-int}
for the integral is calculated via a solution of an appropriate differential equation. Somehow, in Example~3, for the calculation
of the asymptotics, we kept \texttt{Digits} equal to $20$ and did not notice
visual differences while increasing its value. In this example, however, we found a difference when the value of
\texttt{Digits} was increased from 20 to $100$. To verify that the plots remain stable, we increased \texttt{Digits} up to $180$:
the plots presented below correspond to this value of the parameter \texttt{Digits}. In principle, $80$ to $90$ digits for
this calculation suffices.

For the generation of the plots with `peaks', we found that the most appropriate value for the number of points was $5237$,
while in Example~3, this value was kept at $1837$. In this example, we were obliged to keep it that high because the
height of the three highest peaks were not stable until the number of points reached the value $5237$, after which,
these heights underwent only minor variations. The calculation with that many digits and number of points increases
the required time by about $20\%$, which, in our case, is not crucial.

In Example~3, we explained that the `staircase-structure' of the plot for $\mathrm{Re}\,I(r)$ is a consequence of
the fact that the tail of the numeric plot for $\mathrm{Re}\,H(r)$ is floating slowly upwards from the third to the
second quadrant as $r\to-\infty$; the same behaviour, of course, is replicated in its asymptotic plot. If the tails of both plots
float above the negative real axis almost simultaneously, then the numeric and asymptotic mole dwellings almost coincide
(cf. Figure~\ref{fig:H0=-0148+i0191+IntReH}). The rate of floating of the tails, however, depends on the initial value $H(0)$.
It may happen (and in many cases it does) that, although both plots are very close to one another, as in the present case,
one tail (the red one) floats before the other (the blue one); in Figure~\ref{fig:H0=-100-i300+ReH}, the red tail floats
at $r\approx-55$, prior to the fifth peak, while the blue tail floats at $r\approx-120$, prior to the seventh peak.
Due to scaling, this fact is not clearly seen in Figure~\ref{fig:H0=-100-i300+ReH}, but we verified it numerically.
Thus, the upward-trending float of the asymptotic tail
lags the numeric one by two peaks, so that the `asymptotic' mole's dwelling appears to be two steps lower than the `numeric'
dwelling (cf. Figure~\ref{fig:H0=-100-i300+IntReH160}). This example demonstrates that the integer $k$ in the asymptotic
formula~\eqref{eq:asympt-sing-int} for $I(r)$ appears not only because this formula originates from the asymptotics of the
function $\varphi(r)$, which is defined modulo $2\pi k$ (cf. Appendix~\ref{app:asympt0}), but also because it is related to
the quality of the approximation of the numeric plot by its asymptotics at finite values of $r$.
Note that the parameter $k$ undergoes a shift of $2$ when one approximates the small-$r$ part
of the numeric solution and its large-$r$ part via the asymptotics~\eqref{eq:asympt-sing-int}
(cf. Figs.~\ref{fig:H0=-100-i300+IntReH160} and \ref{fig:H0=-100-i300+IntReH}, respectively).
\begin{figure}[htpb]
\begin{center}
\includegraphics[height=50mm,width=100mm]{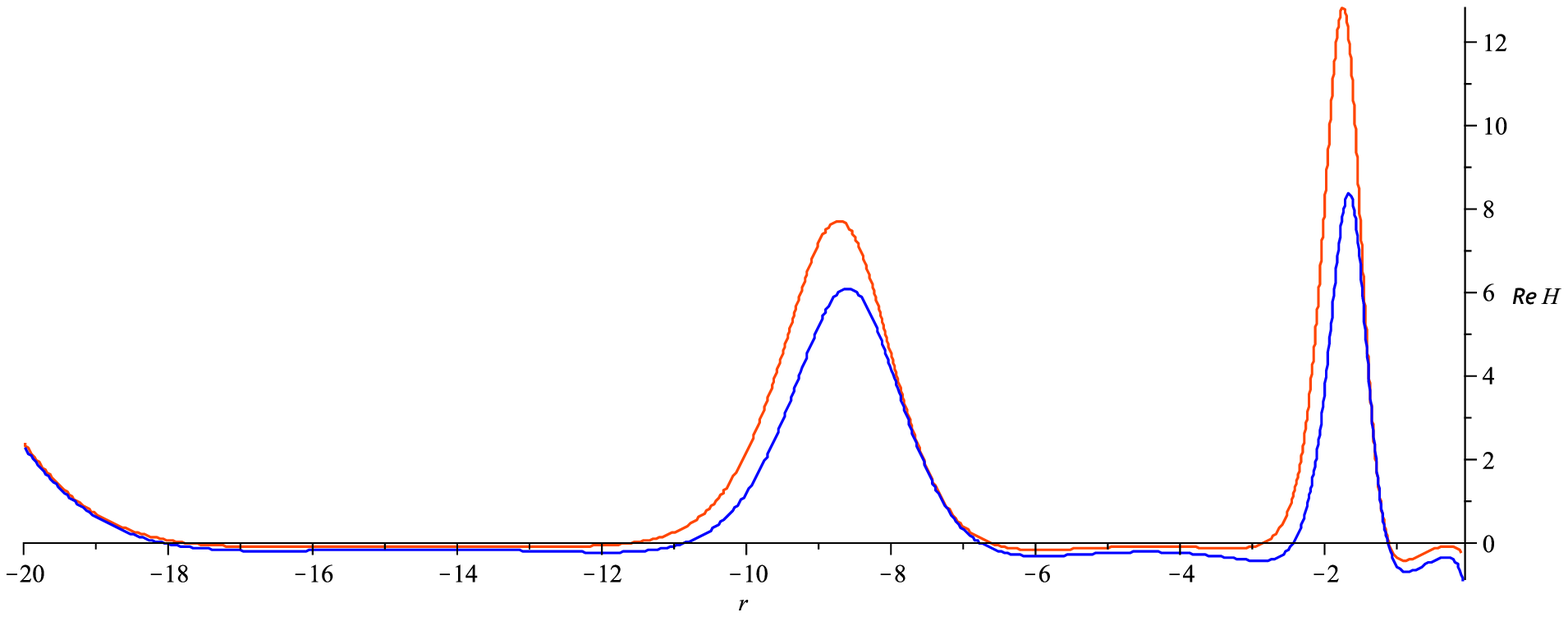}
\caption{The red and blue plots are, respectively, the real parts of the numeric and large-$r$ asymptotic
(cf. \eqref{eq:H-asympt-Large-singular}) values of the function $H(r)$ for $r\leqslant-0.1$ corresponding to the initial
value $H(0)=-100-\mi300$.}
\label{fig:H0=-100-i300+ReH20}
\end{center}
\end{figure}
\begin{figure}[htpb]
\begin{center}
\includegraphics[height=70mm,width=100mm]{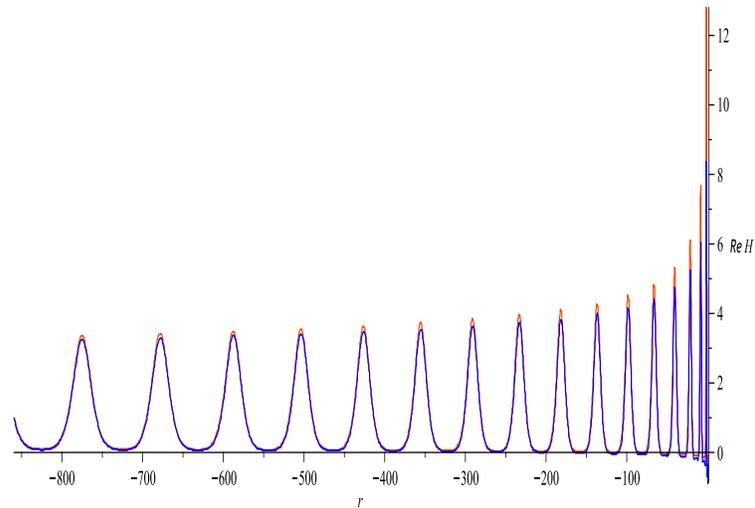}
\caption{Extended version of Figure~\ref{fig:H0=-100-i300+ReH20}, where only the first two peaks are shown.
On this plot, the first peak virtually coincides with the vertical axis.}
\label{fig:H0=-100-i300+ReH}
\end{center}
\end{figure}
\begin{figure}[htpb]
\begin{center}
\includegraphics[height=50mm,width=100mm]{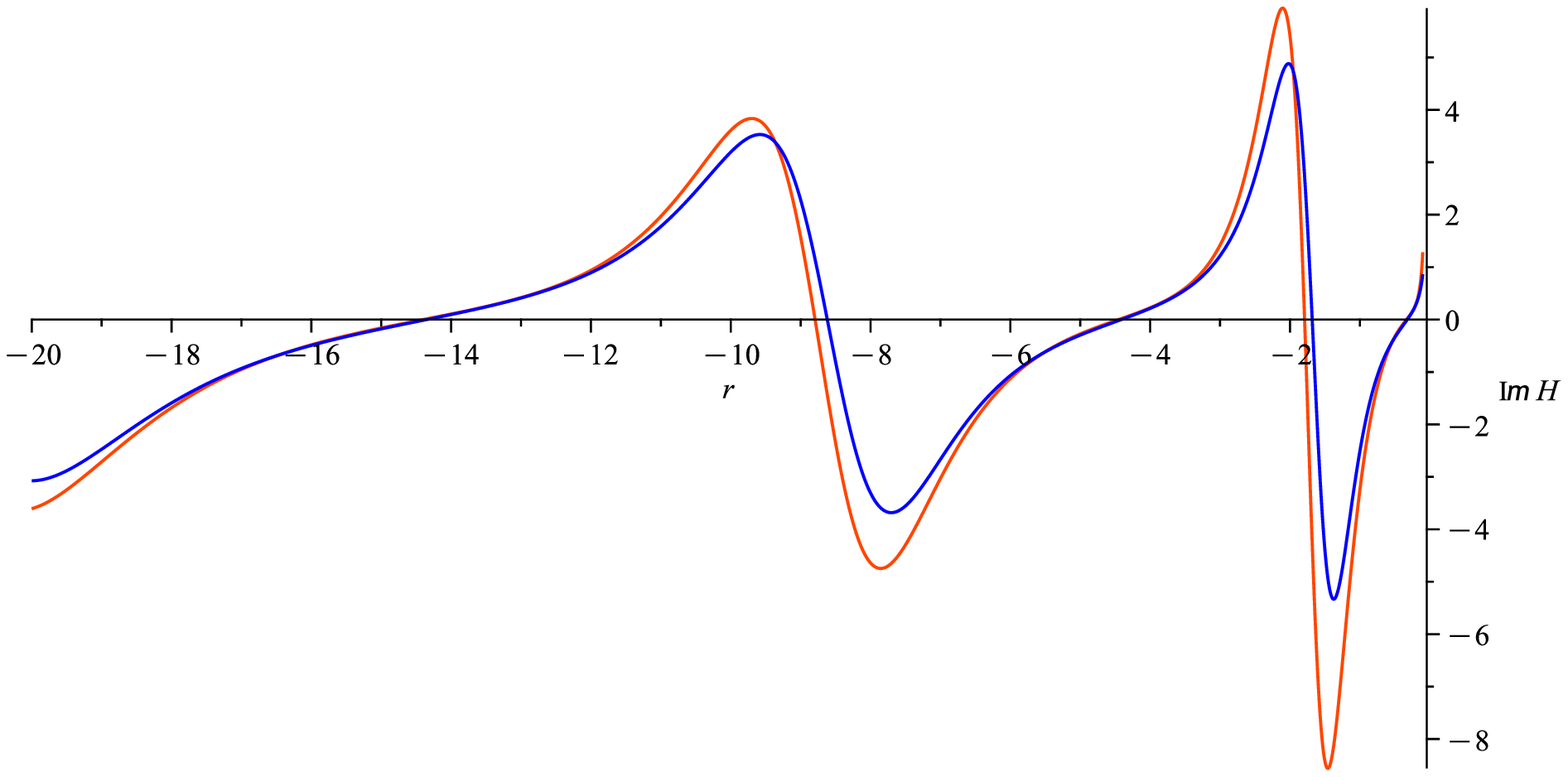}
\caption{The red and blue plots are, respectively, the imaginary parts of the numeric and large-$r$ asymptotic
(cf. \eqref{eq:H-asympt-Large-singular}) values of the function $H(r)$ for $r\leqslant-0.1$ corresponding to the initial value
$H(0)=-100-\mi300$.}
\label{fig:H0=-100-i300+ImH20}
\end{center}
\end{figure}
\begin{figure}[htpb]
\begin{center}
\includegraphics[height=70mm,width=100mm]{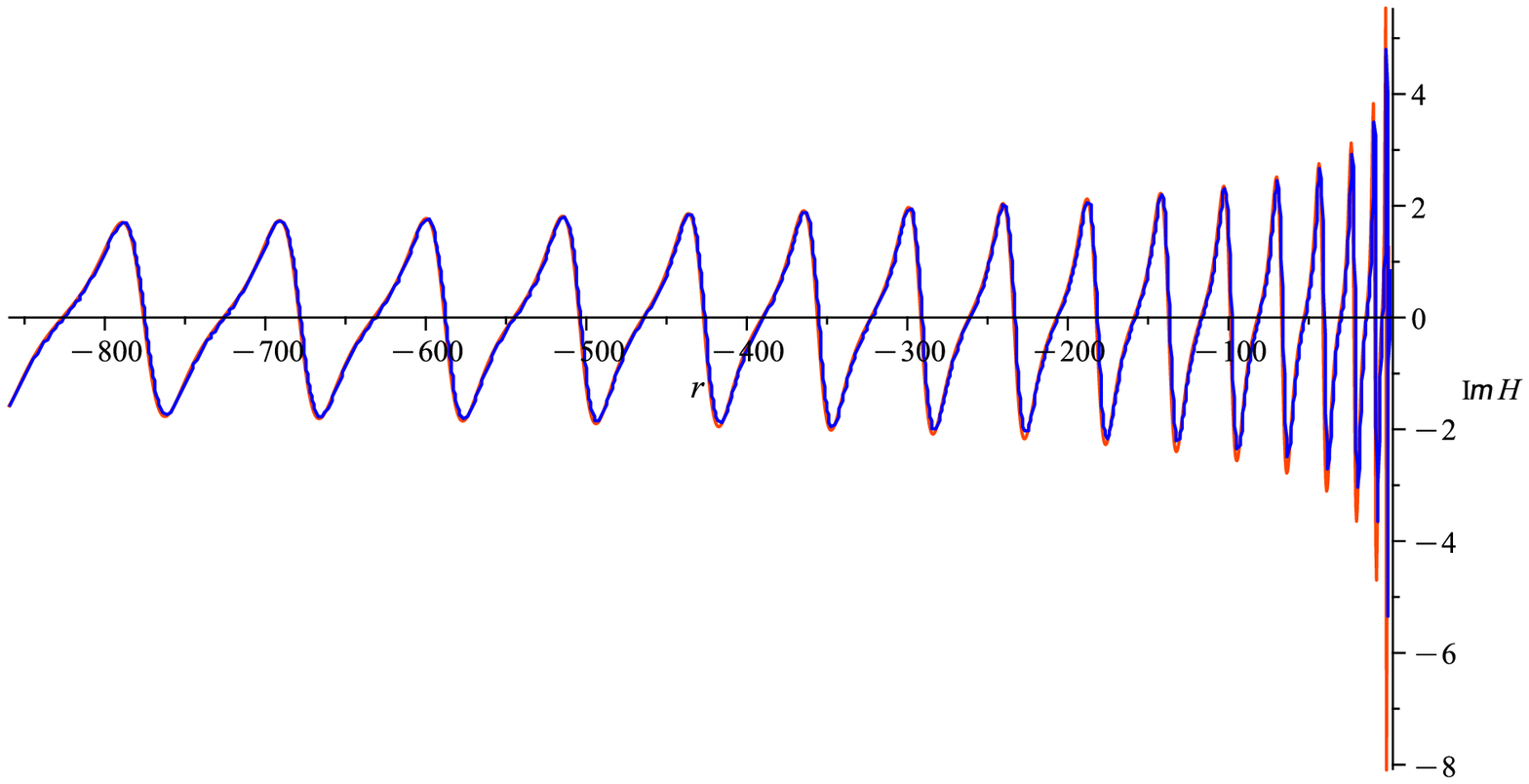}
\caption{Extended version of Figure~\ref{fig:H0=-100-i300+ImH20}. On this plot, the first down-up peak
of Figure~\ref{fig:H0=-100-i300+ImH20} almost coincides with the vertical axis.}
\label{fig:H0=-100-i300+ImH}
\end{center}
\end{figure}
\begin{figure}[htpb]
\begin{center}
\includegraphics[height=70mm,width=100mm]{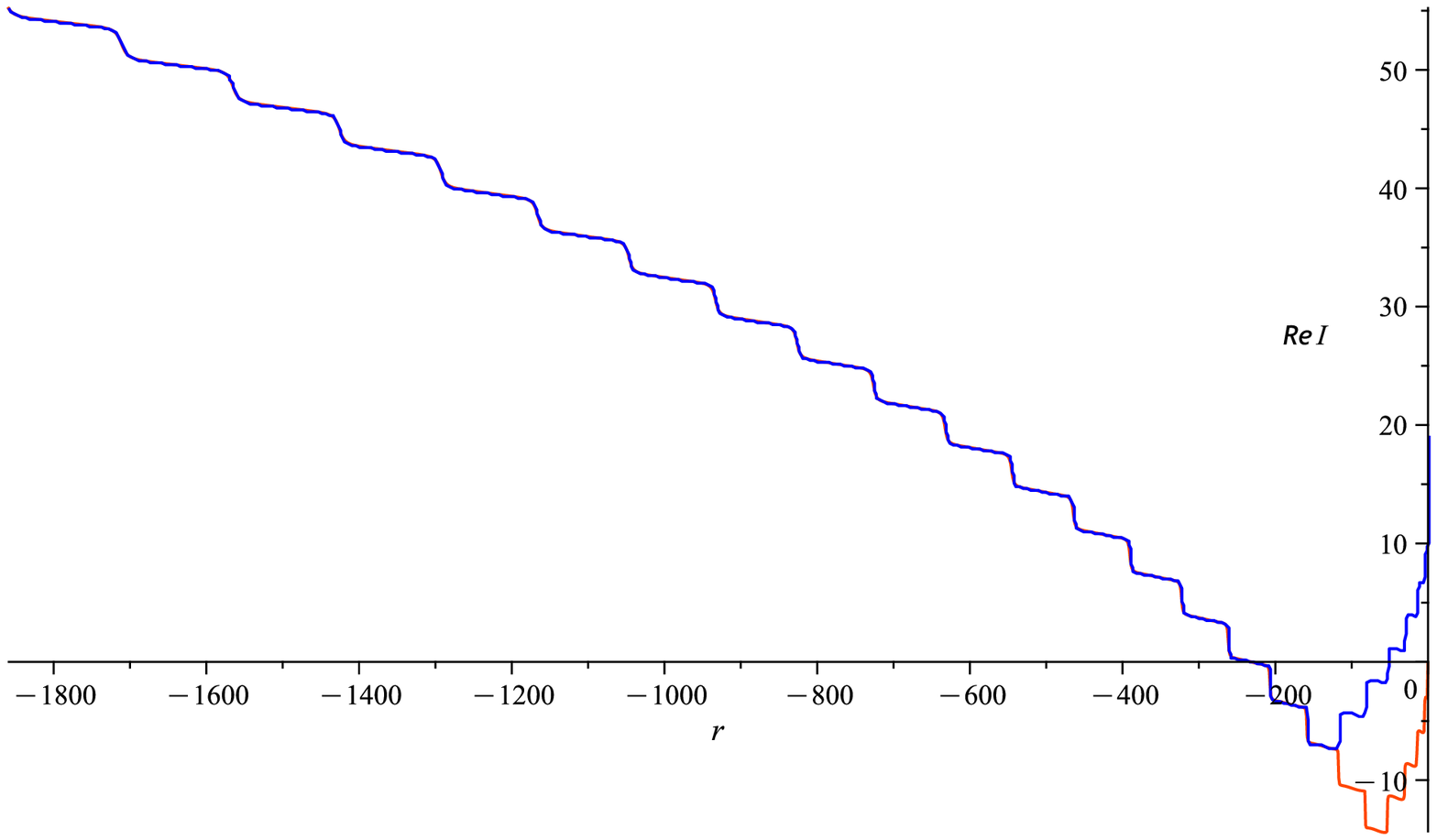}
\caption{The red and blue plots are, respectively, the real parts of the numeric and large-$r$ asymptotic
(cf. \eqref{eq:asympt-sing-int} with $k=3$) values
of $I=\smallint_{r}^0\frac{1}{\sqrt{-r}H(r)}\,\md r$ for $r\leqslant-10^{-8}$ corresponding to the solution $H(r)$ with
initial value $H(0)=-100-\mi300$.}
\label{fig:H0=-100-i300+IntReH}
\end{center}
\end{figure}
\begin{figure}[htpb]
\begin{center}
\includegraphics[height=60mm,width=100mm]{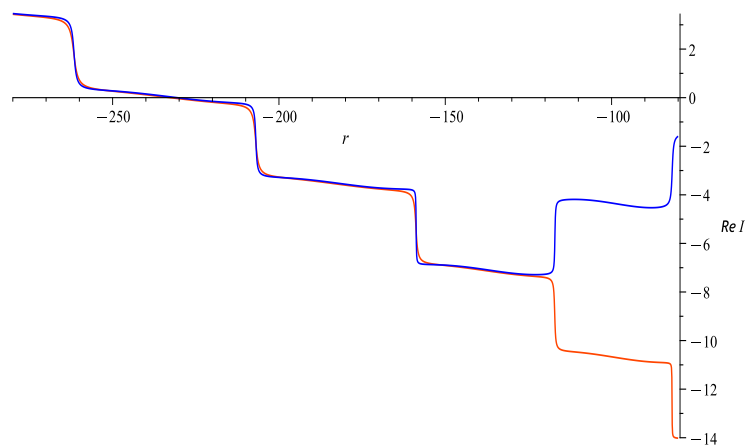}
\caption{A close-up of the segment of the plot on Figure~\ref{fig:H0=-100-i300+IntReH} corresponding to
$-280\leqslant r\leqslant-80$ where the numeric and asymptotic plots merge at $r\approx-120$. One can actually distinguish
two different intertwining (for $r<-120$) curves which are very close to each other.}
\label{fig:H0=-100-i300+IntReH280-80}
\end{center}
\end{figure}
\begin{figure}[htpb]
\begin{center}
\includegraphics[height=60mm,width=100mm]{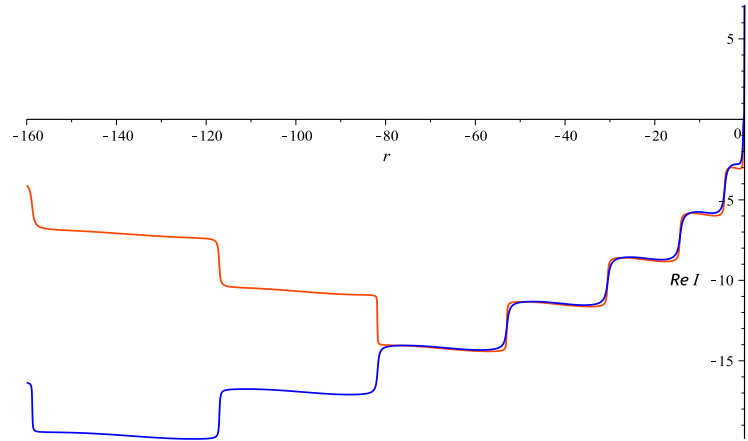}
\caption{A close-up of the part of the plot of the numerical solution on Figure~\ref{fig:H0=-100-i300+IntReH}
for $-160<r<-10^{-9}$ and the corresponding asymptotic plot (cf. \eqref{eq:asympt-sing-int} with---attention!---$k=1$).
One can actually distinguish two separate intertwining curves which are very close to one another. The `numeric' mole's
dwelling corresponds to $r\in(-80,-55)$, while the `asymptotic' dwelling is located two steps below with $r$-coordinate
in $(-160,-120)$.}
\label{fig:H0=-100-i300+IntReH160}
\end{center}
\end{figure}
\begin{figure}[htpb]
\begin{center}
\includegraphics[height=60mm,width=100mm]{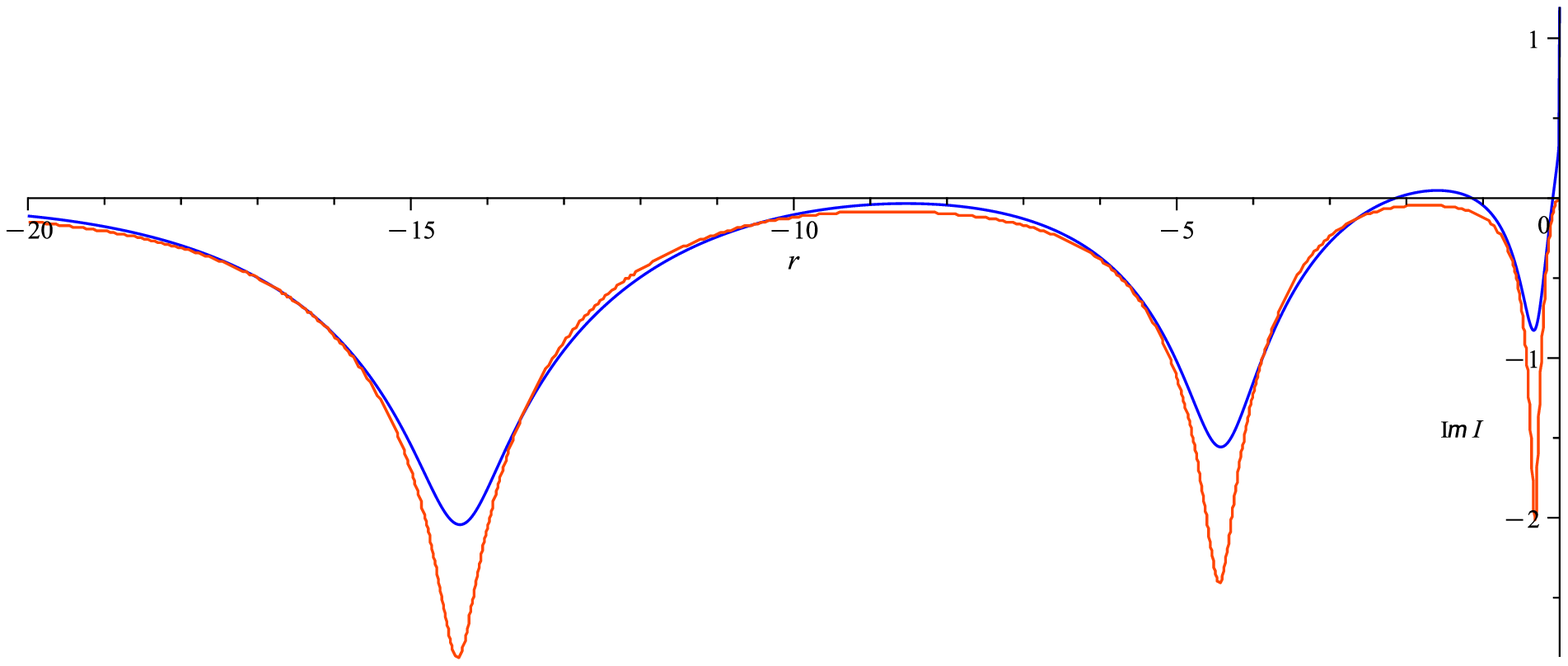}
\caption{The red and blue plots are, respectively, the imaginary parts of the numeric and large-$r$ asymptotic
(cf. \eqref{eq:asympt-sing-int}: $k$ is not important here) values
of $I=\smallint_{r}^0\frac{1}{\sqrt{-r}H(r)}\,\md r$ for $-20\leqslant r\leqslant-10^{-9}$ corresponding to the solution
$H(r)$ with initial value $H(0)=-100-\mi300$.}
\label{fig:H0=-100-i300+IntImH20}
\end{center}
\end{figure}
\begin{figure}[htpb]
\begin{center}
\includegraphics[height=70mm,width=100mm]{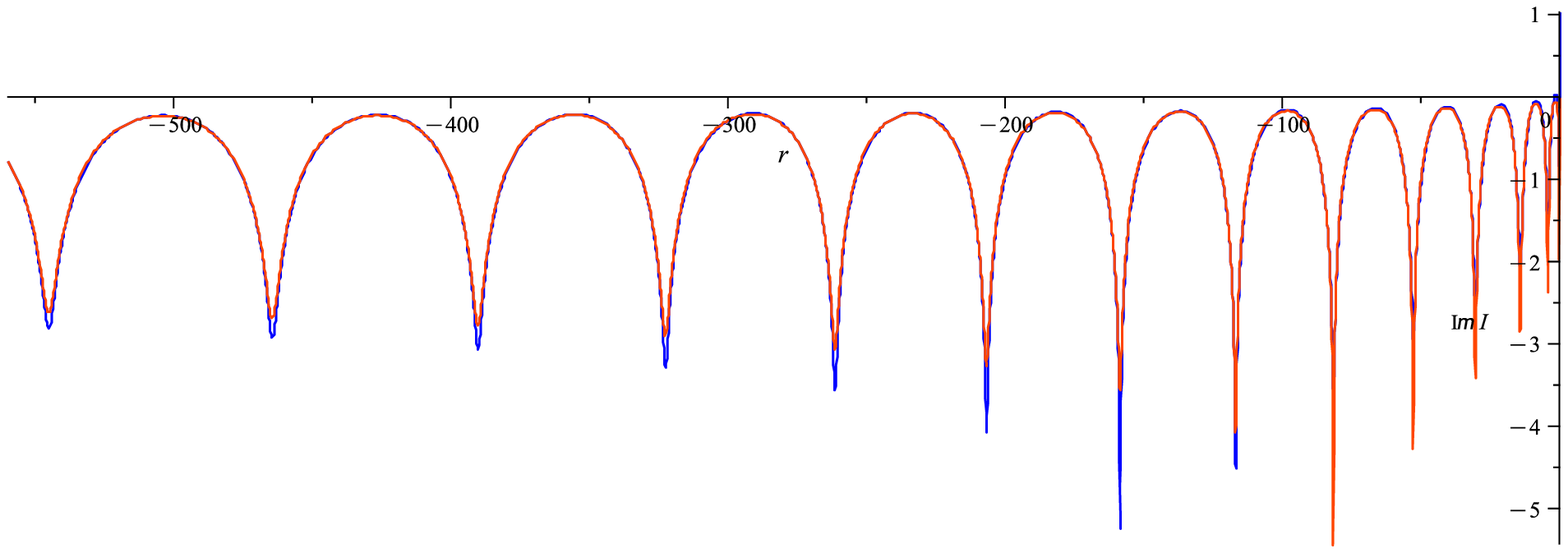}
\caption{Extended version of Figure~\ref{fig:H0=-100-i300+IntImH20} for $-560\leqslant r\leqslant-10^{-9}$. Here, the first
peak on Figure~\ref{fig:H0=-100-i300+IntImH20} virtually coincides with the vertical axis, and is therefore not clearly
visible. The first 6 icicles of the numerical solution are longer than the corresponding asymptotic icicles; however,
from the 7th icicle onward, the asymptotic icicles are longer, but the difference between the corresponding icicles
decreases rapidly.}
\label{fig:H0=-100-i300+IntImH}
\end{center}
\end{figure}
\subsection{Example 5: $H(0)=-\mi300$}\label{subsec:example5}
For this value of $H(0)$, $\nu_1=0.732934\ldots-\mi0.249469\ldots$, so that the
condition~\eqref{eq:nu1-singular} is satisfied, and we can use the asymptotic formulae~\eqref{eq:H-asympt-Large-singular}
and \eqref{eq:asympt-sing-int}.
For the numerical calculation of the solution and its related integral via \textsc{Maple}, we used the same settings as
in Example~4. In this case, the plots look similar to those in Example~4. The difference in this example is that,
for $r<0$, $\mathrm{Re}\,H(r)>0$; however, for small values of $r$, $\mathrm{Re}\,H(r)$ is very close to zero; for example,
at the first minimum $r_{\mathrm{min}}=-0.485998\ldots$, $H(r_{\mathrm{min}})=0.024587\ldots$
(cf. Figure~\ref{fig:H0=-i300+ReH20}).
As in Example~4, the  approximation for $H(r)$ via its large-$r$ asymptotics is not as good as that for the initial data
for which $10^{-3}<|H(0)|<10^{3/2}$. For small values of $r$, though, there are two segments wherein the asymptotics
for $\mathrm{Re}\,H(r)$ is negative (cf. Figure~\ref{fig:H0=-i300+ReH20}). These facts imply that the numerical plot for
$I(r)$ looks similar to the plot of Figure~\ref{fig:H0=60-i100+IntReH} (no mole's dwelling (!)); however, the plot for the
asymptotics does, indeed, contain two descending steps to the numerical plot of $I(r)$ until they merge. So, in lieu of the
mole's dwelling, we have an `asymptotic pigeon hole'. On the segment $[-5.5,0]$, the asymptotic formula does not even
approximate qualitatively the solution.

When we increase the value of $|\mathrm{Im}\,H(0)|$, the asymptotic pigeon hole shifts to the left; for $H(0)=-\mi10^5$, say,
the asymptotic pigeon hole is located on the segment $[-79,-59]$, and, as a result, the asymptotics does not approximate $I(r)$
on $[-59, 0]$. Our numerical experiments show that, in this case, a good approximation for $I(r)$ on the segment $[-79,-3]$ is
attained by a reflection of the asymptotic plot with respect to a straight line passing through two points at the bottom of
the asymptotic pigeon hole. These two points are not uniquely defined; therefore, one may vary the location of these points
in order to achieve a better approximation.
\begin{figure}[htpb]
\begin{center}
\includegraphics[height=50mm,width=100mm]{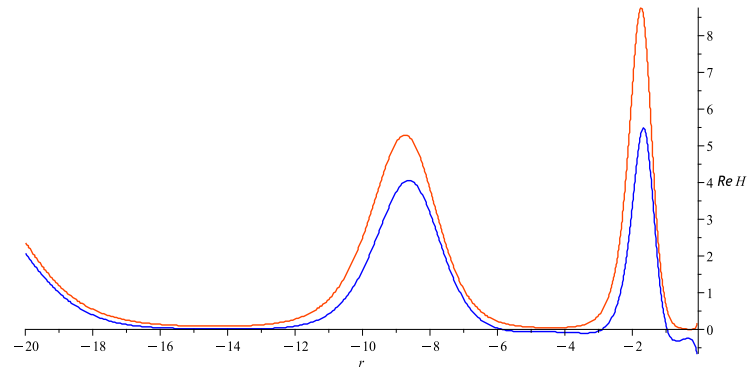}
\caption{The red and blue plots are, respectively, the real parts of the numeric and large-$r$ asymptotic
(cf. \eqref{eq:H-asympt-Large-singular}) values of the function $H(r)$ for $r\leqslant-0.1$ corresponding to the initial value
$H(0)=-\mi300$. The numeric solution (in red) is positive, while the asymptotic solution (in blue) has two segments
where it is negative.}
\label{fig:H0=-i300+ReH20}
\end{center}
\end{figure}
\begin{figure}[htpb]
\begin{center}
\includegraphics[height=70mm,width=100mm]{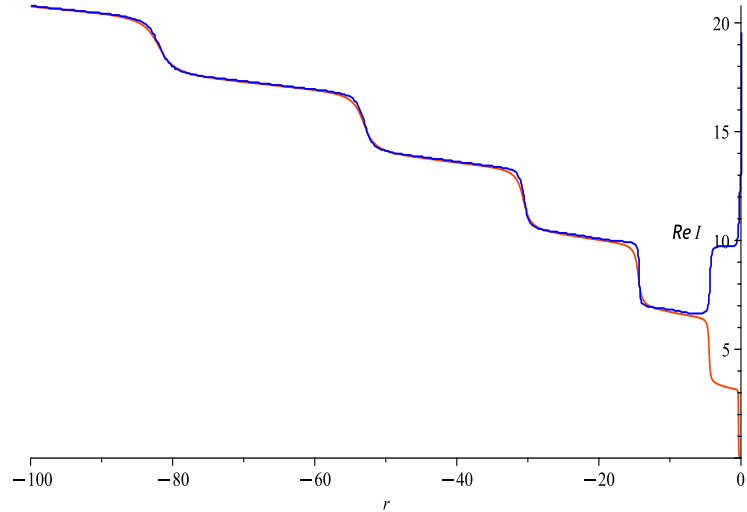}
\caption{The numerical solution (in red) merges with the asymptotic
(cf. \eqref{eq:asympt-sing-int} with---attention!---$k=3$) solution (in blue) on the segment $(-6,-5.5)$. The right-most
point of both plots is $r=-10^{-9}$.}
\label{fig:H0=-i300+IntReH100}
\end{center}
\end{figure}
\subsection{Example 6: $H(0)=-0.2+\mi0.045$}\label{subsec:example6}
For this value of $H(0)$, $\nu_1=0.049319\ldots-\mi0.459650\ldots$, so that the
condition~\eqref{eq:nu1-singular} is satisfied, and we can use the asymptotic formulae~\eqref{eq:H-asympt-Large-singular}
and \eqref{eq:asympt-sing-int}. For the numerical calculation of the solution and its related integral via \textsc{Maple},
we used the same settings as in Example~4.
The plots presented in Figs.~\ref{fig:H0=-02+i0045+ReH20}--\ref{fig:H0=-02+i0045+IntImH} were generated in $957\text{s}$
with the help of the newer notebook described in footnote~\ref{foot:notebooksCalc1}.

It is instructive to compare this example with Example 3: the initial values in both examples are close, but
the behaviour of the solutions is very different. This occurs because, in Example~6, $\textrm{Im}\,\nu_1$ is much closer
to $-0.5$.  This last fact also causes the appearance of very sharp peaks in the plots for the real and imaginary parts
of $H(r)$. These peaks create an additional problem for the visualization of asymptotics because we were not able to obtain
the right heights for the first peaks in our plots: not even 15000 points for constructing the plots was sufficient!
Our strategy is two-fold: (i) we consider two close-up plots for $\textrm{Re}\,H(r)$ and $\textrm{Im}\,H(r)$
(cf. Figs.~\ref{fig:H0=-02+i0045+ReH20} and \ref{fig:H0=-02+i0045+ImH20}) in order to determine the correct heights
of the first two peaks; and (ii) to construct the extended plots (cf. Figs.~\ref{fig:H0=-02+i0045+ReH} and
\ref{fig:H0=-02+i0045+ImH}) we randomly choose the number of points in order to get the heights of the first two peaks to be
as close as possible, like in the close-up figures. For item (i), we used 5237 and 1800 points for the construction of the
plots in Figs.~\ref{fig:H0=-02+i0045+ReH20} and \ref{fig:H0=-02+i0045+ImH20}, respectively. In this case, the heights of
the first two peaks of the real part of the numerical solution were found to be approximately 182 and 121, respectively.
On the extended version of Figure~\ref{fig:H0=-02+i0045+ReH20} (cf. Figure~\ref{fig:H0=-02+i0045+ReH}), the best possible
heights that we were able to achieve were 163 and 119, respectively, with the number of points equal to 6237.
To get good correspondence between the peaks on the close-up and extended figures for the imaginary part of the numerical
solution, we constructed the extended plot with the help of 10800 points. For the corresponding asymptotic plots, 2200 points
was enough for both cases.
The real part of $I(r)$ does not have any sharp peaks, so we used 1800 points. The imaginary part of $I(r)$ has
only one sharp peak for small negative $r$, and we achieved its right height with 5237 points for both the numerical and
asymptotic plots.

An intriguing feature of this example is that, practically, it is not possible
to reach the mole's dwelling by numerically calculating $I(r)$. Our calculation reveals that the last zero of
$\mathrm{Re}\,H(r)$ is located at $\approx-2.625\times10^{24}$: the right wall of the mole's dwelling is located at this point,
and what is termed `asymptotics' only begins here! Therefore, practically speaking, it is not possible to numerically
generate a plot for the solution and verify that, finally, the leading term of asymptotics is given by a shifted parabola
$2\sqrt{-r}$, i.e., the non-oscillatory part of the
asymptotics~\eqref{eq:asympt-reg-int}.\footnote{The mole has burrowed far too deep!\label{foot:deepmole}}
We see, however, that the asymptotic formula~\eqref{eq:asympt-sing-int} very well describes the right staircase leading to
the mole's dwelling starting from the first step; therefore, the asymptotics~\eqref{eq:asympt-sing-int} is correct, from
which it follows that, for very large negative
values of $r$, the solution should behave like the shifted parabola $2\sqrt{-r}$. This illustrates the theoretical application
of the asymptotic formula~\eqref{eq:asympt-sing-int} mentioned in Remark~\ref{rem:correctionFORsingularINTEGRAL};
without knowledge of the asymptotics, one may assume that the right staircase constitutes the asymptotic behaviour
of $I(r)$ for all large negative $r$.

\begin{figure}[htpb]
\begin{center}
\includegraphics[height=50mm,width=100mm]{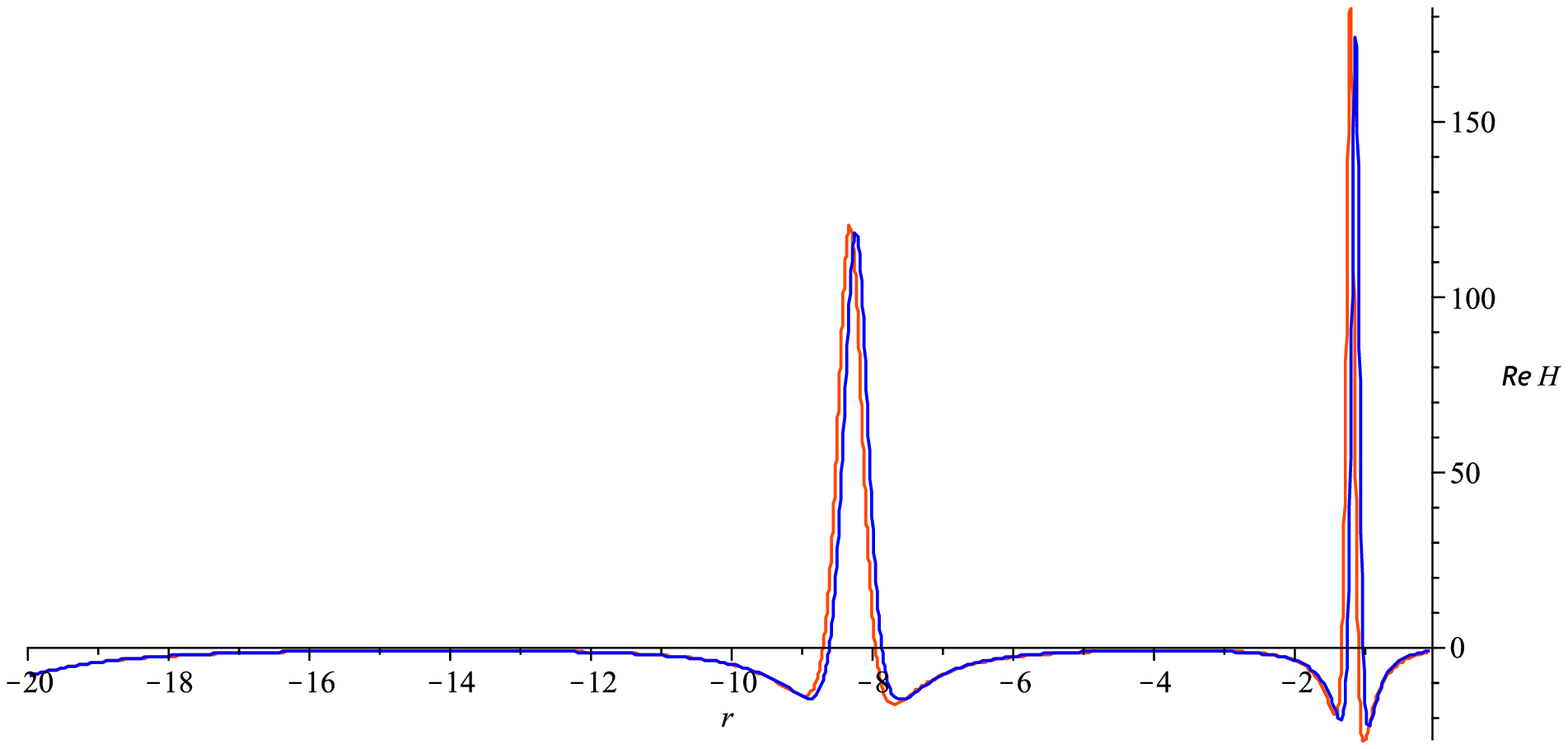}
\caption{The red and blue plots are, respectively, the real parts of the numeric and large-$r$ asymptotic
(cf. \eqref{eq:H-asympt-Large-singular}) values of the function $H(r)$ for $r\leqslant-0.1$ corresponding to the initial
value $H(0)=-0.2+\mi0.045$.}
\label{fig:H0=-02+i0045+ReH20}
\end{center}
\end{figure}
\begin{figure}[htpb]
\begin{center}
\includegraphics[height=70mm,width=100mm]{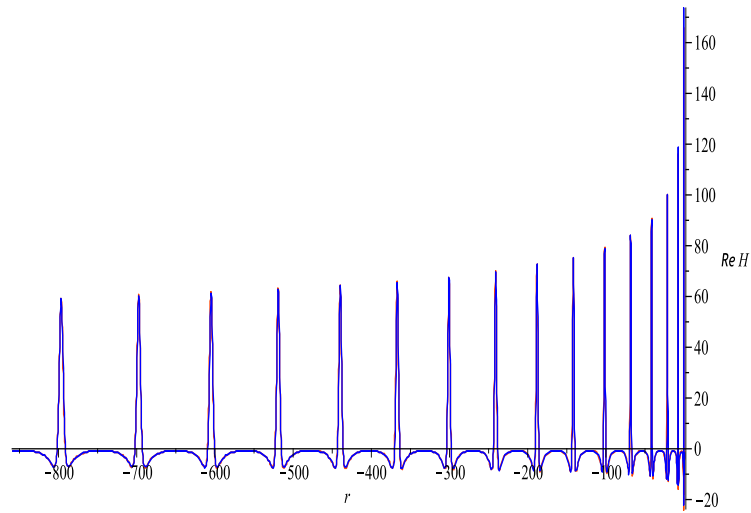}
\caption{Extended version of Figure~\ref{fig:H0=-02+i0045+ReH20} where the first two peaks of the plot are shown.
On this plot, the first peak almost coincides with the vertical axis.}
\label{fig:H0=-02+i0045+ReH}
\end{center}
\end{figure}
\begin{figure}[htpb]
\begin{center}
\includegraphics[height=50mm,width=100mm]{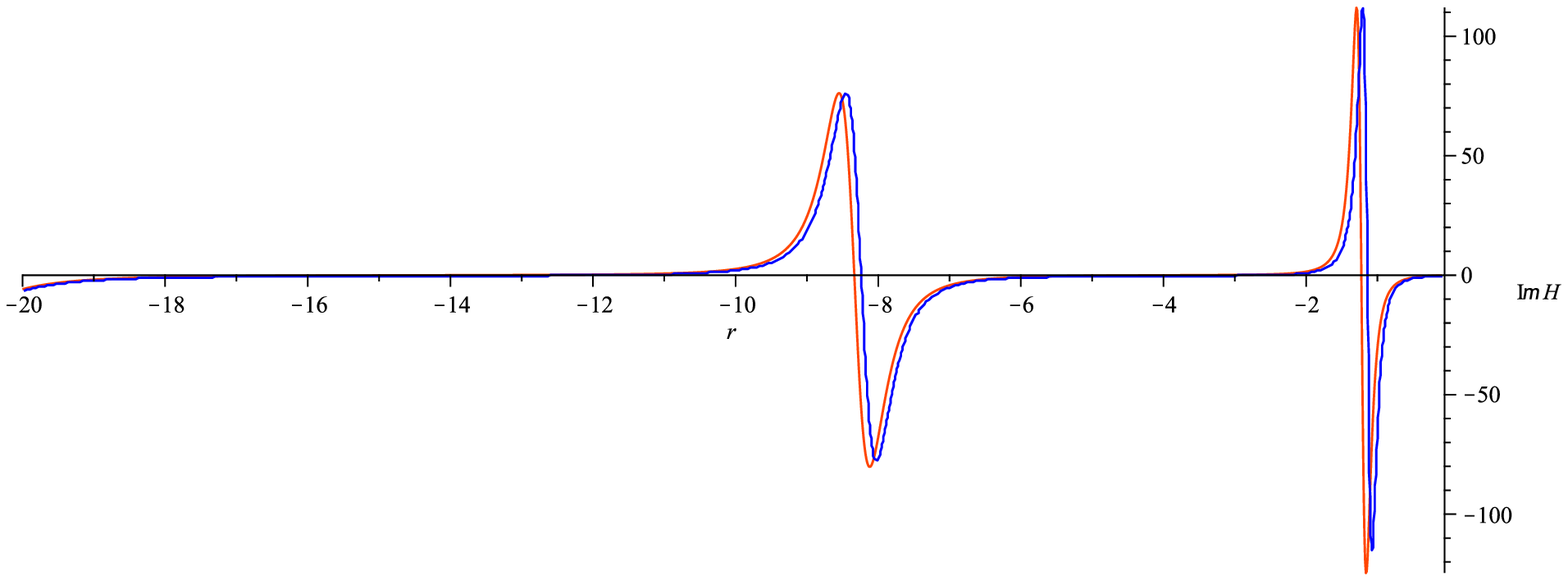}
\caption{The red and blue plots are, respectively, the imaginary parts of the numeric and large-$r$ asymptotic
(cf. \eqref{eq:H-asympt-Large-singular}) values of the function $H(r)$ for $r\leqslant-0.1$ corresponding to the initial
value $H(0)=-0.2+\mi0.045$.}
\label{fig:H0=-02+i0045+ImH20}
\end{center}
\end{figure}
\begin{figure}[htpb]
\begin{center}
\includegraphics[height=70mm,width=100mm]{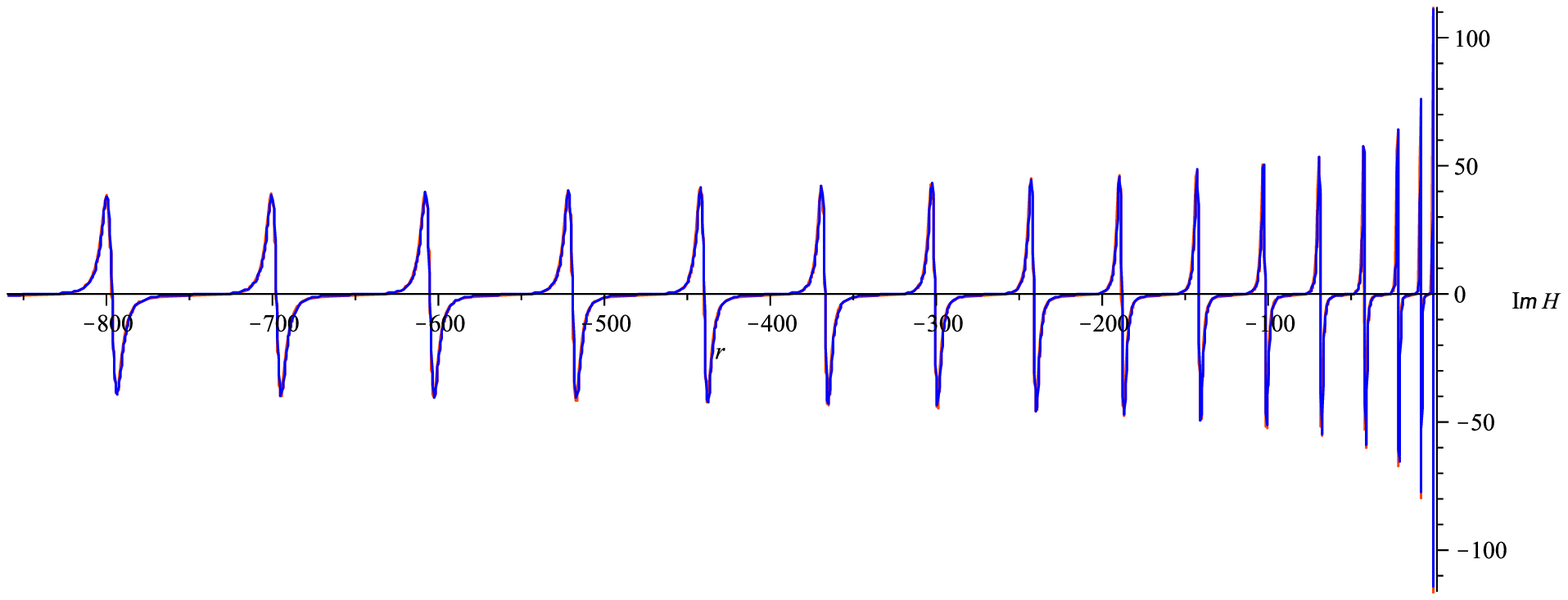}
\caption{Extended version of Figure~\ref{fig:H0=-02+i0045+ImH20}. On this plot, the first down-up peak
of Figure~\ref{fig:H0=-02+i0045+ImH20} virtually coincides with the vertical axis.}
\label{fig:H0=-02+i0045+ImH}
\end{center}
\end{figure}
\begin{figure}[htpb]
\begin{center}
\includegraphics[height=70mm,width=100mm]{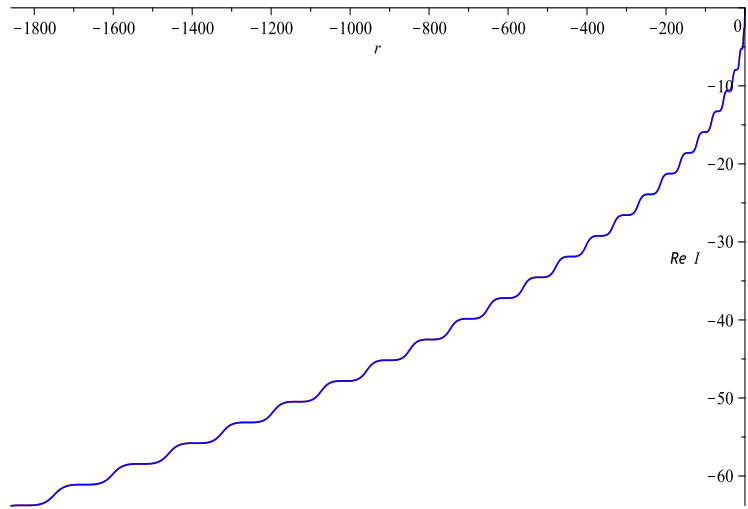}
\caption{The red and blue plots are, respectively, the real parts of the numeric and large-$r$ asymptotic
(cf. \eqref{eq:asympt-sing-int} with $k=0$) values
of $I=\smallint_{r}^0\frac{1}{\sqrt{-r}H(r)}\,\md r$ for $r\leqslant-10^{-8}$ corresponding to the solution $H(r)$ with
initial value $H(0)=-0.2+\mi0.045$.
The red plot is overlapped by the blue plot, and is therefore not visible. On the close-up
Figure~\ref{fig:H0=-02+i0045+IntReH20} that follows,
one can distinguish the red colour.}
\label{fig:H0=-02+i0045+IntReH}
\end{center}
\end{figure}
\begin{figure}[htpb]
\begin{center}
\includegraphics[height=70mm,width=100mm]{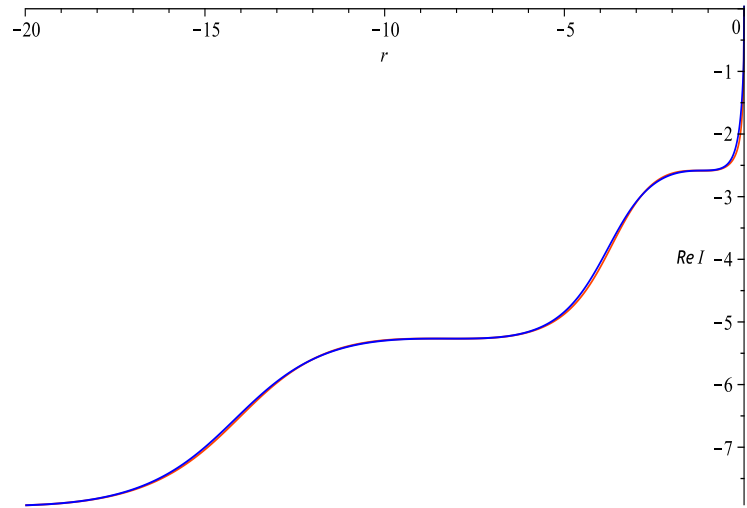}
\caption{A close-up of a part of the plot of the numerical solution in Figure~\ref{fig:H0=-02+i0045+IntReH}
for $-20\leqslant r\leqslant-10^{-8}$, and the corresponding asymptotic plot (cf. \eqref{eq:asympt-sing-int} with $k=0$).
On the coloured picture, one can see that the steeper slopes from below are in red while those from above are in blue.}
\label{fig:H0=-02+i0045+IntReH20}
\end{center}
\end{figure}
\begin{figure}[htpb]
\begin{center}
\includegraphics[height=70mm,width=100mm]{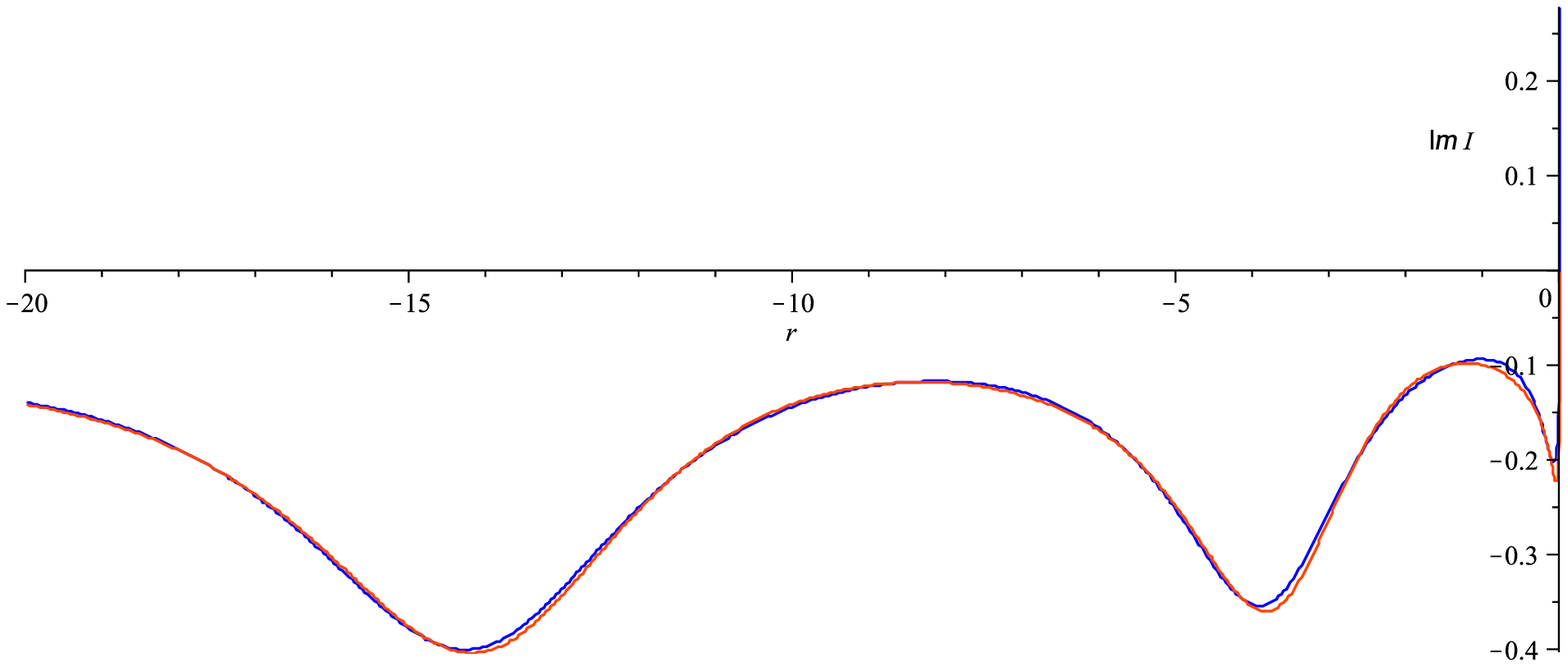}
\caption{The red and blue plots are, respectively, the imaginary parts of the numeric and large-$r$ asymptotic
(cf. \eqref{eq:asympt-sing-int}) values
of $I=\smallint_{r}^0\frac{1}{\sqrt{-r}H(r)}\,\md r$ for $-20\leqslant r\leqslant-10^{-9}$ corresponding to the
solution $H(r)$ with initial value $H(0)=-0.2+\mi0.045$.}
\label{fig:H0=-02+i0045+IntImH20}
\end{center}
\end{figure}
\begin{figure}[htpb]
\begin{center}
\includegraphics[height=70mm,width=100mm]{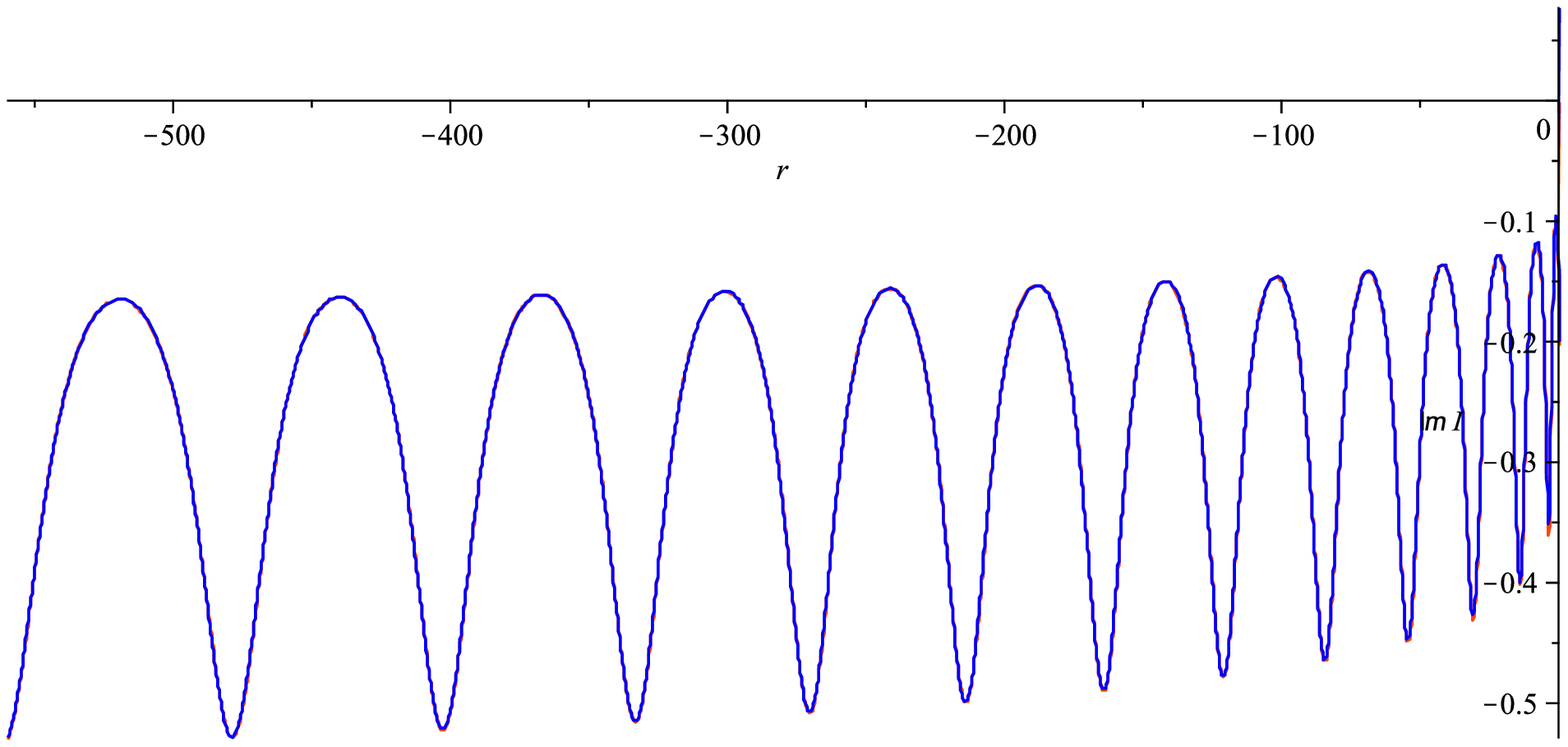}
\caption{Extended version of Figure~\ref{fig:H0=-02+i0045+IntImH20} for $r\leqslant-10^{-5}$. In this figure, the first peak
of Figure~\ref{fig:H0=-02+i0045+IntImH20} almost coincides with the vertical axis, and is therefore not clearly visible.
At the end-points of the first few peaks, one can distinguish small red dots.}
\label{fig:H0=-02+i0045+IntImH}
\end{center}
\end{figure}

\begin{remark}\label{rem:mole's-office}
\textsc{Question}: Where is the mole's dwelling in Figure~\ref{fig:H0=-02+i0045+IntReH}?
\textsc{Answer}: No one can see it, but it exists!
As explained in Example~3, the location of the dwelling (its right wall) corresponds to the last zero
of the function $\mathrm{Re}\,H(r)$. Observing Figure~\ref{fig:H0=-02+i0045+ReH}, one may have a concern as to whether or not
this last zero does, in fact, exist, because we see that the asymptotic formula~\eqref{eq:H-asympt-Large-singular} works very
well within the given plot, and we expect that it should replicate, with minor changes, as $r\to-\infty$.
The major tendency of the large-$r$ behaviour of the solution that follows from Figure~\ref{fig:H0=-02+i0045+ReH} is that the
heights of the peaks becoming lower and the distances between them increase. Another tendency in the transformation of the plot
which is more subtle, however, is that its tail is floating up slowly above the negative real semi-axis as one proceeds to the
the left ($r\to-\infty$): finally, $\mathrm{Re}\,H(r)\to1$; therefore, the last zero of $\mathrm{Re}\,H(r)$ should exist.

In order to comprehend these tendencies, one has to consider the asymptotics of the asymptotic
formula~\eqref{eq:H-asympt-Large-singular}: equation~~\eqref{eq:H-asympt-Large-singular}, in fact, can be simplified and
transformed into an expression resembling the ``regular asymptotics''~\eqref{eq:H-asympt-Large-regular};
so, in this sense, the latter asymptotics serves as ``the asymptotics of asymptotics''. This is reasonable
terminology, because the beginning of the plot where the first asymptotics~\eqref{eq:H-asympt-Large-singular} is already
working, and the tail of the plot, where it is, of course, still working, look radically different! At the same time, though,
we explain below that, practically, one cannot visualize this difference!

The plots in Figs.~\ref{fig:H0=-02+i0045+ReH20} and \ref{fig:H0=-02+i0045+ReH} show that the asymptotics accurately
approximates $\mathrm{Re}\,H(r)$; as a result, for the calculation of the last zero of $\mathrm{Re}\,H(r)$, we can use
the asymptotics of this function instead of its numerical evaluation. This is important because, as is evident from the
result~\eqref{eq:mole-x} stated below, such a calculation is not possible due to the ``astronomical distances'' required
for this purpose. According to our calculations, the last zero of $\text{Re}\,H(r)$ is located
at\footnote{Recall that $10^{24}$ is an estimate for the number of stars in the observable universe.\label{foot:stars}}
\begin{equation}\label{eq:mole-x}
r_0=-2.6279340765216450944920718115\ldots\times10^{24}.
\end{equation}
On the plot of $\mathrm{Re}\,I(r)$, after the point, $r_0$, one steps onto the floor of the mole's dwelling, from which the
left stairway to heaven begins.
In order to evaluate the depth where the dwelling is located, one has to count the number of zeros of $\text{Re}\,H(r)$.
In Figure~\ref{fig:H0=-02+i0045+ReH}, we see that the plot has a quasi-periodic structure: by quasi-period we mean a part of
the plot located between two neighbouring peaks. Each peak ``stands'' on two legs, so that one leg of a peak
belongs to the left quasi-period, whereas the other leg belongs to the right quasi-period; therefore, each leg intersects
the negative real axis, and gives rise to one zero of $\text{Re}\,H(r)$. Thus, each quasi-period has at least two zeros.
There is one more point that requires verification, namely, the maximum point on the ``bridge'' connecting the legs. This point
is always negative until one arrives at the quasi-period with the last zero. It is evidently not a trivial matter to establish
this for the Painlev\'e function; however, it can be proved for its asymptotics~\eqref{eq:H-asympt-Large-singular}.
Numerically, for the first two quasi-periods shown in Figure~\ref{fig:H0=-02+i0045+ReH20}, we found that
$r_{\max,1}=-3.756\ldots$, with $H(r_{\max,1})=-0.46471\ldots$, and $r_{\max,2}=-14.118\ldots$, with
$H(r_{\max,2})=-0.47042\ldots$. So, the floating process of the tail evolves as follows: the legs are lifting up and the
distance between them is growing at the same time that the ``bridge'' between the
legs straightens out so that its maximum point is moving down. Thus, the shape of the bridge changes from convex to concave,
the legs disappear, the maximum turns smoothly to the minimum, and the floating process continues:
the tail (on some appropriate scale) becomes similar to the plot shown in Figure~\ref{fig:H0=-0148+i0191+ReH}. This transformation
of the plot is progressing very slowly on astronomically large distances for $|r|$; however, it can still be observed with the
help of the asymptotic formula~\eqref{eq:H-asympt-Large-singular}. At this stage of the evolution, the transformation of the
plot continues: the spikes become more and more ``plateaued'', so that, finally, the tail, after an appropriate re-scaling,
resembles the black plot in Figure~\ref{fig:H0=-02+i0045+ReHintro}. Practically, the graph of this part of the tail cannot be
plotted even with the help of asymptotics, because the distance between the spikes, as well as the spikes themselves, are
becoming progressively more and more stretched along the negative semi-axis, unless some special scaling for $r$ is used which
changes concomitant with the growth of $|r|$.

Reverting back to the determination of the mole's dwelling, we note that the pair of zeros between two neighbouring peaks
of $\mathrm{Re}\,H(r)$ correspond to one step down in the right staircase. If $2N$ denotes the number of zeros, then there are
$N$ steps down in the right staircase. The depth of each step-down, starting from the second one, should be $\pi$, so that
the depth of the mole's dwelling, for large $N$,  is $\pi N +\mathcal{O}(1)$.
On the other hand, this depth can be calculated (cf. equations~\eqref{eq:asympt-sing-int} and \eqref{eq:corr-asympt-sing-int})
as follows, $\mathrm{Re}\big((\hat\psi(r_0)-\theta_0)/2+(\hat\psi(r_0)+\theta_0)/2\big)-2\sqrt{-r_0}+o(1)$,
so that, for large $|r_0|$, mole's dwelling  is located at a depth of about
$2\sqrt{-3r_0}-2\sqrt{-r_0}+\mathcal{O}\big(\ln\sqrt{-3r_0}\big)$.
A more careful estimation requires the $\mathcal{O}(1)$ contribution
which, on the scale of our distances~\eqref{eq:mole-x}, is comparable to the logarithmic error estimate; but, for the
purposes of the rough evaluation we are looking for, such contributions do not matter. Numerically, for the $y$-co-ordinate
of the mole's dwelling, we get the following estimate:\footnote{Taking into account the co-ordinates of the mole's
dwelling~\eqref{eq:mole-x}, \eqref{eq:mole-y}, its size, and the fact that the left staircase goes up, far beyond the level of
the negative semi-axis, we may suspect that the strange mole from Example 6 fell from the heavens some time ago.\label{foot:dev}}
\begin{equation}\label{eq:mole-y}
-2(\sqrt{3}-1)\sqrt{-r_0}=-2.373441069108\ldots\times10^{12}.
\end{equation}
We expect that $10$ to $11$ digits after the decimal point in \eqref{eq:mole-y} are correct. A more careful estimate for the
depth of the mole's dwelling can be obtained with the help of the so-called stair-stringer~\eqref{eq:right-stringer}
discussed in the following subsection.
\hfill $\blacksquare$\end{remark}
\subsection{Stair-Stringers}\label{subsec:stair-stringer}
In conclusion, we would like to formulate some observations concerning the asymptotic behaviour of $I(r)$.
These observations require further investigations and more careful formulations. In Examples 1 and 2,  the real
part of $I(r)$ looks like a staircase. In Examples 3--6, we observe two staircases connected with the mole's dwelling.
The appearance of the mole's dwelling on the plot of $I(r)$ was explained in terms of the floating of the tail of the plot
of the function $H(r)$ above the negative-$r$ semi-axis. We conjecture that the straight line
$H_-(r)=-1/2=\Re\big(\me^{\pm2\pi\mi/3}\big)$
plays the role of an unstable attracting line, and the appearance of the mole's dwelling and its location can be explained via
the interaction of the stable attracting line, $H_+(r)=+1$, and the unstable attracting line, $H_-(r)$.\footnote{The numbers $1$
and $\me^{\pm2\pi\mi/3}$ are the roots of the non-differentiated part of equation~\eqref{eq:hazzidakis}
(cf. Corollary~\ref{cor:algebraicMONDATA}).\label{foot:stair-stringer-cubic roots}}

By the term \emph{stair-stringer} we mean the non-oscillatory part of the leading term of
asymptotics describing the staircases. The stair-stringer for the left staircase, denoted by $Str_l$, can be deduced
immediately from the asymptotics~\eqref{eq:asympt-reg-int}, namely,
\begin{equation}\label{eq:left-stringer}
Str_l=2\sqrt{-r}+2\mathrm{Re}(\nu_1)\ln(2+\sqrt{3})+2\pi k-\mathrm{Im}\left(\ln\left(
\frac{\me^{\frac{2\pi\mi}{3}}H(0)-\me^{-\frac{2\pi\mi}{3}}}{\me^{\frac{2\pi\mi}{3}}-H(0)\me^{-\frac{2\pi\mi}{3}}}\right)\right).
\end{equation}
To get a formula for the stair-stringer for the right staircase, denoted by $Str_r$, is more complicated; however,
the assiduous reader of this section should be able to derive the following formula:
\begin{equation}\label{eq:right-stringer}
\begin{aligned}
Str_r=&-2(\sqrt{3}-1)\sqrt{-r}-\mathrm{Re}(\nu_1)\ln\sqrt{-r}
+\mathrm{Re}(\nu_1)\ln\left(\frac{(2+\sqrt{3})^2}{(2\sqrt{3})^3}\right)
-\frac{3\pi}{2}\mathrm{Im}(\nu_1)\\
&+\frac{\pi}{4}+2\pi(k-1)-\mathrm{Im}\left(\ln\left(
\frac{\me^{\frac{2\pi\mi}{3}}H(0)-\me^{-\frac{2\pi\mi}{3}}}{\me^{\frac{2\pi\mi}{3}}-H(0)\me^{-\frac{2\pi\mi}{3}}}\right)\right)
+\mathrm{Im}\left(\ln\left(\tilde{g}_1\Gamma(\mi\nu_1)\right)\right).
\end{aligned}
\end{equation}
In equations \eqref{eq:left-stringer} and \eqref{eq:right-stringer}, $k\in\mathbb{Z}$ is an explicit manifestation of the
$2\pi$-indeterminacy of the imaginary part of the asymptotics \eqref{eq:asympt-reg-int} and \eqref{eq:asympt-sing-int}.

The stair-stringer $Str_l$ does not require additional visualization since it is a major part of the leading term of asymptotics
which has already been verified in Examples ~1 and 2. The ``intermediate'' stair-stringer $Str_r$ is illustrated with the help
of Figs.~\ref{fig:H0=-04+i078+IntReHstr} and \ref{fig:H0=-100-i200+IntReHstr}: the integer $k$ in both of these figures
is equal to $1$, while $k=0$ for $Str_r$ corresponding to Example 6, which is not presented here.

\begin{figure}[htpb]
\begin{center}
\includegraphics[height=70mm,width=100mm]{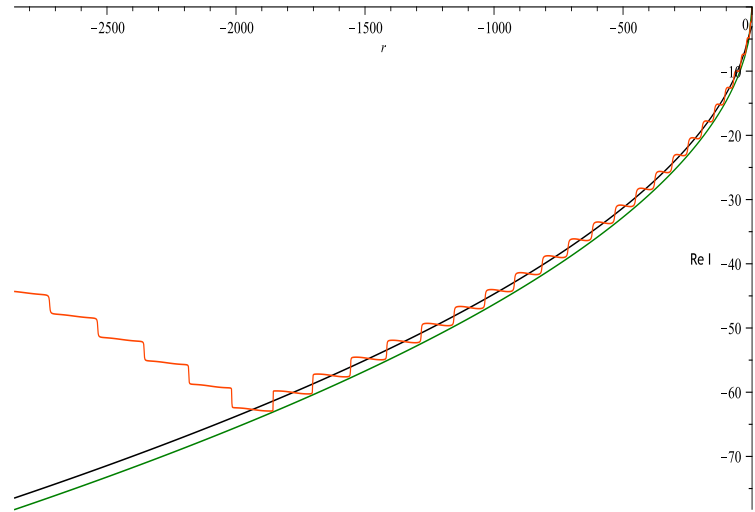}
\caption{The red curve represents the numerical values of $\mathrm{Re}\,I(r)$ corresponding to the initial value
$H(0)=-0.4+\mi0.78$. The {\sc Maple} settings for the generation of this curve are similar to those of Example 6.
The parameter $\nu_1$ for this solution equals $-0.396664\dotsc-\mi0.198139\dotsc$. The black line represents $Str_r$
with $k=1$, and the lowest (green) line is the unstable attracting parabola.}
\label{fig:H0=-04+i078+IntReHstr}
\end{center}
\end{figure}
\begin{figure}[htpb]
\begin{center}
\includegraphics[height=70mm,width=100mm]{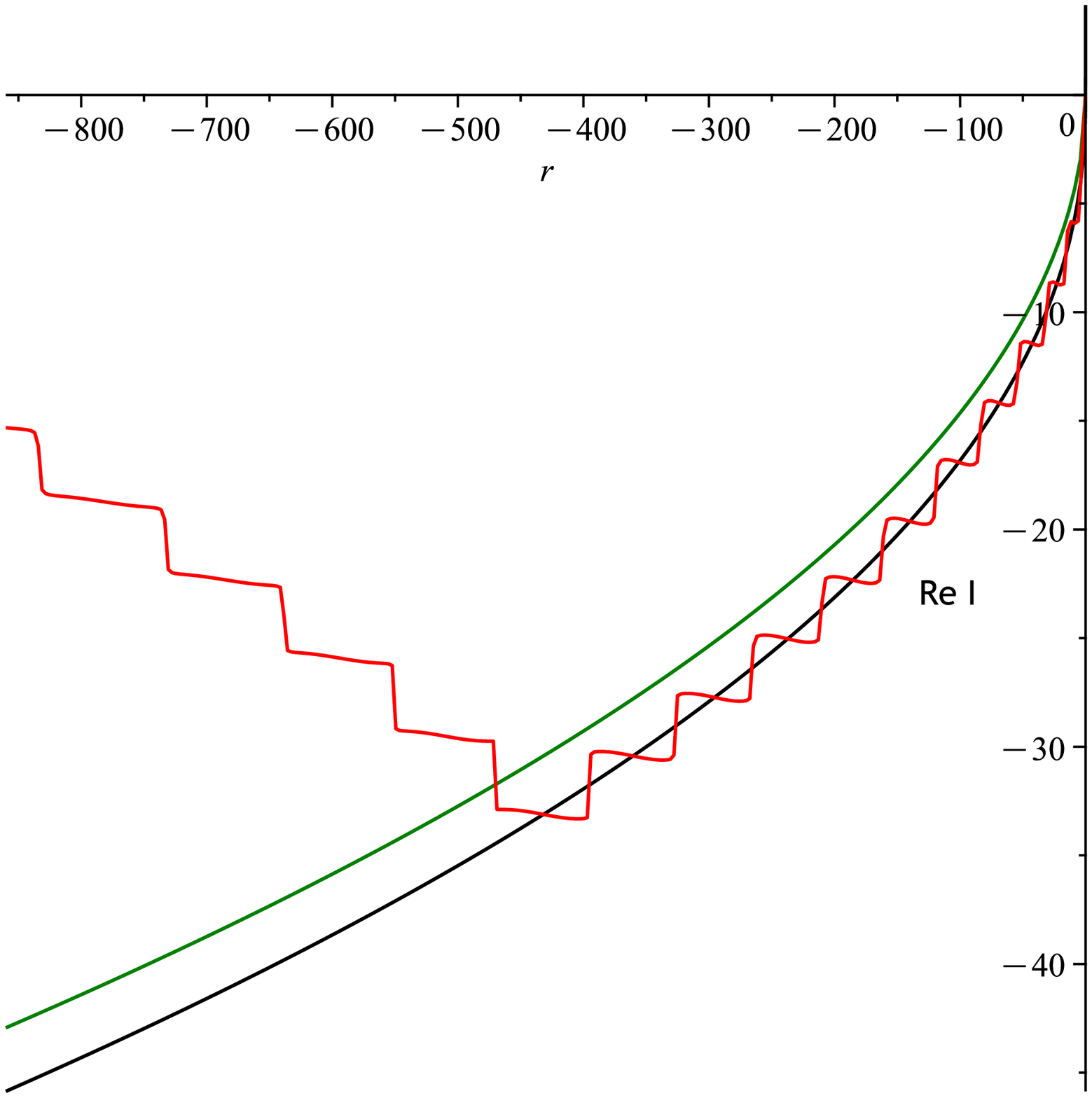}
\caption{The red curve represents the numerical values of $\mathrm{Re}\,I(r)$ corresponding to the initial value
$H(0)=-100-\mi200$. The {\sc Maple} settings for the generation of this curve are similar to those of Example 4.
The parameter $\nu_1$ for this solution equals $0.686841\dotsc-\mi0.323156\dotsc$. The black line represents $Str_r$
with $k=1$, and the upper (green) line is the unstable attracting parabola.}
\label{fig:H0=-100-i200+IntReHstr}
\end{center}
\end{figure}

The asymptotic behaviour as $r\to-\infty$ of almost all solutions of equation~\eqref{eq:hazzidakis} (in particular, all solutions
considered in this work) can be described as follows: there exists some $r$  which could be small or large, depending on
a solution, such that solutions are attracted to the parabola $2\sqrt{-r}$ and remain oscillatory at a distance
(cf. \eqref{eq:left-stringer}) $Str_l-2\sqrt{-r}$ which can be small or large, depending on a particular solution.
In this sense, we can call $r\to2\sqrt{-r}$ the \emph{stable attracting parabola}. Many solutions, prior to their behaviours being
equilibrated by the stable attracting parabola, are being ``captured'' by the \emph{unstable attracting parabola}, that is,
$r\to-2(\sqrt{3}-1)\sqrt{-r}$. For any $r=r_0$, there is a solution which is attracted by the unstable parabola on the distance
$|r_0|$. These solutions are oscillating near $Str_r$. Contrary to the case for $Str_l$, the distance between the
unstable parabola and $Str_r$ has logarithmic growth; however, this growth is not the reason why the unstable parabola
cannot hold the solution on the infinite-$r$ interval. Our formal explanation for this behaviour of the solutions,
which relies on the notion of `asymptotics of asymptotics', is given in Remark~\ref{rem:mole's-office}; it implies that the
tail of the plot for $\textrm{Re}\,H(r)$ is floating above the negative real semi-axis, and therefore the number of down-steps
should be finite.

The function $\mathrm{Re}\,I(r)$ can serve as a mathematical model describing the mole's behaviour. This model
gives a simpler and even more convincing explanation for the instability of the right parabola.
The mole is moving leftward-downward leaving behind a trajectory resembling the right staircase;
the deeper and deeper the mole burrows, the less and less food it finds, so that the right staircase becomes less and less
concave. Finally, the mole realizes that there is no food. At this stage, the mole constructs its dwelling, and, after a brief
respite gathering its thoughts, the mole, regardless of its nature, realizes that it is necessary to change its direction of
movement, and, finally, does so! Consequently, the mole starts to build the left staircase. Initially, the mole moves up in a
near-vertical manner; however, since it is difficult to move vertically, its trajectory becomes ``wavy''. After a few steps,
the mole sees more and more food in front of itself, thus its horizontal motion becomes longer and longer, whilst
its vertical movement continues to shrink; hence, the left staircase becomes less and less convex.
Depending on the particular situation, regulated by the value of $H(0)$, the location of the mole's dwelling
will be  different. We can also interpret the steps of the staircase as tunnels made by the mole, the lengths of which
can be regulated with the help of a scaling parameter, $c>0: r\to cr$. What about the imaginary part of $I(r)$? We see
(cf. Figure~\ref{fig:H0=-0148+i0191+IntImH}) that each icicle corresponds to a step/tunnel in its own right. We know that
the mole is making repositories for its food; so, in each tunnel there is a single food repository, and the lengths of the
icicles can be interpreted as numbers that are proportional to the food supply accrued in the corresponding repositories.
In accordance with this model, the largest repository of food is located precisely in the mole's dwelling.

A more interesting model describing the ``underground mole's geometry'' could be related to a generalization of the
function $I(r)$ that depends on two variables. It seems plausible, therefore, to ``dig'' for such a function amongst the
integrals associated with the higher Painlev\'e equations depending on two variables, namely, the second member of a
hierarchy related to the degenerate third Painlev\'e equation~\eqref{eq:dp3u}.
\appendix
\section{Appendix: The Function $g_2(r)$}\label{app:g2}
In Section~\ref{sec:2}, the definition of the generating functions $g_k(r)$, $k=1,2,\dotsc$, are given,
and the first function of this sequence, $g_1(r)$, is constructed. The second function, $g_2(r)$, is constructed here, and
we also explain how one calculates $P_n^{'}(-1)$ with the help of this result.

The differential equation for $g_2(r)$ is
\begin{equation}\label{eq:g2}
\big(rg_2^{'}(r)\big)'=3g_2(r)-\frac13+\frac13\big(I_1(2\sqrt{3r})\big)^2,
\end{equation}
where $I_1$ is the modified Bessel function of order $1$. This ODE is an inhomogeneous modified Bessel equation of order $0$.
The small-$r$ expansion of $g_2(r)$ does not contain logarithmic terms, and its leading term is $-r/3$. This fact allows us to uniquely
specify the proper solution of this equation:
\begin{equation}\label{eq:g2Bessel}
g_2(r)=\frac19-\frac19I_0(2\sqrt{3r})+
\frac23\int_0^r\left(I_0(2\sqrt{3r})K_0(2\sqrt{3x})-K_0(2\sqrt{3r})I_0(2\sqrt{3x})\right)\left(I_1(2\sqrt{3x})\right)^2\md x.
\end{equation}
The following expansion is not widely known; however, it is very helpful in our study:
\begin{equation}\label{eqI1squaredSERIES}
\left(I_1(2\sqrt{3x})\right)^2=3x\sum_{n=0}^{\infty}\frac{\binom{2n+2}{n}(3x)^n}{((n+1)!)^2},
\end{equation}
where $\tbinom{m}{k}=\tfrac{m!}{k!(m-k)!}$ is the binomial coefficient.

Since, in fact, we need the Taylor series expansion of $g_2(r)$, we solve equation~\eqref{eq:g2} with the help of power series.
Substituting the expansion
\begin{equation}\label{eq:g2SERIEScn}
g_2(r)=\sum_{k=1}^{\infty}c_k\,r^k,
\end{equation}
where the coefficients $c_k$, $k=1,2,\ldots$, are independent of $r$,
into equation~\eqref{eq:g2}, one deduces the following recurrence relation,
$$
(k+1)^2c_{k+1}=3c_k+\frac{3^{k-1}}{(k!)^2}\binom{2k}{k-1},\quad
c_1=-\frac13.
$$
Solving the last relation, we find that
\begin{equation}\label{eq:cnForg2Explicit}
c_n=\frac{3^{n-2}}{(n!)^2}\left(-1+\sum_{k=1}^{n-1}\binom{2k}{k-1}\right),\quad
n=1,2,\ldots.
\end{equation}
\begin{remark}\label{rem:App-g2}
The integer sequence $\sum_{k=1}^{n-1}\binom{2k}{k-1}$, $n=2,3,\ldots$, coincides with sequence A057552
in \cite{OEIS3}.\footnote{See, also, the sequence A279561 in OEIS enumerating the number of length-$n$ inversion sequences
avoiding the patterns $101$, $102$, $201$, and $210$.\label{foot:A279561}}
Shifting $n$ to $n-2$ in a formula given in \cite{OEIS3}, we find that
$$
\sum_{k=1}^{n-1}\binom{2k}{k-1}=\frac{\binom{2n-2}{n-1}}{2n}\left(4n-2-n\,{}_2F_1(1,-n+1;-n+3/2;1/4)\right)-\frac12,\quad
n\in\mathbb{N},
$$
where ${}_2F_1(1,-n+1;-n+3/2;1/4)$ is the value of the Gauss hypergeometric function at $x=1/4$, and
$\tfrac{\binom{2n-2}{n-1}}{n}$ is the $(n-1)$th Catalan number. It is interesting to note that {\sc Maple} gives a more
complicated presentation for this sum:
$$
\sum_{k=1}^{n-1}\binom{2k}{k-1}=-\frac12+\frac{\mi\sqrt{3}}{6}-\binom{2n}{n-1}\,{}_3F_2(1,n+1,n+1/2;n,n+2;4),
$$
where ${}_3F_2(1,n+1,n+1/2;n,n+2;4)$ is the value of the generalized hypergeometric function at $x=4$.
\hfill $\blacksquare$\end{remark}
\begin{corollary}\label{cor:g_2-generating function}
The function~\eqref{eq:g2Bessel} is the generating function for the sequence $c_n$ \emph{(}cf. \eqref{eq:cnForg2Explicit}\emph{)}.
\end{corollary}
\begin{proof}
The function~\eqref{eq:g2Bessel} is obtained from the general solution by specifying the initial conditions
$g_2(0)=0$ and $g_2'(0)=-1/3$. Since $r=0$ is the singular point of the equation, we, strictly speaking, have to prove
that this solution can be developed into a power series in $r$ of the form  \eqref{eq:g2SERIEScn}. The proof
is straightforward: one employs the expansions for $I_0(2\sqrt{3r})$, $I_1(2\sqrt{3r})$, and $K_0(2\sqrt{3r})$ at $r=0$.
The expansions for the functions $I_0(2\sqrt{3r})$ and $I_1(2\sqrt{3r})$ are convergent power series. The expansion for
$K_0(2\sqrt{3r})$ reads
$$
K_0(2\sqrt{3r})=-I_0(2\sqrt{3r})\ln\sqrt{3r}+\sum_{m=0}^{\infty}\frac{\psi(1+m)(3r)^m}{(m!)^2},\quad
\psi(1+m)=1+\frac12+\ldots+\frac1m-\gamma,
$$
where $\gamma=0.57721\ldots$ is the Euler-Mascheroni constant. Then, the integral in \eqref{eq:g2Bessel} can be decomposed
into two parts: the first part is the integral with logarithms, that is,
\begin{equation}\label{eq:INTwithLOGS}
\frac12I_0(2\sqrt{3r})\int_0^r(\ln(r)-\ln(x))I_0(2\sqrt{3x})\left(I_1(2\sqrt{3x})\right)^2\md x,
\end{equation}
whilst the second one is the integral which contains functions that can be expanded in power series in the variable of
integration and $r$. Substituting into the second integral the corresponding power series, one can integrate it successively
and obtain a power series with rational (!) coefficients, because, in this expansion, there appear the differences
$\psi(m+1)-\psi(l+1)$ for integers $m$ and $l$ \cite{BE1}.
The integral~\eqref{eq:INTwithLOGS} should be integrated by parts; then, the explicitly integrated term
vanishes, while the remaining integral contains only those functions that can be developed into Taylor series due to the
expansion~\eqref{eqI1squaredSERIES}.
\end{proof}
\begin{remark}
We have executed the scheme discussed in the proof of Corollary~\ref{cor:g_2-generating function} explicitly;
however, the corresponding formula for the numbers $c_n$ which we obtained via substitution of the corresponding series
into the integral~\eqref{eq:g2Bessel} is quite cumbersome. We were not able to convert the latter formula to the simple
expression~\eqref{eq:cnForg2Explicit}, but we confirmed numerically that both formulae coincide.
\hfill $\blacksquare$\end{remark}
\begin{proposition}\label{prop:Pn'(-1)}
\begin{equation}\label{eq:Pn'(-1)}
(-1)^{\lfloor\frac{n+1}{2}\rfloor}3^{\nu_3(n+1)}P_n^{'}(-1)=
3^{b_n-2}\left(\frac{n}{2}-\frac54-(-1)^n\frac34+\sum_{k=1}^{n-1}\binom{2k}{k-1}\right),\quad
n\in\mathbb{N},
\end{equation}
where the sequences $\nu_3(n+1)$ and $b_n$ are defined in Conjecture~\ref{con:structure-ak} and equation~\eqref{eq:Cloitre},
respectively.
\end{proposition}
\begin{proof}
We now calculate the coefficients $c_n$ with the help of the ansatz~\eqref{eq:a-n-conjecture}. Substitute into
ansatz~\eqref{eq:a-n-conjecture} $a_0=-(1-\varepsilon)^{1/3}$, and consider the expansion of $a_n$ as $\varepsilon\to0$.
The numbers $c_n$ coincide with the coefficients of this expansion at $\varepsilon^2$; the result of this calculation reads:
$$
c_n=(-1)^{\left\lfloor\frac{n-1}{2}\right\rfloor+1}\frac{3^{\nu_3(n+1)}}{(n!|n!|_3)^2}\left(P_n^{'}(-1)+
\frac13\left(\frac{n}{2}-\frac14-(-1)^n\frac34\right)P_n(-1)\right).
$$
Taking into account the formula for $P_n(-1)$ given in \eqref{eq:Pn(-1)}, and equation~\eqref{eq:cnForg2Explicit}, we, after
a straightforward calculation, arrive at equation~\eqref{eq:Pn'(-1)}.
\end{proof}

\section{Appendix: Asymptotics as $\tau \! \to \! 0$ of $u(\tau)$ and $\varphi
(\tau)$ for $a \! \in \! \mathbb{C}$} \label{app:asympt0}
The result presented in Theorem~\ref{th:B1asympt0} below is based on Theorem~3.4 of \cite{KitVar2004}. Here, we formulate
only the key asymptotic result of Theorem~3.4 of \cite{KitVar2004} that concerns asymptotics as $\tau\to+0$
($\arg(\tau)=0$), which means that the parameters $\varepsilon_{1}$ and $\varepsilon_{2}$ appearing in Theorem~3.4
must be set equal to $0$. This, in turn, allows one to simplify the notation in Theorem~3.4 of \cite{KitVar2004}, namely,
we use  $(a,s_{0}^{0},s_{0}^{\infty}, s_{1}^{\infty},g_{11},g_{12},g_{21},g_{22})$ instead of
$(a,s_{0}^{0}(0,0),s_{0}^{\infty}(0,0),s_{1}^{\infty}(0,0),g_{11}(0,0),g_{12}(0,0),g_{21}(0,0),g_{22}(0,0))$
for the monodromy co-ordinates. Furthermore, we simplify the notation for the parameters in the asymptotic formulae:
(i) $\varpi_{k}^{\natural}(0,0;\lambda)$ is changed to $\varpi_{k}(\lambda)$, $k=1,2$; (ii) $\chi_{k}(\mathbf{g};\lambda)$
is denoted here as $\chi_{k}(\lambda)$; and (iii) $\mathbf{p}(a,\lambda):=\mathbf{p}_1(\lambda)$ and
$\mathbf{p}(-a,\lambda)\me^{-\mi\pi\lambda}:=\mathbf{p}_2(\lambda)$. Moreover, we used standard identities for
the (Euler) gamma function in order to simplify the expression for $\mathbf{p}(z_{1},z_{2})$ that appears in Theorem~3.4
of \cite{KitVar2004}.
Since the Hamiltonian function corresponding to equation~\eqref{eq:dp3u} is not studied in this work, its small-$\tau$
asymptotics is omitted in Theorem~\ref{th:B1asympt0} below; at the same time, though, Theorem~\ref{th:B1asympt0} contains
the small-$\tau$ asymptotis of the function  $\varphi(\tau)$. The latter asymptotics was not included
in the formulation of Theorem~3.4 of \cite{KitVar2004}, even though, in fact, it was obtained in the course of the proof of
Theorem 3.4.

Recall that, for any solution $u(\tau)$ of
equation~\eqref{eq:dp3u}, the function $\varphi(\tau)$ is defined as the indefinite integral
\begin{equation}\label{eq:app:varphi-definition}
\varphi{'}(\tau)=\frac{2a}{\tau}+\frac{b}{u(\tau)}.
\end{equation}
This function appears in the parametrization of the coefficients of the corresponding isomonodromy system
(see \cite{KitVar2004}, Proposition~1.2) in the form $\me^{\mi\varphi(\tau)}$; therefore, as long as the monodromy data
for the isomonodromy system are given, the function $\varphi(\tau)$ is fixed modulo $2\pi$, or,
in other words, the constant of integration in \eqref{eq:app:varphi-definition} is defined via the monodromy data modulo
$2\pi$, i.e., from the isomonodromy point of view the function $\varphi(\tau)$ is an ``almost definite integral''.
This fact allows one to calculate some particular integrals that are related to the function $u(\tau)$ \cite{KitVar2019}.

There is a one-to-one correspondence between the pair of functions $(u(\tau),\me^{\mi\varphi(\tau)})$,
where the pair $(u(\tau), \varphi(\tau))$ is a solution of the system~\eqref{eq:dp3u}, \eqref{eq:app:varphi-definition}, and
the monodromy data of the first equation of the system~(12) in \cite{KitVar2004}. It is the monodromy data of this equation
that is denoted as $(a,s_{0}^{0},s_{0}^{\infty}, s_{1}^{\infty},g_{11},g_{12},g_{21},\linebreak[1] g_{22})\in\mathbb{C}^8$.
These complex co-ordinates satisfy a set of algebraic equations where, instead of $a$, the variable $\me^{\pi a}$ appears;
this system of equations in $\mathbb{C}^8$ defines a monodromy manifold. By treating $\me^{\pi a}$ as a parameter, this manifold
can also be considered as an algebraic variety in $\mathbb{C}^7$.
We call a point of this manifold the monodromy data corresponding to $(u(\tau),\me^{\mi\varphi(\tau)})$, or, for brevity,
the monodromy data corresponding to the pair $(u(\tau),\varphi(\tau))$, with the \emph{proviso} that under
the function $\varphi(\tau)$ is understood the class of functions defined by the equivalence relation
$\varphi\equiv\varphi+2\pi k, k\in\mathbb{Z}$.
\begin{theorem}\label{th:B1asympt0}
Let $(u(\tau), \varphi(\tau))$ be a solution of the system~\eqref{eq:dp3u}, \eqref{eq:app:varphi-definition} for
$\varepsilon b>0$ corresponding to the monodromy data $(a,s_{0}^{0}$, $s_{0}^{\infty}$, $s_{1}^{\infty}$,
$g_{11},g_{12},g_{21},g_{22})$. Suppose that
\begin{equation} \label{eq:app:Asympt0:mondataConditions}
\lvert \Im (a) \rvert \! < \! 1, \qquad \rho \! \neq \! 0, \qquad \lvert \Re (\rho)
\rvert \! < \! 1/2, \qquad g_{11}g_{22} \! \neq \! 0,
\end{equation}
where
\begin{equation} \label{eq:appAsympt0:rho-s}
\cos (2 \pi \rho) \! := \! -\frac{\mi s_{0}^{0}}{2} \! = \! \cosh (\pi a) \! + \!
\frac{1}{2}s_{0}^{\infty}s_{1}^{\infty} \me^{\pi a}.
\end{equation}
Then, $\exists$ $\delta \! > \! 0$ such that
\begin{equation}\label{eq:app:Asympt0:u}
u(\tau)\underset{\tau\to+0}{=} \, \frac{\tau b \me^{\pi a/2}}{16 \pi}\!
\left(\varpi_{1}(\rho)\tau^{2\rho}+\varpi_{1}(-\rho)\tau^{-2\rho}\right)\!
\left(\varpi_{2}(\rho)\tau^{2\rho}+\varpi_{2}(-\rho)\tau^{-2\rho}\right)\!
\left(1+o\big(\tau^{\delta})\right),
\end{equation}
and
\begin{equation}\label{eq:app:varphi}
\me^{\mi\varphi(\tau)}\underset{\tau\to+0}{=}\me^{\mi \pi}
2^{\mi a} \tau^{\mi 2a}\!
\left(\frac{\varpi_{2}(\rho)\tau^{2\rho}\! + \! \varpi_{2}(-\rho)\tau^{-2\rho}}
{\varpi_{1}(\rho)\tau^{2\rho}\! + \! \varpi_{1}(-\rho)\tau^{-2\rho}}
\right) \!
\left(1+o\big(\tau^{\delta}\big)\right),
\end{equation}
where
\begin{gather}
\varpi_{k}(\lambda)=\mathbf{p}_k(\lambda)\chi_{k}(\lambda), \qquad
k=1,2,\label{eq:app:varpi-pk-chi}\\
\chi_{k}(\lambda)=g_{1k}\me^{\mi\pi(\lambda+1/4)}+g_{2k}\me^{-\mi\pi(\lambda+1/4)} \label{eq:app:chi-g},\\
\mathbf{p}_k(\lambda)=\me^{-(-1)^k\frac{\mi\pi}{2}\lambda}\left(\frac{\varepsilon b}{2}\right)^{\lambda}
\frac{\Gamma(1-2\lambda)}{\Gamma(1+2\lambda)}\frac{\Gamma\big(1+\lambda-(-1)^k\mi a/2\big)}{\lambda},\label{eq:app:pk-gamma}
\end{gather}
and $\Gamma (\ast)$ is the gamma function {\rm \cite{BE1}}.
\end{theorem}
\section{Appendix: Asymptotics as $\tau \! \to \! +\infty$ of $u(\tau)$
and $\varphi(\tau)$ for $a \! \in \! \mathbb{C}$} \label{app:infty}
In one of our previous works (see \cite{KitVar2010}, Appendix B), we noted that the phase shift in the large-$\tau$
asymptotics of the function $u(\tau)$ reported in \cite{KitVar2004} contained a mistake which was corrected in
Appendix B of \cite{KitVar2010}.
Subsequently, the author of \cite{KitSIGMA2019} (see Section 7 of \cite{KitSIGMA2019}) uncovered an additional inconsistency
in \cite{KitVar2004} which, ``hereditarily'', transferred to the paper~\cite{KitVar2010}: this error does not affect
the asymptotcs of the function $u(\tau)$; it does, however, impact the asymptotics of the function $\varphi(\tau)$, since
the term $a\ln\tau$ was not accounted for in its small- and large-$\tau$ asymptotics. The latter asymptotics for $\varphi(\tau)$
were not included in list of the results obtained in the aforementioned papers. Nevertheless, we wrote the paper~\cite{KitVar2019}
which dealt with the asymptotics of an integral for the meromorphic solution of equation~\eqref{eq:dp3u}, where it was explained
how to rectify the asymptotics for $\varphi(\tau)$ in case it is extracted from \cite{KitVar2004}, and the particular example
considered in \cite{KitVar2019} was verified numerically. The paper~\cite{KitVar2019} contains the asymptotics of the function
$\varphi(\tau)$ for the meromorphic solution together with the correction term, which was stated, but not derived; however,
it was checked numerically for the particular example considered therein. In this appendix, we present the asymptotics with
the correction term for the general solution and also give the derivation for the correction term.

Furthermore, during the course of the numerical calculations for the present work, we found one more ``mysterious'' misprint
in the large-$\tau$ asymptotics of the function $u(\tau)$ (see \cite{KitVar2004}, Section 3, p. 1174, Theorem 3.1):
the term $-\tfrac{\pi\mi}{2}$ in the definition of $z(\varepsilon_1,\varepsilon_2)$ must be changed to $+\tfrac{\pi\mi}{2}$.
Since our original calculations presented in \cite{KitVar2004} do, in fact, contain the term $+\tfrac{\pi\mi}{2}$ , it is
clear that this is a typographical error.
This error in the sign appears, sadly, to have leapfrogged to the asymptotics obtained in the subsequent work
\cite{KitVar2010}, because the method used therein (see \cite{KitVar2010}, the paragraph at the top of p. 43) allows us
to determine the corresponding phase shift modulo $\pi\mi$.
In order to determine, ultimately, whether or not $\pi\mi$ should be added, we matched the asymptotics obtained in
\cite{KitVar2004} with those of \cite{KitVar2010}. Moreover, in \cite{KitSIGMA2019}, where this discrepancy with $\pi\mi$
was noticed, and where, in Subsections 8.1 and 8.2, even though the correct formulae
are written and used for the numerical calculations, the general result for the asymptotics of $u(\tau)$ stated on p. 47
still contains the wrong sign for $\tfrac{\pi\mi}{2}$ in the formula for $z$ in equation (8.7)!

In this appendix, we restate the main results obtained in the papers \cite{KitVar2004} and \cite{KitVar2010} but with all the
corrections outlined above included. As in Appendix~\ref{app:asympt0}, we present here only asymptotics on the positive real
semi-axis and exclude all transformation issues, i.e., we set $\varepsilon_1=\varepsilon_2=0$ and simplify the notation for
the monodromy data and the parameters used in Theorem 3.1 of \cite{KitVar2004} and in Theorems~2.1--2.3 and B.1
of \cite{KitVar2010}, namely,
$\linebreak[1]
(s_{0}^{0}(0,0),s_{0}^{\infty}(0,0),s_{1}^{\infty}(0,0),g_{11}(0,0),g_{12}
(0,0),g_{21}(0,0),g_{22}(0,0),\tilde{\nu}(0,0),z(0,0)):=
(s_{0}^{0},s_{0}^{\infty},s_{1}^{\infty},g_{11}, g_{12},g_{21},g_{22},\tilde{\nu},z).
$
\begin{remark}\label{rem:app:inffty:branches}
Here and throughout the paper we use non-single-valued functions such as roots, logarithms, fractional powers, and power
functions with complex exponent to construct single-valued asymptotics. We use the natural agreement that, for real positive
argument, all these functions are positive. The branch of such a function for complex argument, in case it appears several
times in a formula or within a group of related formulae, is assumed to be the same within the formula or the group.
In case a choice of the branch is not discussed, it means that its particular choice does not affect the value of the resulting
formula(e); for example, the branch of $\ln(z)$ in $\sinh(\ln(z))$ or the branch of $\sqrt{\tilde\nu+1}$ in
equation~\eqref{eq:app:u-reg-as} below. Here, as in the paper, $\Gamma (\ast)$ is the gamma function \cite{BE1}.
\hfill $\blacksquare$\end{remark}

Since the Hamiltonian function corresponding to equation~\eqref{eq:dp3u} is not studied in this work, its large-$\tau$
asymptotics stated in Theorem~3.1 of \cite{KitVar2004} is not included in Theorem~\ref{th:asympt-infty2004} below; instead,
it is supplanted by the asymptotics of the function $\varphi(\tau)$: with this, the asymptotic description of the
isomonodromy deformation for the first equation of the system (12) in \cite{KitVar2004} is complete.
\begin{theorem} \label{th:asympt-infty2004}
Let $(u(\tau), \varphi(\tau))$ be a solution of the system~\eqref{eq:dp3u}, \eqref{eq:app:varphi-definition} for
$\varepsilon b>0$ corresponding to the monodromy data $(a,s_{0}^{0}$, $s_{0}^{\infty}$, $s_{1}^{\infty}$,
$g_{11},g_{12},g_{21},g_{22})$.
Define
\begin{equation} \label{eq:app:nu+1}
\tilde{\nu} \! + \! 1 \! := \! \frac{\mi}{2 \pi} \ln (g_{11}g_{22}),
\end{equation}
and suppose that
\begin{equation} \label{eq:app:u-reg-as-conditions}
g_{11}g_{12}g_{21}g_{22} \! \neq \! 0, \, \qquad \, \lvert \Re (\tilde{\nu} \! + \! 1)
\rvert \! < \! 1/6.
\end{equation}
Then, $\exists$ $\delta \! > \! 0$ such that
\begin{equation}\label{eq:app:u-reg-as}
u(\tau)\underset{\tau\to+\infty}{=}\,\frac{\varepsilon(\varepsilon b)^{2/3}}{2}\tau^{1/3}\!
\left(1+\frac{2\sqrt{\tilde{\nu}+1}\me^{3\pi\mi/4}}{3^{1/4}(\varepsilon b)^{1/6}\tau^{1/3}}
\cosh\Big(\mi\theta(\tau)+(\tilde{\nu}+1)\ln\theta(\tau)+z+o\big(\tau^{-\delta}\big)\Big)\right),
\end{equation}
where
\begin{gather}
\theta(\tau):=3^{3/2}(\varepsilon b)^{1/3}\tau^{2/3},\label{eq:app:infty:regular-theta-tau}\\
z:=\,\frac{\ln2\pi}{2}+\frac{\pi\mi}{2}-\frac{3\pi\mi}{2}(\tilde{\nu}+1)+\mi a\ln(2 +\sqrt{3})+(\tilde{\nu}+1)\ln12
-\ln\!\left(g_{11}g_{12}\sqrt{\tilde{\nu}+1}\,\Gamma(\tilde{\nu}+1)\right), \label{eq:app:asympt-regular-phase}
\end{gather}
and
\begin{equation}\label{eq:app:infty:regular-varphi}
\begin{gathered}
\varphi(\tau)\underset{\tau\to+\infty}{=} \,
3(\varepsilon b)^{1/3}\tau^{2/3}+2a\ln\big(\tau^{2/3}\big)-a\ln\left((\varepsilon b)^{1/3}/4\right)+\pi-2\pi(\tilde{\nu}+1)\\
+\mi\ln\big(g_{11}^2\big)-2\mi(\tilde{\nu}+1)\ln(2+\sqrt{3})+\rm{E}_{\varphi}(\tau),\qquad
\rm{E}_{\varphi}(\tau)=o\big(\tau^{-\delta}\big).
\end{gathered}
\end{equation}
\end{theorem}
\begin{remark}\label{rem:app:varphi-error}
The explicit expression for the correction term $\mathrm{E}_{\varphi}(\tau)$ (cf. equation~\eqref{eq:app:infty:regular-varphi})
is obtained in Proposition~\ref{prop:app:error} below.
\hfill $\blacksquare$\end{remark}
\begin{remark}\label{rem:app:cond-regular}
Here, we discuss the second restriction in \eqref{eq:app:u-reg-as-conditions}. The methodology used in \cite{KitVar2004}
is such that the function $\cosh(\cdot)$ in the asymptotics~\eqref{eq:app:u-reg-as} is obtained during the course of the
appropriate estimates as the half-sum of two exponentials: if $\mathrm{Re}(\tilde\nu+1)\neq0$, then, one of the exponentials
has a power-like growth, and the other has a power-like decay.
Each exponential should represent a leading term of asymptotics at the corresponding stage of the derivation which leads
to the second restriction in \eqref{eq:app:u-reg-as-conditions}. In principle, had we required that only the growing
exponential represent the leading term of asymptotics, then we would have obtained a less restrictive condition, namely,
$|\mathrm{Re}(\tilde\nu+1)|<1/2$. Had we followed the latter scheme of the derivation, we would have had to consider
three different conditions on $\mathrm{Re}(\tilde\nu+1)$, i.e.,
$\mathrm{Re}(\tilde\nu+1)=0$ and $\pm\mathrm{Re}(\tilde\nu+1)\in(0,1/2)$, and would have derived three separate asymptotic
formulae which, in our case, are amalgamated in the unique formula \eqref{eq:app:u-reg-as}.

This fact can be revealed from a completely different perspective: if we address the complete asymptotic expansion given in
equation~\eqref{eq:app:u-asympt-regular-complete} below, then we can verify that, in fact, the smallest exponential in
the leading term of this expansion is larger than the largest exponential in the second correction term when
the second condition in \eqref{eq:app:u-reg-as-conditions} is valid.

The value $1/6$ is quite clearly observed in our numerical examples presented in Section~\ref{sec:asymptnumerics}
where, in the first two examples  $|\mathrm{Re}(\tilde\nu+1)|<1/6$, whilst in the other examples, this condition is violated.
According to our numerical studies, the solutions with the monodromy data satisfying this condition swiftly reach their
asymptotic behaviour provided that none of the factors in the product of the first condition in
\eqref{eq:app:u-reg-as-conditions} is close to zero. By ``not close to zero'' we mean that $|g_{i,j}|\geqslant0.1$;
however, the specification of this numerical aspect requires further investigation.

If one understands the $\cosh(\cdot)$ term as a unique function, then it plays the role of the leading term of asymptotics
(as its growing exponential) for a wider range of monodromy data, namely, $|\text{Re}(\tilde{\nu}+1)|<1/2$.
When $|\text{Re}(\tilde{\nu}+1)|$ increases farther and farther away from the value $1/6$, the corresponding solutions
reach their asymptotic behaviour at larger and larger values of $\tau$. Application of the higher-order terms
of the expansion \eqref{eq:app:u-asympt-regular-complete} is helpful, but not radically, though. A much better correspondence
between the asymptotic and numeric results for finite values of $\tau$ for the case
$|\text{Re}(\tilde{\nu}+1)|\geqslant1/6$ is achieved with the help of the asymptotics given in
Theorem~\ref{th:asympt-infty2010} below.
\hfill $\blacksquare$\end{remark}
\begin{remark}\label{rem:app:varphi}
The branch of $\ln\big(g_{11}^2\big)$ in the asymptotics~\eqref{eq:app:infty:regular-varphi} can not be fixed because the
function $\varphi(\tau)$ enters into the corresponding isomonodromy system in the form $\me^{\mi\varphi(\tau)}$
(see \cite{KitVar2004}, Propositions 1.1 and 1.2); therefore, the real part of the asymptotics for $\varphi(\tau)$ is defined
modulo $2\pi$.

On the other hand, the function $u(\tau)$ has movable zeros with leading term $\pm\mi b(\tau-\tau_0)$, and
equation~\eqref{eq:app:varphi-definition} implies that $\varphi(\tau)$ has movable logarithmic singularities with
leading term $\mp\mi\ln(\tau-\tau_0)$. Even with a branch cut
on the complex $\tau$-plane from the origin to the point at infinity, the analytic continuation of $\varphi(\tau)$ depends
on the path of continuation: we thus get infinitely many branches of $\varphi(\tau)$ that differ by $2\pi k$, $k\in\mathbb{Z}$.
As a matter of fact, the asymptotics~\eqref{eq:app:u-reg-as} is valid not only on the positive semi-axis, but also in the
larger domain $\mathcal{D}=\{\tau\in\mathbb{C};\mathstrut\, |\text{Im}\,\tau^{2/3}|<h, h\in\mathbb{R}_+\}$: for large
$|\tau|$, the width of $\mathcal{D}$ is growing as $\mathcal{O}\big(|\tau|^{1/3}\big)$.
There may be some logarithmic singularities of $\varphi(\tau)$ in the interior of this domain along with a few non-homotopic
paths going around them.
The locations of such singularities for finite values of $\tau$ are not known, and, consequently, the corresponding
asymptotics ``are not cognizant'' of the paths along which $\varphi(\tau)$ is continued; therefore, the formula for the
asymptotics should contain an ambiguity in order to account for all possible analytic continuations of $\varphi(\tau)$
along non-homotopic paths, if any, which is retained in the ambiguity in the choice of the imaginary part of
$\ln\big(g_{11}^2\big)$.

During the course of our numerical studies, we noticed one additional mechanism for the $2\pi$-ambiguity in the asymptotic
formula $\varphi(\tau)$. This mechanism is related to the fact that our plots are made starting from
finite---and quite small---values of $\tau$. For such values of $\tau$, the asymptotics does (sometimes) not approximate
the solution $u(\tau)$ as well as it does for very large values of $\tau$; therefore, if, in some segment,
$u(\tau)$ is close, but not equal, to zero, its asymptotics may vanish in this segment, which provides a $2\pi$-discrepancy in
the approximation of the function $\varphi(\tau)$ via its asymptotics (cf. Section~\ref{sec:asymptnumerics}, Example~4).

This remark brings to the forefront the fact that the formulation of a general rule which would fix the $2\pi$-ambiguity
in the $\varphi(\tau)$ asymptotics, applicable for all solutions, is a nontrivial problem; however, for some classes of
the solutions, such a rule could, feasibly, be formulated.
Assume one considers the positive semi-axis as the path of analytic continuation for $\varphi(\tau)$: on this semi-axis,
there exists a class of solutions, $u(\tau)$, with strictly positive real part. Suppose one finds the correct
branch of $\ln\big(g_{11}^2\big)$ for one solution from this class; then, since the asymptotics of $u(\tau)$ is given
explicitly, one can determine the proper branch for all solutions from this class.
\hfill $\blacksquare$\end{remark}
In order to obtain the explicit formula for $\mathrm{E}_{\varphi}(\tau)$, we are going to use
equation~\eqref{eq:app:varphi-definition}: for this purpose, the first few terms of the complete asymptotic expansion for
the function $1/u(\tau)$ are required. The asymptotic expansion we need was obtained by Shimomura~\cite{ShSh-inf2015}
with the help of a unique method, together with asymptotic expansions for the other Painlev\'e transcendents. Here, we
present a ``visualization'' for the first few terms of this expansion for the functions $u(\tau)$ and
$1/u(\tau)$, and suggest some conjectures for their coefficients:
\begin{equation}\label{eq:app:u-asympt-regular-complete}
u(\tau)\underset{\tau\to+\infty}{=}\frac{\varepsilon(\varepsilon b)^{2/3}}{2}\tau^{1/3}\left(1+
\sum_{k=1}^\infty\frac{1}{\tau^{k/3}}\sum_{j=-k}^{j=k}a_{k,j}w^{j}\right),\qquad
w:={\tau}^{\frac23(\tilde{\nu}+1)}\me^{\mi\theta(\tau)}.
\end{equation}
Comparing the expansion~\eqref{eq:app:u-asympt-regular-complete} with the asymptotics~\eqref{eq:app:u-reg-as}, we find that
\begin{equation}\label{eq:app:a1j}
a_{1,0}=0,\qquad
a_{1,\pm1}=\frac{\sqrt{\tilde{\nu}+1}\,\me^{3\pi\mi/4}}{3^{1/4}(\varepsilon b)^{1/6}}
\me^{\pm\left((\tilde{\nu}+1)\ln\big(3^{3/2}(\varepsilon b)^{1/3}\big)+z\right)},\qquad
a_{1,1}\,a_{1,-1}=-\frac{\mi(\tilde{\nu}+1)}{\sqrt{3}(\varepsilon b)^{1/3}}.
\end{equation}
\begin{remark}\label{rem:app:full-asympt-exp}
The finite inner sum in the expansion~\eqref{eq:app:u-asympt-regular-complete} is considered as a unique function
(the $k$th correction term): in this case, we get an asymptotic expansion of $u(\tau)$ for $|\text{Re}(\tilde{\nu}+1)|<1/2$.
One can consider different types of summations for the double series in \eqref{eq:app:u-asympt-regular-complete} and obtain
different asymptotic formulae for $u(\tau)$.
\hfill $\blacksquare$\end{remark}

To derive the explicit formula for the leading-order correction $\rm{E}_{\varphi}(\tau)$,
the coefficients $a_{1,\pm1}$, $a_{2,\pm1}$, and $a_{3,0}$ must be known explicitly: the coefficients $a_{1,\pm1}$
are given in equations~\eqref{eq:app:a1j}; therefore, it remains to determine $a_{2,\pm1}$ and $a_{3,0}$.
It is known that, in order to obtain the coefficients $a_{k,j}$ for $k\geqslant2$ and $j\neq\pm1$ and $a_{k-2,\pm1}$
in terms of the lower-order coefficients $a_{l,j}$ with $l<k$, one can substitute the
expansion~\eqref{eq:app:u-asympt-regular-complete}, with the terms up to order $\tau^{-k/3}$ retained, into
equation~\eqref{eq:dp3u} and equate to zero (in the obtained expression) the first $k-1$ higher-order coefficients
as $\tau\to+\infty$. To get the last two coefficients `at level $k$', i.e., $a_{k,\pm1}$, one has to repeat the analogous
procedure, but keep in the expansion~\eqref{eq:app:u-asympt-regular-complete} the terms of two more orders, namely,
$\tau^{-(k+1)/3}$ and $\tau^{-(k+2)/3}$. To calculate, for example, the coefficient $a_{3,0}$, one has to keep in the
expansion~\eqref{eq:app:u-asympt-regular-complete} the $\mathcal{O}\big(\tau^{-4/3}\big)$ terms.

It is convenient to introduce new variables:
$$
\varkappa=\tilde{\nu}+1,\qquad
\alpha=2\mi\sqrt{3}\,a.
$$
In terms of these variables, the coefficients of orders $\mathcal{O}\left(\tau^{-2/3}\right)$
and $\mathcal{O}\left(\tau^{-1}\right)$ read:\footnote{In all formulae with $\pm$'s, one has to choose, on both
the left- and right-hand sides, either the upper or the lower sign.\label{foot:app:sign-rule}}
\begin{gather}
a_{2,\pm2}=\frac{a_{1,\pm1}^2}{3},\quad
a_{2,\pm1}=0,\quad
a_{2,0}=\frac{\alpha+12\varkappa}{6\mi\sqrt{3}(\varepsilon b)^{1/3}},\label{eq:app:a2j}\\
a_{3,\pm3}=\frac{a_{1,\pm1}^3}{12},\quad
a_{3,\pm2}=a_{3,0}=0,\quad
a_{3,\pm1}=\pm\frac{\mi\sqrt{3}\,a_{1,\pm1}}{6^3(\varepsilon b)^{1/3}}
\big(3(\alpha^2+8\varkappa\alpha+10\varkappa^2)-3\mp12\alpha\mp80\varkappa\big).
\label{eq:app:a3j}
\end{gather}
We have explicitly calculated the coefficients of the expansion~\eqref{eq:app:u-asympt-regular-complete} up to the level $k=10$;
based on these calculations, we formulate below some conjectures regarding the coefficients $a_{k,j}$. Since the formulae
become progressively more cumbersome for higher values of the level $k$,
we present below the coefficients for levels $k=4,5,6$ in order to demonstrate our hypotheses: these examples might
be useful for the proofs of our conjectures.
\begin{equation}\label{eq:app:a4j}
\begin{gathered}
a_{4,\pm4}=\frac{a_{1,\pm1}^4}{54},\qquad
a_{4,\pm3}=a_{4,\pm1}=a_{4,0}=0,\\
a_{4,\pm2}=\pm\frac{\mi\sqrt{3}\,a_{1,\pm1}^2}{2^23^4(\varepsilon b)^{1/3}}
\big(3\alpha^2+24\varkappa\alpha+30\varkappa^2+1\mp12\alpha\mp54\varkappa\big).
\end{gathered}
\end{equation}
\begin{equation}\label{eq:app:a5j}
\begin{gathered}
a_{5,\pm5}=\frac{5a_{1,\pm1}^5}{1296},\qquad
a_{5,\pm4}=a_{5,\pm2}=a_{5,0}=0,\\
a_{5,\pm3}=\pm\frac{\mi\sqrt{3}\,a_{1,\pm1}^3}{2^53^4(\varepsilon b)^{1/3}}
\big(9(\alpha^2+8\varkappa\alpha+10\varkappa^2)+1\mp36\alpha\mp138\varkappa\big),\\
a_{5,\pm1}=-\frac{a_{1,\pm1}}{2^73^5(\varepsilon b)^{2/3}}
\left(9(\alpha^2+8\varkappa\alpha+10\varkappa^2)^2\mp(24\alpha^3+48\alpha^2\varkappa-912\alpha\varkappa^2
-3440\varkappa^3)\right.\\
\left.-90\alpha^2-1296\varkappa\alpha-4168\varkappa^2\pm(216\alpha+240\varkappa)+81\right).
\end{gathered}
\end{equation}
\begin{equation}\label{eq:app:a6j}
\begin{gathered}
a_{6,\pm6}=\frac{a_{1,\pm1}^6}{1296},\qquad
a_{6,\pm5}=a_{6,\pm3}=a_{6,\pm1}=0,\\
a_{6,0}=-\frac{\mi\sqrt{3}}{2^33^6\varepsilon b}\left(\alpha^3+18\alpha^2\varkappa+72\alpha\varkappa^2+60\varkappa^3
-12\alpha-90\varkappa\right),\\
a_{6,\pm4}=\pm\frac{\mi\sqrt{3}\,a_{1,\pm1}^4}{2^33^6(\varepsilon b)^{1/3}}
\big(6(\alpha^2+8\varkappa\alpha+10\varkappa^2)-1\mp24\alpha\mp84\varkappa\big),\\
a_{6,\pm2}=-\frac{a_{1,\pm1}^2}{2^53^6(\varepsilon b)^{2/3}}
\left(9(\alpha^2+8\varkappa\alpha+10\varkappa^2)^2-171\mp(48\alpha^3+396\alpha^2\varkappa+576\alpha\varkappa^2
-880\varkappa^3)\right.\\
\left.-30\alpha^2-816\varkappa\alpha-2770\varkappa^2\pm(384\alpha+1408\varkappa)\right).
\end{gathered}
\end{equation}
The coefficients calculated above suggest the following
\begin{conjecture}\label{con:app:akk+parity}
If $k\in\mathbb{N}$ and $j\in\mathbb{Z}$, $|j|\leqslant k$, have different parity, then $a_{k,j}=0$, and
\begin{equation*}\label{eq:app:conj:akk}
a_{k,\pm k}=\frac{k\,a_{1,\pm1}^k}{2^{k-1}3^{k-1}}.
\end{equation*}
\end{conjecture}
In addition, we calculated the coefficients for the levels $k=7$, $8$, $9$, and $10$ in order to arrive at the following
\begin{conjecture}\label{con:app:ak(k-2)}
For $k\geqslant4$,
\begin{equation}\label{eq:app:ak(k-2)}
\begin{aligned}
a_{k,\pm(k-2)}=&\pm\frac{\mi\sqrt{3}\,a_{1,\pm1}^{k-2}}{2^k3^k(\varepsilon b)^{1/3}}\left(3(k-2)^2(\alpha^2+
8\varkappa\alpha+10\varkappa^2)-5(k-2)^2+24(k-2)-24\right.\\
&\left.\mp12(k-2)^2\alpha\mp6\big(5(k-2)^2+8(k-2)\big)\varkappa\right).
\end{aligned}
\end{equation}
\end{conjecture}
\begin{remark}\label{rem:app:conjectures}
The coefficients $a_{3,\pm1}$ look similar in form to \eqref{eq:app:ak(k-2)}, but are different. An analogous situation
occurs with the Taylor coefficients for the function $H(r)$ studied in Section~\ref{sec:2}, where the first few members
of some sequences of the coefficients do not conform to the general formula for the sequence.
The quickest way of proving these conjectures, and other analogous formulae, is to use the technique of generating functions,
presented, in particular, in Section~\ref{sec:2}.
\hfill $\blacksquare$\end{remark}
Using the expansion~\eqref{eq:app:u-asympt-regular-complete} for $u(\tau)$, we find a similar expansion for $1/u(\tau)$:
\begin{equation}\label{eq:app:1/u-asympt-regular-complete}
\begin{aligned}
\frac{b}{u(\tau)}\underset{\tau\to+\infty}{=}&\frac{2(\varepsilon b)^{1/3}}{\tau^{1/3}}-
\frac{2(\varepsilon b)^{1/3}}{\tau^{2/3}}\left(a_{1,-1}w^{-1}+a_{1,1}w\right)+
\frac{1}{\tau}\left(\frac43(\varepsilon b)^{1/3}\big(a_{1,-1}^2w^{-2}+a_{1,1}^2w^2\big)-\frac23a\right)\\
+&\sum_{k=4}^\infty\frac{1}{\tau^{k/3}}\sum_{j=-k+1}^{j=k-1}b_{k,j}w^{j},\qquad\quad
w={\tau}^{\frac23(\tilde{\nu}+1)}\me^{\mi\theta(\tau)},
\end{aligned}
\end{equation}
where the coefficients $b_{k,j}$ have similar expressions in terms of $a_{1,\pm1}$ as do the $a_{k,j}$'s. For
our goal of determining the leading term of asymptotics for $\rm{E}_{\varphi}(\tau)$, the coefficient
$b_{4,0}$ is important: it can be determined with the help of equations~\eqref{eq:app:a2j} and \eqref{eq:app:a3j}; in fact,
we determined all the coefficients at level $k=4$:
\begin{align*}
b_{4,0}&=b_{4,\pm2}=0,\qquad
b_{4,\pm3}=-\frac56(\varepsilon b)^{1/3}a_{1,\pm1}^3,\\
b_{4,\pm1}&=\mp\frac{\mi\sqrt{3}\,a_{1,\pm1}}{108}\big(3(\alpha^2+8\varkappa\alpha+10\varkappa^2-1)
\pm12\alpha\pm40\varkappa\big).
\end{align*}
\begin{proposition}\label{prop:app:error}
\begin{equation}\label{eq:app:error}
\mathrm{E}_{\varphi}(\tau)\underset{\tau\to+\infty}{=}
\frac{2\sqrt{\tilde{\nu}+1}\,\me^{-3\pi\mi/4}}{3^{3/4}(\varepsilon b)^{1/6}\tau^{1/3}}
\sinh\big(\mi\theta(\tau)+(\tilde{\nu}+1)\ln\theta(\tau)+z\big)+
\mathcal{O}\left(\frac{\me^{\pm2\mi\theta(\tau)}}{\tau^{\frac23(1-2|\mathrm{Re}(\tilde{\nu}+1)|)}}\right)
+\mathcal{O}\left(\frac{1}{\tau^{2/3}}\right),
\end{equation}
where $\theta(\tau)$ and $z$ are defined in Theorem~\ref{th:asympt-infty2004}.
\end{proposition}
\begin{proof}
For $\tau>0$, define the function $f(\tau)$ via
\begin{equation}\label{eq:app:f}
\frac{b}{u(\tau)}=\frac{2(\varepsilon b)^{1/3}}{\tau^{1/3}}-
\frac{2(\varepsilon b)^{1/3}}{\tau^{2/3}}\left(a_{1,-1}w^{-1}+a_{1,1}w\right)-\frac{2a}{3\tau}+f(\tau).
\end{equation}
This definition suggests the following estimates as $\tau\to+\infty$,
\begin{equation}\label{eq:app:f-estimates}
f(\tau)\underset{\tau\to+\infty}{=}\mathcal{O}\left(\frac{w^2}{\tau}\right)+\mathcal{O}\left(\frac{w^{-2}}{\tau}\right)
+\mathcal{O}\left(\frac{1}{\tau^{5/3}}\right),
\end{equation}
\begin{equation}\label{eq:app:f-integral-estimate}
\int_{+\infty}^{\tau}f(\tau)\,\md\tau\underset{\tau\to+\infty}{=}
\mathcal{O}\left(\frac{\me^{\pm2\mi\theta(\tau)}}{\tau^{\frac23(1-2|\text{Re}(\tilde{\nu}+1)|)}}\right)
+\mathcal{O}\left(\frac{1}{\tau^{2/3}}\right),
\end{equation}
where the contour of integration is assumed to be taken along the positive semi-axis, or, more generally,
in the domain $\mathcal{D}$ described in Remark~\ref{rem:app:varphi}. The right-most estimates in the
asymptotics~\eqref{eq:app:f-estimates} and \eqref{eq:app:f-integral-estimate} are proportional to
$b_{5,0}=\big(\alpha^2-3(4\varkappa+\alpha)^2\big)/\big(2^23^3(\varepsilon b)^{1/3}\big)$.

Now, integrating equation~\eqref{eq:app:varphi-definition} from $\tau_0$ to $\tau$ along the contour in $\mathcal{D}$,
and taking into account the definition of the function $f(\tau)$ (cf.  equation~\eqref{eq:app:f}), we obtain the following
equation:
\begin{equation}\label{eq:app:varphi-f-integral}
\varphi(\tau)-\varphi(\tau_0)=3(\varepsilon b)^{1/3}\big(\tau^{2/3}-\tau_0^{2/3}\big)+\frac{4a}{3}\ln\frac{\tau}{\tau_0}
-\int_{\tau_0}^{\tau}\frac{2(\varepsilon b)^{1/3}}{\tau^{2/3}}\left(a_{1,-1}w^{-1}+a_{1,1}w\right)\,\md\tau
+\int_{\tau_0}^{\tau}f(\tau)\,\md\tau.
\end{equation}
Substituting into \eqref{eq:app:varphi-f-integral} the asymptotics for $\varphi(\tau)$  given in
\eqref{eq:app:infty:regular-varphi} and separating the $\tau$-dependent and $\tau$-independent parts, one finds that
\begin{gather}
\mathrm{E}_{\varphi}(\tau)=-\int_{+\infty}^{\tau}
\frac{2(\varepsilon b)^{1/3}}{\tau^{2/3}}\left(a_{1,-1}w^{-1}+a_{1,1}w\right)\md\tau
+\int_{+\infty}^{\tau}f(\tau)\,\md\tau,\label{eq:app:Error-f-integral}\\
3(\varepsilon b)^{1/3}\tau_0^{2/3}+\frac{4a}{3}\ln\tau_0+M-\varphi(\tau_0)=
\int_{\tau_0}^{+\infty}\left(\frac{b}{u(\tau)}-\frac{2(\varepsilon b)^{1/3}}{\tau^{1/3}}+\frac{2a}{3\tau}\right)\md\tau,
\label{eq:app:integral-b/u}
\end{gather}
where, in the integral \eqref{eq:app:integral-b/u}, we substituted for $f(\tau)$ its definition given in \eqref{eq:app:f},
and $M$ denotes the monodromy constant from equation~\eqref{eq:app:infty:regular-varphi}, namely,
\begin{equation*}\label{eq:app:M}
M=-a\ln\big((\varepsilon b)^{1/3}/4\big)+\pi-2\pi(\tilde{\nu}+1)+\mi\ln\big(g_{11}^2\big)-2\mi(\tilde{\nu}+1)\ln(2+\sqrt{3}).
\end{equation*}
Consider the first integral in \eqref{eq:app:Error-f-integral}: using equations~\eqref{eq:app:a1j}, convert the integrand
back to $\cosh$-form,
\begin{equation}\label{eq:app:Error-derivation}
-\int_{+\infty}^{\tau}\frac{2(\varepsilon b)^{1/3}}{\tau^{2/3}}\left(a_{1,-1}w^{-1}+a_{1,1}w\right)\md\tau=
\frac{2\sqrt{\tilde{\nu}+1}\,\me^{-\frac{3\pi\mi}{4}}}{3^{3/4}(\varepsilon b)^{1/6}}
\int_{+\infty}^{\tau}\frac{\cosh\tilde\psi(\tau)}{\tau^{1/3}\big(1+\mathcal{O}(\tau^{-2/3})\big)}\,\md\tilde\psi(\tau),
\end{equation}
where $\tilde\psi(\tau)=\mi\theta(\tau)+(\tilde{\nu}+1)\ln\theta(\tau)+z$. Integrating by parts with the help of the relation
for the differential
$\md\tilde\psi(\tau)=2\mi\sqrt{3}(\varepsilon b)^{1/3}\tau^{-1/3}\big(1+\mathcal{O}(\tau^{-2/3})\big)\md\tau$, one finds the
leading term of asymptotics in equation~\eqref{eq:app:error} with the correction
$\mathcal{O}\left(\frac{\me^{\pm\mi\theta(\tau)}}{\tau^{1-\frac23|\text{Re}(\tilde{\nu}+1)|}}\right)$. The last correction,
however, is smaller than the estimate for the second integral in \eqref{eq:app:Error-f-integral}
(cf.~\eqref{eq:app:f-integral-estimate}); thus, we arrive at the result stated in \eqref{eq:app:error}.
\end{proof}
\begin{remark}
The procedure presented in the proof of Proposition~\ref{prop:app:error} can surely be extended to obtain the corrections for
$\varphi(\tau)$ to all orders. We restricted our attention to only the leading-order correction term because it is visible on
the plots of $\varphi(\tau)$.

Equation~\eqref{eq:app:integral-b/u} calculates the integral of $1/u(\tau)$ that is regularized at infinity.
One can derive asymptotics of such integrals as $\tau_0\to0$ with the help of Theorem~\ref{th:B1asympt0}.
For $a=0$, as studied in Section~\ref{sec:asymptnumerics}, the integral of $1/u(\tau)$ is convergent at the origin,
and we studied the large-$\tau$ asymptotics when the upper limit of integration is not fixed but approaches the point at infinity.
Another possibility is to regularize the integral at the origin for $a\neq0$ as in \cite{KitVar2019}.
\hfill $\blacksquare$\end{remark}

The following Theorem~\ref{th:asympt-infty2010} is a reformulation of Theorem~2.1 in \cite{KitVar2010} with
$\varepsilon_{1} \! = \! \varepsilon_{2} \! = \! 0$ and where the asymptotics of the
Hamiltonian function has been supplanted with the asymptotics of the function $\varphi(\tau)$. As per the discussion
given at the beginning of this appendix, the function $\vartheta(\tau)$ which appears in
equation~\eqref{eq:app:vartheta-singular} below is obtained by subtracting $\pi$ from equation~(2.5) in \cite{KitVar2010}.
\begin{theorem}\label{th:asympt-infty2010}
Let $(u(\tau), \varphi(\tau))$ be a solution of the system~\eqref{eq:dp3u}, \eqref{eq:app:varphi-definition} for
$\varepsilon b>0$ corresponding to the monodromy data
$(a,s_{0}^{0},s_{0}^{\infty},s_{1}^{\infty},g_{11},g_{12},g_{21},g_{22})$.
Assume that
\begin{equation} \label{eq:app:u-sing-as-conditions}
g_{11}g_{12}g_{21}g_{22} \! \neq \! 0, \, \qquad \,
|g_{11}g_{22}|\neq-g_{11}g_{22},
\end{equation}
and define
\begin{equation}\label{eq:app:tilde-nu-def-singular}
\tilde{\nu}+1=\frac{\mi}{2\pi}\ln(g_{11}g_{22}),\quad
\text{where}\quad
\Re (\tilde{\nu} \! + \! 1) \! \in \!(0,1) \setminus \lbrace 1/2 \rbrace.
\end{equation}
Then, $\exists$ $\delta_{G} \! > \! 0$
such that
\begin{align}
u(\tau)\underset{\tau\to+\infty}{=}&\,\frac{\varepsilon(\varepsilon b)^{2/3}}{2} \tau^{1/3}
\! \left(1 \! - \! \frac{3}{2\sin^2\!\left(\frac{1}{2}\vartheta(\tau)\right)}\!\right) \label{eq:app:u-sing-as}\\
\underset{\tau \to +\infty}{=}&\,\frac{\varepsilon (\varepsilon b)^{2/3}}{2}\,
\tau^{1/3} \frac{\sin \! \left(\frac{1}{2} \vartheta (\tau) \! - \! \vartheta_{0} \right) \! \sin \!
\left(\frac{1}{2} \vartheta (\tau) \! + \! \vartheta_{0} \right)}{\sin^{2} \! \left(\frac{1}{2}
\vartheta (\tau) \right)}, \label{eq:app:u-sing-as-factor}
\end{align}
where
\begin{equation}\label{eq:app:vartheta-singular}
\begin{aligned}
\vartheta (\tau) \! = &\, \theta (\tau) \! - \! \mi ((\tilde{\nu} \! + \! 1) \! - \! 1/2) \ln
\theta (\tau) \! - \! \frac{3\pi}{4} \! - \! \frac{3 \pi}{2}(\tilde{\nu} \! + \! 1) \! - \! \mi
((\tilde{\nu} \! + \! 1) \! - \! 1/2) \ln 12 \! - \! \frac{\mi}{2} \ln 2 \pi \! + \! a \ln
(2 \! + \! \sqrt{3})\\
+& \, \mi \ln \! \left(g_{11}g_{12} \Gamma (\tilde{\nu} \! + \! 1) \right) \! + \!
\mathcal{O}\big(\tau^{-\delta_{G}} \ln \tau\big),\qquad\quad
\theta(\tau)=3^{3/2}(\varepsilon b)^{1/3}\tau^{2/3},
\end{aligned}
\end{equation}
\begin{equation}\label{eq:app:theta0}
\vartheta_{0} \! := \! -\frac{\pi}{2} \! + \! \frac{\mi}{2} \ln (2 \! + \! \sqrt{3}), \,
\quad \, \sin \vartheta_{0} \! = \! -(3/2)^{1/2}, \, \quad \, \cos \vartheta_{0}
\! = \! \mi/\sqrt{2},
\end{equation}
and $\exists$ $\delta>0$ satisfying the inequality $0 <\delta<2/3$ such that
\begin{equation}\label{eq:app:infty:irregular-varphi}
\begin{gathered}
\varphi(\tau)\underset{\tau\to+\infty}{=} \,
3(\varepsilon b)^{1/3}\tau^{2/3}+2a\ln\big(\tau^{2/3}\big)-a\ln\big((\varepsilon b)^{1/3}/4\big)+\pi
-2\pi\big((\tilde{\nu}+1)-1/2\big)+\mi\ln\big(g_{11}^2\big)\\
-2\mi\big((\tilde{\nu}+1)-1/2\big)\ln(2+\sqrt{3})+\mathcal{E}_{\varphi}(\tau),\qquad
\mathcal{E}_{\varphi}(\tau)=-\mi\,\mathrm{Ln}\left(\frac{\sin (\frac{1}{2} \vartheta (\tau) \! + \!
\vartheta_{0})}{\sin (\frac{1}{2} \vartheta (\tau) \! - \! \vartheta_{0})}\right) +o\big(\tau^{-\delta}\big).
\end{gathered}
\end{equation}
\end{theorem}
\begin{remark}\label{rem:app:nu+1definition}
Even though Theorem~\ref{th:asympt-infty2010} uses the same equation for the definition of $\tilde{\nu}+1$
as in Theorem~\ref{th:asympt-infty2004} (cf. \eqref{eq:app:tilde-nu-def-singular} and \eqref{eq:app:nu+1}),
it must be emphasized that the branch of $\ln(\cdot)$ in these equations is fixed in different domains. The function
$\theta(\tau)$ in equation~\eqref{eq:app:vartheta-singular} is given by the same formula as in
Theorem~\ref{th:asympt-infty2004}, and is rewritten here for the convenience of the reader.
\hfill $\blacksquare$\end{remark}
\begin{remark}\label{rem:app:varphi-LnfracSIN-corr}
As mentioned in Remark~\ref{rem:app:varphi}, the function $\varphi(\tau)$ does not possess the Painlev\'e property, and is
defined modulo the additive constant $2\pi k$, for some integer $k$, which depends on the path of analytic continuation for the
function $\varphi(\tau)$. This ambiguity is manifested by the presence of the term $\mi\ln\big(g_{11}^2\big)$ in the
asymptotic formula~\eqref{eq:app:infty:irregular-varphi}; moreover, there is an alternative mechanism which corroborates this
ambiguity, namely, the function $\mathcal{E}_{\varphi}(\tau)$. The $2\pi k$ ambiguity in $\mathcal{E}_{\varphi}(\tau)$
reflects the quality of the approximation of $\varphi(\tau)$ by its asymptotics in the finite domain! How might this occur?
In the definition of $\mathcal{E}_{\varphi}(\tau)$, we use $\mathrm{Ln}$ to denote the continuous branch of the corresponding
logarithmic function: this means that we have to fix at some point in a neighbourhood of the positive semi-axis a value of
the logarithm, and then consider a path for the analytic continuation of this logarithm from the chosen point to the point
where we want to know the value of the logarithm.
Of course, the simplest way would be to take the same path as for the definition of the function $\varphi(\tau)$, but then,
the appearance of the additional constant $2\pi k$ depends on the respective location of the zeros of $u(\tau)$ and of
its asymptotics. Since the function $u(\tau)$ and its asymptotics have a finite number of zeros along the positive semi-axis
and the path of analytic continuation is fixed, the integer $k$ is uniquely defined.
Our asymptotics provides a very good approximation for $\varphi(\tau)$ not only for large values
of $\tau$, but also for finite values as well (see Section~\ref{sec:asymptnumerics}, Examples 3, 4, 5, and 6);
therefore, we define the continuous branch of $\mathcal{E}_{\varphi}(\tau)$ starting from small values of $\tau$.

Note that, for computing the function $\mathcal{E}_{\varphi}(\tau)$, one cannot use the standard commands in {\sc Maple} and
{\sc Mathematica} for the calculation of $\ln(\cdot)$ because they calculate the principal branch of the logarithm,
and as a result, instead of the plot of the asymptotics, one would see a saw-like line. We have discussed this
issue as it relates to the function $I(r)$ in Remark~\ref{rem:correctionFORsingularINTEGRAL} of Section~\ref{sec:asymptnumerics}.
\hfill $\blacksquare$\end{remark}
\begin{remark}\label{rem:app:thC2validity}
The notation $\tau\to+\infty$ in Theorem~\ref{th:asympt-infty2010} has the same meaning as
in Theorem~\ref{th:asympt-infty2004}, i.e., $\tau\in\mathcal{D}$ and $|\tau|\to\infty$ (cf. Remark~\ref{rem:app:varphi}).
We have excluded from consideration in Theorem~\ref{th:asympt-infty2010} monodromy data satisfying the condition
$\text{Re}(\tilde{\nu}+1)=1/2$ because, in this case, the domain $\mathcal{D}$ contains an infinite number of poles and zeros
accumulating at the point at infinity. Theorem~\ref{th:asympt-infty2010} also remains valid in this case, but in the
domain $\mathcal{D}_{u}$ defined below. It is a matter of convenience, therefore, to formulate this result separately;
we formulate this result in Theorem~\ref{th:asympt-poles-zeros} below.

Our numeric studies demonstrate that good correspondence between the numeric solution and asymptotics presented in
Theorem~\ref{th:asympt-infty2010} for finite values of $\tau$ is achieved for $\Re(\tilde{\nu}+1)$ in the intervals
$[1/6,1/2)\cup(1/2,5/6]$.
For $\Re(\tilde{\nu}+1)$ in the intervals $[0,1/6)\cup(5/6,1)$, the correspondence between the numerical solution and the
asymptotics  given in Theorem~\ref{th:asympt-infty2004} for finite values of $\tau$ is better.
\hfill $\blacksquare$\end{remark}

We now turn our attention to the case $\Re(\tilde{\nu}+1)=1/2$. In this case, the domain of validity $\mathcal{D}$
(cf. Remark~\ref{rem:app:thC2validity}) for the asymptotics presented in Theorem~\ref{th:asympt-infty2010}
contains zeros and poles of the function $u(\tau)$, and therefore requires a more delicate formulation.
The following Theorem~\ref{th:asympt-poles-zeros} is, substantially, a reproduction of Theorem~2.2 in \cite{KitVar2010},
but with $\varepsilon_{1}=\varepsilon_{2}= 0$, and with $+\pi/2$ in equation~(2.14) of \cite{KitVar2010} corrected to $-\pi/2$
(see \eqref{eq:app:vahrrho2} below); furthermore, in \eqref{eq:app:vahrrho2-arg}, the arguments of two monodromy functions
are combined into the argument of one function.
\begin{theorem} \label{th:asympt-poles-zeros}
Let $(u(\tau), \varphi(\tau))$ be a solution of the system~\eqref{eq:dp3u}, \eqref{eq:app:varphi-definition} for
$\varepsilon b>0$ corresponding to the monodromy data $(a,s_{0}^{0},s_{0}^{\infty},s_{1}^{\infty},g_{11},g_{12},g_{21},g_{22})$.
Assume that
\begin{equation*} \label{eq:app:u-sing-as-conditionsRe<nu+1>=1/2}
g_{11}g_{12}g_{21}g_{22} \! \neq \! 0, \, \qquad \,
|g_{11}g_{22}|=-g_{11}g_{22}.
\end{equation*}
Define
\begin{equation*} \label{eq:app:varrho1}
\varrho_{1} \! := \! \frac{1}{2 \pi} \ln (-g_{11}g_{22}) \quad (\in \! \mathbb{R}),
\end{equation*}
and
\begin{align}
\varrho_{2} :=&\,\varrho_{1}\ln(24\pi)-\frac{3\pi}{2}+a\ln(2+\sqrt{3})-\frac{3\pi\mi}{2}\varrho_{1}-\frac{\mi}{2}\ln(2\pi)
+\mi\ln\left(g_{11}g_{12}\Gamma\big(\tfrac12+\mi\varrho_{1}\big)\right) \label{eq:app:vahrrho2}\\
=&\,\varrho_{1}\ln(24\pi)-\frac{3\pi}{2}+a\ln(2+\sqrt{3})-\arg \left(\Gamma\big(\tfrac{1}{2}+\mi\varrho_{1}\big)
\frac{\sqrt{g_{11}g_{12}}}{\sqrt{g_{21}g_{22}}}\right)+\frac{\mi}{2}\ln\left\lvert\frac{g_{11}g_{12}}{g_{21}g_{22}}\right\rvert.
\label{eq:app:vahrrho2-arg}
\end{align}
The right-most logarithm in \eqref{eq:app:vahrrho2} is complex, and the principal branch is assumed.
The branches of the square roots in \eqref{eq:app:vahrrho2-arg} are defined such that
$\sqrt{g_{11}g_{12}}\,\sqrt{g_{21}g_{22}}>0$, and $\arg(\cdot)\in(-\pi,\pi]$ denotes the principal value of the argument
of the corresponding complex function.

Then, $\exists$ $\delta \! \in \! (0,1/39)$ such that the function $u(\tau)$ has, for
all large enough $m \! \in \! \mathbb{N}$, second-order poles, $\tau_{m}^{\infty}$,
accumulating at the point at infinity, and, the function $u(\tau)$ (resp., $\varphi(\tau))$
has, for all large enough $m \! \in \! \mathbb{N}$, a pair of first-order zeros
(resp., movable logarithmic branch points), $\tau_{m}^{\pm}$, accumulating at the
point at infinity, where
\begin{equation} \label{eq:app:tau-m-poles}
\tau_{m}^{\infty} \underset{m \to \infty}{=} \left(\frac{2 \pi m}{3^{3/2}
(\varepsilon b)^{1/3}} \right)^{3/2} \! \left(1 \! - \! \frac{3 \varrho_{1}}{4 \pi}\,
\frac{\ln m}{m} \! - \! \frac{3 \varrho_{2}}{4 \pi}\, \frac{1}{m} \right) \! + \!
\mathcal{O} \! \left(m^{(1-3 \delta)/2} \right),
\end{equation}
and
\begin{equation} \label{eq:app:tau-m-zeros}
\tau_{m}^{\pm} \underset{m \to \infty}{=} \left(\frac{2 \pi m}{3^{3/2}
(\varepsilon b)^{1/3}} \right)^{3/2} \! \left(1 \! - \! \frac{3 \varrho_{1}}{4 \pi}\,
\frac{\ln m}{m} \! - \! \frac{3}{4 \pi}(\varrho_{2} \! \pm \! 2 \vartheta_{0})\,
\frac{1}{m} \right) \! + \! \mathcal{O} \! \left(m^{(1-3 \delta)/2} \right),
\end{equation}
with $\vartheta_{0}$ defined in equation~\eqref{eq:app:theta0}.
\end{theorem}

Assume that the conditions stated in Theorem~\ref{th:asympt-poles-zeros} are valid.
Denote by $\hat{\tau}_{m}^{\infty}$ and $\hat{\tau}_{m}^{\pm}$, respectively, the leading terms of the
asymptotics \eqref{eq:app:tau-m-poles} and \eqref{eq:app:tau-m-zeros} of the poles and zeros of $u(\tau)$.

Define the multiply-connected domain $\mathcal{D}_{u} \! := \! \lbrace \mathstrut \tau \! \in \! \mathcal{D}; \, \lvert
\theta (\tau) \! - \! \theta (\hat{\tau}_{m}^{\ast}) \rvert \! \geqslant \! C \lvert\hat{\tau}_{m}^{\ast} \
\rvert^{-\delta} \rbrace$, $\ast \! \in \! \lbrace \infty,\pm \rbrace$, where the strip
domain $\mathcal{D}$ is defined in Remark~\ref{rem:app:varphi}, $\theta (\tau)$ is defined by
equation~\eqref{eq:app:infty:regular-theta-tau}, and $C \! > \! 0$ and $\delta \! \in \! (0,1/39)$ are parameters.
In Theorem~\ref{th:app:asymptRe<nu+1>=1/2} below, the notation $\tau\to+\infty$ means $\tau\in\mathcal{D}_{u}$ and
$|\tau|\to\infty$.
\begin{theorem}\label{th:app:asymptRe<nu+1>=1/2}
For the conditions of Theorem~\ref{th:asympt-poles-zeros}, there exists $\delta_{G}>0$ satisfying the inequality
$0<\delta<\delta_{G}<\tfrac{1}{15}-\tfrac{8\delta}{5}$, where $\delta\in(0,1/39)$, such that
\begin{align} \label{eq:app:u-asymptRe<nu+1>=1/2}
u(\tau)\underset{\tau \to +\infty}{=}& \, \frac{\varepsilon (\varepsilon
b)^{2/3}}{2} \tau^{1/3} \! \left(1 \! - \! \frac{3}{2\sin^{2} \! \left(\frac{1}{2} \Theta
(\tau) \right)} \! \right)\\
\underset{\tau \to +\infty}{=}& \, \frac{\varepsilon
(\varepsilon b)^{2/3}}{2} \tau^{1/3} \frac{\sin \! \left(\frac{1}{2} \Theta (\tau) \! - \! \vartheta_{0}
\right) \! \sin \! \left(\frac{1}{2} \Theta (\tau) \! + \! \vartheta_{0} \right)}{\sin^{2} \!
\left(\frac{1}{2} \Theta (\tau) \right)}, \nonumber
\end{align}
and $\exists$ $\delta_{1} \! > \! 0$ satisfying the inequality $0 \! < \! \delta_{1} \! <
\! 1/3$ such that
\begin{equation} \label{eq:app:exp-varphi-asymptRe<nu+1>=1/2}
\me^{\mi\varphi(\tau)}\underset{\tau \to +\infty}{=}
\frac{\me^{\mi\Phi(\tau)}}{g_{11}^{2}}\,
\frac{\sin(\frac{1}{2}\Theta(\tau)+\vartheta_{0})}{\sin (\frac{1}{2}\Theta(\tau)-\vartheta_{0})}\!
\left(1+\mathcal{O}\big(\tau^{-2 \delta_{1}+\delta}\big)\right),
\end{equation}
where
\begin{align}
\Theta (\tau) \! &= \,\theta (\tau)+\varrho_{1} \ln \theta (\tau)-\frac{3 \pi}{2}+\varrho_{1}\ln 12 \!
+a\ln (2+\sqrt{3})-\frac{3\pi\mi}{2}\varrho_{1}-\frac{\mi}{2}\ln(2\pi)\nonumber\\
&+\mi\ln \!\left(g_{11}g_{12}\Gamma\big(\tfrac12+\mi\varrho_{1}\big)\right)+\mathcal{O}(\tau^{-\delta_{G}}\ln \tau),
\label{eq:app:Theta}\\
\Phi(\tau) \! =& \, \frac{\theta (\tau)}{ \sqrt{3}} \! + \! \frac{4a}{3}
\ln \tau \! - \! 2\mi\pi \varrho_{1} \! +\pi+\! 2\varrho_{1} \ln\big(2+\sqrt{3}\big)-
a\ln\big((\varepsilon b)^{1/3}/4\big)+\mathcal{O}\big(\tau^{-\delta_{G}}\big). \label{eq:app:Phi}
\end{align}
\end{theorem}
\begin{remark}\label{rem:app:Th4paths}
The functions $u(\tau)$ and $\me^{\mi\varphi(\tau)}$ have the Painlev\'e property, so that they, as well as their asymptotics,
are uniquely defined in $\mathcal{D}_u$. The situation with respect to the function $\varphi(\tau)$ is more complicated because
of the infinite number of zeros in $\mathcal{D}_u$, and therefore requires further investigation.
\hfill $\blacksquare$\end{remark}
\section{Appendix: Comments on the Paper~\cite{Kit87}}\label{app:Kit87}

In this appendix, the notations introduced in the paper~\cite{Kit87} are used; these notations deviate from that employed
in the main body of this work, where the notation of ~\cite{KitVar2004,KitVar2010} is adopted. It is assumed that the reader has the
paper~\cite{Kit87} at hand, because the purpose of this appendix is to discuss not only the results presented there, but also
to correct some typographical errors and omissions; for example, in the discussion below, the function $e^{u(\tau)}$
is identical to the solution $H(r)$, with $r=\tau$, studied in this paper.

The paper~\cite{Kit87} is devoted to the asymptotic analysis of the following second-order ODE
\begin{equation}\label{eq:tzieica-sym}
\varepsilon(\tau u'(\tau))'={\me}^{u(\tau)}-{\me}^{-2u(\tau)},\quad
\varepsilon=\pm1,
\end{equation}
via isomonodromy deformations of a $3\times3$ matrix linear ODE. As a result of this paper, the leading terms of asymptotics
as $\tau\to+0$ of all solutions are obtained; moreover, asymptotics as $\tau\to+\infty$ of ``regular'' solutions to
equation~\eqref{eq:tzieica-sym} are constructed. The asymptotics are parametrised in terms of the monodromy data of an
associated $3\times3$ matrix linear ODE, which, consequently, allows one to obtain the connection formulae for asymptotics as $\tau\to+0$
and as $\tau\to+\infty$.

In \cite{Kit87} the monodromy data are defined differently, depending on whether $\varepsilon=+1$ or $\varepsilon=-1$; even
though they are denoted by the same letters, this should not, however, cause any confusion, since each group of formulae
where such data appear is indicated the corresponding value of $\varepsilon$.
These monodromy data formally define different homeomorphic manifolds; however, because of the obvious symmetry reduction
$\tau\to-\tau$ and $\varepsilon\to-\varepsilon$ of equation~\eqref{eq:tzieica-sym}, one can easily deduce the relation
between the monodromy co-ordinates of the same solution on these manifolds, and, therefore, derive connection
formulae for asymptotics as $\tau\to0$ and $\tau\to\infty$ of the corresponding solution. This relation between the manifolds
is important for the connection results, because some asymptotics are given for $\varepsilon=-1$, whilst others are presented
for $\varepsilon=+1$.

In the English translation of \cite{Kit87}, we noticed that the following line is absent:
$$
a.\quad
\varepsilon=+1.
\qquad
S_k^0=S_k^{\infty},
\quad
k=1,2,3,4,5,6;
$$
this line should be inserted directly above the last line of p. 2079.
In both the Russian and English versions of \cite{Kit87}, the symbol $\thicksim$ is used to denote the leading terms of
asymptotics. In the first equation of the list (18), in lieu of $\thicksim$ the symbol $=$ is incorrectly used.

Throughout \cite{Kit87}, the correction terms for asymptotics are not presented: they have the standard form for all
error estimations that are obtainable via the isomonodromy deformation method, namely,
for asymptotics as $\tau\to\infty$ the leading terms should be multiplied by $(1+\mathcal{O}(\tau^{-\delta}))$, whilst
for asymptotiics as $\tau\to0$ by $(1+\mathcal{O}(\tau^{\delta}))$, where $\delta>0$, which is different in
different formulae, is some small enough number. A more precise value for $\delta$, and even full asymptotic expansions,
can be obtained by other local asymptotic methods.

The expression for $\omega$ in the list of formulae enumerated as (18) in \cite{Kit87} should be
$\omega^{\pm1}=\me^{\mp\frac{2\pi\mi}3\mu}$, instead of $\omega^{\mp1}$ (as typed).

The main formulae defining the asymptotics as $\tau\to+0$ for the function $\me^{u(\tau)}$ are given by the lists of equations
enumerated as (18) and (19) in \cite{Kit87}.
The range of validity of these formulae is presented geometrically with the help of a hyperbola separating the complex
plane of the monodromy parameter $s$ into two parts (see Figure~1 in the English translation of \cite{Kit87}; in the
Russian version, the corresponding figure is not numbered). In terms of the auxiliary parameter $\mu$, the range of validity of
the asymptotics (18) can be written as $|\Re\mu|<1$. Usually, such formulae are valid for a wider range of values for
parameters like $\mu$; in this case, it could be $|\Re\mu|<3/2$, i.e., the complex $s$-plane with the negative semi-axis deleted:
the last claim remains to be verified.
It is assumed that the square root in the logarithm defining $\mu$ is positive for $s>0$.
The choice for this branch of the square root, however, is unimportant: the asymptotics (18) is invariant under the symmetry
$\mu\to-\mu$; therefore, assuming that the main branch of the logarithm is chosen so that $\ln(1)=0$ and it changes like
$\ln(x)=-\ln(1/x)$, the asymptotics (18) remains invariant under the change of the branch of the square root.

The asymptotic formula that appears next corresponds to the pole of the hyperbola ($s=3$): it is obtained from the
asymptotics (18) by taking the limit $\mu\to0$. This formula is followed by two lines (each beginning with $s=3$) describing
the complete set of the monodromy data for the solutions with this asymptotics. In fact, these formulae are obtained by
substituting $s=3$ into equations~(15): doing so, one gets
$$
s=3,\quad
g_1=\frac12+g_2+\sqrt{\frac14+2g_2},\quad
g_3=\frac12+g_2-\sqrt{\frac14+2g_2}.
$$
In the above equations one can choose either branch for the square root, so that these equations define two one-parameter
families of solutions with corresponding logarithmic asymptotics. In \cite{Kit87}, some other related formulae for the monodromy
data are given, but with an arithmetic mistake.

Now, we comment on asymptotics (19) in \cite{Kit87}. This asymptotics contains the parameter $\nu$. Similar to equations
(18), which are invariant under the transformation $\mu\to-\mu$, the asymptotics (19) is invariant under the change $\nu\to-\nu$.
Compared to the parameter $\mu$, $\nu$ is defined by a slightly different formula: the latter formula
contains the same long logarithm (with the square root) as the one for $\mu$. The branches of the logarithm and the square root
are assumed to be the same as the ones in equations (18); in particular, $\ln(-1)=\pi\mi$.
When we change the branch of the square root in this definition, the branch of the logarithm is chosen such
it changes like $\ln(x)=-\ln(1/x)+2\pi\mi$. This guarantees that the parameter $\nu$ changes to $-\nu$,
so that the asymptotics (19) does not change. In terms of the parameter $\nu$, the validity of the asymptotics (19) can be
formulated as $|\Re\nu|<1$. Similar to the situation with equation (18), the asymptotics (19) is, most probably, valid
for a wider range of the parameter $\nu$, namely, $|\Re\nu|<3$; however, a better approximation of solutions via
the asymptotics (19) is achieved when $|\Re\nu|<1$: for larger values of $|\Re\nu|$, the asymptotics (18) works better.
Note that, in terms of the parameter $\nu$, one writes $\omega^{\pm1}=-\me^{\pm\frac{\pi\mi\nu}3}$. The formula given in
\cite{Kit87} depends on the branch of the square root, and, therefore, is less accurate.

We now demonstrate how one specifies the solution studied in this paper with the help of the results of ~\cite{Kit87}.
First, we have to specify the monodromy data by using the asymptotics as $\tau\to+0$. For this purpose, one can use
any of the formulae (18), (19), or (20);
begin with, say, the formula (18). Obviously, the only opportunity to get a finite non-vanishing value for $\me^{u(0)}$ is
to choose $\mu=1$ (or, surely, $\mu=-1$). Substituting $\mu=1$ into the equation for $\mu$, one finds that
$(s-1)/2=\cos(2\pi/3)$, or $s=0$. The parameter $\omega=\me^{-\tfrac{2\pi\mi}3}$; therefore, employing formulae given in (18),
we obtain
\begin{equation}\label{eq:r1}
r_1=\me^{-\tfrac{2\pi\mi}3}g_1+\me^{\tfrac{2\pi\mi}3}g_2+g_3,\qquad
\frac{c_2}{c_0}=\frac12.
\end{equation}
Dividing both sides of the asymptotic formula in (18) by $\tau$ and taking the limit $\tau\to0$, we obtain
$\me^{u(0)}=H(0)=r_1$.
Combined with equations (15), which define the monodromy data, we arrive at three equations that completely define the
monodromy data in terms of $H(0)$:
\begin{equation}\label{sys:mondata3x3}
\begin{aligned}
g_1&=\frac{\me^{\frac{2\pi\mi}3}}{3H(0)}\left(H(0)-\me^{\frac{2\pi\mi}3}\right)(H(0)-1),\quad
g_2=\frac{\me^{-\frac{2\pi\mi}3}}{3H(0)}\left(H(0)-\me^{-\frac{2\pi\mi}3}\right)(H(0)-1),\\
g_3&=\frac{1}{3H(0)}\left(H(0)-\me^{\frac{2\pi\mi}3}\right)\left(H(0)-\me^{-\frac{2\pi\mi}3}\right)
=\frac13\left(H(0)+1+\frac{1}{H(0)}\right),\quad
s=0.
\end{aligned}
\end{equation}
Before considering asymptotics as $\tau\to\infty$, we check how the same result can be obtained via the asymptotics (19).
Since the branch of the long logarithm in (18) and (19) is the same, we get that $\nu=\mu=1$, in which case
\begin{equation}\label{eq:r}
r=g_3-\me^{\tfrac{\pi\mi}3}g_1-\me^{-\tfrac{\pi\mi}3}g_2,\qquad
\frac{c_{-}}{c_{+}}=-1.
\end{equation}
Multiplying both sides of (19) by $\tau^{1/2}$ and taking the limit $\tau\to+0$, one obtains $\me^{u(0)}=H(0)=r$.
Since, according to equations~\eqref{eq:r1} and ~\eqref{eq:r}, $r=r_1$, the asymptotics (19) provides us with the
same monodromy data for the given initial value of $H(0)$.
Now, we turn our attention to equation (20). The case of interest to us corresponds to $\varphi=0$; then, the parameter $p=1$,
and equation (20) reads
$$
H(\tau)=\me^{u(\tau)}\underset{\tau\to0}{\thicksim} r\left(1+(r-1/r^2)\tau\right),
$$
where $r=H(0)$, as established above (cf. equation~\eqref{eq:Hat0short}).

Now, we are ready to consider asymptotics as $\tau\to\infty$. For $\varepsilon=+1$, there is only one regular
asymptotics decaying exponentially to $1$. This asymptotic is proportional to $s$. Since $s=0$, it follows that this
asymptotics corresponds to the exact solution $H(\tau)=1$. The other monodromy parameters corresponding to
this solution are $g_1=g_2=0$ and $g_3=1$. Comparing them with \eqref{sys:mondata3x3}, we find that, actually, $H(0)=1$.

The other asymptotic results as $\tau\to+\infty$ concern the case of equation~\eqref{eq:tzieica-sym} with $\varepsilon=-1$.
If $u(\tau)$ is a solution of \eqref{eq:tzieica-sym} with $\varepsilon=+1$, then $u(-\tau)$ is a solution for $\varepsilon=-1$.
It is straightforward to deduce the corresponding mapping between the monodromy manifolds of equation~\eqref{eq:tzieica-sym}
with $\varepsilon=\pm1$, and, thus, to obtain regular asymptotics as $\tau\to-\infty$ for given asymptotics as $\tau\to+0$
of the solution $u(\tau)$ in case $u(\tau)$ has such regular asymptotics on the negative semi-axis. Note that the general
solution of equation~\eqref{eq:tzieica-sym} is not single-valued;
therefore, mappings between the monodromy manifolds corresponding  to the solutions $u(\me^{\pi\mi}\tau)$ and
$u(\me^{-\pi\mi}\tau)$ are different. Our case, $s=0$, is especially simple, because the corresponding solution is single-valued,
$u(\me^{\pi\mi}\tau)=u(\me^{-\pi\mi}\tau)$, and the corresponding mapping between the monodromy manifolds is just the identity
transformation. Therefore, we can use the results for asymptotics as $\tau\to+\infty$ for the case $\varepsilon=-1$
by merely changing $\tau\to-\tau$ in them and assuming that $\tau\to-\infty$, with the monodromy parameters in these formulae
coinciding with those in \eqref{sys:mondata3x3}: the formulae in ~\cite{Kit87} are not enumerated; rather, they are given at
the bottom of p. 2080 and at the top of p. 2081 in the English translation (resp., the bottom of p. 49 and at
the beginning of p. 50 in the Russian version). For the case $H(0)=1$, the first asymptotic formula at the bottom of
p. 2080 (resp., p. 49) is applicable, and gives $H(\tau)=1$, but now on the negative semi-axis. In a more general situation,
we have to use the second asymptotic formula on the bottom of p. 2080 (which continues to the top of the next page); in our
case ($s=0$), it reads:
\begin{equation}\label{eq:infinity-asymptotics}
\me^{u(\tau)}-1\underset{\tau\to-\infty}{=}-\frac{a\sqrt{6}}{(-3\tau)^{1/4}}\cos\left(2\sqrt{-3\tau}+a^2\ln\sqrt{-3\tau}
+\varphi-\frac{\pi}4\right)\left(1+\mathcal{O}\left(\tau^{-\delta}\right)\right),
\end{equation}
where
$$
a=\sqrt{\frac{|\ln g_3|}{2\pi}}\exp\left(\frac{\mi}2\arg\ln g_3\right).
$$
and
$$
\varphi=a^2\ln 24-\frac{\mi}{2}\ln\left|\frac{\Gamma(-\mi a^2)}{\Gamma(\mi a^2)}\frac{g_2}{g_1}\right|+
\frac12\arg\frac{\Gamma(-\mi a^2)}{\Gamma(\mi a^2)}+\frac12\arg\frac{g_2}{g_1}.
$$
Note that, in the paper~\cite{Kit87}, the formula for $a$ contains a conspicuous misprint in
$\exp\left(\tfrac{1}2\arg\ln g_3\right)$, i.e., the factor $\tfrac12$ must be changed to $\tfrac{\mi}2$.

The solution has regular asymptotics given by \eqref{eq:infinity-asymptotics} provided that
$g_1\neq0$, $g_2\neq0$, and $|\arg g_3|<\pi/2$. Note that if $g_1g_2\neq0$, then $g_3\neq0$.
The argument, $\arg\ln g_3$, is calculated via the principal branch of $\ln$ in the standard range
$(-\pi,\pi]$.

We restrict our consideration, hereafter, to the case $H(0)\in\mathbb{R}\setminus\{0\}$. The qualitative behaviour of the
corresponding solutions are studied in Proposition 5.3.2. of \cite{BobEitLMN2000}.
Briefly, solutions $H(r)$ with $H(0)>0$ are positive and bounded on the negative semi-axis. On the positive semi-axis,
these solutions monotonically approach a pole. Each solution for $H(0)<0$ increases monotonically from some pole on the
negative semi-axis to a zero on the positive semi-axis.

If $H(0)<0$, then the equation for $g_3$ in ~\eqref{sys:mondata3x3} implies $g_3\leqslant-1/3$, so that $\arg g_3=\pi$,
and the asymptotics~\eqref{eq:infinity-asymptotics} is not applicable. For $H(0)>0$, one proves that $g_3>1$, so that
$\arg g_3=0$, and  the asymptotics~\eqref{eq:infinity-asymptotics} is applicable. In this case, $a>0$ and
$\overline{g_2}=g_1\Rightarrow|g_2|=|g_1|$. This allows one to simplify the equations for $a$ and $\varphi$:
\begin{equation}\label{eq:a-varphi-positive-H0}
a=\sqrt{\frac{\ln g_3}{2\pi}},\qquad
\varphi=a^2\ln 24+\arg{\Gamma(-\mi a^2)}+\arg{g_2}.
\end{equation}
Below, we present plots for the numerical solutions together with the plots of their large-$\tau$ asymptotics calculated via
equations \eqref{eq:infinity-asymptotics} and \eqref{eq:a-varphi-positive-H0}. In the figures below, as in the main body of the
text, the notation $\tau=r$ is adopted (this is not related to, nor should it be confused with, the variable $r$ used in
\cite{Kit87}).
Observing these figures, one notes good qualitative correspondence between the numeric and asymptotic behaviours. This
correspondence starts from very small values of $\tau$. In case $H(0)>0$ is very large, the solution decreases ``very rapidly''
down to the real axis, but does not cross it: for such large values of $H(0)$, the first minimum of asymptotics is, certainly,
achieved below the real axis, because the large-$\tau$ asymptotics contains the factor $\tau^{1/4}$ in the denominator.
The asymptotics, though, continue to follow the behaviour of the solution even for $r\ll 1$ and for very large $H(0)$;
for example, for $H(0)=15$, the first minimum of asymptotics, denoted by $H_{as}(r)$, occurs at $r=r_m=-0.8181\ldots$,
with $H_{as}(r_m)=-0.6219\ldots$, and the first minimum of the solution is $H(-0.887801\ldots)=0.439959\ldots$.
For $H(0)=100$, the minimum of asymptotics is $H_{as}(-0.393734\ldots)=-0.72913780\ldots$, and the first minimum of
the solution is $H(-0.4622134\ldots)=0.289185\ldots$. One notes that the approximation for both values of $H(0)$ is of the
same order of accuracy: for larger values of $r$, the large-$\tau$ asymptotics approximates the solution
more precisely, as expected. The closer $H(0)$ is to $1$, the better, in the numerical sense, works the asymptotic
formula~\eqref{eq:infinity-asymptotics}, i.e., it more accurately (numerically) approximates the solution for smaller
values of $\tau$. These conclusions are illustrated in Figs.~\ref{fig:H0=5over7}--\ref{fig:H0=100}.

\begin{figure}[tbhp]
\begin{center}
\includegraphics[height=50mm,width=100mm]{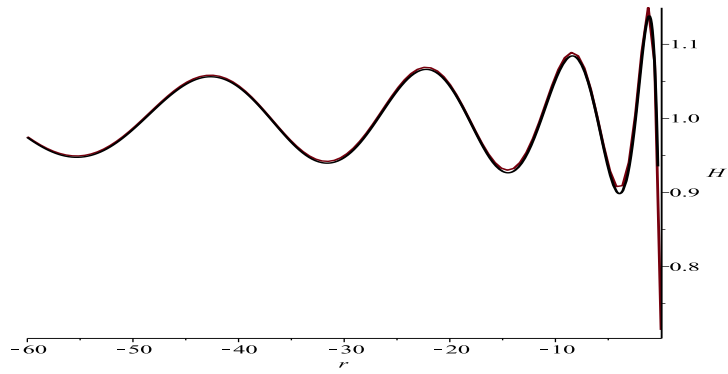}
\caption{The red and black plots are, respectively, the numeric and large-$r$ asymptotic solutions
($\me^{u(\tau)}=H(r)$, $\tau=r$)
of equation~\eqref{eq:tzieica-sym} corresponding to the initial datum $\me^{u(0)}=H(0)=5/7$.}
\label{fig:H0=5over7}
\end{center}
\end{figure}
\begin{figure}[tbhp]
\begin{center}
\includegraphics[height=50mm,width=100mm]{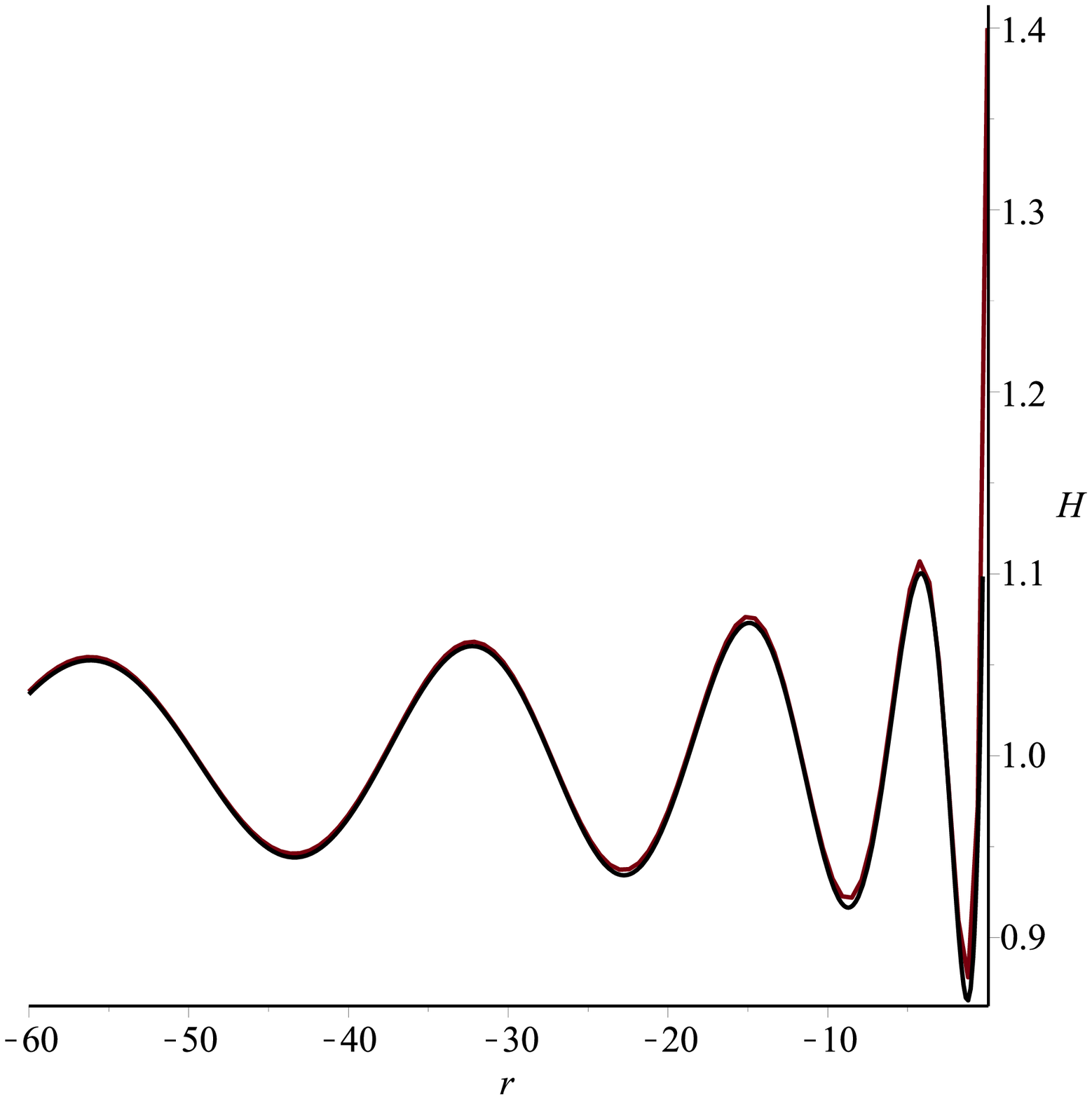}
\caption{The red and black plots are, respectively, the numeric and large-$r$ asymptotic solutions
($\me^{u(\tau)}=H(r)$, $\tau=r$) of equation~\eqref{eq:tzieica-sym} corresponding to the initial datum $\me^{u(0)}=H(0)=7/5$.}
\label{fig:H0=7over5}
\end{center}
\end{figure}
\begin{figure}[tbhp]
\begin{center}
\includegraphics[height=50mm,width=100mm]{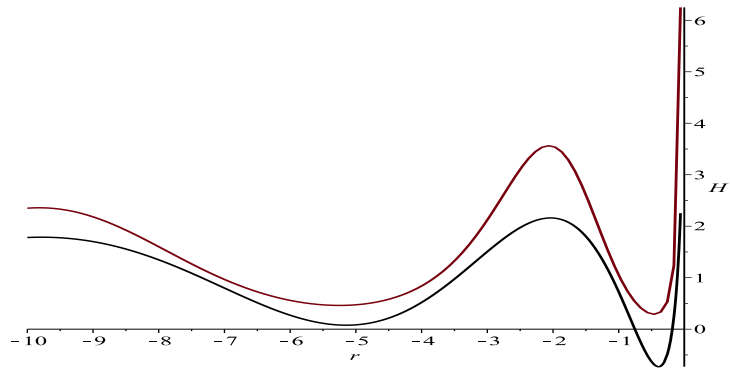}
\caption{The red and black plots are, respectively, the numeric and large-$r$ asymptotic solutions
($\me^{u(\tau)}=H(r)$, $\tau=r$) of equation~\eqref{eq:tzieica-sym} corresponding to the initial datum $\me^{u(0)}=H(0)=100$.}
\label{fig:H0=100}
\end{center}
\end{figure}

\begin{remark}
In this appendix, we used the manifold defined in \cite{Kit87}. As mentioned above, two equivalent monodromy manifolds
were introduced in \cite{Kit87}. For the case $s=0$, this equivalency is established via the identity mapping,
because the monodromy data of these manifolds are denoted by the same letters, namely, $g_1$, $g_2$, and $g_3$, which
are used in this appendix. On the other hand, in Section~\ref{sec:mondata}, we use the monodromy manifold defined
in \cite{KitVar2004}. Since both monodromy manifolds describe the same set of the solutions of equation~\eqref{eq:dP3y}
for $a=0$, they should be equivalent.
Strictly speaking, though, the manifold considered in Section~\ref{sec:mondata} contains one more
parameter which allows one to find, additionally, asymptotics of the function $\varphi(\tau)$ considered in
Appendices~\ref{app:asympt0} and \ref{app:infty} (see, also, the integral $I(r)$ in Section~\ref{sec:asymptnumerics});
however, the quadratic contraction of the manifold in Section~\ref{sec:mondata} enumerates only solutions of
equation~\eqref{eq:dP3y} with $a=0$, so that this contraction should be equivalent to the manifold employed in this appendix.

We do not consider the mappings between the manifolds discussed above in the general setting; however, we present
the explicit mapping between the manifolds for the case $s=0$, which follows from the comparison of the results of this
appendix and Lemma~\ref{lem:mondataH0}:
\begin{gather*}
s=\tilde{s}=0
\qquad
\Leftrightarrow
\qquad
s_0^0=i,\quad
s_0^\infty s_1^\infty=-1,\\
g_1=\tilde{g}_1,\quad
g_2=-\tilde{g}_2,\quad
g_3=\tilde{g}_3,
\end{gather*}
which coincides, for $s=\tilde{s}=0$, with the transformation~\eqref{eqs:contraction-symmetries-g2=-g2}.
\hfill$\blacksquare$\end{remark}
\vspace*{0.35cm}
\noindent
\textbf{\Large Acknowledgements}
\vspace*{0.25cm}

\noindent
The authors are grateful to R. Conte for discussions related to the Painlev\'e Property.
A.~V. is very grateful to the St.~Petersburg Branch of the Steklov Mathematical Institute (POMI) for
hospitality during the Summer of 2019, when this work began.

\end{document}